\documentclass[reqno,12pt]{amsart}

\usepackage{a4}
\usepackage[latin1]{inputenc}
\usepackage[T1]{fontenc}
\usepackage{amsmath}
\usepackage{amssymb}
\usepackage{vmargin}
\usepackage[all]{xy}
\xyoption{curve} \xyoption{v2} \xyoption{2cell} \UseAllTwocells
\SelectTips{eu}{11}

\usepackage{mathrsfs}
\usepackage[mathcal]{euscript}
\usepackage{dsfont}
\usepackage{verbatim}

\newcounter{proposition}[section]
\newcounter{definition}[section]
\newcounter{theorem}[section]
\newcounter{lemma}[section]
\newcounter{corollary}[section]
\newcounter{remark}[section]
\newcounter{conjecture}[section]
\newcounter{example}[section]
\newcounter{hypothesis}[section]

\makeatletter

\def\theproposition{\thesection.\@arabic\c@proposition}
\def\thedefinition{\thesection.\@arabic\c@definition}
\def\thetheorem{\thesection.\@arabic\c@theorem}
\def\thelemma{\thesection.\@arabic\c@lemma}
\def\thecorollary{\thesection.\@arabic\c@corollary}
\def\theremark{\thesection.\@arabic\c@remark}
\def\theconjecture{\thesection.\@arabic\c@conjecture}
\def\theexample{\thesection.\@arabic\c@example}
\def\thehypothesis{\thesection.\@arabic\c@hypothesis}

\makeatother

\newenvironment{enonce}[1]{
\smallskip
\noindent{\textbf{\scshape{}
#1}}---\,\,\begin{itshape}}{\end{itshape}
\smallskip}

\newenvironment{enoncenonitalique}[1]{
\noindent{\textbf{{ #1}}\scshape{---}} }

\newenvironment{proposition}[1][]{\refstepcounter{proposition}\refstepcounter{definition}\refstepcounter{theorem}\refstepcounter{lemma}\refstepcounter{corollary}\refstepcounter{remark}\refstepcounter{conjecture} \refstepcounter{example}\refstepcounter{hypothesis}
\begin{enonce}{Proposition \theproposition{} #1 }}{\end{enonce}}

\newenvironment{definition}[1][]{\refstepcounter{proposition}\refstepcounter{definition}\refstepcounter{theorem}\refstepcounter{lemma}\refstepcounter{corollary}\refstepcounter{remark}\refstepcounter{conjecture}\refstepcounter{example}\refstepcounter{hypothesis}
\begin{enonce}{Definition \thedefinition{} #1 }}{\end{enonce}}

\newenvironment{theorem}[1][]{\refstepcounter{proposition}\refstepcounter{definition}\refstepcounter{theorem}\refstepcounter{lemma}\refstepcounter{corollary}\refstepcounter{remark}\refstepcounter{conjecture}\refstepcounter{example}\refstepcounter{hypothesis}
\begin{enonce}{Theorem \thetheorem{} #1 }}{\end{enonce}}

\newenvironment{lemma}[1][]{\refstepcounter{proposition}\refstepcounter{definition}\refstepcounter{theorem}\refstepcounter{lemma}\refstepcounter{corollary}\refstepcounter{remark}\refstepcounter{conjecture}\refstepcounter{example}\refstepcounter{hypothesis}
\begin{enonce}{Lemma \thelemma{} #1 }}{\end{enonce}}

\newenvironment{corollary}[1][]{\refstepcounter{proposition}\refstepcounter{definition}\refstepcounter{theorem}\refstepcounter{lemma}\refstepcounter{corollary}\refstepcounter{remark}\refstepcounter{conjecture}\refstepcounter{example}\refstepcounter{hypothesis}
\begin{enonce}{Corollary \thecorollary{} #1 }}{\end{enonce}}

\newenvironment{remark}[1][]{\refstepcounter{proposition}\refstepcounter{definition}\refstepcounter{theorem}\refstepcounter{lemma}\refstepcounter{corollary}\refstepcounter{remark}\refstepcounter{conjecture}\refstepcounter{example}\refstepcounter{hypothesis}
\smallskip
\begin{enoncenonitalique}{Remark \theremark{} #1 }}{\end{enoncenonitalique} \smallskip}

\newenvironment{example}[1][]{\refstepcounter{proposition}\refstepcounter{definition}\refstepcounter{theorem}\refstepcounter{lemma}\refstepcounter{corollary}\refstepcounter{remark}\refstepcounter{conjecture}\refstepcounter{example}
\refstepcounter{hypothesis}
\begin{enoncenonitalique}{Example \theexample{} #1 }}{\end{enoncenonitalique}}

\newcommand{\Spec}{{\rm Spec}}

\newcommand{\Z}{\mathbb{Z}}
\newcommand{\Aff}{\mathbb{A}}

\newcommand{\Proj}{\mathbb{P}}
\newcommand{\Sch}{{\rm Sch}}
\newcommand{\Smooth}{{\rm Sm}}

\newcommand{\DM}{\mathbf{DA}}
\newcommand{\Dia}{\mathsf{Dia}}

\newcommand{\et}{\textit{\emph{\'et}}}

\newcommand{\id}{{\rm id}}

\newcommand{\ctp}{\rm cla}

\newcommand{\un}{\mathds{1}}

\newcommand{\Q}{\mathbb{Q}}

\newcommand{\N}{\mathbb{N}}

\newcommand{\R}{\mathbb{R}}

\newcommand{\C}{\mathbb{C}}

\newcommand{\M}{{\rm M}}
\newcommand{\Aut}{{\rm Aut}}
\newcommand{\Sm}{{\rm Sm}}
\newcommand{\SmAn}{{\rm SmAn}}
\newcommand{\Shv}{\mathbf{Shv}}
\newcommand{\Spect}{\mathbf{Spect}}

\newcommand{\eff}{{\rm eff}}
\newcommand{\coh}{{\rm coh}}

\newcommand{\An}{{\rm An}}

\newcommand{\EE}{\mathbb{E}}

\newcommand{\ucarre}{\, \xy <0cm,-.08cm> ;<.22cm,-.08cm>:
(0,0) ; (0,1) **\dir{-}, (1,1) ; (0,1) **\dir{-}, \endxy \, }

\newcommand{\smallucarre}{\, \xy <0cm,-.05cm> ;<.12cm,-.05cm>:
(0,0) ; (0,1) **\dir{-}, (1,1) ; (0,1) **\dir{-}, \endxy \, }

\begin{document}

\title[Artin motives and the reductive Borel-Serre compactification]{Relative
Artin motives and the reductive Borel-Serre compactification
of a locally symmetric variety}

\author[J. Ayoub]{Joseph Ayoub}
\address{\newline
Institut f\"ur Mathematik\\ Universit\"at
Z\"urich\\Winterthurerstr. 190\\
CH-8057 Zurich\\
Switzerland  \newline  CNRS\\ LAGA Universit\'e Paris 13\\
99 avenue J.B.~Cl\'ement \\ 93430 Villetaneuse\\ France}
\email{joseph.ayoub@math.uzh.ch}
\author[S. Zucker]{Steven Zucker}
\address{\newline
Department of Mathematics\\ Johns Hopkins University\\ Baltimore\\ MD 21218\\
USA}
\email{zucker@math.jhu.edu}
\thanks{The first author was
supported in part by the Swiss National Science
Foundation through grant no.~ 2000201-124737/1.
The second author was
supported in part by the National Science Foundation through grant
DMS0600803.}

\hfill

\begin{abstract} \begin{sloppypar}
We introduce the notion of Artin motives and cohomological motives over a scheme $X$.
Given a cohomological motive $M$ over $X$, we construct its
punctual weight zero part $\omega^0_X(M)$
as the universal Artin motive mapping to $M$.
We use this to define a motive $\EE_X$ over $X$ which is an invariant of the singularities of $X$.
The first half of the paper is devoted to the study of the functors $\omega^0_X$ and the computation of the motives $\EE_X$.

In the second half of the paper, we develop the application to locally symmetric varieties.
More specifically, let $\Gamma \backslash D$ be a locally symmetric variety and denote by
$p:\overline{\Gamma\backslash D}^{rbs} \to \overline{\Gamma\backslash D}^{bb}$ the projection
of its reductive Borel-Serre compactification to its Baily-Borel Satake compactification.
We show that ${\rm R}p_*\Q_{\overline{\Gamma\backslash D}^{rbs}}$ is naturally isomorphic
to the Betti realization of the motive $\EE_{\overline{X}^{bb}}$, where $\overline{X}^{bb}$ is the
scheme such that $\overline{X}^{bb}(\C)=\overline{\Gamma \backslash D}^{bb}$. In particular,
the direct image of $\EE_{\overline{X}^{bb}}$ along the projection of $\overline{X}^{bb}$ to $\Spec(\C)$ gives a motive whose Betti
realization is naturally isomorphic to the cohomology of
$\overline{ \Gamma\backslash
D}^{rbs}$.

\end{sloppypar}
\end{abstract}

\maketitle

\bigskip

\setcounter{tocdepth}{3}

\begin{tableofcontents}

\end{tableofcontents}

\bigskip

\subsection*{Introduction}

Let $X$ be a noetherian scheme. By the work of F.~Morel and V.~Voevodsky
\cite{morel-voevodsky}, R.~Jardine \cite{jardine}, and
others, one can associate to $X$ a triangulated category $\DM(X)$,
whose objects are called motives over $X$.  Any quasi-projective
$X$-scheme $Y$ has a cohomological motive $\M_{\coh}(Y)$, an
object of $\DM(X)$. Many of the expected properties of these
categories are still unknown, notably the existence of a motivic
$t$-structure, usual and perverse, and a filtration by weights and punctual weights on their respective hearts.\footnote{In \cite{weight-filtration} M.~Bondarko defined a weight
filtration for effective motives over a field. However, this weight
filtration is sensitive to the suspension functor $[1]$; the latter
increase the weight by $1$. In particular, his truncation functors are not
triangulated. In this paper, we have in mind weight filtrations compatible
with the triangulated structure, more in the spirit of S. Morel's weight
truncation \cite[\S 3]{sophie-morel}. However, contrary to
\cite{sophie-morel}, we use punctual weights instead of weights on
perverse sheaves, for the latter are very far from being available in the
motivic setting.}

By definition, a general \emph{cohomological motive} is an object
of $\DM(X)$ which can be obtained from the motives $\M_{\coh}(Y)$
by an iteration of the following operations: direct sums,
suspensions and cones. Similarly, one defines \emph{Artin motives}
by taking only the motives $\M_{\coh}(Y)$ with $Y$ finite
over $X$. We propose a candidate for the punctual weight zero part of a cohomological motive $M$. It is the universal
Artin motive $\omega^0_X(M)$ that maps to $M$.
That $\omega^0_X(M)$ exists is a
consequence of general existence theorems for compactly generated
triangulated categories. What is less formal is that the functor
$\omega^0_X$ satisfies the properties one expects from a
punctual weight
truncation functor. The preceding is the subject of \S
\ref{sub-sect:zero-slice-of-the-weight-filt}.

Next, in \S2.3, we use the functors $\omega^0_X$ to define a
motive $\EE_X$ over $X$ as follows. Assume that $X$ is
reduced and
quasi-projective over a field $k$ of characteristic zero, and let
$j:U \hookrightarrow X$ be the inclusion of a dense and smooth
open subset. Then $\EE_X$ is defined to be $\omega^0_X(j_*\un_U)$,
where $\un_U$ is the unit of the tensor product on $\DM(U)$, which
is independent of the choice of $U$. Moreover, $\EE_X$ is an
invariant of the singularities of $X$. Indeed, if $X$ is smooth
$\EE_X\simeq \un_X$. Moreover, given a smooth morphism $f:Y \to
X$, there is a canonical isomorphism $f^*\EE_X\simeq \EE_Y$. The
large \S \ref{subsection:compute-E-X} is devoted to the
computation of the motive $\EE_X$ in terms of a stratification of
$X$ by smooth locally closed subsets and a compatible family of
resolutions of the closure of each stratum.
To compute $\EE_X$ from the aforementioned resolution
data, we introduce a diagram of schemes $\mathcal{X}$
in \S2.5.2 that breaks
down the determination into a ``corner-like'' decomposition of the
boundary in the resolutions.  We further break it down, by means
of the diagram $\mathcal{Y}$ in \S2.5.4, to the strata in the
objects of $\mathcal{X}$.
Unfortunately, the
outcome is not very elegant, but it is useful nonetheless.

Section \ref{sect:compactif-locally-symmetric} treats the relevant
compactifications of a locally symmetric variety, and gathers
their essential properties.  Let $D$ be a bounded symmetric
(complex) domain and $\Gamma\subset {\rm Aut}(D)$ an
arithmetically-defined subgroup. Then $\Gamma\backslash D$ is a
complex analytic space with (at worst) quotient
singularities.\footnote{Under mild conditions on $\Gamma$ (see \S
\ref{subsect:generalities}), $\Gamma\backslash D$ is non-singular
and its boundary strata in each compactification are likewise
well-behaved.}  In fact, it has a canonical structure of an
algebraic variety \cite{baily-borel}.  Though some of its
well-known compactifications are projective varieties, e.g., the
Baily-Borel Satake compactification $\overline{\Gamma\backslash
D}^{bb}$, that is not the case for the rather prominent reductive
Borel-Serre compactification $\overline{\Gamma\backslash D}^{rbs}$
(see \cite{rbs-1}), which was introduced as a technical device
without name in \cite[\S4]{warped}. It is only a real stratified
space whose boundary strata can have odd real dimension.

In Section \ref{sect:main-comput}, we state and prove the main
theorem of this article, which concerns the reductive Borel-Serre
compactification. By \cite{satake-compactifications}, there is a
natural stratified projection $p:\overline{\Gamma\backslash
D}^{rbs} \to \overline{\Gamma\backslash D}^{bb}$ from the
reductive Borel-Serre compactification to the Baily-Borel Satake
compactification. The latter is the variety of $\C$-points
(strictly stated, the associated analytic variety) of a projective
scheme $\overline{X}^{bb}$, by \cite{baily-borel} again. Our
theorem asserts that the Betti realization of
$\EE_{\overline{X}^{bb}}$ is canonically isomorphic to ${\rm
R}p_*\Q_{\overline{\Gamma\backslash D}^{rbs}}$. Our main theorem
signals that the non-algebraic reductive Borel-Serre
compactification is a natural object in our algebro-geometric
setting; in a sense, this justifies the repeated presence of
$\overline{\Gamma\backslash D}^{rbs}$ in the literature \cite
{weighted,GT,rbs-1, rbs-2,rbs-mixed-hodge-str}. It is natural to
define the motive of the reductive Borel-Serre compactification
$\overline{\Gamma\backslash D}^{rbs}$ to be
$\M^{rbs}(\Gamma\backslash D)=\pi_*(\EE_{\overline{X}^{bb}})$ with
$\pi$ the projection of $\overline{X}^{bb}$ to $\Spec(\C)$. Then
the Betti realization of $\M^{rbs}(\Gamma\backslash D)$ is
canonically isomorphic to the cohomology of the topological space
$\overline{\Gamma \backslash D}^{rbs}$. We add that a construction
of a mixed Hodge structure on the cohomology of
$\overline{\Gamma\backslash D}^{rbs}$ is given in
\cite{rbs-mixed-hodge-str},\footnote{Indeed, it was the {\it
raison d'\^etre} of our collaboration. We believe it was Kazuya
Kato who first suggested, on the basis of
\cite{rbs-mixed-hodge-str}, that there might be a reductive
Borel-Serre motive.} though it has flaws that appear to be
fixable. Though it is natural to expect the latter to coincide
with the mixed Hodge structure one gets from the motive
$\M^{rbs}(\Gamma \backslash D)$, we do not attempt to address it
in this article. (See also Remark \ref{huber}.)

An important technique in the proof of our main result, Theorem
\ref{thm:main-thm}, is the use of diagrams of schemes (already
mentioned above) and motives over them.  A diagram of schemes is
simply a covariant functor $\mathcal{X}$ from a small category
$\mathcal{I}$ (the \emph{indexing category}) to the category of
schemes. Roughly speaking, a motive $M$ over the diagram of
schemes $\mathcal{X}$ is a collection of motives $M(i)\in
\DM(\mathcal{X}(i))$, one for each object $i\in \mathcal{I}$,
which are strictly contravariant (i.e., and not only up to
homotopy) with respect to the arrows of $\mathcal{I}$. Diagrams of
schemes and motives over them are used extensively in Sections
\ref{sect:artin-zero-slice-weight} and \ref{sect:main-comput} to
encode the way some motives are functorially reconstructed from
simpler pieces.

Here is a simple illustration of this principle.
Let $X$ be a
scheme and $M$ a motive over $X$. We assume that $M$ is defined as a homotopy
pull-back of
a diagram of the form
$$\xymatrix@C=1.5pc{M_{(1,0)} \ar[r]^-{u_{10}} &  M_{(0,0)} & \ar[l]_-{u_{01}}
M_{(0,1)},}$$
i.e., as $Cone\{u_{10}-u_{01}:M_{(1,0)}\bigoplus M_{(0,1)} \to M_{(0,0)}\}$.
Then $M$
depends only loosely (i.e., not functorially) on the above diagram.
However, in
good situations,
the above diagram can be promoted naturally to
an object of $\DM(X,\ucarre)$,
where $\ucarre$ (cf.~Lemma \ref{texnical-commut-holim})
is the category $\{(1,0)\leftarrow (0,0) \rightarrow (0,1)\}$ and
$(X,\ucarre)$
is the constant
diagram of schemes with value $X$. As homotopy
pullback is a well-defined
functor from
$\DM(X,\ucarre)$ to $\DM(X)$, it is, for technical reasons, much better
to work with objects
of $\DM(X,\ucarre)$ rather than diagrams of motives in $\DM(X)$ having the
shape of $\ucarre^{\rm op}$.

The construction of the isomorphism in our main theorem uses, as a
starting point, the computation of $\EE_X$ in
\S\ref{subsection:compute-E-X} (especially Theorem
\ref{thm:final-form-for-applic-main-thm}).
In the case of $\overline{X}^{bb}$ (playing the
role of $X$ in \S2.5), we use the toroidal compactifications of
\cite{toroidal-comp,pink-thesis} for the compatible family of
resolutions, which are determined by compatible sets of
combinatorial data.  From there, we use the specifics of the
situation to successively modify the diagram of schemes that
appears in Theorem \ref{thm:final-form-for-applic-main-thm},
without changing the (cohomological) direct image of the diagram
of motives along the projection onto $\overline{X}^{bb}$. (See Proposition
\ref{prop:comparison-cal-Y-and-cal-V} and Theorem
\ref{thm:final-form-comput-EE-X-bb}.) When we
finally arrive at the diagram $\mathcal W^{tor}$ in \S4.2.4, we
can escape the confines of schemes and pass to diagrams of
topological spaces in \S4.2.5, where the role of the reductive
Borel-Serre compactification emerges naturally.

\subsubsection*{Acknowledgments} We are indebted to Marc Levine, who suggested
that the authors meet. We
wish to thank J\"org Wildeshaus for his interest in this work. We are grateful
to Ching-Li Chai for answering questions about toroidal compactifications and
Shimura varieties.
We are particularly appreciative of the helpful, thorough refereeing that our submitted article received.

{\footnotesize
\subsubsection*{Notation and conventions}
There are places in the article where we have used
somewhat different notation from what appears in the literature.
For instance, $\DM(\mathcal{X},\mathcal{I})$ is really the triangulated category
$\mathbb{SH}^T_{\mathfrak{M}}(\mathcal{X},\mathcal{I})$ of
\cite[D\'ef.~4.5.21]{ayoub-these-II}, with $\mathfrak{M}$ the category of complexes of $\Q$-vector spaces,
$\tau=\et$, the \'etale topology, and
$T$ the Tate motive
as in \S\ref{Quick}.
We also note that in
\S\ref{subsect:The-Baily-Borel-compactification} and the sequel,
we have deviated from the notation of \cite{toroidal-comp}.
Starting in \S\ref{subsect:toroidal-compact}, the usage of the
symbols $\Sigma^{\circ}$ and $\Sigma^c$ is the opposite of that in
\cite{HZ2,rbs-2}. (We do this to conform with the relation between
the corresponding open and closed schemes.)

The category with one object and one arrow
is denoted $\mathbf{e}$.
For a scheme $X$ over $\C$, we often identify
$X(\C)$ with the associated complex analytic space. We use bold
capital letters for a linear algebraic group defined over $\Q$, e.g., $\mathbf{G}$, and
use the same letter in ordinary mathematical font, $G$ in the
example, to denote $\mathbf{G}(\R)$, viewed as a real Lie group,
beginning in Section 3. In talking
about cone complexes in \S\ref{subsect:toroidal-compact}, the
notation for a cone refers to the {\it open} cones. We have used
throughout the convention that when we state that something is an
almost direct product, we use notation for it as though it were a
direct product. Remark \ref{g-gprime} establishes a convention
that the use of a certain symbol includes the context in which it
is being used.}

\section{Triangulated categories of motives}
\label{sect:triang-cat-of-mot}

\subsection{Quick review of their construction}
\label{Quick}
We briefly describe the construction of a triangulated category
$\DM(X)$ whose objects will be called \emph{relative motives} over
the scheme $X$.
The details of our construction are to be found in
\cite[\S 4.5]{ayoub-these-II}: our category $\DM(-)$ is
the category $\mathbb{SH}_{\mathfrak{M}}^T(-)$
of \cite[D\'ef.~4.5.21]{ayoub-these-II} when we take
for $\mathfrak{M}$, the category of complexes of $\Q$-vector spaces, and for $\tau$,
the \'etale topology.
(The notation $\DM$ is probably due to F.~Morel and it appears already in \cite[Def.~1.3.2]{mot-var-rig}; most probably the $\mathbf{A}$ stands for abelian.)
Roughly
speaking, we follow, without lots of imagination, the recipe of Morel and Voevodsky
\cite{morel-voevodsky}, replacing simplicial sets by complexes of
$\Q$-vector spaces and then use spectra to formally invert the
tensor product by the Tate motive, as in \cite{jardine}. In particular, we do not use
the theory of finite correspondences from \cite{livre-orange} in defining
$\DM(X)$.
However, it can be shown
that, for $X=\Spec(k)$ the spectrum of a perfect
field, we have an equivalence of categories $\DM(k)\simeq
\mathbf{DM}(k)$, where $\mathbf{DM}(k)$ is Voevodsky's category of mixed
motives with rational coefficients (see Proposition \ref{prop:compare} below).

For the reader convenience, we now review some elements of the
construction of $\DM(X)$.
For a noetherian scheme $X$, we denote by $\Sm/X$ the category of
smooth $X$-schemes of finite type. We consider $\Sm/X$ as a site
for the \'etale topology. The category $\Shv(\Sm/X)$, of \'etale
sheaves of $\Q$-vector spaces over $\Sm/X$, is a Grothendieck abelian
category. Given a smooth $X$-scheme
$Y$, we denote by $\Q_{\et}(Y\to X)$ (or just $\Q_{\et}(Y)$ when
$X$ is understood) the \'etale sheaf associated to the presheaf
$\Q(Y)$ freely generated by $Y$, i.e., $\Q(Y)(-)=\Q(\hom_{\Sm/X}(-,Y))$.

\begin{definition}
\label{defn:DMeff-of-X} The category $\DM_{\eff}(X)$ is the
Verdier quotient of the derived category $\mathbf{D}(\Shv(\Sm/X))$
by the smallest triangulated subcategory $\mathbf{A}$ that is stable under
infinite sums and contains the complexes $[\Q_{\et}(\Aff^1_Y)\to
\Q_{\et}(Y)]$ for all smooth $X$-schemes $Y$.
\end{definition}

As usual, $\Aff^1_Y$ denotes the relative affine line over $Y$.
Given a smooth $X$-scheme $Y$, we denote by $\M_{\eff}(Y)$ (or
$\M_{\eff}(Y\to X)$ if confusion can arise) the object
$\Q_{\et}(Y)$ viewed as an object of $\DM_{\eff}(X)$. This is the
{\it effective homological motive} of $Y$. We also write $\un_X$ (or
simply $\un$) for the motive $\M_{\eff}(\id_X)$ where $\id_X$ is the identity mapping of $X$.
This is a unit for the tensor product on $\DM_{\eff}(X)$.

One can alternatively define $\DM_{\eff}(X)$ as the homotopy
category of a model structure in the sense of \cite{Quillen} (see
\cite{GJ}). More
precisely, the category $\mathbf{K}(\Shv(\Sm/X))$ of complexes of
\'etale sheaves on $\Sm/X$ can be endowed with the $\Aff^1$-local model structure
$(\mathbf{W}_{\Aff^1},\mathbf{Cof},\mathbf{Fib}_{\Aff^1})$, for
which $\DM_{\eff}(X)$ is the homotopy category
$$\mathbf{K}(\Shv(\Sm/X))[\mathbf{W}_{\Aff^1}^{-1}].$$
Here, the class $\mathbf{W}_{\Aff^1}$ (of $\Aff^1$-weak
equivalences) consists of morphisms which become invertible in
$\DM_{\eff}(X)$; the cofibrations are the injective morphisms of
complexes; the class $\mathbf{Fib}_{\Aff^1}$ (of
$\Aff^1$-fibrations) is defined by the right lifting property
\cite{Quillen} with respect to the arrows in $\mathbf{Cof}\cap
\mathbf{W}_{\Aff^1}$.

In this paper we need to use some of the Grothendieck operations
on motives (see \cite{ayoub-these-I,ayoub-these-II}). These
operations are defined on the categories $\DM(X)$ obtained from
$\DM_{\eff}(X)$ by formally inverting the operation $T\otimes -$,
tensor product with the Tate motive. Here, we will take as a model
for the Tate motive\footnote{Usually the Tate motive $\Q_X(1)$ is
defined to be $T_X[-1]$ viewed as an object of $\DM_{\eff}(X)$. As
the shift functor $[-1]$ is already invertible in $\DM_{\eff}(X)$,
it is equivalent to invert $(T_X\otimes -)$ or
$(\Q_X(1)\otimes -$).} the \'etale sheaf
$$T_X=\ker \left\{\Q_{\et}\left((\Aff^1_X- o(X))\to X\right) \longrightarrow \Q_{\et}(\text{id}_X:X\to X)\right\},$$
where $o:X \to \Aff^1_X$ is the zero section. We
denote $T_X$ simply by $T$ if the base scheme $X$ is clear.

The process of inverting $T\otimes -$ is better understood via the
machinery of spectra, borrowed from algebraic topology \cite{adams}.
We denote
the category of $T$-spectra of complexes of \'etale sheaves on
$\Sm/X$ by
$$\mathbb{M}_T(X)=\Spect_T(\mathbf{K}(\Shv(\Sm/X))).$$
Objects of $\mathbb{M}_T(X)$ are collections
$\mathbf{E}=(E_n,\gamma_n)_{n\in \N}$, in which the $E_n$'s are
complexes of \'etale sheaves on $\Sm/X$ and the $\gamma_n$'s are
morphisms of complexes
$$\gamma_n:T\otimes E_n\to E_{n+1},$$ called
{\it assembly maps}. We note that $\gamma_n$ determines by adjunction a morphism
$\gamma'_n:
E_n \to\underline{Hom}(T,E_{n+1})$. There is a stable $\Aff^1$-local model
structure on $\mathbb{M}_T(X)$ such that a $T$-spectrum
$\mathbf{E}$ is stably $\Aff^1$-fibrant if and only if each $E_n$
is $\Aff^1$-fibrant and each $\gamma'_n$ is a
quasi-isomorphism of complexes of sheaves. This model structure is denoted by
$(\mathbf{W}_{\Aff^1\text{-st}},\mathbf{Cof},\mathbf{Fib}_{\Aff^1\text{-st}})$.

\begin{definition}
\label{defn:DM-of-X-stable-version}
The category $\DM(X)$ is the homotopy category of $\mathbb{M}_T(X)$
with respect to the stable $\Aff^1$-local model structure:
$$\DM(X)=\mathbb{M}_T(X) [(\mathbf{W}_{\Aff^1\text{-}{\rm st}})^{-1}].$$

\end{definition}

There is an infinite suspension functor
$\Sigma^{\infty}_T:\DM_{\eff}(X) \to \DM(X)$ which takes a complex
of \'etale sheaves $K$ to the $T$-spectrum
$$(K, T
\otimes K , \dots, T^{\otimes r}\otimes K, \dots ),$$
where
the assembly maps are the identity maps. In $\DM(X)$, the homological motive of a smooth
$X$-scheme $Y$ is then $\M(Y)=\Sigma^{\infty}_T(\M_{\eff}(Y))$ (we
write $\M(Y\to X)$ if confusion can arise). The motive
$\M(\id_X)$ will be denoted by $\un_X$ (or simply $\un$).
There is also a tensor product on $\DM(X)$ which makes it a
closed monoidal symmetric category
with unit object $\un_X$. Then the functor
$\Sigma^{\infty}_T$ becomes monoidal symmetric and unitary. Moreover, the Tate motive
$\un_X(1)=\Sigma^{\infty}_T(T_X)[-1]$ is invertible for the tensor product of $\DM(X)$. For $n\in
\Z$, we define the Tate twists $M(n)$ of a motive $M\in \DM(X)$ in
the usual way.

By \cite{ayoub-these-I, ayoub-these-II}, we have the full machinery
of Grothendieck's six operations on the triangulated
categories $\DM(X)$.
Two of these operations,
$\otimes_X$ and $\underline{\rm Hom}_X$, are part of the monoidal
structure on $\DM(X)$.
Given a morphism of noetherian schemes $f:X
\to Y$, we have the operations $f^*$ and $f_*$ of inverse image
and cohomological direct image along $f$. When $f$ is
quasi-projective, we also have the operations $f_!$ and
$f^!$ of direct image with compact support and extraordinary
inverse image along $f$. The usual properties from
\cite{SGA4} hold.

\begin{definition}
\label{defn:cohomological-motive} Let $X$ be a noetherian scheme
and $Y$ a quasi-projective $X$-scheme. We define $\M_{\coh}(Y)$, or
$\M_{\coh}(Y\to X)$ if confusion can arise, to be
$(\pi_{Y})_*\un_Y$, where $\pi_Y:Y \to X$ is the structural
morphism of the $X$-scheme $Y$. This is the {\rm cohomological
motive} of $Y$ in $\DM(X)$.
\end{definition}

It is easy to check that this defines a contravariant functor
 $\M_{\coh}(-)$
from the category of quasi-projective $X$-schemes to $\DM(X)$.
In contrast to homological motives, $\M_{\coh}(Y)$
is defined without assuming $Y$ to be smooth over $X$.\footnote{
In the stable motivic categories, $\M(Y)$ can be extended for all quasi-projective $X$-schemes $Y$ by setting
$\M(Y)=(\pi_Y)_!(\pi_Y)^!\un_X$. We do not use this in the paper.}

We write $\mathbf{DM}(X)$ for Voevodsky's category of motives over
the base-scheme $X$. $\mathbf{DM}(X)$ is obtained
in the same way as $\DM(X)$ using the category
$\Shv^{tr}_{\rm Nis}(\Smooth/X)$ of Nisnevich sheaves with transfers (cf.~\cite[Lect.~13]{motivic-lectures} for $X$ the spectrum of a field) instead of the category $\Shv(\Smooth/X)$ of \'etale sheaves.
A detailed construction of this category (at least
for $X$ smooth over a field) can be obtained as the particular case
of \cite[Def.~2.5.27]{mot-var-rig} where
the valuation of the base field is taken to be trivial.
A more recent account
of the construction can be found in
\cite{motif-transfert}.
The effective and geometric version of this category
is also constructed in \cite{ivorra}.

As we work with sheaves
of $\Q$-vector spaces, a Nisnevich sheaf with transfers is
automatically an \'etale sheaf. This gives a forgetful functor
${\rm o}^{tr}:\mathbf{DM}(X)\to \DM(X)$, which has a left adjoint
$${\rm a}^{tr}:\xymatrix@C=1.7pc{\DM(X) \ar[r] &  \mathbf{DM}(X).}$$
Thus, a motive $M\in\DM(X)$ determines a motive ${\rm a}^{tr}(M)$ in the sense of
Voevodsky. Moreover, when $X=\Spec(k)$ is the spectrum of a field
$k$ of characteristic zero, it follows from \cite{morel-unpublished} (cf.~\cite[Cor.~15.2.20]{motif-transfert} for a complete proof that works more generally for any excellent and unibranch base-scheme $X$) that:

\begin{proposition}
\label{prop:compare}
The functor ${\rm a}^{tr}:\DM(k) \to \mathbf{DM}(k)$
is an equivalence of categories.
\end{proposition}

\begin{remark}
\label{rem:for-our-choices}
The main reason we are working with coefficients in $\mathbb{Q}$ (rather than in $\mathbb{Z}$) is technical.
For computing the functors
$\omega^0_X$ (see Proposition
\ref{prop:main-computation-omega0-pi-0} below), we need to invoke
Proposition \ref{prop:compare}, which holds only with rational coefficients. Also, some of the arguments in the proof of
Theorem \ref{thm:final-form-for-applic-main-thm} use in an essential way that the coefficients are in $\mathbb{Q}$.
Also, we choose to work with the
categories $\DM(X)$ rather than $\mathbf{DM}(X)$. We do this in order
to have a context in which the formalism of the six operations of
Grothendieck is available. Indeed, there is an obstacle to having
this formalism in $\mathbf{DM}(X)$, at least with integral coefficients, as the localization axiom (see
\cite[\S 1.4.1]{ayoub-these-I}) is still unknown for relative
motives in the sense of Voevodsky. Moreover, as
\cite[Cor.~15.2.20]{motif-transfert} indicates, there is no essential difference between these categories, as long as we are concerned with rational coefficients and unibranch base-schemes.

\end{remark}

\subsection{Motives over a diagram of schemes}\label{subsect:motives over a diagram}
Later, we will need a generalization of the notion of relative motive where the scheme $X$ is replaced by a \emph{diagram of schemes}. The main references for this are \cite[Sect.~2.4]{ayoub-these-I} and \cite[Sect.~4.5]{ayoub-these-II}. We will denote by $\Dia$ the $2$-category of small categories.

Let $\mathcal{C}$ be a category. A \emph{diagram} in $\mathcal{C}$
is a covariant functor $\mathcal{X}:\mathcal{I}\to \mathcal{C}$ with $\mathcal{I}$ a small category (i.e., $\mathcal{I}\in \Dia$).
A diagram in $\mathcal{C}$ will be denoted $(\mathcal{X},\mathcal{I})$ or simply $\mathcal{X}$ if no confusion can arise.
Given an object $X\in \mathcal{C}$, we denote by $(X,\mathcal{I})$ the constant diagram with value $X$, i.e., sending any object to $X$ and any arrow to the identity of $X$.

A morphism of diagrams $(\mathcal{Y},\mathcal{J}) \to (\mathcal{X},\mathcal{I})$ is a pair $(f,\alpha)$ where $\alpha:\mathcal{J} \to \mathcal{I}$ is a functor and $f:\mathcal{Y}\to \mathcal{X}\circ \alpha$ is a natural transformation. Such a morphism admits a natural factorization
\begin{equation}
\label{eq:factorization-1-morph-diag-C}
\xymatrix@C=1.7pc{(\mathcal{Y},\mathcal{J}) \ar[r]^-f & (\mathcal{X}\circ \alpha, \mathcal{J}) \ar[r]^-{\alpha} & (\mathcal{X},\mathcal{I}).}
\end{equation}
(When $\mathcal{C}$ is a category of spaces, $f$ and $\alpha$ are respectively called the \emph{geometric} and the \emph{categorical} part of $(f,\alpha)$.)
We denote by $\Dia(\mathcal{C})$ the category of diagrams in $\mathcal{C}$ which is actually a strict $2$-category where the $2$-morphisms are defined as follows.
Let $(f,\alpha)$ and $(g,\beta)$ be two morphisms from $(\mathcal{Y},\mathcal{J})$ to $(\mathcal{X},\mathcal{I})$. A $2$-morphism $t:(f,\alpha)\to (g,\beta)$ is a natural transformation $t:\alpha \to \beta$ such that for every $j\in \mathcal{J}$, the following triangle
$$\xymatrix@C=1.7pc@R=1.7pc{\mathcal{Y}(j) \ar[r]^-{f(j)} \ar@/_/[dr]_-{g(j)} & \mathcal{X}(\alpha(j)) \ar[d]^-{\mathcal{X}(t(j))}
\\
& \mathcal{X}(\beta(j)) }$$
commutes.

We have a fully faithful embedding $\mathcal{C} \hookrightarrow
\Dia(\mathcal{C})$ sending an object $X\in \mathcal{C}$ to the
diagram $(X,\mathbf{e})$ where $\mathbf{e}$ is the category with
one object and one arrow. We will identify $\mathcal{C}$ with a
full subcategory of $\Dia(\mathcal{C})$ via this embedding. Given
a diagram $(\mathcal{X},\mathcal{I})$ and an object $i\in
\mathcal{I}$, we have an obvious morphism $i:\mathcal{X}(i) \to
(\mathcal{X},\mathcal{I})$.

Now, we consider the case $\mathcal{C}=\Sch$ (schemes). Objects in
$\Dia(\Sch)$ are called \emph{diagrams of schemes}.
For $(\mathcal{X},\mathcal{I})\in \Dia(\Sch)$, let
$\Sm/(\mathcal{X},\mathcal{I})$ be the category whose objects are pairs $(U,i)$ with $i\in \mathcal{I}$ and $U$ a smooth $\mathcal{X}(i)$-scheme. Morphisms $(V,j)\to (U,i)$ are given by an arrow
$j\to i$ in $\mathcal{I}$ and a morphism of schemes $V \to U$ making the following square
$$\xymatrix@C=1.5pc@R=1.5pc{V \ar[r] \ar[d] & U \ar[d] \\
\mathcal{X}(j) \ar[r] & \mathcal{X}(i)}$$
commutative.
As in the case of a single scheme, we may use the category
$\Sm/(\mathcal{X},\mathcal{I})$, endowed with the \'etale topology,
to define a triangulated category
$\DM(\mathcal{X},\mathcal{I})$ of motives over
$(\mathcal{X},\mathcal{I})$. The full details of the construction can be found in
\cite[Ch. 4]{ayoub-these-II}.
Objects of $\DM(\mathcal{X},\mathcal{I})$ are called relative motives over $(\mathcal{X},\mathcal{I})$.

Let $(\mathcal{X},\mathcal{I})$ be a diagram of schemes and
$\mathcal{J}$ a small category. We call ${\rm pr}_1:\mathcal{I}\times \mathcal{J} \to \mathcal{I}$ the projection to the first factor.
There is a functor
\begin{equation}
\label{eq:partial-skeleton-definition}
{\rm sk}_{\mathcal{J}}:\xymatrix@C=1.7pc{\DM(\mathcal{X}\circ {\rm pr_1},\mathcal{I}\times \mathcal{J})\ar[r] & \underline{HOM}(\mathcal{J}^{\rm op},\DM(\mathcal{X},\mathcal{I}))}
\end{equation}
which associates to a relative motive $\mathbf{E}$ over $(\mathcal{X}\circ {\rm pr}_1,\mathcal{I}\times \mathcal{J})$ the
contravariant functor $j\rightsquigarrow \mathbf{E}(-,j)\in \DM(\mathcal{X},\mathcal{I})$,
called the {\it $\mathcal{J}$-partial skeleton} of $\mathbf{E}$.
When $\mathcal{X}(i)$ is not the empty scheme for at least one $i\in \mathcal{I}$, this functor is an
equivalence of categories only if $\mathcal{I}$ is discrete, i.e., equivalent to a category where every arrow is an identity.

The basic properties concerning the functoriality of $\DM(\mathcal{X},\mathcal{I})$ with respect to $(\mathcal{X},\mathcal{I})$ are summarized in
\cite[\S 2.4.2]{ayoub-these-I}. Note that a morphism of diagrams of schemes $(f,\alpha):(\mathcal{Y},\mathcal{J})\to (\mathcal{X},\mathcal{I})$ induces a functor $(f,\alpha)^*:\DM(\mathcal{X},\mathcal{I}) \to \DM(\mathcal{Y},\mathcal{J})$.
The assignment $(f,\alpha)\rightsquigarrow (f,\alpha)^*$ is contravariant with respect to $2$-morphisms and
$(f,\alpha)^*$
 admits a right adjoint $(f,\alpha)_*$. When $f$ is objectwise smooth (i.e., $f(j)$ is smooth for all $j\in \mathcal{J}$), $(f,\alpha)^*$ admits also a left adjoint $(f,\alpha)_{\sharp}$.

Now we gather some additional properties which will be needed later.

\begin{lemma}
\label{rem:i-*-j-diaise-}
For $i,\, j\in \mathcal{I}$, $M\in \DM(\mathcal{X}(i))$ and $N\in \DM(\mathcal{X}(j))$, there are canonical isomorphisms
$$\bigoplus_{j\to i\in \hom_{\mathcal{I}}(j,i)}\mathcal{X}(j\to i)^* M \simeq j^*i_{\sharp}M \qquad \text{and} \qquad  i^*j_* N \simeq \prod_{j\to i \in \hom_{\mathcal{I}}(j,i)} \mathcal{X}(j \to i)_*N.$$

\end{lemma}

\begin{proof}
The second isomorphism is a special case of the axiom \textbf{DerAlg  4$^\prime$g} in \cite[Rem.~2.4.16]{ayoub-these-I}.
The first isomorphism is obtained from the second one using
the adjunctions $(\mathcal{X}(j\to i)^*,\mathcal{X}(j\to i)_*)$,
$(i_{\sharp},i^*)$ and $(j^*,j_*)$.
\end{proof}

\begin{proposition}
\label{prop:cohomol-diag-sch-adjunction}
Let $S$ be a noetherian scheme and $(\mathcal{X},\mathcal{I})$ a diagram of $S$-schemes. Let $\mathcal{J}$ be a small category and $\alpha: \mathcal{J} \to \mathcal{I}$
a functor. We form the commutative triangle in $\Dia(\Sch)$
$$\xymatrix@C=1.5pc@R=1.7pc{(\mathcal{X}\circ \alpha,\mathcal{J}) \ar@/_/[dr]_-{(f,q)}  \ar[r]^-{\alpha} & (\mathcal{X},\mathcal{I}) \ar[d]^-{(f,p)}\\
& S.}$$
Assume that $\alpha$ admits a left adjoint.
Then the composition
$$(f,p)_* \xymatrix@C=1.5pc{\ar[r] &} (f,p)_*\alpha_*\alpha^* \simeq (f,q)_*\alpha^*$$
is invertible.
\end{proposition}

\begin{proof}
We have a commutative diagram in $\Dia(\Sch)$
$$\xymatrix@C=1.5pc@R=1.5pc{(\mathcal{X}\circ \alpha,\mathcal{J}) \ar[d]^-{\alpha} \ar[r]^-f &  (S,\mathcal{J}) \ar[d]^-{\alpha} \ar@/^/[dr]^-q &  \\
(\mathcal{X},\mathcal{I}) \ar[r]^-f & (S,\mathcal{I}) \ar[r]^-p & S.}$$
We need to show that $(f,p)_*\to (f,p)_*\alpha_*\alpha^*$ is invertible, or equivalently, that
$p_*f_*\to p_*f_*\alpha_*\alpha^*$ is invertible.
But we have a commutative square
$$\xymatrix@C=1.5pc@R=1.5pc{p_*f_* \ar[r]^-{\eta} \ar[d]^-{\eta} & p_*f_*\alpha_*\alpha^* \ar[d]^-{\sim} \\
p_*\alpha_*\alpha^* f_*  \ar[r]^-{\sim} & p_*\alpha_*f_*\alpha^*  }$$
where the bottom arrow is invertible by
axiom \textbf{DerAlg 3d} of \cite[\S 2.4.2]{ayoub-these-I}.
Thus, it is sufficient to show that $p_*\to p_*\alpha_*\alpha^*$ is invertible.
This follows from
\cite[Lem.~2.1.39]{ayoub-these-I}, as $\alpha$ has a left adjoint.\footnote{There is a misprint in the statement of \cite[Lem.~2.1.39]{ayoub-these-I}. The $u$'s and $v$'s should be interchanged in the two natural transformations that are asserted to be invertible. The proof in \emph{loc.~cit.} remains the same.}
\end{proof}

Before stating a useful corollary of Proposition
\ref{prop:cohomol-diag-sch-adjunction} we need some preliminaries.
Let $\mathcal{J}:\mathcal{I} \to \Dia$ be a functor, i.e., an object of $\Dia(\Dia)$. We define the \emph{total category} $\int_{\mathcal{I}}\mathcal{J}$, or simply $\int\mathcal{J}$, as follows:
\begin{itemize}

\item objects are pairs $(i,j)$ where $i\in \mathcal{I}$ and
$j\in \mathcal{J}(i)$,

\item arrows $(i,j)\to (i',j')$ are pairs $(i\to i', \mathcal{J}(i\to i')(j) \to j')$.

\end{itemize}
This gives a covariant functor $\int:\Dia(\Dia)\to \Dia$. We have
a functor $\rho:\int_{\mathcal{I}}\mathcal{J} \to \mathcal{I}$ sending $(i,j)$ to $i$.
For $i\in \mathcal{I}$, we have
an inclusion
$\epsilon_i:\mathcal{J}(i) \hookrightarrow \int_{\mathcal{I}}\mathcal{J}$ sending
$j\in \mathcal{J}(i)$ to $(i,j)$. We may factor this inclusion through
the comma category\footnote{Recall that, given a functor $\alpha:\mathcal{T} \to \mathcal{S}$ and an object $s\in \mathcal{S}$, the {\it comma category} $\mathcal{T}/s$ is the category of pairs $(t,\alpha(t)\to s)$ where morphisms are defined in the obvious way.}
$(\int_{\mathcal{I}}\mathcal{J})/i$ by sending
$j\in \mathcal{J}(i)$ to $((i,j),\id_i)$. We get in this way an inclusion
$\epsilon_i':\mathcal{J}(i)\hookrightarrow (\int_{\mathcal{I}}\mathcal{J})/i$
which has a left adjoint $(\int_{\mathcal{I}}\mathcal{J})/i \to \mathcal{J}(i)$ sending
$((i',j'),i'\to i)$ to $\mathcal{J}(i'\to i)(j')$.

\begin{definition}
Let $(\mathcal{Y},\mathcal{J}):\mathcal{I} \to \Dia(\mathcal{C})$
be an object of $\Dia(\Dia(\mathcal{C}))$, i.e., a functor sending an object $i\in \mathcal{I}$ to a diagram $(\mathcal{Y}(i),\mathcal{J}(i))$ in $\mathcal{C}$. The assignment $(i,j) \rightsquigarrow \mathcal{Y}(i,j)$ defines a functor on $\int_{\mathcal{I}}\mathcal{J}$. We get in this way a diagram $(\mathcal{Y},\int_{\mathcal{I}}\mathcal{J})$ in $\mathcal{C}$ called the \emph{total diagram} associated with $(\mathcal{Y},\mathcal{J})$.
\end{definition}

\begin{corollary}
\label{cor:very-useful-texnical-cor}
Let $(\mathcal{X},\mathcal{I})$ be a diagram of schemes.
Let $((\mathcal{Y},\mathcal{J}),\mathcal{I})$ be a diagram in
$\Dia(\Sch)$. Assume we are given a morphism
$$f:((\mathcal{Y},\mathcal{J}),\mathcal{I}) \to ((\mathcal{X},\mathbf{e}),\mathcal{I})$$
in $\Dia(\Dia(\Sch))$ which is the identity on $\mathcal{I}$.
Passing to the total diagrams, we get a morphism
$$(f,\rho):(\mathcal{Y},\textstyle{\int_{\mathcal{I}}\mathcal{J}}) \to (\mathcal{X},\mathcal{I}).$$
Then, for every $i\in \mathcal{I}$, there
is a canonical isomorphism
$$i^*(f,\rho)_*\simeq f(i)_*\epsilon_i^*,$$
where, as before,
$\epsilon_i:\mathcal{J}(i)\hookrightarrow
\int_{\mathcal{I}}\mathcal{J}$ denotes the inclusion.
\end{corollary}

\begin{proof}
By axiom \textbf{DerAlg 4'g} in \cite[Rem.~2.4.16]{ayoub-these-I},
$i^*(f,\rho)_*\simeq (f/i)_*u_i^*$ where $u_i:(\int_{\mathcal{I}}\mathcal{J})/i \to \int_{\mathcal{I}}\mathcal{J}$ is the natural morphism and $f/i$ is the projection of
$(\mathcal{Y}\circ u_i,(\int_{\mathcal{I}}\mathcal{J})/i)$
to $\mathcal{X}(i)$.

Now, recall that we have an inclusion $\epsilon'_i:\mathcal{J}(i) \hookrightarrow (\int_{\mathcal{I}}\mathcal{J})/i$ which admits a left adjoint. By Proposition
\ref{prop:cohomol-diag-sch-adjunction},
we have isomorphisms
$$(f/i)_*u_i^* \overset{\sim}{\to} (f/i)_*\epsilon'_{i*} \epsilon'^*_iu_i^*\simeq
f(i)_*\epsilon_i^*.$$
This ends the proof of the corollary.
\end{proof}

\begin{remark}
\label{rem-gen:cor-useful-texnical-to-derivators}
The same method of proof of Corollary
\ref{cor:very-useful-texnical-cor} yields a similar result for triangulated derivators which we describe here for later use; for a working definition of a derivator, see \cite[D\'ef.~2.1.34]{ayoub-these-I}.
Let $\mathbb{D}$ be a derivator, $\mathcal{I}$ a small category and $\mathcal{J}:\mathcal{I}\to \Dia$ an object of $\Dia(\Dia)$. Let $p:\int_{\mathcal{I}} \mathcal{J} \to \mathcal{I}$ and $p(i):\mathcal{J}(i) \to \{i\}$ denote the obvious projections, and $\epsilon_i:\mathcal{J}(i) \hookrightarrow \int_{\mathcal{I}} \mathcal{J}$ the inclusion. Then for all $i\in \mathcal{I}$, the natural transformation
$i^* p_* \to p(i)_* \epsilon_i^*$
(of functors from $\mathbb{D}(\int_{\mathcal{I}}\mathcal{J})$ to $\mathbb{D}(\{i\})$)
is invertible.

\end{remark}

A particular case of Corollary
\ref{cor:very-useful-texnical-cor} yields the following:

\begin{corollary}
\label{cor:making-the-diag-connected}
Let $(\mathcal{X},\mathcal{I})$ be a diagram of schemes.
Denote by $\Pi:\mathcal{I} \to \Dia$ the functor which associate to
$i\in \mathcal{I}$ the set of connected components of $\mathcal{X}$ considered
as a discrete category. Let $\mathcal{I}^{\flat}=\int_{\mathcal{I}}\Pi$ and
$(\mathcal{X}^{\flat},\mathcal{I}^{\flat})$
the diagram of schemes which takes a pair $(i,\alpha)$ with
$i\in \mathcal{I}$ and $\alpha\in \Pi_0(i)$ to the connected component
$\mathcal{X}_{\alpha}(i)$ of $\mathcal{X}(i)$ that corresponds to $\alpha$.
There is a natural morphism of diagrams of schemes
${\scriptstyle \amalg}:(\mathcal{X}^{\flat},\mathcal{I}^{\flat})\to
(\mathcal{X},\mathcal{I})$. Moreover,
$\id \to {\scriptstyle \amalg}_*{\scriptstyle \amalg}^*$ is invertible.

\end{corollary}

\begin{proof}
Only the last statement needs a proof.
For $i\in \mathcal{I}$,
$\id \to {\scriptstyle \amalg}(i)_*{\scriptstyle
\amalg}(i)^*$ is invertible with
${\scriptstyle \amalg}(i):(\mathcal{X}_{\alpha}(i))_{\alpha\in \Pi(i)}\to
\mathcal{X}(i)$ the natural morphism from the
discrete diagram
of schemes $(\mathcal{X}_{\alpha}(i))_{\alpha\in \Pi(i)}$.
Using Corollary
\ref{cor:very-useful-texnical-cor},
applied to the functor which takes $i\in \mathcal{I}$ to $(\mathcal{X}_{\alpha}(i))_{\alpha\in \Pi(i)}$,
we obtain that $i^*\to i^*{\scriptstyle \amalg}_*{\scriptstyle
\amalg}^*$ is invertible.
\end{proof}

Before going further, we introduce the following terminology.

\begin{definition}
Let $\mathcal{I}$ be a small category. We say that $\mathcal{I}$ is
\emph{universal for homotopy limits} if it satisfies to the following
property. For every $1$-morphism of triangulated derivators
$m:\mathbb{D}_1 \to \mathbb{D}_2$ in the sense of
\cite[D\'ef.~2.1.46]{ayoub-these-I}, the natural transformation
between functors from $\mathbb{D}_1(\mathcal{I})$ to $\mathbb{D}_2(\mathbf{e})$:
$$m(\mathbf{e}) (p_{\mathcal{I}})_* \to (p_{\mathcal{I}})_* m(\mathcal{I}),$$
where $p_{\mathcal{I}}$ is the projection of $\mathcal{I}$ to $\mathbf{e}$, is invertible.

\end{definition}

\begin{lemma}
\label{lemma:permanence-universal-for-holim}
If a category has a final object, it is universal for homotopy limits.
The class of small categories which are universal for homotopy limits is
stable by finite direct products. If $\mathcal{J}:\mathcal{I}\to \Dia$
is an object of $\Dia(\Dia)$ such that $\mathcal{I}$ and all the $\mathcal{J}(i)$ are universal
for homotopy limits, for $i\in \mathcal{I}$, then $\int_{\mathcal{I}}\mathcal{J}$ is universal for homotopy limits.
\end{lemma}

\begin{proof}
If $e$ is a final object of $\mathcal{I}$, then $(p_{\mathcal{I}})_*\simeq e^*$. But any morphism of
triangulated derivators commutes with $e^*$ by definition. Hence the first claim of the lemma.

The second claim of the lemma is a special case of the last one. To prove the latter, consider the sequence
$$\xymatrix@C=1.7pc{\int_{\mathcal{I}} \mathcal{J} \ar[r]^-p & \mathcal{I} \ar[r]^-{p_{\mathcal{I}}} & \mathbf{e}.}$$
As $\mathcal{I}$ is universal for homotopy limits, it suffices to show that the natural transformations
$$m(\mathcal{I}) p_* \to p_* m(\textstyle{\int_{\mathcal{I}}\mathcal{J}})$$
are invertible for any $1$-morphism of triangulated derivators
$m$. It suffices to show this after applying $i^*$ for $i\in \mathcal{I}$. With the notation of
Remark \ref{rem-gen:cor-useful-texnical-to-derivators},
we have
$$i^*m(\mathcal{I}) p_*\simeq m(\{i\}) i^* p_*\simeq m(\{i\}) p(i)_*\epsilon_i^*$$
$$\text{and} \quad
i^*p_*m(\textstyle{\int_{\mathcal{I}} \mathcal{J}})
\simeq p(i)_* \epsilon^*_i m(\textstyle{\int_{\mathcal{I}} \mathcal{J}})
\simeq p(i)_*m(\mathcal{J}(i)) \epsilon_i^*.$$
Thus, it suffices to show that
$m$ commutes with $p(i)_*$.
Our claim follows as $\mathcal{J}(i)$ is universal for
homotopy limits.
\end{proof}

Recall that $\underline{\mathbf{1}}$ denotes
the ordered set $\{0\to 1\}$. Let $\ucarre$ be the complement of $(1,1)$ in
$\underline{\mathbf{1}}\times \underline{\mathbf{1}}$.
Recall also that an ordered set is just a small category with at most one arrow between each pair of objets.

\begin{lemma}
\label{texnical-commut-holim}
For $n\in \N$, the category $\ucarre^n$
is universal for homotopy limits.
\end{lemma}

\begin{proof}
It suffices to show that $\ucarre$ is universal for homotopy limits. Fix a morphism of triangulated derivators $m:\mathbb{D}_1\to \mathbb{D}_2$.
For $A_i\in \mathbb{D}_i(\ucarre)$, we have a distinguished triangle in $\mathbb{D}_i(\mathbf{e})$:
$$\xymatrix@C=1.7pc{(p_{\smallucarre})_* A_i \ar[r] & (1,0)^*A_i \bigoplus \,(0,1)^*A_i \ar[r] & (0,0)^*A_i \ar[r] & .}$$
As the $m(-):\mathbb{D}_1(-)\to \mathbb{D}_2(-)$ are triangulated functors, we deduce for $A\in \mathbb{D}_1(\ucarre)$ a morphism of distinguished triangles in
$\mathbb{D}_2(\mathbf{e})$:
$$\xymatrix@C=1.2pc@R=1.4pc{m(\mathbf{e}) \, (p_{\smallucarre})_* A  \ar[r] \ar[d] & m(\mathbf{e}) (1,0)^*A \bigoplus  m(\mathbf{e})(0,1)^*A \ar[r] \ar[d]^-{\sim}
& m(\mathbf{e})(0,0)^*A \ar[r] \ar[d]^-{\sim} & \\
(p_{\smallucarre})_* m(\ucarre) A  \ar[r] &  (1,0)^* m(\ucarre) A \bigoplus \, (0,1)^* m(\ucarre)A \ar[r] & (0,0)^*m(\ucarre)A \ar[r] & }$$
where the second and third vertical arrows are invertible by the
definition of a morphism of derivators. This implies that the first vertical arrow is also invertible. The lemma is proven.
\end{proof}

\begin{proposition}
\label{prop:limits-with-respect-ordered-set}
A finite ordered set is universal for homotopy limits.

\end{proposition}

\begin{proof}
Let $I$ be a finite ordered set.
We argue by induction on ${\rm card}(I)$. When ${\rm card}(I)\leq 2$, the claim is clear. Thus, we may assume that $I$ has more than 2 elements.
Fix $x\in I$ a maximal element of $I$. Let $A(1,0)=I-\{x\}$,
$A(0,0)=\{y\in I, y<x\}$ and $A(0,1)=\{x\}$.
Then we have a diagram of ordered sets
$$\xymatrix@C=1.5pc{A(1,0) & A(0,0) \ar[l] \ar[r] & A(0,1)}$$
indexed by $\ucarre$. The backward arrow is the inclusion and the onward arrow is the unique projection to the singleton $\{x\}$.
Using Lemmas
\ref{lemma:permanence-universal-for-holim} and
\ref{texnical-commut-holim}, and induction on ${\rm card}(I)$, we deduce that $\int_{\smallucarre}A$ is universal for homotopy limits.

On the other hand, we have a diagram of ordered sets $(B,I)$
given by
$$B(y)=\left\{
\begin{array}{ccc}
\{0\} & \text{if} & y=x,\\
\underline{\mathbf{1}}=\{0\to 1\} & \text{if} & y<x,\\
\{1\} & \text{if} & y \text{ is not comparable with } x.
\end{array} \right.$$
It is easy to see that the categories
$\int_{\smallucarre}A$ and $\int_IB$ are isomorphic.
Now, consider the natural functor
$p:\int_IB \to I$ and denote by $q$ the projection of $I$ to $\mathbf{e}$. By Remark \ref{rem-gen:cor-useful-texnical-to-derivators} and the fact that $B(y)$ has a largest element for every $y\in I$,
the unit morphism $\id \to p_*p^*$ is invertible. It follows that
$q_*\simeq (q\circ p)_*p^*$.
This finishes the proof of the proposition, as
$\int_IB\simeq \int_{\smallucarre} A$ is universal for homotopy limits.
\end{proof}

\begin{proposition}
\label{prop-most-gen-proj-base-change}
Let $(\mathcal{X},\mathcal{I})$ be a diagram of schemes.
Let $((\mathcal{Y},\mathcal{J}),\mathcal{I})$ be a diagram in
$\Dia(\Sch)$. Assume we are given a morphism
$$f:((\mathcal{Y},\mathcal{J}),\mathcal{I}) \to ((\mathcal{X},\mathbf{e}),\mathcal{I})$$
in $\Dia(\Dia(\Sch))$ which is the identity on $\mathcal{I}$. Passing to the total diagrams, we get a morphism
$$(f,\rho):(\mathcal{Y},\textstyle{\int_{\mathcal{I}}\mathcal{J}}) \to (\mathcal{X},\mathcal{I}).$$
Let $(g,\alpha):(\mathcal{X}', \mathcal{I}') \to (\mathcal{X}, \mathcal{I})$ be a morphism of diagrams of schemes. We define a diagram of schemes $\mathcal{Y}':\int_{\mathcal{I}'} \mathcal{J}\circ \alpha\to \Sch$ by sending a pair $(i',j)$, with
$i'\in \mathcal{I}'$ and $j\in \mathcal{J}(\alpha(i'))$, to
$\mathcal{X}'(i')\times_{\mathcal{X}(\alpha(i'))} \mathcal{Y}(\alpha(i'),j)$. Then, we have a cartesian square in $\Dia(\Sch)$:
$$\xymatrix@R=1.7pc{(\mathcal{Y}', \int_{\mathcal{I}'} \mathcal{J}\circ \alpha) \ar[r]^-{(g',\alpha')} \ar[d]_-{(f',\rho')} & (\mathcal{Y},\int_{\mathcal{I}}\mathcal{J}) \ar[d]^-{(f,\rho)} \\
(\mathcal{X}',\mathcal{I}') \ar[r]^-{(g,\alpha)} & (\mathcal{X},\mathcal{I}).}$$
Moreover, if $f$ is objectwise projective, $g$ objectwise quasi-projective and the $\mathcal{J}(i)$, for $i\in \mathcal{I}$, are universal for homotopy limits, then the base change morphism
\begin{equation}
\label{eq-prop-most-gen-proj-base-change}
\xymatrix@C=1.7pc{(g,\alpha)^*(f,\rho)_*\ar[r] & (f',\rho')_*(g',\alpha')^*}
\end{equation}
is invertible.

\end{proposition}

\begin{proof}
Everything is clear except the last statement. It suffices to show that
\eqref{eq-prop-most-gen-proj-base-change} is invertible after applying $i'^*$ for $i'\in \mathcal{I}'$. Let $i=\alpha(i')$.
Using Corollary
\ref{cor:very-useful-texnical-cor}
to rewrite $i^*(f,\rho)_*$ and $i'^*(f',\rho')_*$, we immediately reduce to show that the base change morphism associated to the cartesian square
$$\xymatrix@C=1.7pc@R=1.5pc{(\mathcal{Y}'(i'),\mathcal{J}(i)) \ar[r] \ar[d] & (\mathcal{Y}(i),\mathcal{J}(i)) \ar[d] \\
\mathcal{X}'(i') \ar[r] & \mathcal{X}(i)}$$
is invertible. Our square is the vertical composition of the following two squares
$$\xymatrix@C=1.7pc@R=1.5pc{(\mathcal{Y}'(i'),\mathcal{J}(i)) \ar[r] \ar[d] & (\mathcal{Y}(i),\mathcal{J}(i)) \ar[d] \\
(\mathcal{X}'(i'),\mathcal{J}(i)) \ar[r] & (\mathcal{X}(i),\mathcal{J}(i)),} \qquad \quad
\xymatrix@C=1.7pc@R=1.5pc{(\mathcal{X}'(i'),\mathcal{J}(i)) \ar[r] \ar[d] & (\mathcal{X}(i),\mathcal{J}(i))\ar[d] \\
\mathcal{X}'(i') \ar[r] & \mathcal{X}(i).\!}$$
The base change morphism associated to the first square is
invertible by \cite[Th.~2.4.22]{ayoub-these-I}.
Also, the base change morphism associated to the second square is invertible as $\mathcal{J}(i)$ is universal for homotopy limits and
$(\mathcal{X}'(i) \to \mathcal{X}(i))^*$ defines a $1$-morphism of derivators $\xymatrix@C=1.3pc{\DM(\mathcal{X}(i),-)\ar[r] & \DM(\mathcal{X}(i'),-)}\!$.
This proves the proposition.
\end{proof}

\subsection{Stratified schemes}
Recall that a \emph{stratification} on a topological space $X$ is a
partition $\mathcal{S}$ of $X$ by locally closed subsets such that:
\begin{enumerate}

\item[(i)] Any point of $X$ admits an open neighborhood $U$ such that $S\cap U$ has finitely many connected components for every $S\in \mathcal{S}$, and is empty
except for finitely many $S\in \mathcal{S}$.

\item[(ii)] For $T\in \mathcal{S}$ we have, as sets, $\overline{T}=\bigsqcup_{S\in \mathcal{S}, \, S\subset \overline{T}} S$.

\end{enumerate}
As $\mathcal{S}$ is a partition of $X$, for
$S_1, \, S_2\in \mathcal{S}$,
either $S_1=S_2$ or $S_1\bigcap S_2=\emptyset$.

A connected component of an element of $\mathcal{S}$ will be called an $\mathcal{S}$-\emph{stratum} or simply \emph{stratum} if no confusion can arise. Two stratifications $\mathcal{S}$ and $\mathcal{S}'$ are equivalent if they determine the same set of strata. The set of $\mathcal{S}$-strata is a stratification on $X$ which is equivalent to $\mathcal{S}$. We usually identify equivalent stratifications.
When $X$ is a noetherian scheme, every stratification of $X$ has finitely many strata.

An open (resp.~closed) stratum is a stratum which is open
(resp.~closed) in $X$.
Given two strata $S$ and $T$, one writes $S\preceq T$ when $S\subset
\overline{T}$. Under mild conditions (satisfied when $X$ is a
noetherian scheme), a stratum $S$ is maximal (resp. minimal) for this partial order if and only if $S$ is an open (resp. a closed) stratum.
Finally, a subset of $X$ is called $\mathcal{S}$-\emph{constructible} if it is a union of $\mathcal{S}$-strata.

\begin{example}
\label{ex:strat-ass-family-subschemes}
Let $X$ be a noetherian scheme and suppose we are given a finite family $(Z_{\alpha})_{\alpha\in I}$ of closed subschemes of $X$.
For $J\subset I$, we put
$$X^0_J=\left(\bigcap_{\beta\in J} Z_{\beta}\right)- \left(\bigcup_{\alpha\in I- J} Z_{\alpha}\right).$$
This clearly give a stratification on $X$ such that any connected component of $X^0_{\emptyset}$ is an open stratum and any connected component of $X_I$ is a closed stratum.

\end{example}

We record the following lemma for later use:

\begin{lemma}
\label{lemma:form-locality-closed-cover}
Let $X$ be a noetherian scheme endowed with a stratification $\mathcal{S}$.
Denote $A$ the set of $\mathcal{S}$-strata
ordered by the relation $\preceq$.
Let $\mathcal{X}:A \to \Sch$ be the diagram of schemes sending an $\mathcal{S}$-stratum $S$ to its closure $\bar{S}$ (with its reduced scheme-structure).
Let
$s:(\mathcal{X},A)\to X$ be the natural morphism. Then
the unit morphism
$\id \to s_*s^*$ is invertible.
\end{lemma}

\begin{proof}
$X$ is a disjoint union of its $\mathcal{S}$-strata. By the locality axiom (cf.~\cite[Cor.~4.5.47]{ayoub-these-II}) it then sufficient to show that
$u^*\to u^*s_*s^*$ is invertible for
any $\mathcal{S}$-stratum $U$; $u:U \hookrightarrow X$ being the inclusion morphism.
Let $s':(\mathcal{X}\times_X U,A)\to U$ be the base-change of $s$ by
$u:U \hookrightarrow X$.
Using Propositions
\ref{prop:limits-with-respect-ordered-set} and
\ref{prop-most-gen-proj-base-change}, we are reduced to showing that
$\id \to s'_*s'^*$ is an isomorphism.
Now, for every $S\in A$, $\bar{S}\cap U$ is either empty or equal to $U$.
Let $A^{\flat}$ be the subset of $A$ consisting of those $S$'s such that $U\subset \bar{S}$, i.e., $U\preceq S$. Then, by
Corollary \ref{cor:making-the-diag-connected},
we are reduced to showing that
$\id \to t_*t^*$ is invertible with
$t:(U,A^{\flat}) \to U$ given objectwise by $\id_U$. But $A^{\flat}$ has a smallest element, namely the $\mathcal{S}$-stratum $U$. We may now use \cite[Prop.~2.1.41]{ayoub-these-I} to finish the proof.
\end{proof}

\subsection{Direct image along the complement of a \emph{sncd}}
\label{sub:direct-image-sncd}
Let $k$ be a field and $X$ a smooth $k$-scheme. Recall that a \emph{simple normal crossing divisor} (sncd) in $X$
is a divisor $D=\cup_{\alpha\in I} D_{\alpha}$ in $X$ such that
the scheme-theoretic intersection
$D_J=\bigcap_{\beta\in J} D_{\beta}$ is smooth of pure codimension
${\rm card}(J)$ for every $J\subset I$. In particular, we do not allow self-intersections of components in $D$. For the purpose of this article,
we need to extend the notion of \emph{sncd} to $k$-schemes having quotient singularities.

\begin{definition}
\label{defn:sncd-in-quotient-sing}

\begin{enumerate}

\item[(a)] A finite type $k$-scheme $X$ is said to have
\emph{only quotient singularities}
if locally for the \'etale topology, $X$ is the quotient of a smooth $k$-scheme by the action of a finite group with order prime to the exponent characteristic of $k$.

\item[(b)] Let $X$ be a finite type $k$-scheme having only quotient singularities. A \emph{simple normal crossing divisor} (sncd) of $X$ is a Weil divisor $D=\cup_{\alpha\in I} D_{\alpha}$ in $X$ such that all the $D_{\alpha}$ are normal schemes and the following condition is satisfied. Locally for the \'etale topology on $X$, there exist:

\begin{itemize}

\item a smooth affine $k$-scheme $Y$ and
a sncd $F=(F_{\alpha})_{\alpha\in I}$ in $Y$,

\item a finite group $G$ with order prime to the exponent characteristic of $k$, acting on $Y$ and globally fixing each $F_{\alpha}$,

\item an isomorphism $Y/G \simeq X$ sending $F_{\alpha}/G$ isomorphically to $D_{\alpha}$ for all $\alpha\in I$.

\end{itemize}

\end{enumerate}

\end{definition}

For every $J \subset I$, $D_J=\bigcap_{\beta\in J} D_{\beta}$ is, locally for the \'etale topology, the quotient of $F_J=\bigcap_{\beta\in J} F_{\beta}$. Hence it has codimension
${\rm card}(J)$ in $X$ and only quotient singularities. Moreover, $\bigcup_{\alpha\in I- J}(D_{\alpha}\cap D_J)$ is a \emph{sncd} in $D_J$.
If $X$ is smooth, it can be shown that the $D_J$ are necessarily
smooth, and thus $D$
is a \emph{sncd} in the usual sense. However, we omit the proof as this is not needed later.

\begin{proposition}
\label{prop:direct-image-compl-sncd}
Let $k$ be a field and $X$ a quasi-projective $k$-scheme having only quotient singularities.
Let $D=\bigcup_{\alpha \in I} D_{\alpha}$ be a simple normal crossing divisor in $X$ and denote by $j:U \to X$ the inclusion of its complement.
Let $T\subset Z\subset X$ be closed subschemes such that
there exist a subset $J \subset I$ satisfying the following conditions:
\begin{enumerate}

\item[(i)] $Z$ is constructible with respect to the stratification induced by the family $(D_{\beta})_{\beta\in J}$ (as in
\emph{Example \ref{ex:strat-ass-family-subschemes}}),

\item[(ii)] $T$ is contained in $\bigcup_{\alpha\in I- J} D_{\alpha}$.

\end{enumerate}
Put $Z^0=Z- T$ and let $z:Z\to X$ and $u:Z^0\to Z$ denote the inclusions.
Then
the morphism
$$\xymatrix@C=1.7pc{z^*j_*\un_U \ar[r] & u_*u^*z^*j_*\un_U,}$$
given by the unity of the adjunction $(u^*,u_*)$, is invertible.

\end{proposition}

\begin{proof}
We split the proof in two steps. The first one is a reduction to the case where $X$ is smooth (and $D$ is a \emph{sncd} in the usual sense).

\smallskip

\noindent
\underbar{Step 1:}
The problem being local for the \'etale topology on $X$, we may assume that $X=Y/G$ and $D_{\alpha}=F_{\alpha}/G$ with $Y$, $(F_{\alpha})_{\alpha\in I}$ and $G$ as in
Definition \ref{defn:sncd-in-quotient-sing}, (b).
Let $e$ denote the projection $Y \to X$, $V=e^{-1}(U)$,
$Z'=e^{-1}(Z)$ and $T'=e^{-1}(T)$. Then $Z'$ is constructible with respect to the stratification induced by the family $(F_{\beta})_{\beta\in J}$ and $T'$ is contained in $\bigcup_{\alpha\in I- J}F_{\alpha}$. Let $Z'^0=Z'- T'=e^{-1}(Z^0)$.

Consider the commutative diagram
$$\xymatrix@C=1.7pc@R=1.4pc{Z'^0 \ar[r]^-{u'} \ar[d]^-e & Z' \ar[r]^-{z'} \ar[d]^-e & Y \ar[d]^-e & V \ar[l]_-{j'} \ar[d]^-e \\
Z^0 \ar[r]^-u & Z \ar[r]^-z & X & U \ar[l]_-j}$$
where the squares are cartesian (up to nil-immersions).
The group $G$ acts on $e_*\un_V\simeq e_*e^*\un_U$, and the morphism
$\un_U \to e_*e^*\un_U$ identifies $\un_U$ with the image of the projector $\frac{1}{|G|}\sum_{g\in G} g$ (see \cite[Lem.~2.1.165]{ayoub-these-I}). Hence, $\un_U \to e_*e^*\un_U$ admits a retraction $r:e_*e^*\un_U\to \un_U$.
It is then sufficient to show that
\begin{equation}
\label{eq:new-for-reuction-direct-im-quotient}
\xymatrix@C=1.7pc{z^*j_*e_*\un_V \ar[r] & u_*u^*z^*j_*e_*\un_V}
\end{equation}
is invertible. But we have a commutative diagram
$$\xymatrix@C=1.5pc@R=1.5pc{z^*j_*e_* \ar[r]^-{\sim}  \ar[d] & z^*e_*j'_* \ar[d] \ar[r]^-{\sim} & e_* z'^*j'^*\ar[d] \ar@/^1.pc/[drr] & &  \\
u_*u^* z^* j_* e_* \ar[r]^-{\sim} & u_*u^*z^*e_* j'_* \ar[r]^-{\sim} & u_*u^* e_* z'^*j'^* \ar[r]^-{\sim}  & u_*e_*u'^*z'^* j'^* \ar[r]^-{\sim} & e_*u'_*u'^* z'^* j'_* }$$
where the all the horizontal arrows are invertible, either for trivial reasons or because of the base change theorem for projective morphisms \cite[Cor.~1.7.18]{ayoub-these-I} applied to $e$. This shows that
\eqref{eq:new-for-reuction-direct-im-quotient}
is isomorphic to push-forward along $e:Z' \to Z$ of
$z'^* j'_* \un_V \to u'_*u'^* z'^*j'_*\un_V$.
Thus, it suffices to show that the latter is invertible, i.e, we only need to consider the smooth case.

\smallskip

\noindent
\underbar{Step 2:}
{\it We assume now that $X$ is smooth.}
We argue by induction on the dimension of $X$. We may assume $X$ is connected and hence irreducible.
Because the problem is local on $X$, we may assume that each $D_{\alpha}$ is given as the zero locus of some global function in $\mathcal{O}_X(X)$.
Then the normal sheaf $\mathcal{N}_L$ to the closed subscheme $D_L=\bigcap_{\alpha\in L} D_{\alpha} \subset X$ is free for every $L\subset I$.

When $Z=X$, condition (ii) implies that $T\subset \bigcup_{\alpha\in I} D_{\alpha}$, or equivalently that $U\subset Z^0$. In this case, we need to show that
$j_*\un_U \to u_*u^*j_*\un_U$ is an isomorphism. Writing $v$ for the inclusion of $U$ in $Z^0$, so that $j=u\circ v$, we get
$$u_*u^*j_*\simeq u_*u^*u_*v_*\simeq u_*v_*\simeq j_*.$$
This proves our claim in this case.

{\it Next, we assume that $Z\subset X- U$.} Let $J_0\subset J$ be of
minimal cardinality with $Z\subset \bigcup_{\beta\in J_0} D_{\beta}$.  We argue by induction on the cardinality of $J_0$:

\smallskip

\noindent
\emph{First Case:}
First assume that $J_0$ has only one element, i.e., we may find
$\beta_0\in J$ such that $Z\subset D_{\beta_0}$. Write
$z_0:Z\hookrightarrow D_{\beta_0}$ and $d_0:D_{\beta_0} \hookrightarrow X$,
so that
$z=d_0\circ z_0$. With these notations, we need to show that
$$z_0^* (d_0^*j_*\un_U) \!\xymatrix@C=1.5pc{\ar[r] &}\! u_*u^*z_0^* (d_0^*j_*\un_U) $$
is invertible.
Let $D_{\beta_0}^0=D_{\beta_0}- \bigcup_{\alpha\neq \beta_0} D_{\alpha}$, and denote by $e_0:D_{\beta_0}^0 \hookrightarrow D_{\beta}$ the inclusion. By \cite[Th.~3.3.44]{ayoub-these-II},
the morphism
$$d_0^*j_*\un_U \!\xymatrix@C=1.5pc{\ar[r] &}\! e_{0*}e_0^*d_0^*j_*\un_U$$
is invertible. Moreover, as the normal sheaf to $D_{\beta_0}^0$ is assumed to be free, $e_0^*d_0^*j_*\un_U\simeq \un_{D_{\beta_0}^0} \oplus \un_{D_{\beta_0}^0}(-1)[-1]$.
As the Tate twist commutes with the operations of inverse and direct images, we are reduced to showing that
$$z_0^*e_{0*} \un_{D^0_{\beta_0}} \!\!\xymatrix@C=1.5pc{\ar[r] & }\! u_*u^*z_0^* e_{0*} \un_{D^0_{\beta_0}}$$
is invertible. This follows by our induction hypothesis on the dimension of $X$.

\smallskip

\noindent
\emph{Second Case:}
Now we assume that $J_0$ has at least two elements. Fix $\beta_0\in J_0$ and let $J'_0=J_0- \{\beta_0\}$.
Define $Z_0=Z\bigcap D_{\beta_0}$, $Z'=Z\bigcap (\bigcup_{\beta\in J'_0} D_{\beta})$ and $Z'_0=Z\cap Z'$.
Also Let $T_0$, $T'$ and $T_0'$ be the intersection of $T$ with
$Z_0$, $Z'$ and $Z_0'$. Finally, let
$Z_0^0$, $Z'^0$ and $Z'^0_0$ be the complements of $T$ in
$Z_0$, $Z'$ and $Z_0'$.

Writing $t_0$, $t'$ and $t'_0$ for the inclusion of $Z_0$, $Z'$ and $Z_0'$ in $Z$, we have a morphism of distinguished triangles
$$\xymatrix@C=1pc@R=1.5pc{z^*j_*\un_U \ar[r] \ar[d] & t_{0*} t_0^* z^*j_*\un_U \bigoplus t'_*t'^* z^* j_*\un_U \ar[r] \ar[d] & t'_{0*} t'^*_0 z^*j_* \un \ar[r] \ar[d] & \\
u_*u^* z^*j_*\un_U \ar[r]  & u_*u^* t_{0*} t_0^* z^*j_*\un_U \bigoplus u_*u^* t'_*t'^* z^* j_*\un_U \ar[r] & u_*u^* t'_{0*} t'^*_0 z^*j_* \un \ar[r] &  }$$
We are reduced to showing that the second and third vertical arrows are invertible.
We do it only for the second factor of the second arrow as the other cases are similar.
Let $u':Z'^0 \subset Z'$. Then
$u_*u^*t'_*\simeq t'_*u'_*u'^*$. Thus, with $z'=z\circ t'$, it suffices to show that
$z'^* j_*\un_U \to u'_*u'^* z'^* j_*\un_U$
is invertible. This follows from the induction hypothesis, as
$Z'$ is contained in $\bigcup_{\beta \in J_0'}D_{\beta}$ and ${\rm card}(J_0')={\rm card}(J_0)-1$.
The proof of the proposition is complete.
\end{proof}

We note the following corollary for later use.

\begin{corollary}
\label{cor:new-for-compar-beta-prime-et-theta-dp}
Let $P$ be a smooth quasi-projective $k$-scheme and $F=\cup_{\alpha\in I} F_{\alpha}$ a sncd in $P$. Let $G$ be a finite group with order prime to the exponent characteristic of $k$ acting on $P$ and stabilizing the smooth divisors $F_{\alpha}$.
Let $H\subset G$ be a subgroup and set
$X=P/G$, $X'=P/H$, $D_{\alpha}=F_{\alpha}/G$ and $D'_{\alpha}=F_{\alpha}/H$. Call $U$ and $U'$ the complements of the sncd $D=\cup_{\alpha \in I} D_{\alpha}$ and $D'=\cup_{\alpha\in I} D'_{\alpha}$. Let $T\subset Z \subset X$ be as in
\emph{Proposition \ref{prop:direct-image-compl-sncd}} and set $Z^0=Z- T$. We form the commutative diagram with cartesian squares
$$\xymatrix@C=1.5pc@R=1.5pc{Z'^0 \ar[r]^-{u'} \ar[d]^-{d'} & Z' \ar[r]^-{z'} \ar[d]^-d & X' \ar[d]^-c & \ar[l]_-{j'} U' \ar[d]^-{c'} \\
Z^0 \ar[r]^-u & Z \ar[r]^-z &  X &  \ar[l]_-j U}$$
Then, the base change morphism $d^*u_*\to u'_*d'^*$ applied to
$u^*z^*j_*\un_U$ is invertible.

\end{corollary}

\begin{proof}
It suffices to consider the case $Z=X$ and $Z^0=U$. Indeed, assume
that $c^*j_*\un_U \to j'_*c'^*\un_U$ is invertible. From
Proposition \ref{prop:direct-image-compl-sncd} applied over $X'$,
we get that $z'^*j_*\un_{U'} \to u'_*u'^*z'^*j'_*\un_{U'}$ is
invertible. Using, the commutative diagram
$$\xymatrix@C=1.7pc@R=1.5pc{d^*z^*j_*\un_U \ar[d] \ar[r]^-{\sim} & z'^*c^*j_*\un_U \ar[r]^-{\sim} \ar[d] & z'^*j'_*c'^* \un_U \ar[d]^-{\sim} \\
u'_*u'^*d^*z^*j_*\un_U \ar[r]^-{\sim} & u'_*u'^* z'^*c^*j_*\un_U \ar[r]^-{\sim}  & u'_*u'^*z'^*j'_*c'^* \un_U}$$
we get that $d^*z^*j_*\un_U \to u'_*u'^*d^*z^*j_*\un_U$ is invertible.
We conclude using the commutative diagram:
$$\xymatrix@C=1.7pc@R=1.5pc{d^*z^*j_*\un_U \ar[d]_-{\sim} \ar@{=}[rr]  && d^*z^*j_*\un_U \ar[d]^-{\sim} \\
d^*u_*u^*z^*j_*\un_U \ar[r] & u'_*d'^*u^*z^*j_*\un_U \ar[r]^-{\sim} & u'_*u'^*d^*z^*j_*\un_U.}$$

To finish the proof, it remains to show that $c^*j_*\un_U \to j'_* c'^* \un_U$ is invertible. As $(P\to X')^*$ is conservative, we easily reduce to the case $H=1$ and $X'=P$.
If $I$ is empty, there is nothing to show.
Next, assume that $I$ has one element, i.e., $F$ is a smooth divisor.
Let $F_0$ be a connected component of $F$. Then, $P/{\rm Stab}_G(F_0) \to X$ is \'etale in the neighborhood of
$F_0/{\rm Stab}_G(F_0)$. Thus, we may replace $X$ by $P/{\rm Stab}_G(F_0)$ and assume that $G$ globally fixes $F_0$. In other words, we may assume that $F$ is connected and hence irreducible.
Also, the question being local on $P$ (for $G$-equivariant Zariski covers), we may assume that the divisor $F\subset P$
is defined by a single equation $t=0$. Then sending $g\in G$ to
$g^{-1}t/t$ yields a character $\chi:G \to \Gamma(P,\mathcal{O}^{\times})$. When $F$ is geometrically irreducible, which we may assume without loss of generality, this character takes values in $k^{\times}$.

Now, let $W\subset P$ be a globally $G$-invariant open subset such that
$W\cap F$ is non-empty. Assume that our claim is true for the cartesian square
$$\xymatrix@C=1.7pc@R=1.5pc{W- F \ar[r]^-{e'} \ar[d]_-{q'} & (W- F)/G\ar[d]^-q \\
W \ar[r]^-e & W/G,}$$
i.e., $e^*q_*\un_{(W- F)/G} \to q'_* e'^*
\un_{(W- F)/G}$ is invertible.
It follows that $c^*j_*\un_U \to j_*'c'^* \un_U$ is invertible over $W$.
Clearly, both $(c^*j_*\un_U)_{|F}$ and $(j'_*c'^*\un_U)_{|F}$ are isomorphic to $\un_F \oplus \un_F(-1)[-1]$. (This can be derived easily from \cite[Cor.~1.6.2]{ayoub-these-I} and the base change theorem by smooth morphisms
\cite[Prop.~4.5.48]{ayoub-these-II}.)
For all $i, \, j \in \Z$, there is a canonical isomorphism
$$\hom_{\DM(F)}(\un_F,\un_F(i)[j])\simeq \hom_{\DM(k)}(\M(F),\un_{\Spec(k)}(i)[j])$$
given by the adjunction $(p_{F\sharp},p_F^*)$ with $p_F$ the projection of $F$ to $\Spec(k)$ and the fact that $\M(F)=p_{F\sharp}\un_F$.
Using Proposition \ref{prop:compare} and
\cite[Cor.~4.2 and Th. 16.25]{motivic-lectures},
it follows that every endomorphism of
$\un_F \oplus \un_F(-1)[-1]$ is given by a matrix
$$\left(\begin{array}{cc} a & b\\
0 & a'\end{array}\right)$$
where $a, \, a'\in \Q$ and $b\in \mathcal{O}^{\times}(F)\otimes \Q$.
The same holds true for $F$ replaced by $W\cap F$.
As
$c^*j_*\un_U \to j_*'c'^* \un_U$ was assumed to be invertible on $W$ and in particular over $W\cap F$, we deduce that it is also invertible over $F$. This implies that
$c^*j_*\un_U \to j_*'c'^* \un_U$ is an isomorphism.

Replacing $P$ by a well-chosen $W\subset P$ as above, we may assume that
$F\to F/G$ is an \'etale cover.
With $K=\chi^{-1}(1)$, the morphism $P \to P/K$ is then \'etale in the neighborhood of $F$. Thus, we may replace $P$ by $P/K$. In other words, we may assume that $\chi:G \to k^*$ is injective. Then $G$ is cyclic of order $m$ and $P\to P/G$ is locally for the \'etale toplogy, isomorphic to $e_m:\Aff^1_k \times_k F \to \Aff^1_k \times_k F$,
where $e_m$ is the elevation to the $m$-th power. Our claim in this case follows from \cite[Lem.~3.4.13]{ayoub-these-II}
as we work with rational coefficients.

Now we prove the general case by induction on $I$.
By the previous discussion, we may assume that $I$ has more than two elements.
It suffices to show that
$c^*j_*\un_U \to j'_*c'^*\un_U$ is invertible over each divisor $F_i$.
Fix $i_0\in I$ and let $I'=I- \{i_0\}$.
As our problem is local over $P$ (for $G$-equivariant Zariski covers), we may assume that the normal bundle to $F_{i_0}$ is trivial.
Let $F_{i_0}^0=F_{i_0}- \cup_{i\in I'} F_i$ and
consider the commutative diagram with cartesian squares
$$\xymatrix@C=1.7pc@R=1.5pc{F_{i_0}^0 \ar[r]^-{u'} \ar[d]^-{c'_{i_0}} & F_{i_0} \ar[d]^-{c_{i_0}} \ar[r]^-{z'} & P \ar[d]^-c & \ar[l]_-{j'} P- F \ar[d]^-{c'}\\
D_{i_0}^0 \ar[r]^-{u} & D_{i_0} \ar[r]^-{z} & X & U.\! \ar[l]_-j}$$
We know by Proposition \ref{prop:direct-image-compl-sncd} that
$$z^*j_*\un_U\simeq u_*u^*z^*j_*\un_U \simeq u_*u^*(\un_{D^0_{i_0}} \oplus \un_{D^0_{i_0}}(-1)[-1]) \qquad \text{and}$$
$$z'^*j'^* \un_{P- F} \simeq  u'_*u'^*z'^*j'^* \un_{P- F} \simeq u'_*u'^* (\un_{F^0_{i_0}} \oplus \un_{F^0_{i_0}}(-1)[-1]).$$
(Again, the last two isomorphisms follow from
\cite[Cor.~1.6.2]{ayoub-these-I}
and the base change theorem by smooth morphisms
\cite[Prop.~4.5.48]{ayoub-these-II}.)
Moreover, modulo these isomorphisms, the restriction of
$c^*j_*\un_U \to j'_*c'^*\un_U$ to $F_{i_0}$ is isomorphic to
the base change morphism $c_{i_0}^* u_* \to u'_* c'^*_{i_0}$
applied to $\un_{D^0_{i_0}}\oplus \un_{D^0_{i_0}}(-1)[-1]$. Thus we may use the induction hypothesis to conclude.
\end{proof}

\subsection{The Betti realization}
\label{The Betti realization}
In this paragraph we briefly describe the construction of the
Betti realization of relative motives and describe the
compatibilities with the Grothendieck operations. The main reference for
the material in the subsection is \cite{realiz-oper}.

Let $X$ be an analytic space (for example, the space of
$\C$-points of an algebraic variety defined over $\C$).  Let
$\SmAn/X$ be the category of smooth morphisms of analytic spaces
$U\to X$ (called {\it smooth $X$-analytic spaces}). The category
$\SmAn/X$ is a site when endowed with the classical topology and
we denote by $\Shv(\SmAn/X)$ the associated category of sheaves of
$\Q$-vector spaces. Given a smooth $X$-analytic space $Y$, we
let $\Q_{\ctp}(Y)$ denote the sheaf on $\SmAn/X$ associated to the presheaf
of $\Q$-vector spaces freely generated by $Y$
($\ctp$ stands for "classical topology").

Let $\mathbb{D}^1=\{z\in \C; |z|<1\}$ be the unit disc. If $Y$ is an
$X$-analytic space, write
$\mathbb{D}^1_Y$ for the $X$-analytic space $\mathbb{D}^1\times Y$.
As for schemes, there is a $\mathbb{D}^1$-local model structure $(\mathbf{W}_{\mathbb{D}^1},\mathbf{Cof},\mathbf{Fib}_{\mathbb{D}^1})$ on the category
$\mathbf{K}(\Shv(\SmAn/X))$ of complexes of sheaves on $\SmAn/X$
for which the morphisms
$\Q_{\ctp}(\mathbb{D}^1_Y)\to
\Q_{\ctp}(Y)$ are $\mathbb{D}^1$-weak equivalences.
Our construction of the Betti realization is
based on the following proposition which is a particular case of
\cite[Th.~1.8]{realiz-oper}:

\begin{proposition}
\label{prop:main-for-betti-real}
There is a natural
equivalence of categories
\begin{equation}
\label{eq-prop:main-for-betti-real}
\xymatrix{\mathbf{D}(\Shv(X)) \ar[r]^-{\sim} & \mathbf{K}(\Shv(\SmAn/X))[\mathbf{W}_{\mathbb{D}^1}^{-1}]}
\end{equation}
where $\Shv(X)$ is the abelian category of sheaves of $\Q$-vector spaces on the topological space $X$.
\end{proposition}

Now, let $X$ be a quasi-projective scheme defined over a subfield $k$ of $\C$.
Whenever we write ``$X(\C)$'',
we mean the analytic space associated to the $\C$-points of $X$.
The functor $\An_X:\Sm/X \to \SmAn/X(\C)$
that takes an $X$-scheme $Y$ to the $X(\C)$-analytic space
$Y(\C)$ induces an adjunction
$$(\An_X^*,\An_{X*}):\xymatrix{\Shv(\Sm/X)\ar@<.2pc>[r] & \ar@<.1pc>[l]
\Shv(\SmAn/X(\C)).}$$
The (unstable) {\it Betti realization functor} is defined to be the
composition
$$\xymatrix@R=1.7pc{\DM_{\eff}(X)=\mathbf{K}(\Shv(\Sm/X))[\mathbf{W}_{\Aff^1}^{-1}] \hspace{4cm}
\ar@<-1cm>[d]^{{\rm L}\An_X^*} \\
\hspace{2cm} \mathbf{K}(\Shv(\SmAn/X(\C)))[\mathbf{W}_{\mathbb{D}^1}^{-1}] \simeq
\mathbf{D}(\Shv(X(\C))),}$$
and will be denoted simply
$\An_X^*:\DM_{\eff}(X) \to \mathbf{D}(\Shv(X(\C)))$. The
realization of the Tate motive $T_X$ is the constant sheaf $\Q[1]$,
which is already an invertible object. For this reason, $\An_X^*$
can be extended to $T$-spectra, yielding a
stable realization functor
\begin{equation}
\label{eq:stable-realiz-funct}
\An_X^*:\xymatrix@C=1.7pc{\DM(X) \ar[r] &  \mathbf{D}(\Shv(X(\C))).}
\end{equation}

It is shown in \cite{realiz-oper} that the realization functors
\eqref{eq:stable-realiz-funct} respect the four operations
$f^*$, $f_*$, $f_!$ and $f^!$. More precisely, for $f:Y \to X$, there is an isomorphism of
functors $(f^{an})^*\An_X^*\simeq \An_Y^*f^*$
inducing a natural transformation $\An_X^*f_*\to {\rm R}(f^{an})_*\An_Y^*$ which is invertible
when applied to compact motives.
A similar statement holds for the operations $f_!$ and $f^!$,
but will not be used in the paper.
We recall that $M\in \DM(X)$
is said to be {\it compact} when $\hom(M,-)$ commutes with infinite
direct sums, or equivalently, when $M$ is in the triangulated
subcategory generated by the homological motives of smooth
$X$-schemes of finite type.

We end this subsection with a discussion of the Betti realization
over a diagram of schemes. A diagram of analytic spaces is an object of $\Dia({\rm AnSpc})$ where ${\rm AnSpc}$ is the category of analytic spaces. Given a diagram of analytic spaces
$(\mathcal{X},\mathcal{I})$, let ${\rm Ouv}(\mathcal{X},\mathcal{I})$ be the category whose objects are pairs $(U,i)$ with
$i\in \mathcal{I}$ and $U$ an open subset of $\mathcal{X}(i)$.
The classical topology of analytic spaces makes ${\rm Ouv}(\mathcal{X},\mathcal{I})$ into a site whose category of sheaves (with values in the category of $\Q$-vector spaces) will be denoted
$\Shv(\mathcal{X},\mathcal{I})$. The derived category of the latter is denoted $\mathbf{D}(\Shv(\mathcal{X},\mathcal{I}))$.

Now, let $(\mathcal{X},\mathcal{I})$ be a diagram of quasi-projective $k$-schemes. Taking complex points, we obtain a diagram of analytic spaces $(\mathcal{X}(\C),\mathcal{I})$. Moreover, as in the case of a single $k$-scheme, we have a triangulated functor
$$\An^*_{\mathcal{X},\mathcal{I}}:\xymatrix@C=1.7pc{\DM(\mathcal{X},\mathcal{I})
\ar[r] & \mathbf{D}(\Shv(\mathcal{X}(\C),\mathcal{I})).}$$
The details of this construction can be found in
\cite[Sect.~4]{realiz-oper}.


\section{Relative Artin motives and the punctual weight filtration}

\label{sect:artin-zero-slice-weight}

\subsection{Cohomological motives and Artin motives}
We begin with the definitions:

\begin{definition}
\label{defn:cohomological-vs-artin-motive} Let $X$ be a noetherian
scheme. We denote by $\DM_{\coh}(X)$ (resp.~$\DM_0(X)$) the
smallest triangulated subcategory of $\DM(X)$ stable under infinite
sums and containing $\M_{\coh}(U)$ for all quasi-projective
$X$-schemes $U$ (resp.~all finite $X$-schemes $U$). A motive $M \in
\DM_{\coh}(X)$ (resp.~$M\in \DM_0(X)$) is called a {\rm
cohomological motive} (resp.~an {\rm Artin motive}).
\end{definition}

\begin{lemma}
\label{lemma:generation-proj-reg}
Assume that $X$ is of finite
type over a perfect field $k$. Then $\DM_{\coh}(X)$ is the
smallest triangulated subcategory stable under infinite sums
and containing $\M_{\coh}(Y)$ for all $X$-schemes $Y$ that
are projective over $X$ and smooth over $k$.
\end{lemma}

\begin{proof}
We denote by $\DM'_{\coh}(X)$ the smallest triangulated
subcategory stable, etc., as in the statement of the lemma. We
want to prove that $\DM'_{\coh}(X)=\DM_{\coh}(X)$. We clearly have
$\DM'_{\coh}(X)\subset \DM_{\coh}(X)$. As both triangulated
subcategories are stable under infinite sums, we must verify that
for $U$ a quasi-projective $X$-scheme, $\M_{\coh}(U)\in
\DM'_{\coh}(X)$. We argue by induction on the dimension of $U$
over $k$. As
$\M_{\coh}(U)=\M_{\coh}(U_{red})$, we may assume that $U$ is
reduced.

A reduced finite-type $X$-scheme of dimension zero consists of just
points, so it is smooth over
$k$ and projective over $X$. Its cohomological motive is in
$\DM'_{\coh}(X)$ by definition. We may then assume that ${\rm
dim}(U)>0$. We split the proof into two steps.

\smallskip

\noindent
\underbar{Step 1:}
Using de Jong resolution of singularities by alterations \cite{deJong},
we can find:
\begin{itemize}

\item A projective morphism $Y'\to X$ with $Y'$ smooth over $k$,

\item An open subset $U'\subset Y'$ with $Y'- U'$ a simple normal
crossings divisor and an $X$-morphism $e:U'\to U$ projective and generically \'etale.

\end{itemize}

Let $Z\subset U$ be a closed subscheme with everywhere positive
codimension and such that $U'- e^{-1}(Z)\to U- Z$ is an \'etale
cover. We show that $Cone\{M_{\coh}(U)\to \M_{\coh}(Z)\}$ is
isomorphic to a direct factor of $Cone\{\M_{\coh}(U')\to
\M_{\coh}(e^{-1}(Z))\}$. For this, we form the commutative diagram
$$\xymatrix@C=1.7pc@R=1.5pc{U'- e^{-1}(Z) \ar[r]^-{j'} \ar[d]_-{e_0} &  U' \ar[d]^-e & e^{-1}(Z) \ar[l]_-{i'}\ar[d]^-{e_1} \\
U-Z \ar[r]^-j & U & Z \ar[l]_-i }$$ Then $Cone\{\M_{\coh}(U')\to
\M_{\coh}(e^{-1}(Z))\}[-1]$ is isomorphic to the direct image of
$j'_!\un_{U'-e^{-1}(Z)}$ along the projection $U'\to X$. Similarly,
$Cone\{\M_{\coh}(U)\to \M_{\coh}(Z)\}[-1]$ is isomorphic to the
direct image of $j_!\un_{U- Z}$ along the projection $U \to X$.
Thus, we need to show
that $e_*j'_!\un_{U'-e^{-1}(Z)}$ contains $j_!\un_{U-Z}$ as a
direct factor. Using that $e_*j'_!=e_!j'_!=j_!e_{0!}$ we are
reduced to showing that $\un_{U-Z}$ is a direct factor of
$e_{0*}\un_{U'-e^{-1}(Z)}\simeq e_{0*}e_0^*\un_{U-Z}$.
This follows from the first part of
\cite[Lem.~2.1.165]{ayoub-these-I}.

Using the induction hypothesis for $\M_{\coh}(Z)$ and
$\M_{\coh}(e^{-1}(Z))$, we are reduced to showing that
$\M_{\coh}(U')\in \DM_{\coh}'(X)$.

\smallskip

\noindent
\underbar{Step 2}: We return to the original notation.
By Step 1, we may assume that $U$ is the complement of a simple normal
crossing divisor in a projective $X$-scheme $Y$ which is
smooth over $k$.

Let $j:U\subset Y$, $p:Y \to X$ and $q=p\circ j:U \to X$. Then
$\M_{\coh}(U)=q_*\un_U = p_* j_*\un_U$. Let $(D_i)_{i=1, \dots,
n}$ be the irreducible divisors in $Y-U$. For $\emptyset \neq
I\subset [\![1,n]\!]$, we let $D_I=\cap_{i\in I} D_i$ and
$i_I:D_I\subset Y$. Then $j_*\un_U$ is in the triangulated
subcategory of $\DM(Y)$ generated by $\un_Y$ and the following
objects
$$i_{I*}i_{I}^!\un_Y \qquad \text{for} \quad \emptyset\neq I \subset [\![1,n]\!].$$
This follows from
\cite[Prop.~1.4.9]{ayoub-these-I} by standard arguments.
For $\emptyset \neq I \subset [\![1,n]\!]$ denote by
$\mathcal{N}_I$ the normal sheaf of the immersion $i_I$.
The Thom equivalence
${\rm Th}^{-1}(\mathcal{N}_I)$
is the functor $s_I^!p_I^*$ where $p_I$ is the projection of the vector bundle $\mathbb{V}(\mathcal{N}_I)=\Spec(\oplus_{n\in \N} {\rm S}^n(\mathcal{N}_I))$ to $D_I$ and
$s_I$ its zero section.
By \cite[Th.~1.6.19]{ayoub-these-I},
we have an isomorphism $i_I^!\un_Y\simeq {\rm Th}^{-1}(\mathcal{N}_I)\un_{D_I}$.
Moreover, we have for each $\emptyset\neq I \subset [\![1,n]\!]$ a distinguished triangle in $\DM(D_I)$:
$${\rm Th}^{-1}(\mathcal{N}_I)\un_{D_I} \to \M_{\coh}(\Proj(\mathcal{N}_I\oplus \mathcal{O}_{D_I})) \to \M_{\coh}(\Proj(\mathcal{N}_I))\to.$$
The construction of this triangle follows the argument of
\cite[Prop. 2.17(3)]{morel-voevodsky}, which is in the context of $\Aff^1$-homotopy theory.
Taking direct images along $D_I \to X$ and using our earlier observation on $j_*\un_U$, we obtain that $\M_{\coh}(U\to X)$ is in the triangulated subcategory generated by
$\M_{\coh}(Y\to X)$, $\M_{\coh}(\Proj(\mathcal{N}_I)\to X)$ and
$\M_{\coh}(\Proj(\mathcal{N}_I\oplus \mathcal{O}_{D_I})\to X)$ where $\emptyset \neq I \subset [\![1,n]\!]$. This proves that
$\M_{\coh}(U)\in \DM'_{\coh}(X)$.
\end{proof}

\begin{remark}
When $k$ is of characteristic zero, one can
use Hironaka's resolution of singularities \cite{hironaka}
to simplify the argument in Step 1 of the proof of
Lemma
\ref{lemma:generation-proj-reg}.
\end{remark}

\begin{proposition}
\label{prop:cohomological-stability-oper}
For  quasi-projective schemes over a perfect
field $k$.

{\rm 1-} The categories $\DM_{\coh}(-)$ are stable under the following operations:
\begin{enumerate}
\item[(i)] $f^*$, $f_*$ and $f_!$ with $f$ any quasi-projective morphism,

\item[(ii)] $e^!$ with $e$ a quasi-finite morphism (if $k$ is of characteristic zero),

\item[(iii)] tensor product.

\end{enumerate}

{\rm 2-} The categories $\DM_0(-)$ are stable under
the following operations:
\begin{enumerate}

\item[(i$'$)] $f^*$ with $f$ any quasi-projective morphism,

\item[(ii$'$)] $e_!$ with $e$ a quasi-finite morphism,

\item[(iii$'$)] tensor product.

\end{enumerate}

\end{proposition}

\begin{proof}
We consider first the case of cohomological motives.
Fix a quasi-projective morphism
$f:Y \to X$.
The stability
by $f_*$ is clear by the definition of $\DM_{\coh}(-)$ (as $f_*$ commutes with infinite sums).
The stability by $f^*$ follows from
Lemma \ref{lemma:generation-proj-reg}.
Indeed, by the base change theorem for projective morphisms
\cite[Cor.~1.7.18]{ayoub-these-I}, one has
$f^*\M_{\coh}(X')\simeq \M_{\coh}(Y\times_X X')$ for every projective $X$-scheme $X'$.

Stability of $\DM_{\coh}(X)$ with respect to the tensor product
also follows from Lemma \ref{lemma:generation-proj-reg}. Indeed, as
$\otimes_X$ commutes with infinite sums, we are left to show that
$\M_{\coh}(X')\otimes \M_{\coh}(X'')$ is a cohomological motive for
$X'$ and $X''$ projective $X$-schemes. Let $p$ and $q$ denote the projections
of $X'$ and $X''$ to $X$. As $p$ is projective, we have
$p_!\simeq p_*$. Using the projection formula
\cite[Th.~2.3.40]{ayoub-these-I}, we have isomorphisms
$$p_*\un_{X'}\otimes q_*\un_{X''}\simeq
p_!\un_{X'} \otimes q_*\un_{X''} \simeq p_!(\un_{X'}\otimes p^*q_*\un_{X''})\simeq p_*(p^*q_*\un_{X''}).$$
We are done, as $p_*$, $p^*$ and $q_*$ preserves cohomological motives.

We now prove the stability with respect to $f_!$. Let
$p:Y'\to Y$ be a projective
morphism. By Lemma \ref{lemma:generation-proj-reg}, it suffices to
show that $f_!p_*\un_{Y'} \in \DM_{\coh}(X)$. We can form a
commutative square
$$\xymatrix@C=1.5pc@R=1.5pc{ Y' \ar[r]^-j \ar[d]_-p & X' \ar[d]^-g \\
Y \ar[r]^-f & X }$$
with $j$ an open immersion and $g$ a projective morphism. Then
$$f_!p_*\simeq f_!p_!\simeq g_!j_!\simeq g_*j_!.$$
Giving the stability by the operation $g_*$, we only need to show that $j_!\un_{Y'} \in \DM_{\coh}(X')$.
But this is clear as $j_!\un_{Y'}\simeq Cone\{\un_{X'}\to
i_*\un_{X'-Y'}\} [-1]$ for $i$ the inclusion of $X'-Y'$ in $X'$.

Concerning cohomological motives, we still have to prove stability
with respect to $e^!$ for $e:Y \to X$ quasi-finite. We first note that
the case of a closed immersion $i:Y \to X$,
follows from the distinguished triangle
(cf.~\cite[Prop.~1.4.9]{ayoub-these-I})
$$i^!M\to i^*M \to i^*j_*j^*M \to $$
where
$j:X-Y\subset X$ is the complementary open immersion. Indeed by (i)
we know that $i^*M$ and $i^*j_*j^*M$ are cohomological motives for
$M\in
\DM_{\coh}(X)$.

For the general case, we argue by noetherian induction on $X$. If
$\overline{e(Y)}\neq X$, we write
$e^!=e'^!s^!$, with $s:\overline{e(Y)}\subset X$ and $e':Y\to
\overline{e(Y)}$, and then use induction and the case of closed immersions. So we may assume that $e$ is dominant. There
exists a dense open subset $v:V\subset Y$ such that $e_{|V}$ is
\'etale (it is here that we use that $k$ is of characteristic zero). Let
$t:Z=Y-V\subset Y$ be the complementary closed immersion. We then have a distinguished triangle (cf.~\cite[Prop.~1.4.9]{ayoub-these-I})
$$t_*t^!e^!M\to e^!M\to v_*v^!e^!M \to.$$
The functor $(e\circ v)^!=(e\circ v)^*$ preserves cohomological
motives by (i).
Using that $\overline{e(Z)}\neq X$, we see as before (using the induction hypothesis) that
$(e\circ t)^!$ also preserves cohomological motives. This proves
(ii).

As for Artin motives, stability with respect to $f^*$
follows again by base-change.
We prove
stability with respect to $e_!$ for $e:Y \to X$ a quasi-finite
morphism. Let $p:Y'\to Y$ be a finite morphism. We need to show
that $e_!p_*\un_{Y'}$ is an Artin motive. We can find a
commutative square
$$\xymatrix@C=1.5pc@R=1.5pc{ Y' \ar[r]^-j \ar[d]_-p & X' \ar[d]^-g \\
Y \ar[r]^-e & X }$$ with $j$ an open immersion and $g$ a finite
morphism. With $p$ and $g$ finite, we have $p_! =p_*$ and $g_!=g$.
It follows that $e_!p_*\un_{Y'}\simeq g_*j_!\un_{Y'}$. But
again, $j_!\un_{Y'}=Cone\{\un_{X'}\to i_*\un_{X'-Y'}\}[-1]$ for $i$ the
inclusion of $X'-Y'$ in $X'$.
Finally, the stability with respect to the tensor product
is obtained, as in the case of cohomological motives, using the projection formula
\cite[Th.~2.3.40]{ayoub-these-I} and the stability with respect to the operations $f^*$ and $e_!$.
\end{proof}

\begin{lemma}
\label{lemma:tecknix-gener-artin-mot}
Let $X$ be a quasi-projective scheme over a field $k$ of characteristic zero.
The category $\DM_0(X)$ is smallest triangulated subcategory of $\DM(X)$ stable under infinite sums and containing the objects $e_!\un_U$ with $e:U \to X$ \'etale.

\end{lemma}

\begin{proof}
That $e_!\un_U$ is an Artin motive follows from
Proposition \ref{prop:cohomological-stability-oper}, (ii$'$). Let
$\DM_0'(X)$ now denote the smallest triangulated subcategory
of $\DM(X)$ stable under infinite sums and containing the
$e_!\un_U$, with
$e$ as above.
We wish to show $\DM'_0(X)=\DM_0(X)$. For that, we need
to show that for any finite morphism $Y \to X$,
$\M_{\coh}(Y)\in \DM'_0(X)$. We argue by induction on the
dimension of $Y$. As $\M_{\coh}(Y)=\M_{\coh}(Y_{red})$ we may
assume that $Y$ is reduced.

When $Y$ is empty, there is nothing to prove. Otherwise, we may find a dense open subscheme $V\subset Y$ which is \'etale over an affine locally closed subscheme $U\subset X$. Shrinking $U$ and $V$ further, we may assume that
$$V\simeq \Spec(\mathcal{O}(U)[t,u]/(P(t), u Q(t) P'(t)-1))$$
for some polynomials $P, \, Q\in \mathcal{O}(U)[t]$ with $P$ unitary.
By lifting the polynomials $P$ and $Q$ over an affine neighborhood of $U$, we obtain an \'etale morphism $e:W \to X$ such that
the $X$-scheme $V$ is isomorphic to a closed subscheme of $W$. Thus, we have a commutative diagram
$$\xymatrix@C=1.5pc@R=1.5pc{V \ar[r]^-j \ar@/^1.5pc/[rr]^-s \ar[d]_-a & Y \ar@/_/[dr]^-b & W \ar[d]^-e \\
U \ar[rr]^-i & & X}$$
with $e$ and $a$ \'etale, $i$ a locally closed immersion, $j$ an open immersion and $s$ a closed immersion. We let $Z=Y\backslash V$ and $W'=W\backslash V$. We also let $c:Z \to X$ and $e':W'\to X$ be the obvious morphisms.

By the induction hypothesis, we know that $\M_{\coh}(Z)=c_*\un_Z$ is in $\DM'_0(X)$. Using the distinguished triangle
(cf.~\cite[Lem.~1.4.6]{ayoub-these-I})
$$b_*j_!\un_V \to  b_*\un_Y \to c_*\un_Z \to $$
we are reduced to showing that $b_*j_!\un_V$ is in
$\DM'_{0}(X)$. For this, we use another distinguished triangle
(cf.~\cite[Lem.~1.4.6]{ayoub-these-I})
$$e'_!\un_{W'} \to  e_!\un_W \to e_! s_*\un_V \to $$
and the isomorphisms $e_!s_*\un_V\simeq e_!s_!\un_V \simeq b_!j_!\un_V \simeq b_*j_!\un_V$.
This is what we needed to show, as $e$ and $e'$ are \'etale.
\end{proof}

\subsection{The zero part of the punctual weight filtration}
\label{sub-sect:zero-slice-of-the-weight-filt}
We expect that every object in the heart of the conjectural
motivic $t$-structure on $\DM(X)$ has punctual weights which are compatible under realization with the punctual weights of $\ell$-adic sheaves as defined in the introduction of
\cite{Weil-II}. In this paragraph, we will present a candidate for
the derived version of the punctual weight
zero truncation
restricted to cohomological motives.

\begin{definition}
\label{defn:i-x-nu-x-omega-x}
Let $X$ be a noetherian scheme.
\begin{enumerate}
\item [i)]
Denote by $\nu^0_X:\DM_{\coh}(X) \to \DM_0(X)$ the right adjoint
to the inclusion ${\rm i}_X:\DM_0(X) \hookrightarrow \DM_{\coh}(X)$. If
$M$ is a cohomological motive over $X$, $\nu^0_X(M)$ is called the
\emph{punctual weight zero part} of $M$.
\item [ii)]
Put $\omega^0_X={\rm i}_X\circ \nu^0_X:\DM_{\coh}(X) \to
\DM_{\coh}(X)$ and also call $\omega^0_X(M)$ the
\emph{punctual weight zero
part} of $M$. We then have a natural transformation $\delta_X:\omega^0_X\to
\id$, given by the counit of
the adjunction between ${\rm i}_X$ and $\nu^0_X$.

\end{enumerate}

\end{definition}

The existence of a right adjoint to the inclusion ${\rm i}_X$
follows from a general principle.
Specifically, let $\mathcal{T}$ and $\mathcal{T}'$ be compactly
generated triangulated categories with infinite sums. A
triangulated functor $F:\mathcal{T} \to \mathcal{T}'$ admits
a right adjoint if and only if it commutes with infinite sums (see,
for example, \cite[Cor.~2.1.22]{ayoub-these-I}). Moreover, if $F$ preserves
compact objects, its right adjoint commutes with
infinite sums (see \cite[Lem.~2.1.28]{ayoub-these-I}). In particular, $\nu_X^0$ and $\omega_X^0$ commute
with infinite sums.

\begin{remark}
We give some motivation for our terminology in Definition
\ref{defn:i-x-nu-x-omega-x}. If $E$ is an Artin motive over a scheme $X$,
defined over a finite field, its $\ell$-adic realization has the property
that all of its cohomology sheaves (for the standard $t$-structure) have punctual weight zero in the sense of Deligne (see the introduction of \cite{Weil-II}). In fact, more is true as the eigenvalues of Frobenius are roots of unity. Let's say that a complex of $\ell$-adic sheaves on $X$ has punctual weight less or equal to $n$ if its homology sheaves (for the usual $t$-structure) are so in the sense of Deligne. Then, if $M$ is a cohomological motive, we believe that its $\ell$-adic realization has a universal map from a complex of $\ell$-adic sheaves which is of punctual weight less than or equal to zero. We also predict that the latter is given by the $\ell$-adic realization of $\omega^0_X(M)$.
\end{remark}

\begin{remark}
The functors $\nu_X^0$ and $\omega_X^0$ can be extended to all motives (not only the cohomological ones). Indeed, the inclusion
$\DM_0(X) \hookrightarrow \DM(X)$ has a right adjoint $v$ which coincides with $\nu_X^0$ when applied to cohomological motives. However, for general $M\in \DM(X)$,
$v(M)$ is not a reasonable candidate for the $0$-part of the hypothetical punctual weight filtration. Indeed, based on
\cite{slice-filt}, one can show that $v$ does not preserve compact motives even when $X$ is the spectrum of a field. This problem disappears if we restrict to cohomological motives (cf.~Proposition
\ref{prop:additional-prop-omega-0-x}, (vii) below).

\end{remark}

The rest of Section 2 is devoted to developing the properties of $\omega_X^0$.
First, as
${\rm i}_X$ is a full embedding, we have immediately:

\begin{proposition}
\label{prop:univ-property-for-omega-addenda}
For $M\in \DM_{\coh}(X)$,
$\delta_X:\omega_{X}^0(M) \to M$ is the universal morphism from an Artin motive to $M$.
More precisely, every morphism $a:L\to M$,
from an Artin motive $L$, factors uniquely as
$$\xymatrix{L \ar@{.>}[r] \ar@/^1.3pc/[rr]^-{a} & \omega^0_X(M) \ar[r] & M.}$$
In other words, the composition with $\delta_X(M)$ induces a bijection
$$\hom_{\DM(X)}(L,\omega^0_X(M)) \xymatrix{\ar[r]^-{\sim} &}\hom_{\DM(X)}(L,M).
$$
\end{proposition}

\begin{proposition}
\label{prop:main-computation-omega0-pi-0}
Let $X$ be a
quasi-projective scheme over a field $k$ of characteristic zero.
Let $Y$ be a smooth and projective $X$-scheme and consider its
Stein factorization $Y\to \pi_0(Y/X) \to X$. The induced morphism
$\M_{\coh}(\pi_0(Y/X))\to \M_{\coh}(Y)$ factors uniquely through
$\M_{\coh}(\pi_0(Y/X))\to \omega^0_X(\M_{\coh}(Y))$, and the latter
is an isomorphism.
\end{proposition}
\begin{proof}
In the Stein factorization, $\pi_0(Y/X) \to X$ is finite and $Y\to
\pi_0(Y/X)$ has geometrically connected fibers (see
\cite[Rem.~4.3.4]{ega}). Moreover, this factorization is
characterized by these two properties up to universal
homeomorphisms. From this we deduce, for every finite
type $X$-scheme $X'$, a canonical isomorphism
\begin{equation}
\label{eq:change-base-stein-fact}
\pi_0(Y/X)\times_X X'\simeq \pi_0(Y\times_X X'/X').
\end{equation}
(Use that the two $X'$-schemes above are \'etale,
$Y$ being smooth over $X$.)

The existence of $\M_{\coh}(\pi_0(Y/X))\to \omega^0_X(\M_{\coh}(Y))$ follows
from the universal property of $\omega^0_X$, as
$\M_{\coh}(\pi_0(Y/X))$ is an Artin motive. We need to show that
this morphism is an isomorphism. It then suffices to show that
$\M_{\coh}(\pi_0(Y/X))\to \M_{\coh}(Y)$ satisfies the
universal property of Proposition
\ref{prop:univ-property-for-omega-addenda}, i.e, for any Artin motive $L$ on $X$, the
homomorphism
\begin{equation}
\label{eq-prop:main-computation-omega0-pi-0}
\hom(L,\M_{\coh}(\pi_0(Y/X))) \to \hom(L,\M_{\coh}(Y))
\end{equation}
is a bijection. We split the proof into three steps.

\smallskip

\noindent
\underbar{Step 1:} By Lemma
\ref{lemma:tecknix-gener-artin-mot},
it is enough to check that
\eqref{eq-prop:main-computation-omega0-pi-0} is a bijection for
$L=e_!\un_U[r]$ with $r\in \Z$ and $e:U \to X$ \'etale. By adjunction, base-change and the fact that $e^!= e^*$ for $e$ \'etale,
we see that
\eqref{eq-prop:main-computation-omega0-pi-0} can be written
$$\hom(\un_U[r],\M_{\coh}(\pi_0(Y/X)\times_X U)) \to
\hom(\un_U[r],\M_{\coh}(Y\times_X U)).$$
By \eqref{eq:change-base-stein-fact}, we know that
$\pi_0(Y/X)\times_X U\simeq \pi_0((Y\times_X U)/U)$.
Thus we are reduced to showing that
\eqref{eq-prop:main-computation-omega0-pi-0}
is bijective for $L=\un_{X}[r]$.

We label our morphisms of $k$-schemes:
$$\xymatrix{Y \ar[r]_-g  \ar@/^1.3pc/[rr]^-f & \pi_0(Y/X) \ar[r]_-e & X \ar[r]^-p & \Spec(k).}$$
Recall that $\M_{\coh}(Y)=f_*\un_Y$ and
$\M_{\coh}(\pi_0(Y/X))=e_*\un_{\pi_0(Y/X)}$.
Using adjunction, we can
write \eqref{eq-prop:main-computation-omega0-pi-0} when $L=\un_X[r]$ as
\begin{equation}
\label{eq-prop:main-computation-omega0-pi-0-0-313}
\hom_{\DM(\pi_0(Y/X))}(\un[r],\un)\to \hom_{\DM(Y)}(\un[r], \un).
\end{equation}
The homomorphism above is given by the action of the functor $g^*$
on morphisms as $g^*\un_{\pi_0(Y/X)}=\un_Y$.

\smallskip

\noindent
\underbar{Step 2:} In this step, we reduce to check
that \eqref{eq-prop:main-computation-omega0-pi-0-0-313} is invertible in the case where $X$ is smooth over $k$.
We argue by induction on the dimension of $X$.
Using resolution of singularities, we may find a cartesian square
$$\xymatrix@C=1.5pc@R=1.5pc{E \ar[r]^-{j} \ar[d]_-{q}  & X' \ar[d]^-p \\
Z \ar[r]^-i & X}$$
with $p$ a blow-up, $X'$ smooth over $k$, $i$ and $j$ closed immersions of non-zero codimension everywhere, and such that $X'\backslash E \to (X\backslash Z)_{red}$ is an isomorphism. We deduce two similar cartesian squares
$$\xymatrix@C=1.5pc@R=1.5pc{\pi_0(Y\times_X E/E) \ar[r]^-j \ar[d]_-q   & \pi_0(Y\times_X X'/X') \ar[d]^-p \\
\pi_0(Y\times_X Z/Z) \ar[r]^-i & \pi_0(Y/X)} \qquad
\xymatrix@C=1.5pc@R=1.5pc{Y\times_X E \ar[r]^-j \ar[d]_-q  & Y\times_X X' \ar[d]^-p \\
Y\times_X Z \ar[r]^-i & Y.}$$
Let $t=p\circ j=i\circ q$.
We have two distinguished triangles
$$\un_{\pi_0(Y/X)} \to p_*\un_{\pi_0(Y\times_XX'/X')}\oplus i_*\un_{\pi_0(Y\times_X Z/Z)} \to t_*\un_{\pi_0(Y\times_X E/E)} \to$$
$$\text{and} \quad
\un_{Y} \to p_*\un_{Y\times_XX'}\oplus i_*\un_{Y\times_X Z} \to t_*\un_{Y\times_X E} \to.$$
(They are obtained by showing that
$Cone\{\un \to p_*\un\oplus i_*\un\}\!\xymatrix@C=1.3pc{\ar[r] &}\!
t_*\un$ is invertible, which follows from locality
\cite[Cor.~4.5.47]{ayoub-these-II} and the base change theorem
for projective morphisms \cite[Cor~1.7.18]{ayoub-these-I}.)
Using the five Lemma and then adjunction, we are reduced to showing that
$$\hom_{\DM(\pi_0(Y\times_X \dagger /\dagger))}(\un[r],\un)\to \hom_{\DM(Y\times_X \dagger)}(\un[r], \un)$$
is invertible for $\dagger\in \{X',Z, E\}$.
We are done as $X'$ is smooth and $Z$ and $E$ have dimension strictly smaller than ${\rm dim}(X)$.

\smallskip

\noindent
\underbar{Step 3:} It remains to check that
\eqref{eq-prop:main-computation-omega0-pi-0-0-313}
is bijective assuming that $X$ is smooth. In this case,
$Y$ and $\pi_0(Y/X)$ are also smooth.
Using Proposition \ref{prop:compare} and \cite[Cor.~4.2 and Th. 16.25]{motivic-lectures}, we get isomorphisms
$$\hom_{\DM(U)}(\un[r],\un)\simeq \hom_{\DM(k)}(\M(U)[r],\un)\simeq
{\rm H}^{-r}_{\rm Nis}(U,\Q)=\left\{\begin{array}{ccc} \Q^{\pi_0(U)} & \text{if} & r=0,\\
0 & \text{if} & r\neq 0,\end{array}\right.$$
for every smooth $k$-scheme $U$. (In the above $\pi_0(U)$ denotes the set of connected components of $U$.) We are done as $Y$ and $\pi_0(Y/X)$ have the same set of connected components.
\end{proof}

The statement of Proposition \ref{prop:main-computation-omega0-pi-0}
can be slightly generalized as follows:

\begin{corollary}
\label{cor:of-main-comp-omega0-pi-0}
Keep the notation and hypothesis of {\rm Proposition
\ref{prop:main-computation-omega0-pi-0}.} Let $U\subset Y$ be an
open subscheme such that $Y-U$ is a simple normal crossing divisor
relative to $X$, i.e., $Y-U=\cup_{i\in I} D_i$ with
$D_J=\cap_{j\in J} D_j$ smooth over $X$ and of codimension ${\rm
card}(J)$ for all $\emptyset \neq J \subset I$. Then
$\M_{\coh}(\pi_0(Y/X))\to \M_{\coh}(U)$ identifies
$\M_{\coh}(\pi_0(Y/X))$ with $\omega^0_X(\M_{\coh}(U))$.

\end{corollary}

\begin{proof}
Proposition
\ref{prop:main-computation-omega0-pi-0}
gives the analogous assertion for $Y$ instead of $U$. We
show that
$\omega_X^0(\M_{\coh}(U))\to\omega_X^0(\M_{\coh}(Y))$ is an isomorphism,
and for that,
it suffices to show that $\omega_X^0(Cone\{\M_{coh}(Y)\to \M_{coh}(U)\})=0$.

We use the notation and construction from Step 2 of the proof of Lemma \ref{lemma:generation-proj-reg}.
One sees by
basically the same argument that $K=Cone\{\M_{\coh}(Y)\to \M_{\coh}(U)\}$ is
in the
triangulated subcategory of $\DM(X)$ generated by the objects
$$
C_J=Cone\{\M_{\coh}(\mathbb{P}(\mathcal{N}_J\oplus\mathcal{O}_{D_J})\to X) \!\xymatrix@C=1.4pc{\ar[r] & }\!
\M_{\coh}(\mathbb{P}(\mathcal{N}_J) \to X)\}
$$
for $\emptyset\neq J \subset I$.

For a locally free $\mathcal{O}_{D_J}$-module $\mathcal{M}$ of strictly positive rank, $\pi_0(\mathbb{P}(\mathcal{M})/X)\simeq \pi_0(D_J/X)$. Moreover, as $D_J$ is smooth and projective over $X$,
Proposition \ref{prop:main-computation-omega0-pi-0} implies that
$$\omega^0_X(\M_{\coh}(\mathbb{P}(\mathcal{M}) \to X))\simeq \M_{\coh}(\pi_0(\mathbb{P}(\mathcal{M})/X))\simeq \M_{\coh}(\pi_0(D_J/X)).$$
It
follows that
$\omega_X^0(C_J)=0$ for all $\emptyset \neq J \subset I$, and hence $\omega_X^0(K)=0$ as well.
\end{proof}

For the next corollary of Proposition \ref{prop:main-computation-omega0-pi-0}, we introduce the following terminology
\cite{smooth-mot}.

\begin{definition}
Let $X$ be a noetherian scheme. We let $\DM_{\coh}^{\rm
sm}(X)$ be the smallest triangulated
subcategory of $\DM(X)$ closed under infinite sums and containing
$\M_{\coh}(Y)$ whenever $Y$ is a smooth and projective
$X$-scheme. Motives in $\DM_{\coh}^{\rm sm}(X)$ are called
{\rm smooth cohomological motives}.

\end{definition}

The proof of Corollary \ref{cor:of-main-comp-omega0-pi-0}
shows that the cohomological motive of the complement in a smooth and projective $X$-scheme of a relative \emph{sncd} is a smooth motive.

\begin{corollary}
\label{invertible}
Let $X$ be a quasi-projective scheme over a field $k$ of
characteristic zero. Let $M$ be a smooth cohomological motive on
$X$. Then $\omega^0_X(M)$ is a smooth motive. Moreover, for any quasi-projective morphism $f:X'\to X$, the natural
morphism (cf.~\emph{Proposition \ref{prop:additional-prop-omega-0-x}, (ii)})
$$\xymatrix@C=1.5pc{f^*\omega^0_X(M)\ar[r] & \omega^0_{X'} (f^*M)}$$
is invertible.
\end{corollary}

\begin{proof}
That $\omega^0_X(M)$ is a smooth motive if $M$ is a smooth motive follows from
Proposition \ref{prop:main-computation-omega0-pi-0}
and the fact that $\pi_0(Y/X)\to X$ is an \'etale cover when $Y$ is a smooth and projective $X$-scheme.

Now, let $M$ be a smooth motive over $X$. Applying $f^*$ to $\omega^0_X(M) \to M$ we obtain a morphism
$f^*\omega^0_X(M) \to f^*(M)$ from an Artin motive to a
cohomological motive. It factors uniquely through
$f^*\omega^0_X(M)\to \omega^0_{X'} (f^*M)$. This is the natural
morphism in question.

To show that this morphism is invertible for smooth cohomological
motives, it suffices to consider the case $M=\M_{\coh}(Y)$ for $Y$
a smooth and projective $X$-scheme.  Our assertion follows
then from Proposition \ref{prop:main-computation-omega0-pi-0} and
the isomorphism \eqref{eq:change-base-stein-fact}.
\end{proof}

\begin{remark}
The assertion of the corollary above is false for non-smooth cohomological motives.
Proposition \ref{prop:cube-strat-F-omega-j-star} below can be used to construct examples where it fails.
\end{remark}

The next proposition, whose proof occupies the rest of
this subsection, gives some additional properties of the functors $\omega_X^0$.

\begin{proposition}
\label{prop:additional-prop-omega-0-x}
Let $X$ be a
quasi-projective scheme over a field $k$ of characteristic zero.
The functors $\omega^0_X$ and its coaugmentation $\delta_X:\omega^0_X\to \id$ satisfy
the following:

\begin{enumerate}

\item[(i)] If $L$ is an Artin motive over $X$, we have an
isomorphism $\delta_X:\omega_X^0(L)\overset{\sim}{\to} L$. In
particular, the natural transformation
$\delta_X(\omega_X^0):\omega_X^0\circ \omega_X^0 \overset{\sim}{\to}
\omega_X^0$ is invertible. Moreover, $\delta_X(\omega_X^0)=\omega_X^0(\delta_X)$.

\item[(ii)] Let $f:Y \to X$ be a quasi-projective morphism. There is a natural transformation $\alpha_f: f^*\omega_X^0 \to \omega_Y^0 f^*$ making the triangles
$$\xymatrix@C=1.5pc@R=1.5pc{f^*\omega_X^0 \ar[r]^-{\alpha_f} \ar@/_/[dr]_-{f^*(\delta_X)}  & \omega_Y^0f^* \ar[d]^-{\delta_Y(f^*)} \\
& f^*}
\quad \text{and} \quad
\xymatrix@C=1.5pc@R=1.5pc{\omega^0_Y f^* \omega^0_X \ar@/_/[drr]_-{\omega^0_Yf^*(\delta_X)} \ar[rr]^-{\delta_Y(f^*\omega^0_X)} & & f^* \omega^0_X \ar[d]^-{\alpha_f}  \\
& & \omega_Y^0 f^* }$$
commutative.
Moreover, $\alpha_f$ is invertible when $f$ is smooth.

\item[(iii)] Let $f:Y \to X$ be a quasi-projective morphism.
The natural transformation $\omega^0_X f_* \omega^0_Y \to \omega^0_X f_*$, obtained by applying $\omega^0_X f_*$ to $\delta_Y$, is invertible. Moreover, there exists a natural transformation $\beta_f:\omega^0_X f_* \to f_*\omega^0_Y$ such that:

\begin{enumerate}

\item the following two triangles
$$\xymatrix@C=1.5pc@R=1.5pc{\omega^0_Xf_* \ar[r]^-{\beta_f} \ar@/_/[dr]_-{\delta_X(f_*)} & f_*\omega^0_Y \ar[d]^-{f_*(\delta_Y)} \\
& f_*} \quad \text{and} \quad
\xymatrix@C=1.5pc@R=1.5pc{\omega^0_X f_* \omega^0_Y \ar[rr]^-{\omega^0_X\!f_*(\delta_Y)} \ar@/_/[drr]_-{\delta_X(f_*\omega_Y^0)} && \omega^0_X f_* \ar[d]^-{\beta_f} \\
&& f_*\omega_Y^0}$$
commute,

\item $\omega^0_X(\beta_f)$ is invertible for any $f$,

\item $\beta_f$ is invertible when $f$ is finite.

\end{enumerate}

\item[(iv)] Let $e:Y \to X$ be a quasi-finite morphism.
There exists a natural transformation $\eta_e:e_!\omega_Y^0 \to \omega_X^0 e_!$ making the triangles
$$\xymatrix@C=1.5pc@R=1.5pc{e_!\omega_Y^0 \ar[r]^-{\eta_e} \ar@/_/[dr]_-{e_!(\delta_Y)} & \omega_X^0 e_! \ar[d]^-{\delta_X(e_!)} \\
& e_!} \quad \text{and} \quad
\xymatrix@C=1.5pc@R=1.5pc{\omega^0_X e_! \omega^0_Y \ar@/_/[drr]_-{\omega^0_X e_! (\delta_Y)} \ar[rr]^-{\delta_X(e_!\omega^0_Y)} & & e_!\omega^0_Y \ar[d]^-{\eta_e} \\
& & \omega^0_X e_!}$$
commutative.
Moreover, when $e$ is finite, $\eta_e$ is invertible and coincides with $\beta_e^{-1}$ modulo the natural isomorphism $e_!\simeq e_*$.

\item[(v)] Let $e:Y \to X$ be a quasi-finite morphism.
The natural transformation $\omega^0_Y e^! \omega^0_X \to \omega^0_Y e^!$, obtained by applying $\omega^0_Y e^!$ to $\delta_X$, is invertible. Moreover, there exists
a natural transformation $\gamma_e:\omega_Y^0 e^!\to e^!\omega_X^0$ such that:

\begin{enumerate}

\item the following two triangles
$$\xymatrix@C=1.5pc@R=1.5pc{\omega^0_Y e^! \ar[r]^-{\gamma_e}\ar@/_/[dr]_-{\delta_Y (e^!)} & e^!\omega^0_X \ar[d]^-{e^!(\delta_X)}  \\
& e^!} \quad \text{and} \quad \xymatrix@C=1.5pc@R=1.5pc{\omega^0_Y e^!\omega^0_X \ar[rr]^-{\omega^0_Ye^!(\delta_X)} \ar@/_/[drr]_-{\delta_Y(e^!\omega^0_X)} & & \omega^0_Y e^! \ar[d]^-{\gamma_e} \\
& & e^!\omega^0_X}$$
commute,

\item $\omega^0_Y(\gamma_e)$ is invertible for any quasi-finite $e$,

\item $\gamma_e$ is invertible when $e$ is \'etale.

\end{enumerate}

\item[(vi)] Let $U\subset X$ be an open subscheme with complement $Z=X-U$,
and $j:U\hookrightarrow X$ and $i:Z\hookrightarrow X$ be the inclusions.
Let $M\in \DM_{\coh}(X)$ and assume that $j^*M\in \DM_{0}(U)$.
Then the morphism
$i^*\omega^0_X(M) \to \omega^0_Z(i^*M)$ is invertible.

\item[(vii)] The functor $\omega^0_X$ preserves compact objects.

\end{enumerate}

\end{proposition}

\begin{proof}
The first statement in property (i) is clear from the universal property of
$\omega^0_X(M) \to M$ for cohomological motives $M$ over $X$.
The equality $\delta_X(\omega^0_X)=\omega^0_{X}(\delta_X)$ follows from the commutative square
$$\xymatrix{\omega_X^0\omega_X^0(M)\ar[rr]_-{\sim}^-{\omega^0_X(\delta_X(M))} \ar[d]_-{\delta_X(\omega^0_X(M))}^-{\sim} && \omega^0_X(M)\ar[d]^-{\delta_X(M)}\\
\omega^0_X(M)\ar[rr]^-{\delta_X(M)} && M }$$
and the universal property (and more precisely the uniqueness of the factorization through $\omega^0_X(M)$).

As for (ii), the
natural transformation $\alpha_f$ has already appeared in Corollary
\ref{invertible}, where its restriction to $\DM_{\coh}^{\rm sm}(X)$ was
shown to be invertible. Recall its construction.
For $M\in \DM_{\coh}(X)$, consider the morphism
$f^*(\delta_X): f^*\omega^0_X(M) \to f^*(M)$. By Proposition
\ref{prop:cohomological-stability-oper}, $f^*\omega^0_X(M)$ is an Artin motive. By the universal property of $\omega^0_Y$,
$f^*(\delta_X)$ factors uniquely through $\omega^0_Yf^*(M)$ yielding
$\alpha_f(M):f^*\omega^0_X(M) \to \omega^0_Y f^*(M)$.
The commutation of the first triangle in (ii) is clear from the above construction. For the commutation of the second triangle in (ii), we use the commutative diagram
$$\xymatrix{&& \ar@/_/[dll]_-{\omega^0_Y(f^*\delta_X)}  \omega^0_Y f^* \omega^0_X \ar[rr]^-{\delta_Y(f^* \omega^0_X)} \ar[d]^-{\omega^0_Y(\alpha_f)} & & f^*\omega^0_X \ar[d]^-{\alpha_f} \\
\omega^0_Yf^*  &&  \omega^0_Y \omega^0_Y f^* \ar[ll]^-{\sim}_-{\omega^0_Y (\delta_Y f^*)} \ar[rr]^-{\delta_Y(\omega^0_Y f^*)}_-{\sim} &  &  \omega^0_Y f^*  }$$
and the equality $\omega_Y^0(\delta_Y)=\delta_Y(\omega^0_Y)$
of (i).
The verification that $\alpha_f$ is invertible for
smooth $f$, will be postponed to the end of the proof.

In (iii), the natural transformation $\beta_f$ is the
composition
$$\xymatrix{\omega_X^0 f_*\ar[r] & f_*f^*\omega_X^0f_*\ar[r]^-{f_*\alpha_f f_*} &   f_*\omega_Y^0f^*f_*\ar[r] & f_*\omega_Y^0.}$$
The commutation of the first triangle follows from the more precise commutative diagram
$$\xymatrix{\omega_X^0 f_* \ar[r] \ar[d]^-{\delta_X f_*} & f_*f^*\omega_X^0 f_* \ar[r]^-{f_*\alpha_ff_*} \ar[d]^-{f_*f^*\delta_Xf_*} & f_*\omega_Y^0f^*f_* \ar[r] \ar[d]^-{f_*\delta_Yf^*f_*} & f_*\omega_Y^0\ar[d]^-{f_*\delta_Y} \\
 f_* \ar[r]  & f_*f^* f_* \ar@{=}[r] & f_*f^*f_* \ar[r] & f_*}$$
where the composition in the bottom line is the identity of $f_*$.
Note that the commutation of the middle square follows from the commutation of the triangle in (ii).
For the commutation of the second triangle in (iii), we use the commutative diagram
$$\xymatrix{&& \ar@/_/[dll]_-{\delta_X(f_*\omega^0_Y)} \omega_X^0 f_*\omega_Y^0 \ar[rr]^-{\omega^0_Xf_*(\delta_Y)} \ar[d]^-{\beta_f(\omega^0_Y)} && \omega^0_Xf_* \ar[d]^-{\beta_f} \\
f_*\omega^0_Y  && f_*\omega_Y^0\omega_Y^0 \ar[rr]^-{f_*\omega^0_Y(\delta_Y)}_-{\sim} \ar[ll]^-{\sim}_-{f_*\delta_Y(\omega^0_Y)}   && f_*\omega_Y^0}$$
and the equality $\delta_Y(\omega^0_Y)=\omega^0_Y(\delta_Y)$
of (i).

We now show property (b). Applying $\omega^0_X$ to the commutative triangles from (a) we get
$$\xymatrix{\omega^0_X \omega^0_Xf_* \ar[r]^-{\omega^0_X \beta_f} \ar@/_/[dr]_-{\omega^0_X \delta_X f_*}^-{\sim} & \omega^0_X f_*\omega^0_Y \ar[d]^-{\omega^0_X f_* \delta_Y} \\
& \omega^0_Xf_*} \quad \text{and} \quad
\xymatrix{\omega^0_X \omega^0_X f_* \omega^0_Y \ar[rr]^-{\omega^0_X\omega^0_Xf_* \delta_Y} \ar@/_/[drr]_-{\omega^0_X \delta_X f_*\omega^0_Y }^-{\sim} && \omega^0_X \omega^0_X f_* \ar[d]^-{\omega^0_X\beta_f} \\
&& \omega^0_X f_*\omega_Y^0.\!}$$
The diagonal arrows are indeed invertible as $\omega^0_X(\delta_X)$ is invertible by (i).
This shows that $\omega^0_X(\beta_f)$ has a right and a left inverse.
Using the first triangle above, we see also that
$\omega^0_X f_* \delta_Y$ is also invertible, which is our first claim in (iii).
Property (c) follows from (b).
Indeed, as $f$ is finite, $f_*$
preserves Artin motives. This implies that
the right vertical arrow in the commutative square
$$\xymatrix{\omega^0_X \omega^0_Xf_* \ar[rr]^-{\omega^0_X(\beta_f)}_-{\sim} \ar[d]_-{\delta_X(\omega^0_Xf_*)}^-{\sim}  & & \omega^0_X f_*\omega^0_Y \ar[d]^-{\delta_X(f_*\omega^0_Y)} \\
\omega^0_X f_* \ar[rr]^-{\beta_f} & & f_*\omega^0_Y}$$
is invertible, hence $\beta_f$ is likewise.

The part of Property (iv) concerning general quasi-finite morphisms is proved using the same arguments as in the proof of (ii).
That $\eta_e$ is invertible and coincides with $\beta_e^{-1}$ when $e$ is finite follows from
part (c) of (iii) and Lemma
\ref{lemma:for-the-long-prop-e-!-e-*} below.
Indeed, the vertical arrows in
\eqref{eq-prop:additional-prop-omega-0-x-00831} are then invertible.

\begin{lemma}
\label{lemma:for-the-long-prop-e-!-e-*}
Let $e:Y \to X$ be a quasi-finite morphism. The square
\begin{equation}
\label{eq-prop:additional-prop-omega-0-x-00831}
\xymatrix@C=1.5pc@R=1.5pc{e_!\omega^0_Y \ar[r]^-{\eta_e} \ar[d] &  \omega^0_X e_!\ar[d] \\
e_*\omega^0_Y & \omega^0_Xe_* \ar[l]_-{\beta_e}}
\end{equation}
commutes.
\end{lemma}

\begin{proof}
The square \eqref{eq-prop:additional-prop-omega-0-x-00831}
is part of a larger diagram
$$\xymatrix@C=1.8pc@R=1.9pc{\ar@/^1.9pc/[rrrr]^-{\eta_e} e_!\omega^0_Y \ar[d] & & \ar[ll]_-{(\star)}^-{\delta_X e_! \omega^0_Y} \omega^0_X e_!\omega^0_Y \ar[rr]_-{\omega_X^0 e_!\delta_Y} \ar[d] && \omega^0_X e_! \ar[d]\\
e_*\omega^0_Y & & \ar[ll]_-{\delta_X e_* \omega^0_Y}  \omega^0_X e_*\omega^0_Y \ar[rr]_-{(\star)}^-{\omega^0_X e_*\delta_Y} &&  \omega^0_X e_*\ar@/^1.9pc/[llll]^-{\beta_e}.\!}$$
The two squares and the two triangles that constitute the above diagram are commutative. Hence, it suffices to show that the two arrows labeled with a $(\star)$ are invertible.
But $\delta_X e_! \omega^0_Y$ is invertible as
$e_!\omega^0_Y$ takes values in the category of Artin motives.
Also $\omega^0_X e_*\delta_Y$ is invertible by
Proposition \ref{prop:additional-prop-omega-0-x}, (iii).
\end{proof}

We return to the proof of Proposition
\ref{prop:additional-prop-omega-0-x}.
Property (v) is proven in the same way as (iii). We leave the
details to the reader.
Property (vi) follows easily from Lemma
\ref{lemma:distinguished-triang-for-long-prop} below.
Indeed, as $j^!M=j^*M$ is an Artin motive by hypothesis,
$j_!j^!(M)$ is also Artin and thus $\eta_j:j_!\omega_U^0(j^!M) \to \omega^0_X j_!(j^!M)$ is invertible. This implies that
$i_*(\alpha_i):i_*i^*\omega^0_X M \to i_*\omega^0_Z i^*M$ is invertible. But $i_*$ is a fully faithful embedding as the counit
$i^*i_*\to \id$ is invertible
(cf.~\cite[Cor.~4.5.44]{ayoub-these-II}).

\begin{lemma}
\label{lemma:distinguished-triang-for-long-prop}
Let $j:U \to X$ be an open immersion and $i:Z\to X$ a complementary closed immersion. For $M\in \DM_{\coh}(X)$,
\begin{equation}
\label{eq:distinguished-triang-for-long-prop}
\xymatrix@C=1.5pc@R=1.7pc{ j_!j^!\omega^0_X M \ar[r] \ar[d]^-{\eta_j \circ \gamma_j^{-1}} & \omega^0_X M \ar[r] \ar@{=}[d] &  i_*i^*\omega^0_X M \ar[d]^-{\beta_i^{-1} \circ \alpha_i} \ar[r] & \\
\omega^0_X j_!j^!M \ar[r] & \omega^0_X M \ar[r] &  \omega^0_Xi_*i^*M\ar[r] &}
\end{equation}
is a morphism of distinguished triangles (recall that $\gamma_j$ and $\beta_i$ are invertible by parts \emph{(c)} of \emph{(iii)} and \emph{(v)} in
\emph{Proposition \ref{prop:additional-prop-omega-0-x}} respectively).
\end{lemma}

\begin{proof}
The following two squares
$$\xymatrix@R=1.7pc{j_!\omega^0_U j^! \ar[r]^-{j_!(\gamma_j)} \ar[d]_-{\eta_j(j^!)} & j_!j^! \omega^0_X \ar[d]\\
\omega^0_X j_! j^! \ar[r] & \omega^0_X } \qquad \qquad \xymatrix@R=1.7pc{i_*\omega^0_Z i^* \ar[d]_-{\beta_i(i^*)} \ar[r]^-{i_*(\alpha_i)} & i_*i^*\omega^0_X \ar[d] \\
\omega^0_X i_*i^* \ar[r] & \omega^0_X}$$
commute. We only show this for the first square as the proof is identical for the second one. Using that $j_!\omega^0_U j^!$ takes values in $\DM_0(X)$ it suffices (by the uniqueness of the factorization through $\omega^0_X(-)$) to show that
$$\xymatrix@R=1.7pc{j_!\omega^0_U j^! \ar[r]^-{j_!(\gamma_j)} \ar[d]_-{\eta_j(j^!)} & j_!j^! \omega^0_X \ar[d]\\
\omega^0_X j_! j^! \ar[r] & \id } $$
is commutative. The claim follows now from the commutation of the following two diagrams
$$\xymatrix@C=1.7pc@R=1.7pc{j_!\omega^0_U j^! \ar[r] \ar[d] & j_!j^!\omega^0_X \ar[d] \ar@/^/[dl] \\
j_!j^! \ar[r] & \id} \qquad \xymatrix@C=1.7pc@R=1.7pc{j_!\omega^0_U j^! \ar[r] \ar[d] & \omega^0_X j_!j^! \ar[d] \ar@/^/[dl] \\
 j_!j^! \ar[r] & \id.\!}$$

We now go back to
\eqref{eq:distinguished-triang-for-long-prop}.
By Verdier's axiom (TR3) we may extend the first square of
\eqref{eq:distinguished-triang-for-long-prop}
to a morphism of distinguished triangles.
It is thus sufficient to show that there is at most one morphism
$i_*i^*\omega^0_X M \to \omega^0_X i_*i^*M$
making the triangle
$$\xymatrix@C=1.5pc@R=1.5pc{\omega^0_XM \ar[r] \ar@/_/[dr] & i_*i^*\omega^0_XM \ar[d]\\
& \omega^0_X i_*i^*M}$$
commutative.
Let $a_1$ and $a_2$ be two such morphisms. The composition
$$\xymatrix{\omega^0_XM \ar[r] & i_*i^*\omega^0_XM \ar[r]^-{a_1-a_2} & \omega^0_X i_*i^*M}$$
is zero. Using the top distinguished triangle in
 \eqref{eq:distinguished-triang-for-long-prop}, we may factors $a_1-a_2$ by a morphism $j_!j^!\omega^0_XM[1]\to i_*i^*\omega^0_XM$. Using adjunction and the fact that $i^*j_!\simeq 0$, we deduce that such a morphism is zero. This proves that $a_1=a_2$.
\end{proof}

To complete the proof of Proposition
\ref{prop:additional-prop-omega-0-x}, we still need to show that the functors $\omega^0_X$ preserve compact objects and commute with $f^*$ for $f:Y \to X$ smooth.
We prove both statements by noetherian induction on $X$.
As $f^*$ commutes with infinite sums, we need to show, for $M$ a compact cohomological motive on $X$, that
\begin{itemize}

\item[(a)] $\omega^0_X(M)$ is compact,

\item[(b)] $\alpha_f:f^*\omega_X^0(M)\to \omega_Y^0 (f^*M)$ is invertible.

\end{itemize}
As $M$ is compact,
we may find $j:U\hookrightarrow X$ a dense open immersion
such that $j^*M$ is a smooth cohomological motive.
Indeed, by Lemma \ref{lemma:generation-proj-reg}
there exists finitely many projective $X$-schemes $T_{\alpha}$
which are smooth over $k$ such that $M$ is in the triangulated subcategory of $\DM_{\coh}(X)$ generated by $\M_{\coh}(T_{\alpha})$.
It is thus sufficient to take $U$ such that all $T_{\alpha}\times_X U$ are smooth over $U$.

We first prove (a). Consider the distinguished triangle
(cf.~\cite[Prop.~1.4.9]{ayoub-these-I})
\begin{equation}
\label{eq:main-prop-omega-0-007}
\xymatrix@C=1.5pc{i_!i^!M \ar[r] &  M \ar[r] &  j_*j^*M \ar[r] & }
\end{equation}
with $i$ the inclusion of the complement $Z=X-U$ in $X$.
Applying $\omega_X^0$ and using that $\eta_i$ is invertible, we get a distinguished triangle
$$\xymatrix@C=1.5pc{i_!\omega^0_Z(i^!M) \ar[r] &  \omega^0_X(M) \ar[r] &  \omega_X^0(j_*j^*M) \ar[r] & .}$$
By induction, we know that $\omega^0_Z(i^!M)$ is compact. It is
then sufficient to show that $\omega_X^0(j_*j^*M)$ is compact. By
(iii), we have an isomorphism $\omega_X^0(j_*j^*M)\simeq
\omega_X^0 j_* (\omega^0_U j^*M)$. As $j^*M$ is a compact smooth
cohomological motive, we deduce from Proposition
\ref{prop:main-computation-omega0-pi-0} that $\omega^0_U (j^*M)$
is a compact Artin motive. In particular, $N=j_*\omega^0_U (j^*M)$
is a compact motive such that $j^*N$ is Artin and it suffices to show that $\omega^0_X(N)$ is compact. By (vi), we know
that $i^*\omega^0_X(N)\simeq \omega^0_Z(i^*N)$, which
is compact by induction. From the triangle
(cf.~\cite[Lem.~1.4.6]{ayoub-these-I})
$$\xymatrix@C=1.5pc{j_!j^*N \ar[r] &  \omega^0_X(N) \ar[r] &
 i_*\omega^0_Z(i^*N)\ar[r] & }$$
we deduce that $\omega^0_X(N)$ is compact.

We turn now to the property (b). We form the commutative diagram with cartesian squares
$$\xymatrix@C=1.5pc@R=1.5pc{V \ar[r]^-{j'} \ar[d]_-g & Y \ar[d]^-f &  T \ar[l]_-{i'}\ar[d]^-h \\
U \ar[r]^-j & X & Z. \ar[l]_-i }$$
By the distinguished triangle
\eqref{eq:main-prop-omega-0-007}, we need to show that
\begin{equation}
\label{eq:main-prop-omega-0-011}
f^*\omega^0_X(i_!i^!M) \to \omega^0_Xf^*(i_!i^!M)
\qquad \text{and} \qquad
f^*\omega^0_X(j_*j^*M) \to \omega^0_Xf^*(j_*j^*M)
\end{equation}
are invertible. For the first morphism
of \eqref{eq:main-prop-omega-0-011},
consider the commutative diagram (use Lemma
\ref{lemma:commut-diag-g-star-f-star-omega} below and the equality $i_*=i_!$)
$$\xymatrix@C=1.5pc@R=1.7pc{
f^*\omega^0_X(i_!i^!M) \ar[rr]^-{\alpha_f(i_!i^!M)}  \ar[d]_-{\sim} &&  \omega^0_X(f^*i_!i^!M) \ar[r]^-{\sim} & \omega^0_X(i'_!h^*i^!M)\ar[d]^-{\sim} \\
f^* i_!\omega^0_Z(i^!M) \ar[r]^-{\sim} & i'_!h^*\omega^0_Z(i^!M) \ar[rr]^-{i_!\alpha_h(i^!M)} && i'_!\omega_T^0(h^*i^!M).\!}$$
All the non-labeled arrows are invertible by either (iv) or the base change theorem by smooth morphisms.
As $\alpha_h$ is invertible by induction, we deduce that $\alpha_f(i_!i^!M)$ is also invertible.

For the second morphism of
\eqref{eq:main-prop-omega-0-011}, we use the following commutative diagram
$$\xymatrix@R=1.7pc{f^*\omega^0_X j_* \omega^0_U j^*M \ar[rr]^-{\alpha_f(j_*\omega^0_Uj^*M)}  \ar[d]^-{\sim}_-{\delta_U} && \omega_Y^0 f^*j_*\omega^0_U j^*M \ar[r]^-{\sim}\ar[d]^-{\delta_U}  &
\omega^0_Y j'_*g^*\omega^0_{U}j^*M \ar[d]^-{\delta_U} \\
f^*\omega^0_X j_*j^*M \ar[rr]^-{\alpha_f(j_*j^*M)}  && \omega_Y^0 f^*j_*j^*M  \ar[r]^-{\sim} & \omega^0_Y j'_*g^*j^*M
}$$
The non-labeled morphisms are invertible by the base change theorem by smooth morphisms (cf.~\cite[Prop.~4.5.48]{ayoub-these-II}). The left vertical arrow is invertible by (iii).
Let's show that
$$\delta_U:\omega^0_Y j'_*g^*\omega^0_Uj^*M \to \omega^0_Y j'_*g^*j^*M$$
is also invertible. Using (iii) and the commutative diagram
$$\xymatrix{\omega^0_Y j'_*\omega^0_V g^*\omega^0_Uj^*M \ar[d]_-{\delta_V}^-{\sim} \ar[r]^-{\delta_U} & \omega^0_Y j'_*\omega^0_V g^*j^*M \ar[d]_-{\sim}^-{\delta_V}\\
\omega^0_Y j'_*g^*\omega^0_Uj^*M \ar[r]^-{\delta_U} \ar@/^/[ur]^-{\alpha_g} &  \omega^0_Y j'_*g^*j^*M}$$
we need to show that $\alpha_f: g^*\omega^0_U (j^*M) \to \omega^0_V g^*(j^*M)$
is invertible. This follows from Corollary
\ref{invertible} as $j^*M$ is a smooth cohomological motive.

Putting again $N=j_*\omega^0_U(j^*M)$, we are reduced to
show that
$$f^*\omega_X^0(N) \to \omega^0_Y(f^*N)$$
is invertible. Recall that $j^*N$ is an Artin motive. Using the distinguished
triangle (cf.~\cite[Lem.~1.4.6]{ayoub-these-I})
$$\xymatrix@C=1.5pc{j_!j^*N\ar[r] &  N \ar[r] &  i_*i^*N\ar[r] & }$$
we are reduced to prove that
\begin{equation}
\label{eq:main-prop-omega-0-013}
f^* \omega_X^0(j_!j^*N)\to \omega_Y^0(f^*j_!j^*N) \qquad \text{and}
\qquad
f^*\omega^0_X(i_*i^*N) \to \omega^0_Y(f^*i_*i^*N)
\end{equation}
are invertible. As $j_!j^*N$ and $f^*j_!j^*N$ are already Artin
motives, we have $\omega^0_X(j_!j^*N)=j_!j^*N$ and
$\omega^0_Y(f^*j_!j^*N)=f^*j_!j^*N$ and modulo these
identifications, the first morphism in
\eqref{eq:main-prop-omega-0-013} is the identity. That the second
morphism of \eqref{eq:main-prop-omega-0-013} is invertible,
follows using the induction hypothesis, as we did
for the first morphism of \eqref{eq:main-prop-omega-0-011}.
\end{proof}

\begin{lemma}
\label{lemma:commut-diag-g-star-f-star-omega}
Consider a cartesian square of quasi-projective $k$-schemes
$$\xymatrix@C=1.5pc@R=1.7pc{Y' \ar[r]^-{f'} \ar[d]_-{g'} & Y \ar[d]^-f \\
X' \ar[r]^-g & X}$$
Then the following diagram commutes
$$\xymatrix@C=1.5pc@R=1.5pc{g^* \omega^0_X f_* \ar[r]^-{\beta_f} \ar[d]_-{\alpha_g} & g^* f_*\omega^0_Y \ar[r] & f'_*g'^* \omega^0_Y \ar[d]^-{\alpha_{g'}}\\
\omega_{X'}^0 g^*f_* \ar[r] & \omega_{X'}^0 f'_* g'^* \ar[r]^-{\beta_{f'}} & f'_* \omega_{Y'}^0 g'^*}$$
(where the non-labeled arrows are the base change morphisms).
\end{lemma}

\begin{proof}
Using the construction of $\beta_f$ from $\alpha_f$ and
$\beta_{f'}$ from $\alpha_{f'}$, this follows from the diagram
$$\xymatrix@C=1.5pc@R=1.5pc{g^* \omega^0_X f_* \ar[r] \ar@{=}[dd] & g^*f_*f^* \omega^0_X f_* \ar[r]^-{\alpha_f} \ar[d] & g^*f_*\omega^0_Y f^* f_* \ar[d] \ar[r] & g^*f_*\omega^0_Y \ar[d] \\
& f'_*g'^* f^* \omega^0_X f_* \ar[r]^-{\alpha_f} \ar[d]^-{\sim} & f'_*g'^* \omega^0_Y f^* f_*\ar[r] \ar[d]^-{\alpha_{g'}} & f'_*g'^*\omega^0_Y \ar[d]^-{\alpha_{g'}} \\
 g^* \omega^0_X f_* \ar[r] \ar[d]^-{\alpha_g}  & f'_*f'^*g^*\omega^0_X f_* \ar[d]^-{\alpha_g}  & f'_*\omega^0_{Y'} g'^* f^*f_*  \ar[r]\ar[d]^-{\sim}   & f'_*\omega^0_{Y'} g'^* \ar@{=}[dd]\\
\omega^0_{X'} g^* f_* \ar[r] \ar[d] & f'_*f'^*\omega^0_{X'}g^* f_* \ar[r]^-{\alpha_{f'}} \ar[d] & f'_*\omega^0_{Y'} f'^*g^*f_* \ar[d] & \\
\omega^0_{X'} f'_*g'^* \ar[r] & f'_*f'^* \omega^0_{X'} f'^*g'_* \ar[r]^-{\alpha_{f'}} & f'_*\omega^0_{Y'} f'^* f'_* g'^* \ar[r] &  f'_*\omega^0_{Y'} g'^*}$$
which is clearly commutative.
\end{proof}

\subsection{The motive $\mathbb{E}_X$ and its basic properties}
\label{sect:mot-E-X}
To define this motive, we need the
following corollary of Proposition \ref{prop:additional-prop-omega-0-x}:

\begin{corollary}
\label{defn-cor:E-of-X} Let $X$ be quasi-projective
$k$-scheme (with $k$ of characteristic zero). The motive
$\omega^0_Xj_*\un_U$ does not depend on the choice of a dense open
immersion $j:U\subset X$ with $U_{red}$ smooth.
\end{corollary}

\medskip

\begin{proof}
We may assume that $X$ is reduced.
Let $V\subset U \subset X$ be dense and smooth open subschemes of
$X$. Let $u$ denote the inclusion of $V$ in $U$. We need to
show that the morphism $\omega^0_X(j_*\un_U)\to
\omega^0_X(j_*u_*\un_V)$ is an isomorphism. By Proposition
\ref{prop:additional-prop-omega-0-x}, (iii) we have an
isomorphism $\omega^0_X j_* \omega^0_U(u_*\un_V)\simeq
\omega^0_Xj_*u_*\un_V$. It is then sufficient to show that $\un_U
\to \omega^0_U u_*\un_V$ is invertible.

We do this by induction on the dimension of $U-V$. One can find
an intermediate $V\subset W \subset U$ such that $W-V$ is smooth and
${\rm dim}(U-W)<{\rm dim}(W-V)$.
Let us call $v:V\subset W$ and $w:W\subset U$. We then have a commutative
square
$$\xymatrix@C=1.5pc@R=1.5pc{\un_U \ar[r] \ar[d]_-{a} & \omega^0_U(w_*v_*\un_V)  \ar[d]^-{\sim} \\
\omega^0_U w_*\un_W \ar[r]^-{b} & \omega^0_U(w_*\omega^0_W
v_*\un_V).}$$
By induction, we know that $a$ is invertible. It
then sufficient to show that $b$ is invertible. We prove more
precisely that $\un_W \to \omega^0_W v_*\un_V$ is invertible. As
$Z=W-V$ is smooth, it is a disjoint union of its irreducible
components $Z_1, \dots, Z_n$. Let $s_i:Z_i\hookrightarrow W$ and $\mathcal{N}_i$
the normal sheaf of $Z_i$ in $W$. Then $v_*\un_V$ sits in a
distinguished triangle (use \cite[Prop.~1.4.9]{ayoub-these-I}
and the purity isomorphism \cite[Th.~1.6.19]{ayoub-these-I})
$$\xymatrix@C=1.5pc{\bigoplus_{i=1}^n s_{i*}{\rm Th}^{-1}(\mathcal{N}_i)\un_{Z_i} \ar[r] & \un_W \ar[r] & v_*\un_V\ar[r] &.}$$
As $\omega^0_W (s_{i*}{\rm Th}^{-1}(\mathcal{N}_i)\un_{Z_i})\simeq s_{i*}\omega_{Z_i}^0({\rm Th}^{-1}(\mathcal{N}_i)\un_{Z_i})\simeq 0$,
we get $\un_W\simeq  \omega^0_W v_*\un_V$.
\end{proof}

\begin{definition}
\label{defn:the-motive-ee-x}
If $X$ is a quasi-projective $k$-scheme (with $k$ of characteristic zero), we denote by $\EE_X$
the motive $\omega^0_Xj_*\un_U$ with $U_{red}$ smooth, as in {\rm Corollary \ref{defn-cor:E-of-X}}.

\end{definition}

In particular, if $X_{red}$ is smooth, $\EE_X\simeq \un_X$.
We also deduce from
Proposition \ref{prop:additional-prop-omega-0-x} the following:

\begin{corollary}
\label{cor:functoriality-EE-X-010}
Let $f:Y \to X$ be a morphism of quasi-projective
$k$-schemes (with $k$ of characteristic zero) such that every irreducible component of $Y$ dominates an irreducible component of $X$. Then there is a canonical morphism
$f^*\EE_X\to \EE_Y$ which is invertible if $f$ is smooth.

\end{corollary}

\begin{proof}
We may assume that $X$ and $Y$ are reduced. Let $j:U\hookrightarrow X$
be the inclusion of a dense open subscheme which is smooth over $k$.
Then $f^{-1}(U)$ is dense in $Y$ and we may find a dense an open subset $V\subset f^{-1}(U)$ which is smooth over $k$. Moreover, if $f$ is smooth, we can take $V=f^{-1}(U)$ and we will do so. Let $j':V\hookrightarrow Y$ and $f':V \to U$ denote the obvious morphisms.
Our morphism is then the composition
$$f^*\EE_X\simeq f^*\omega^0_X j_*\un_U \to \omega^0_Y f^*j_*\un_U\to \omega^0_Y j'_*f'^*\un_U\simeq \omega^0_Yj'_*\un_V \simeq \EE_Y.$$
When $f$ is smooth, the above composition is invertible by the
last assertion in
Proposition \ref{prop:additional-prop-omega-0-x}, (ii)
and the base change theorem by smooth morphisms (cf.~\cite[Prop.~4.5.48]{ayoub-these-II}).
\end{proof}

\begin{lemma}
\label{lem:Phi-invariant-EE-X}
Let $G$ be a finite group acting on an integral quasi-projective $k$-scheme
$Y$ (with $k$ a field of characteristic zero). Let $X=Y/G$ and denote by
$e:Y \to X$ the natural morphism. Then,
$G$ acts naturally on the motive $e_*\EE_Y$.
Moreover, the morphism
$\EE_X \to e_*\EE_Y$,
obtained by the adjunction $(e^*,e_*)$
from the morphism $e^*\EE_X \to \EE_Y$ in
\emph{Corollary
\ref{cor:functoriality-EE-X-010}}, identifies
$\EE_X$ with the sub-object of $G$-invariants in $e_*\EE_Y$, i.e.,
with the image of the projector $\frac{1}{|G|}\sum_{g\in G}g$.

\end{lemma}

\begin{proof}
Let $j:U \hookrightarrow X$ be the inclusion of a non-empty open subscheme of $X$ which is smooth over $k$ and such that $V=e^{-1}(U)$ is \'etale over $U$. Let $j':V\hookrightarrow X$ denote the inclusion and $e':V \to U$ the \'etale cover given by the restriction of $e$.
The group $G$ acts on
$e'_*\un_V \simeq e'_*e'^*\un_U$
and the morphism $\un_U \to e'_*\un_V$ identifies $\un_U$ with the sub-object of $G$-invariants (see \cite[Lem.~2.1.165]{ayoub-these-I}).
It follows that
$\omega^0_X(j_*\un_U) \to \omega^0_X(j_*e'_*\un_V)$
identifies $\EE_X=\omega^0_X(j_*\un_U)$ with the sub-object of $G$-invariants in $\omega^0_X(j_*e'_*\un_V)$.

On the other hand, we have a $G$-equivariant isomorphism
$\omega^0_X(j_*e'_*\un_V)\simeq e_*\EE_Y$ given by the composition
$$\omega^0_X(j_*e'_*\un_V)\simeq \omega^0_X (e_*j'_*\un_V)
\!\xymatrix{\ar[r]^-{\beta_e}_-{\sim} & }\! e_*\omega^0_Y (j'_*\un_V)=e_*\EE_Y.$$
The natural transformation $\beta_e$ is indeed invertible
by Proposition
\ref{prop:additional-prop-omega-0-x}, (iii), as $e$ is finite.
Now, remark that
the composition
$\EE_X \to \omega^0_X (j_*e'_*\un_V) \simeq e_*\EE_Y$
coincides with the morphism obtained by the adjunction $(e^*,e_*)$ from the morphism $e^*\EE_X \to \EE_Y$ described in Corollary
\ref{cor:functoriality-EE-X-010}. This proves the lemma.
\end{proof}

\begin{corollary}
\label{cor:new-EE-X-for-quotient-singularities}
Let $X$ be a quasi-projective $k$-scheme (with $k$ of characteristic zero) having only quotient singularities. Then the natural morphism
$\un_X \to \EE_X$ is invertible.
\end{corollary}

\begin{proof}
This is an easy consequence of
Lemma \ref{lem:Phi-invariant-EE-X} and the fact that $\EE_Y\simeq \un_Y$ when $Y$ is smooth. We leave the details to the reader.
\end{proof}

Recall that an algebra $A$ in a monoidal category $(\mathcal{M},\otimes)$
is a pair $(A,m)$ with $A\in \mathcal{M}$ and
$m:A\otimes A \to A$ satisfying the usual associativity condition, i.e., $m(m\otimes \id)=m(\id \otimes m)$.
We say that $A$ is unitary if
there exists a morphism $u:\un \to A$ from a unit object of $\mathcal{M}$ such that $m(u\otimes \id)$ and $m(\id\otimes u)$ are the obvious isomorphisms $\un \otimes A \simeq A$ and $A\otimes \un \simeq A$. When $(\mathcal{M},\otimes)$ is symmetric, we say that $A$ is commutative if $m\circ \tau=m$ where
$\tau$ is the permutation of factors on $A\otimes A$.

Recall, from \cite[D\'ef.~2.1.85]{ayoub-these-I}, that a \emph{pseudo-monoidal} functor $f:(\mathcal{M},\otimes) \to (\mathcal{N},\otimes')$ is a functor $f$ endowed with a bi-natural transformation $f(A)\otimes f(B) \to f(A\otimes' B)$ satisfying some natural coherence conditions. (When this bi-natural transformation is invertible, we say that $f$ is \emph{monoidal}.)  One checks that a pseudo-monoidal functor $f$ takes an algebra of $\mathcal{M}$ to an algebra of $\mathcal{N}$. Moreover, when $f$ is also pseudo-unitary, then $f$ takes a unitary algebra of $\mathcal{M}$ to a unitary algebra of $\mathcal{N}$. Also, if $f$ is symmetric, in the sense of \cite[D\'ef.~2.1.86]{ayoub-these-I}, it preserves commutative algebras.

\begin{lemma}
\label{lem:omega-0-monoidal}
Let $X$ be a quasi-projective scheme over a perfect field $k$. Then $\omega^0_X$ is a symmetric, pseudo-monoidal and pseudo-unitary functor.
\end{lemma}

\begin{proof}
By Proposition \ref{prop:cohomological-stability-oper}, $\DM_0(X)$ and $\DM_{\coh}(X)$ are monoidal subcategories of $\DM(X)$. In particular, the inclusion ${\rm i}_X:\DM_0(X) \hookrightarrow \DM_{\coh}(X)$ is monoidal, symmetric and unitary. It follows form
\cite[Prop.~2.1.90]{ayoub-these-I} that the right adjoint $\nu^0_X$ of ${\rm i}_X$ is pseudo-monoidal, symmetric and pseudo-unitary. The lemma follows as $\omega_X^0={\rm i}_X\circ \nu^0_X$.
\end{proof}

\begin{proposition}\label{E-unitary}
Let $X$ be a quasi-projective $k$-scheme (with $k$ of characteristic zero).
Then $\mathbb{E}_X$ is a commutative unitary algebra in
$\DM(X)$. Also, under the assumptions of \emph{Corollary \ref{cor:functoriality-EE-X-010}}, the morphism $f^*\EE_X \to \EE_Y$ is a morphism of commutative
unitary algebras.

\end{proposition}

\begin{proof}
We use the notation in the proof of
Corollary \ref{cor:functoriality-EE-X-010}.
The claim follows from Lemma
\ref{lem:omega-0-monoidal} above as $j_*\un_U$ is
a commutative unitary algebra in $\DM(X)$.
The second statement follows from the fact that the natural transformations
$f^*\omega^0_X \to \omega^0_Y f^*$,
$f^*j_*\to j'_*f'^*$, used in the construction of
$f^*\EE_X \to \EE_Y$, are morphisms of pseudo-monoidal and pseudo-unitary
functors.
\end{proof}


\subsection{Some computational tools}
\label{subsect:comput-tools}
We describe some tools which are useful for computing
the motives $\EE_X$.
We first extend the definition of the punctual weight zero part to the case of relative
motives over a diagram of schemes.

\begin{definition}
\label{defn:omega-pour-diag}
Let $(\mathcal{X},\mathcal{I})$ be a diagram of quasi-projective $k$-schemes
and
$\mathcal{J}\subset \mathcal{I}$ a full subcategory.
Denote $\DM_{\mathcal{J}-\coh}(\mathcal{X},\mathcal{I})$
(resp. $\DM_{\mathcal{J}-0}(\mathcal{X},\mathcal{I})$) the triangulated
subcategory of $\DM(\mathcal{X},\mathcal{I})$ whose objects are motives $M$
such that for every $j\in \mathcal{J}$,
$j^*M$ is in $\DM_{\coh}(\mathcal{X}(j))$ (resp.
$\DM_{0}(\mathcal{X}(j))$).

For $M\in \DM_{\mathcal{J}-\coh}(\mathcal{X},\mathcal{I})$ denote, if it exists,
$\omega^{0}_{\mathcal{J}|(
\mathcal{X},\mathcal{I})}(M)$
the universal object in
$\DM_{\mathcal{J}-0}(\mathcal{X},\mathcal{I})$
that admits a mapping $\delta_{\mathcal{J}|(\mathcal{X},\mathcal{I})}:
\omega^{0}_{\mathcal{J}|(
\mathcal{X},\mathcal{I})}(M) \to M$.

\end{definition}

\begin{remark}
\label{notational-remark-for-omega-on-diagrams}
We simply denote
$\DM_{\coh}(\mathcal{X},\mathcal{I})$ and
$\DM_{0}(\mathcal{X},\mathcal{I})$ the categories
$\DM_{\mathcal{I}-\coh}(\mathcal{X},\mathcal{I})$ and $\DM_{\mathcal{I}-0}(\mathcal{X},\mathcal{I})$. We also write
$\omega^0_{(\mathcal{X},\mathcal{I})}$ instead of
$\omega^0_{\mathcal{I}|(\mathcal{X},\mathcal{I})}$.
If $X$ is a quasi-projective $k$-scheme and $\mathcal{I}$ a small category, we denote $\omega^0_X$ instead of
$\omega^0_{(X,\mathcal{I})}$, if no confusion can arise. Also, given a diagram of quasi-projective $k$-schemes $(\mathcal{X},\mathcal{I})$,
a full subcategory $\mathcal{J}\subset \mathcal{I}$ and
a small category $\mathcal{K}$, we write again
$\omega^0_{\mathcal{J}|(\mathcal{X},\mathcal{I})}$
instead of
$\omega^0_{\mathcal{J}\times \mathcal{K}|(\mathcal{X}\circ {\rm pr}_1,\mathcal{I}\times \mathcal{K})}$, if no confusion can arise.
Finally, given a diagram $(\mathcal{Y},\mathcal{L}):\mathcal{I} \to \Dia(\Sch/k)$ in the category of diagrams of quasi-projective $k$-schemes, we write
$\omega^0_{\mathcal{J}|(\mathcal{Y},\mathcal{I})}$ instead of
$\omega^0_{\int_{\mathcal{J}}\mathcal{L}| (\mathcal{Y},\int_{\mathcal{I}}\mathcal{L})}$, if no confusion can arise.
\end{remark}

A full subcategory $\mathcal{J}\subset \mathcal{I}$ is said to be \emph{attracting} if for every $j\in \mathcal{J}$ and $i\in \mathcal{I}$, the condition
$\hom_{\mathcal{I}}(j,i)\neq \emptyset$ implies that $i\in \mathcal{J}$.

\begin{lemma}
\label{lemma:existence-w-0-J-X-I}
Keep the notation and assumption of \emph{Definition \ref{defn:omega-pour-diag}}. If $\mathcal{J}\subset \mathcal{I}$ is attracting,
$\omega^{0}_{\mathcal{J}|(
\mathcal{X},\mathcal{I})}(M)$
exists for all $M\in \DM_{\mathcal{J}-\coh}(\mathcal{X},\mathcal{I})$.
Moreover, the functor $\omega^0_{\mathcal{J}|(\mathcal{X},\mathcal{I})}$ commutes with infinite sums.
\end{lemma}

\begin{proof}
The subcategories $\DM_{\mathcal{J}-0}(\mathcal{X},\mathcal{I}), \, \DM_{\mathcal{J}-\coh}(\mathcal{X},\mathcal{I}) \subset \DM(\mathcal{X},\mathcal{I})$ are stable under infinite sums. We show that they are compactly generated. The proof being the same for both categories, we concentrate on $\DM_{\mathcal{J}-\coh}(\mathcal{X},\mathcal{I})$.
For $j\in \mathcal{J}$ and $B\in \DM_{\coh}(\mathcal{X}(j))$,
$j_{\sharp}B$ is in $\DM_{\coh}(\mathcal{X},\mathcal{I})$ (which is contained in $\DM_{\mathcal{J}-\coh}(\mathcal{X},\mathcal{I})$). Indeed,
by Lemma \ref{rem:i-*-j-diaise-}, for any $i \in \mathcal{I}$,
$i^*j_{\sharp}B$ is isomorphic to the coproduct over
the arrows $i\to j$ in $\hom_{\mathcal{I}}(i,j)$ of
$\mathcal{X}(i\to j)^*B$. Similarly, for $i\in \mathcal{I}-\mathcal{J}$ and $A\in \DM(\mathcal{X}(i))$, $i_{\sharp}A$ is in
$\DM_{\mathcal{J}-\coh}(\mathcal{X},\mathcal{I})$. Indeed, for
$j\in \mathcal{J}$, $j^*i_{\sharp}=0$. This follows from
Lemma \ref{rem:i-*-j-diaise-} and the fact that $\hom_{\mathcal{I}}(j,i)=\emptyset$.
For all compact $A$ and $B$,
the motives
$i_{\sharp}A$ and $j_{\sharp}B$ are compact,
and they form a system of compact generators for
$\DM_{\mathcal{J}-\coh}(\mathcal{X},\mathcal{I})$
by \cite[Prop.~2.1.27]{ayoub-these-I}.
Now, by \cite[Cor.~2.1.22 and Lem.~2.1.28]{ayoub-these-I}, the inclusion ${\rm i}_{\mathcal{J}|(\mathcal{X}|\mathcal{I})}:\DM_{\mathcal{J}-0}(\mathcal{X},\mathcal{I}) \hookrightarrow \DM_{\mathcal{J}-\coh}(\mathcal{X},\mathcal{I})$ has a right adjoint $\nu_{\mathcal{J}|(\mathcal{X},\mathcal{I})}^0$ that commutes with infinite sums. It is clear that $\omega^0_{\mathcal{J}|(\mathcal{X}|\mathcal{I})}=
{\rm i}_{\mathcal{J}|(\mathcal{X}|\mathcal{I})}\circ
\nu_{\mathcal{J}|(\mathcal{X},\mathcal{I})}^0$ gives the universal object in $\DM_{\mathcal{J}-0}(\mathcal{X},\mathcal{I})$
that maps to $M\in \DM_{\mathcal{J}-\coh}(\mathcal{X},\mathcal{I})$.
\end{proof}

\begin{proposition}
\label{prop:omega-diagram-raisonnable}
Keep the notation and assumption of \emph{Definition \ref{defn:omega-pour-diag}} and assume that $\mathcal{J}\subset \mathcal{I}$ is attracting.

\begin{enumerate}

\item[(a)] For $j\in \mathcal{J}$, there is a canonical isomorphism $j^*\circ \omega^{0}_{\mathcal{J}|(\mathcal{X},\mathcal{I})}\simeq \omega^0_{\mathcal{X}(j)}\circ j^*$ making the triangle
$$\xymatrix@R=1.8pc{j^*\circ \omega^{0}_{\mathcal{J}|(\mathcal{X},\mathcal{I})} \ar[r]^-{\sim} \ar@/_.8pc/[dr]_-{j^*(\delta_{\mathcal{J}|(\mathcal{X},\mathcal{I})})} & \omega^0_{\mathcal{X}(j)}\circ j^* \ar[d]^-{\delta_{\mathcal{X}(j)}(j^*)} \\
& j^*}$$
commutative.

\item[(b)]
For $i\in \mathcal{I}- \mathcal{J}$,
the natural transformation
$i^*(\delta_{\mathcal{J}|(\mathcal{X},\mathcal{I})}):i^*\circ \omega^{0}_{\mathcal{J}|(\mathcal{X},\mathcal{I})}\to i^*$
is an isomorphism.

\end{enumerate}

\end{proposition}

\begin{proof}
We fix $M\in \DM_{\mathcal{J}-\coh}(\mathcal{X},\mathcal{I})$.
For (a), we need to show that
$j^*(\omega^{0}_{\mathcal{J}|(\mathcal{X},\mathcal{I})}(M))\to j^*M$
is the universal morphism from an Artin motive. Let
$A\in \DM_0(\mathcal{X}(j))$ be an Artin motive. To give a morphism
$a_1:A\to j^*M$
is equivalent, by the adjunction $(j_{\sharp},j^*)$, to giving
a morphism $a_2:j_{\sharp} A \to M$.
Using Lemma \ref{rem:i-*-j-diaise-}, we see that $j_{\sharp}A$ is in
$\DM_{0}(\mathcal{X},\mathcal{I})$ and in particular in
$\DM_{\mathcal{J}-0}(\mathcal{X},\mathcal{I})$. Thus, to give the morphism $a_2$ is equivalent to giving a morphism
$a_3:j_{\sharp} A \to \omega^0_{\mathcal{J}|(\mathcal{X},\mathcal{I})}(M)$. Using again the adjunction $(j_{\sharp},j^*)$, we see that to give $a_3$ is equivalent to giving $a_4:A \to j^*(\omega^0_{\mathcal{J}|(\mathcal{X},\mathcal{I})}(M))$.

For (b), we fix $N\in \DM(\mathcal{X}(i))$. To give a morphism
$b_1:N \to i^*M$ is equivalent, by the adjunction $(i_{\sharp},i^*)$, to giving a morphism
$b_2:i_{\sharp} N \to M$. Now, for $j\in \mathcal{J}$,
$j^*i_{\sharp} N$ is zero
(as in the proof of Lemma \ref{lemma:existence-w-0-J-X-I}).
In particular, $i_{\sharp}N$ is in
$\DM_{\mathcal{J}-0}(\mathcal{X},\mathcal{I})$. Thus, to give the morphism $b_2$ is equivalent to giving a morphism
$i_{\sharp}N \to \omega^0_{\mathcal{J}|(\mathcal{X},\mathcal{I})}(M)$.
Using again the adjunction $(i_{\sharp},i^*)$, we
see that to give $b_3$ is equivalent to giving $b_4:N \to i^*
(\omega^0_{\mathcal{J}|(\mathcal{X},\mathcal{I})}(M))$. Our claim follows now by Yoneda's lemma.
\end{proof}

We introduce some notation. Recall that $\underline{\mathbf{1}}$
denotes the ordered set $\{0\to 1\}$. Let $\ucarre$ be the
complement of $(1,1)$ in $\underline{\mathbf{1}}\times
\underline{\mathbf{1}}$. Given a set $E$, we denote
$\mathcal{P}(E)$ the set of subsets of $E$, partially ordered by
inclusion. Let also $\mathcal{P}_2(E)\subset \mathcal{P}(E)^2$ be
the subset consisting of pairs $(I_0,I_1)$ of subsets of $E$ such
that $(I_0\cap I_1)=\emptyset$. The direct product $\ucarre^E$ can
be identified with $\mathcal{P}_2(E)$ by sending a function $f:E
\to \ucarre$ to the pair $(I_0,I_1)$ where $I_0=\{e\in E, \,
f(e)=(1,0)\}$ and $I_1=\{e\in E, \, f(e)=(0,1)\}$. In particular,
we have an identification $\mathcal{P}_2([\![1,n-1]\!])\times
\ucarre \simeq \mathcal{P}_2([\![1,n]\!])$ which sends
$((J_0,J_1),(0,0))$, $((J_0,J_1),(1,0))$ and $((J_0,J_1),(0,1))$
to $(J_0,J_1)$, $(J_0\bigsqcup\{n\},J_1)$ and $(J_0,J_1\bigsqcup
\{n\})$ respectively for every $(J_0,J_1)\in
\mathcal{P}_2([\![1,n-1]\!])$. This identification will be used
freely in the next statement.

\begin{proposition}
\label{prop:cube-strat-F-omega-j-star}
Let $X$ be a
quasi-projective scheme over a field $k$ of characteristic zero,
endowed with a stratification by locally closed subschemes
$\mathcal{S}=(X_i)_{i\in [\![0,n]\!]}$ such that $X_{i}\subset
\overline{X_{i-1}}$ for $i\in [\![1,n]\!]$. For $i\in
[\![0,n]\!]$, we denote by $u_i$ the inclusion of $X_i$ in $X$.

Then there exists a canonical motive $\theta_{X,\mathcal{S}}\in
\DM(X,\mathcal{P}_2([\![1,n]\!]))$, which is a commutative unitary algebra and which satisfies the following properties.
\begin{enumerate}

\item[(i)] Let $(I_0,I_1)\in \mathcal{P}_2([\![1,n]\!])$.  Then
$$(I_0,I_1)^*\theta_{X,\mathcal{S}}=\psi^{I_0,I_1}_{n}\dots \psi^{I_0,I_1}_{1}(u_{0*}\un_{X_0})$$
where
$$\psi^{I_0,I_1}_j=\left\{\begin{array}{ccc}
\id & \text{if} & j\in I_0,\\
u_{j*}u_j^* & \text{if} & j\not\in  I_0\bigsqcup I_1,\\
u_{j*}\omega^0_{X_j} u_j^* & \text{if} & j\in I_1.
\end{array}\right.$$

\item[(ii)] Suppose that $(I_0,I_1)\subset (I_0',I_1')$ (i.e., $I_0\subset I_0'$ and $I_1\subset I_1'$).
The morphism $(I_0',I_1')^*\theta_{X,\mathcal{S}}\to (I_0,I_1)^*\theta_{X,\mathcal{S}}$ is induced by the natural transformations $\psi^{I_0',I_1'}_j \to \psi^{I_0,I_1}_j$ equal to the identity or one of the two natural transformations
$$\id \to u_{j*}u_j^* \qquad \text{and} \qquad u_{j*}\omega^0_{X_j} u_j^* \to u_{j*}u_j^*$$
depending on the value of $j$.

\item[(iii)] There exists a canonical isomorphism of commutative unitary algebras
$$\omega^0_X (u_{0*}\un_{X_0}) \simeq {\rm holim}\; \theta_{X,\mathcal{S}}.$$
More precisely, ${\rm holim}\; \theta_{X,\mathcal{S}}$ is an Artin
motive, and
$([\![1,n]\!],\emptyset)^*\theta_{X,\mathcal{S}} \simeq u_{0*}\un_{X_0}$
yields a canonical morphism ${\rm holim}\, \theta_{X,\mathcal{S}}
\to u_{0*}\un_{X_0}$ which identifies ${\rm
holim}\, \theta_{X,\mathcal{S}}$ with
$\omega^0_X(u_{0*}\un_{X_0})$.

\end{enumerate}

The motive $\theta_{X,\mathcal{S}}$ is functorial with respect to universally open morphisms\footnote{Recall that a finite presentation morphism $p:T \to S$ is open if the image of every Zariski open subset of $T$ is a Zariski open subset of $S$. We say that $p$ is universally open if any base-change of $p$ is open.} in the following way. Let
$l:\check{X} \to X$ be a universally open morphism of quasi-projective $k$-schemes. For $i\in [\![0,n]\!]$, denote $\check{X}_i=l^{-1}(X_i)$ and $\check{u}_i:\check{X}_i \hookrightarrow \check{X}$ the inclusion.
Then $\check{\mathcal{S}}=(\check{X}_i)_{i\in [\![0,n]\!]}$
is a stratification on $\check{X}$ such that
$\check{X}_i\subset \overline{\check{X}_{i-1}}$ for
$i\in [\![1,n]\!]$, and
there exists a canonical morphism of commutative unitary algebras $l^*\theta_{X,\mathcal{S}}\to \theta_{\check{X},\check{\mathcal{S}}}$
making the following diagram commutative
$$\xymatrix@C=1.5pc@R=1.7pc{l^*\omega^0_X u_{0*} \un_{X_0} \ar[d]_-{\sim}  \ar[r] & \omega^0_{\check{X}} l^*u_{0*}\un_{X_0} \ar[r]  & \omega^0_{\check{X}} (\check{u}_0)_*\un_{\check{X}_0}\ar[d]^-{\sim} \\
l^* {\rm holim}\,\theta_{X,\mathcal{S}} \ar[r]  & {\rm holim}\, l^*\theta_{X,\mathcal{S}} \ar[r]  & {\rm holim}\, \theta_{\check{X},\check{\mathcal{S}}}.}$$
Moreover, when $l$ is smooth, the morphism
$l^*\theta_{X,\mathcal{S}} \to \theta_{\check{X},\check{\mathcal{S}}}$ is invertible.

\end{proposition}

\begin{proof}
The construction of the motive $\theta_{X,\mathcal{S}}$ and the proof of its properties are by induction on the integer $n$. When $n=0$, there is
nothing to do. Indeed, as $\mathcal{P}_2(\emptyset)=\mathbf{e}$, the category with
one object and one arrow, one has to take $\theta_{X,\mathcal{S}}= \un_X\in \DM(X)$.

Let us assume that $n\geq 1$ and that the proposition is proven
for $n-1$. Let $X'=X-X_n$ and $X'_i=X_i$ for $0\leq i \leq n-1$. We
have a stratification $\mathcal{S}'=(X'_i)_{i\in [\![0,n-1]\!]}$ of
$X'$. Denote $u'_i:X'_i \hookrightarrow X'$ and $j:X'\hookrightarrow X$. By induction, we
have a motive $\theta_{X',\mathcal{S}'}\in \DM(X,\mathcal{P}_2([\![1,n-1]\!]))$ satisfying
the properties of the statement.

Let $(\mathcal{A}_n,\ucarre)$ be the following diagram of schemes
$$\xymatrix{X & X_n \ar[l]_-{u_n} \ar@{=}[r] & X_n}$$
where $\mathcal{A}_n(1,0)=X$ and $\mathcal{A}_n(0,0)=\mathcal{A}_n(0,1)=X_n$.
Write $o$ for the non-decreasing map $(-,0):\underline{\mathbf{1}} \to \ucarre$. By restriction, we get a diagram of schemes
$(\mathcal{A}_n \circ o,\underline{\mathbf{1}})$ and a corresponding morphism $o:(\mathcal{A}_n \circ o,\underline{\mathbf{1}})
\to (\mathcal{A}_n,\ucarre)$.
Also we have a morphism
$b:(\mathcal{A}_n\circ o,
\underline{\mathbf{1}}) \to X$
in $\Dia(\Sch/k)$ which is the closed immersion $u_n$ over $0\in \underline{\mathbf{1}}$
and the identity over $1\in \underline{\mathbf{1}}$.
Similarly, we have a morphism
$e:(\mathcal{A}_n,\ucarre) \to (X,\ucarre)$
which is given by $\id_X$ and $u_n$.
Now consider the following diagram in $\Dia(\Sch/k)$
$$\xymatrix@C=1.5pc{X' \ar[r]^-j & X & (\mathcal{A}_n\circ o,\underline{\mathbf{1}}) \ar[r]^-o \ar[l]_-b & (\mathcal{A}_n,\ucarre) \ar[r]^-e & (X,\ucarre). }$$
We define $\theta_{X,\mathcal{S}}$ out of $\theta_{X',\mathcal{S}'}$
by the formula
\begin{equation}
\label{eq:inductive-formula-mot-theta-x-s}
\theta_{X,\mathcal{S}}=e_*\omega^0_{\{(0,1)\}|(\mathcal{A}_n,\smallucarre)} o_*b^*j_*\theta_{X',\mathcal{S}'}.
\footnote{It is possible to give a simpler formula for $\theta_{X,\mathcal{S}}$ by replacing the composition $o_*b^*$ by the operation $p^*$ with $p$ the natural morphism $(\mathcal{A}_n,\ucarre)\to X$. However, the formula
\eqref{eq:inductive-formula-mot-theta-x-s}
is more suited for the proof of
Proposition \ref{prop:comparaison-theta-et-theta-prime}.}
\end{equation}
In the formula above, $\omega^0_{\{(0,1)\}|(\mathcal{A}_n,\smallucarre)}$ is really
$\omega^0_{\mathcal{P}_2([\![1,n-1]\!])\times \{0,1\}|
(\mathcal{A}_n\circ {\rm pr}_2,\mathcal{P}_2([\![1,n-1]\!])\times \smallucarre)}$ (see
Remark \ref{notational-remark-for-omega-on-diagrams}).
As the functors used
in \eqref{eq:inductive-formula-mot-theta-x-s}
are all pseudo-monoidal, symmetric and pseudo-unitary, we see that
$\theta_{X,\mathcal{S}}$ is again a commutative unitary algebra.

The motive $o_*b^*j_*\theta_{X',\mathcal{S}'}$ is given by
$j_*\theta_{X',\mathcal{S}'}$ over $\mathcal{A}_n(1,0)=X$ and
by
$u_n^*j_*\theta_{X',\mathcal{S}'}$ over $\mathcal{A}_n(0,0)=X_n$ and
$\mathcal{A}_n(0,1)=X_n$.
It follows from
Proposition \ref{prop:omega-diagram-raisonnable}
that the $\ucarre$-partial skeleton (cf.~\eqref{eq:partial-skeleton-definition}) of
$\theta_{X,\mathcal{S}}$ is given by
\begin{equation}
\label{eq:needed-in-sect-3-for-theta-X-S-007}
\xymatrix{ \overset{^{(1,0)}}{j_*\theta_{X',\mathcal{S}'}} \ar@<-.5pc>[r]^-{\eta} & \overset{^{(0,0)}} {u_{n*}u_n^* j_*\theta_{X',\mathcal{S}'} } &   \overset{^{(0,1)}}{u_{n*} \omega^0_{X_{n}} u_n^* j_*\theta_{X',\mathcal{S'}}}\ar@<.5pc>[l]_-{\delta_{X_n}}.}
\end{equation}
Properties (i) and (ii) are thus immediate.

We now check (iii). Using the induction hypothesis and Lemma
\ref{texnical-commut-holim}, the homotopy limit of
$\theta_{X,\mathcal{S}}$ can be identified with the homotopy limit of
\begin{equation}
\label{eq-prop:cube-strat-F-omega-j-star-rajoute-123}
\xymatrix{
j_*(\omega^0_{X'} u'_{0*}\un_{X'_0}) \ar[r]^-{\eta} &  u_{n*} u_n^* j_* (\omega^0_{X'} u'_{0*}\un_{X'_0}) & u_{n*} \omega^0_{X_n} u_n^* j_* (\omega^0_{X'} u'_{0*}\un_{X'_0}) \ar[l]_-{\delta_{X_n}}.}
\end{equation}
This shows that $j^*{\rm holim}\,\theta_{X,\mathcal{S}}\simeq \omega^0_{X'}u'_{0*}\un_{X_0'}$ and $u_n^*
{\rm holim}\,\theta_{X,\mathcal{S}} \simeq
\omega^0_{X_n} u_n^* N$ with
$N=j_*(\omega^0_{X'} u'_{0*} \un_{X'_0})$
(for the latter isomorphism, use that $u_n^*(\eta)$ is invertible if $\eta$ is the unit morphism of the adjunction $(u_n^*,u_{n*})$).
In particular, both motives
$j^*{\rm holim}\,\theta_{X,\mathcal{S}}$ and $u_n^*
{\rm holim}\,\theta_{X,\mathcal{S}}$ are Artin.
Using the localization triangle $j_!j^*\to \id \to u_{n*}u_n^* \to$ of \cite[Lem.~1.4.6]{ayoub-these-I}, we deduce that
${\rm holim}\theta_{X,\mathcal{S}}$
is also an Artin motive.

In particular, $\omega^0_X ({\rm holim}\,\theta_{X,\mathcal{S}})\simeq {\rm holim}\,\theta_{X,\mathcal{S}}$.
By Lemma \ref{texnical-commut-holim},
$\omega^0_X$ (which clearly defines an endomorphism of the triangulated derivator $\DM(X,-)$) commutes with homotopy limits indexed
by $\ucarre^n$.
Hence,
${\rm holim}\;\theta_{X,\mathcal{S}}$ is isomorphic to the homotopy limit of
$$\xymatrix{\omega^0_X N \ar[r]^-{\eta} &  \omega^0_X u_{n*} u_n^* N & \omega^0_X  u_{n*}  \omega^0_{X_n} u_n^* N \ar[l]_-{\delta_X}^-{\sim}}$$
where the morphism on the right is invertible by
Proposition \ref{prop:additional-prop-omega-0-x}, (iii).
This shows that ${\rm holim}\,\theta_{X,\mathcal{S}} \simeq \omega^0_X(N)$ and more precisely that the natural morphism
${\rm holim}\, \theta_{X,\mathcal{S}} \to N$ is the universal morphism from an Artin motive to $N$.

To finish the proof of (iii), we recall that $N=j_*\omega^0_{X'} u'_{0*} \un_{X'_0}$.
Again, by Proposition
\ref{prop:additional-prop-omega-0-x},
(iii)
$$\xymatrix{\omega^0_XN=\omega^0_X j_*\omega^0_{X'} u'_{0*} \un_{X'_0}
\ar[r]^-{\delta_{X'}} & \omega^0_X j_*u'_{0*} \un_{X'_0} \simeq
\omega^0_X u_{0*} \un_{X_0}}$$
is invertible. This shows that
${\rm holim}\, \theta_{X,\mathcal{S}}\simeq \omega^0_X (u_{0*} \un_{X_0})$ and more precisely
that the natural morphism ${\rm holim}\, \theta_{X,\mathcal{S}}\to u_{0*}\un_{X_0}$ is the universal morphism from an Artin motive to
$u_{0*}\un_{X_0}$.

It remains to show the functoriality with respect to universally open morphisms.
The condition that $l$ is universally open is assumed to ensure that $(\check{X}_i)_{i\in [\![0,n]\!]}$ is a stratification of $\check{X}$. Indeed, for such $l$, $l^{-1}(X_i)$ is dense in
$l^{-1}(\overline{X_i})$. To prove this, we remark that $l^{-1}(\overline{X_i})- \overline{l^{-1}(X_i)}$ is an open subset
of $l^{-1}(\overline{X_i})$ whose image in $\overline{X_i}$ is open and contained in $\overline{X_{i+1}}$. As $\overline{X_{i+1}}$ is a closed subset which is everywhere of positive codimension, it cannot contain a non-empty open subset of $\overline{X_i}$. This forces
$l^{-1}(\overline{X_i})- \overline{l^{-1}(X_i)}$ to be empty.

Let $\check{X}'=\check{X}\times_X X'$ and $l':\check{X}' \to X'$ be the projection to the second factor. Let also $\check{\mathcal{S}}'$ be the inverse image of the stratification $\mathcal{S}'$ along $l'$. By induction, we may assume that we have a morphism $l'^*\theta_{X',\mathcal{S}'}\to \theta_{\check{X}',\check{\mathcal{S}}'}$ which is invertible if $l$ is smooth. We form the commutative diagram
$$\xymatrix@C=1.7pc@R=1.7pc{\check{X}' \ar[r]^-{j} \ar[d]^-{l'} & \check{X} \ar[d]^-l & \ar[l]_-{b} (\check{\mathcal{A}}_n\circ o) \ar[r]^-o \ar[d]^-l &  \check{\mathcal{A}}_n \ar[r]^-{e} \ar[d]^-l & (\check{X},\ucarre)\ar[d]^-l \\
X' \ar[r]^-j & X & \ar[l]_-b \mathcal{A}_n\circ o \ar[r]^-o & \mathcal{A}_n \ar[r]^-e & (X,\ucarre) }$$
where the diagram of schemes $\check{\mathcal{A}}_n$ is for $\check{X}$ what $\mathcal{A}_n$ is for $X$.
All the squares in the above diagram are cartesian.
We deduce morphisms
$$l^*e_* \simeq e_*l^*, \quad  l^*o_* \to  o_*l^*, \quad l^*b^*\simeq b^*l^* \quad \text{and} \quad l^*j_*\to j_*l'^*.$$
Note that the second and fourth morphisms above are invertible when $l$ is smooth (cf.~\cite[Prop.~4.5.48]{ayoub-these-II}).
Also, we have a natural transformation
$$l^*\omega^0_{(0,1)|\mathcal{A}_n}\to \omega^0_{(0,1)|\check{\mathcal{A}}_n}l^*$$
where we further simplify notation by writing
$\omega^0_{(0,1)|\mathcal{A}_n}$ instead of $\omega^0_{\{(0,1)\}|(\mathcal{A}_n,\smallucarre)}$.
This transformation is invertible when $l$ is smooth, as it follows immediately from
Proposition \ref{prop:omega-diagram-raisonnable} and Proposition \ref{prop:additional-prop-omega-0-x}, (ii).
Thus we get a morphism
$$l^*e_*\omega^0_{(0,1)|\mathcal{A}_n}
o_*b^*j_*\theta_{X',\mathcal{S}'}\to
e_*\omega^0_{(0,1)|\check{\mathcal{A}}_n} o_*b^*j_* l'^*\theta'_{X',\mathcal{S}'}\to e_*\omega^0_{(0,1)|\check{\mathcal{A}}_n} o_*b^*j_*\theta_{\check{X}',\check{\mathcal{S}}'}$$
which is invertible when $f$ is smooth.
By construction, the left hand side is $l^*\theta_{X,\mathcal{S}}$ and the right hand side is $\theta_{\check{X},\check{\mathcal{S}}}$. This gives the morphism $l^*\theta_{X,\mathcal{S}}\to \theta_{\check{X},\check{\mathcal{S}}}$ of the statement.
The commutativity of the last diagram in the statement follows immediately from the commutativity of
$$\xymatrix@C=1.5pc@R=1.7pc{l^* ([\![1,n]\!],\emptyset)^*\theta_{X,\mathcal{S}} \ar[rr]^-{\sim} \ar[d]_-{\sim} & & l^* u_{0*}\un_{X_0} \ar[d] \\
([\![1,n]\!],\emptyset)^* l^*\theta_{X,\mathcal{S}} \ar[r] & ([\![1,n]\!],\emptyset)^* \theta_{\check{X},\check{\mathcal{S}}} \ar[r]^-{\sim} & (\check{u}_0)_*\un_{\check{X}_0}}$$
and the characterization of the isomorphism ${\rm holim}\,\theta_{X,\mathcal{S}}\simeq \omega^0_X u_{0*} \un_{X_0}$ in (iii).
\end{proof}

In terms of Definition \ref{defn:the-motive-ee-x},
we obtain directly from assertion (iii) of Proposition \ref{prop:cube-strat-F-omega-j-star},
whose notation we retain:

\begin{corollary}
\label{cor:E=colim}
When $(X_0)_{red}$ is smooth, $\EE_X\simeq {\rm holim}\; \theta_{X,\mathcal{S}}.$
\end{corollary}

\begin{remark}
Proposition \ref{prop:cube-strat-F-omega-j-star} shares some
similarities with (a particular case of) the formula in
\cite[Th.~3.3.5]{sophie-morel}. However, our statement is sharper as we have an actual isomorphism of motives and not only an equality in a Grothendieck group.
\end{remark}

\subsection{Computing the motive $\EE_X$}
\label{subsection:compute-E-X}
In this section we describe a way
to compute the motive $\EE_X$ using some extra data related to the singularities of $X$.
The proof of the main result of this article, that is Theorem
\ref{thm:main-thm}, is based on this computation.

\subsubsection{The setting}
\label{subsub:setting-for-comput}
Let $X$ be a quasi-projective scheme defined over a field $k$ of
characteristic zero. Suppose we are given the following data:

\begin{enumerate}

\item[\textbf{D1)}] A stratification
$\mathcal{S}=(X_i)_{i\in [\![0,n]\!]}$ of $X$ by locally closed
subschemes $X_i$ which are smooth and such that, for $i\in
[\![1,n]\!]$, $X_i$ is contained in $\overline{X_{i-1}}$ and has
positive codimension everywhere. We do not assume that the $X_i$
are connected. For $i\in [\![0,n]\!]$, we denote by $X_{\geq i}$
the Zariski closure of $X_i$, so that we have the equality of sets
$X_{\geq i}=\bigsqcup_{j\in[\![i,n]\!]} X_j$.

\item[\textbf{D2)}] For $i\in [\![0,n]\!]$, we have a projective morphism
$e_i:Y_i \to X_{\geq i}$ such that $Y_i$ has only quotient
singularities, and $e_i^{-1}(X_i)$ is dense in $Y_i$ and maps
isomorphically to $X_i$. Moreover, $e_i^{-1}(X_{\geq j})$ is a
simple normal crossings divisor (\emph{sncd}) in $Y_i$ for all
$i<j\leq n$.

\item[\textbf{D3)}] For $i\in [\![0,n]\!]$, we have a finite surjective
morphism $c_i:Z_i \to Y_i$ from a smooth $k$-scheme $Z_i$.
Moreover, we assume $(e_i\circ c_i)^{-1}(X_i)$ dense in $Z_i$, and
\'etale and Galois over each connected component of $X_i$. Also,
$Z_i- (e_i\circ c_i)^{-1}(X_i)$ is a \emph{sncd} and the inverse
image along $c_i$ of every irreducible component of $Y_i-
e_i^{-1}(X_i)$ is a smooth sub-divisor of $Z_i- (e_i\circ
c_i)^{-1}(X_i)$ (i.e., the disjoint union of its irreducible
components).

\end{enumerate}

The irreducible components of the \emph{sncd} $Y_i^{\infty}=Y_i-
e_i^{-1}(X_i)$ induce, as in Example
\ref{ex:strat-ass-family-subschemes}, a stratification
$\mathcal{R}^{\infty}_i$ of $Y_i$. More generally, given
$\emptyset \neq I \subset [\![0,n]\!]$, we denote by
$\mathcal{R}(I)$ the stratification on $Y_{{\rm min}(I)}$ induced
by the family of irreducible components of $\bigcup_{j\in I -
\{{\rm min}(I)\}} \overline{e_{{\rm min}(I)}^{-1}(X_j)}$, or
equivalently, by the irreducible components of $Y^{\infty}_i$
whose image in $X$ is an irreducible component of one of the
$X_{\geq j}$ for some $j\in I- \{{\rm min}(I)\}$. Note that
$\mathcal{R}(\{i\})$ is the coarse stratification whose strata are
just the connected components of $Y_i$, and that the
stratifications $\mathcal{R}^{\infty}_i$ and
$\mathcal{R}([\![i,n]\!])$ are the same. We assume the following
two properties:

\begin{enumerate}

\item[\textbf{P1)}]
For $i\leq j$ in $[\![0,n]\!]$, the morphism
$e_i^{-1}(X_{j})\to X_{j}$ extends (uniquely, of course) to a
morphism $e_{i,j}:\overline{e_i^{-1}(X_{j})} \to Y_{j}$, where the
closure is taken inside $Y_i$. Moreover, for $K\subset
[\![j+1,n]\!]$, every $\mathcal{R}(\{i,j\}\bigsqcup K)$-stratum is
mapped by $e_{i,j}$ onto an $\mathcal{R}(\{j\}\bigsqcup K)$-stratum
of $Y_{j}$.

\item[\textbf{P2)}]
For $i\in [\![0,n]\!]$, the morphism
$e_i:Y_i \to X_{\geq i}$ maps an $\mathcal{R}^{\infty}_i$-stratum
$E\subset Y_i$ onto an $\mathcal{S}$-stratum $D\subset X$.
Let $F$ be a connected component of $c_i^{-1}(E)$ endowed with its
reduced scheme structure. Then $F\to E$ is an \'etale cover. Moreover,
if $F'$ is the closure of $F$ in $(c_i\circ e_i)^{-1}(D)$, then
$F' \to D$ is a smooth and projective morphism whose Stein
factorization is dominated by the \'etale Galois cover
$(c_j\circ e_j)^{-1}(D) \to D$, where $j\in [\![i,n]\!]$ is the index such that $D\subset X_j$.

\end{enumerate}

In order to verify part (b) of our main theorem
(Theorem \ref{thm:main-thm}), we need to keep track of the
functoriality of our constructions. For this, we fix a universally
open morphism of quasi-projective $k$-schemes $l:\check{X} \to X$.
Let $\check{X}_i=l^{-1}(X_i)$ which we endow with its reduced scheme
structure. Then $\check{\mathcal{S}}=(\check{X}_i)_{i\in
[\![0,n]\!]}$ is a stratification of $\check{X}$ such that
$\check{X}_i\subset \overline{\check{X}_{i-1}}$ for $i\in
[\![1,n]\!]$ (cf.~the proof of
Proposition \ref{prop:cube-strat-F-omega-j-star}). Moreover, $\check{X}_{\geq i}$, the Zariski closure
of $\check{X}_i$, is equal to the inverse image of $X_{\geq i}$ by
$l$. As in \textbf{D1)}, we assume that each $\check{X}_i$ is
smooth.

Next, we assume that we are given morphisms
$\check{e}_i:\check{Y}_i \to \check{X}_{\geq i}$ and
$\check{c}_i:\check{Z}_i \to \check{Y}_i$ as in $\textbf{D2)}$ and
$\textbf{D3)}$ satisfying to the properties in $\textbf{P1)}$ and
$\textbf{P2)}$. We write $\check{\mathcal{R}}^{\infty}_i$ and
$\check{\mathcal{R}}(I)$ (with $\emptyset \neq I \subset
[\![0,n]\!]$) for the stratifications on $\check{Y}_i$ and
$\check{Y}_{{\rm min}(I)}$, defined as before.
We also assume the existence of a commutative diagram
\begin{equation}
\label{diag:for-naturality-comput-E-X-rajoute}
\xymatrix@C=1.5pc@R=1.5pc{\check{Z}_i \ar[r]^-{\check{c}_i} \ar[d]^-l & \check{Y}_i \ar[r]^-{\check{e}_i} \ar[d]^-l & \check{X}_{\geq i } \ar[d]^-l \\
Z_i \ar[r]^-{c_i} & Y_i \ar[r]^-{e_i} & X_{\geq i}.\!}
\end{equation}
While the morphism $l:\check{Y}_i \to Y_i$
is uniquely determined by $l:\check{X}_{\geq i} \to X_{\geq i}$,
this is not the case for $l:\check{Z}_i \to Z_i$ in general.
Finally, we assume that for $i\in [\![0,n]\!]$ and $I\subset
[\![i+1,n]\!]$, the morphism $l:\check{Y}_i \to Y_i$ maps an
$\check{\mathcal{R}}(\{i\} \bigsqcup I)$-stratum of $\check{Y}_i$
onto an $\mathcal{R}(\{i\}\bigsqcup I)$-stratum of $Y_i$.

We make the following comment concerning notation:

\begin{remark}
We will be constructing some objects (diagrams of schemes,
motives, etc.), using the scheme $X$ and the morphisms $e_i$ and
$c_i$. We will, of course, introduce notation for them.  Analogous
objects will be constructed for $\check{X}$, $\check{e}_i$ and
$\check{c_i}$.  We use parallel notation for these, that is by
just adding $\check{\;}\,$'s.
\end{remark}

\subsubsection{The diagram of schemes $(T,\mathcal{P}^*([\![0,n]\!])^{\rm op})$}
\label{subsub:T}
For $\emptyset \neq I \subset [\![0,n]\!]$
define the scheme $T(I)$ by
\begin{equation}
\label{eq:final-def-T(I)}
T(I)=\bigcap_{i\in I} \overline{e_{{\rm min}(I)}^{-1}(X_i)}.
\end{equation}
By definition, $T(I)$ is an $\mathcal{R}(I)$-constructible closed subscheme of $Y_{{\rm min}(I)}$ and if $\emptyset \neq J \subset I$ with ${\rm min}(J)={\rm min}(I)$, then
$T(I)\subset T(J)$.
The following gives a recursive formula for $T(I)$:

\begin{lemma}
\label{lemma:rec-form-for-T()}
For $i_0\in [\![0,n]\!]$, we have $T(\{i_0\})=Y_{i_0}$.
For $\emptyset \neq
I\subset [\![0,n]\!]$ such that $I'=I-\{{\rm max}(I)\}$ is non-empty, we have
\begin{equation}
\label{rec-form-for-T}
T(I)=\overline{(T(I')\to X)^{-1}(X_{{\rm max}(I)})}.
\end{equation}

\end{lemma}

\begin{proof}
The first claim follows from the definition. For the second claim,
we may assume that $I$ has at least three elements. Indeed,
when $I$ has
two elements, the two formulas \eqref{eq:final-def-T(I)} and
\eqref{rec-form-for-T} are identical.

   From \eqref{eq:final-def-T(I)}, we have $T(I)=T(I')\bigcap
\overline{e^{-1}_{{\rm min}(I)}(X_{{\rm max}(I)})}$. Thus, we need
to show that
$$\overline{T(I')\bigcap e^{-1}_{{\rm min}(I)}(X_{{\rm max}(I)})}=
T(I')\bigcap \overline{e^{-1}_{{\rm min}(I)}(X_{{\rm max}(I)})}.$$
It suffices to show that
\begin{equation}
\label{eq:pour-la-form-rec-T(-)}
\overline{C\bigcap D \bigcap e^{-1}_{{\rm min}(I)}(X_{{\rm max}(I)}) }=C\bigcap D
\end{equation}
for any irreducible component $C$ of $T(I')$ and any irreducible component $D$ of $\overline{e^{-1}_{{\rm min}(I)}(X_{{\rm max}(I)})}$.
As $Y^{\infty}_{{\rm min}(I)}$ is a \emph{sncd}
and because for all $i\in [\![{\rm min}(I)+1,n]\!]$, $\overline{e^{-1}_{{\rm min}(I)}(X_i)}$ is a union of irreducible divisors of $Y_{{\rm min}(I)}^{\infty}$, $C$ is a connected component of
an intersection $\bigcap_{i\in I- \{{\rm min}(I), {\rm max}(I)\}} D_i$
with $D_i$ an irreducible component of
$\overline{e^{-1}_{{\rm min}(I)}(X_i)}$. Moreover, the $D_i$ are uniquely determined by $C$.
Now, let $E$ be a connected component of $C\bigcap D$. As $E$ has only quotient singularities, its is normal and hence irreducible. We claim that $E\cap e^{-1}_{{\rm min}(I)}(X_{{\rm max}(I)})$ is not empty. This will finish the proof of the lemma. Indeed, the image of $E$ in $X$ is contained in $X_{\geq {\rm max}(I)}$. As $X_{{\rm max}(I)}$ is an open subset of
$X_{\geq {\rm max}(I)}$, we see that $E\cap e^{-1}_{{\rm min}(I)}(X_{{\rm max}(I)})$ is an open subset $E$. If the latter is non-empty, it is dense in $E$ and hence $\overline{E\cap e^{-1}_{{\rm min}(I)}(X_{{\rm max}(I)})}=E$. Applying this to all connected components of $C\bigcap D$, we get the equality \eqref{eq:pour-la-form-rec-T(-)}.

To show that $E\cap e^{-1}_{{\rm min}(I)}(X_{{\rm max}(I)})$ is non-empty, we argue by contradiction. Indeed, the contrary implies that ${\rm max}(I)\leq n-1$ and $E\subset e^{-1}_{{\rm min}(I)}(X_{\geq {\rm max}(I)+1})$. Thus, we may find an irreducible component $D'$ of $e^{-1}_{{\rm min}(I)}(X_{\geq {\rm max}(I)+1})$ which contains $E$.
Then $E$, which has codimension ${\rm card}(I)-1$ in
$Y_{{\rm min}(I)}$, is contained in the intersection of ${\rm card}(I)$ distinct irreducible components of $Y_{{\rm min}(I)}^{\infty}$, namely $D$, $D'$ and the $D_i$ for $i\in I-\{{\rm min}(I),\,{\rm max}(I)\}$.
This is a contradiction as $Y_{{\rm min}(I)}^{\infty}$ is a \emph{sncd} in $Y_{{\rm min}(I)}$.
\end{proof}

\begin{lemma}
\label{lemma:T-is-a-funct}
For
$\emptyset \neq J\subset I\subset [\![0,n]\!]$,
let $i_0={\rm min}(I)$ and $j_0={\rm min}(J)$.
Then $T(I)$ is a closed subscheme of $Y_{i_0}$ contained in $\overline{e_{i_0}^{-1}(X_{j_0})}$. Moreover,
the image of $T(I)$ by the morphism $e_{i_0, j_0}:
\overline{e_{i_0}^{-1}(X_{j_0})}
\to Y_{j_0}$ is contained in $T(J)$. This gives
a morphism
$$T(J\subset I): T(I)\to T(J).$$
$T(-)$ becomes thereby a contravariant functor from the partially ordered set
$\mathcal{P}^*([\![0,n]\!])$ of non-empty subsets of $[\![0,n]\!]$ to the category of $X$-schemes.
\end{lemma}

\begin{proof}
As $j_0\in I$, we have
$T(I)\subset T(\{i_0,j_0\})=\overline{e_{i_0}^{-1}(X_{j_0})}$.
We now check that $e_{i_0, j_0}$ sends $T(I)$ into $T(J)$.
When $i_0=j_0$, this is true as
$e_{i_0,j_0}$ is the identity of $Y_{i_0}$ and $T(I)\subset T(J)$. Thus, we may assume that $i_0<j_0$. Using the chain of inclusions
$J\subset \{i_0\}\bigsqcup J \subset I$, we may further assume that
$I=\{i_0\}\bigsqcup J$. We argue by induction on the number of elements in $J$. As $T(\{j_0\})=Y_{j_0}$, there is nothing to prove when $J$ has only one element. When $J$ contains at least two elements, let $J'=J- \{{\rm max}(J)\}$. By induction, we have
$e_{i_0, j_0}(T(\{i_0\}\bigsqcup J')) \subset T(J')$.
It follows that
$$e_{i_0 ,j_0}[(T(\{i_0\}\textstyle{\bigsqcup} J') \to X)^{-1}(X_{{\rm max}(J)})]\subset
(T(J') \to X)^{-1}(X_{{\rm max}(J)}).$$
As $e_{i_0, j_0}$ is continuous for the Zariski topology, we deduce that
$$e_{i_0 , j_0}[\overline{(T(\{i_0\}\textstyle{\bigsqcup} J') \to X)^{-1}(X_{{\rm max}(J)})}]\subset
\overline{(T(J') \to X)^{-1}(X_{{\rm max}(J)})}.$$
We now use \eqref{eq:pour-la-form-rec-T(-)} to conclude.

It remains to check that the morphisms $T(J\subset I)$ define a contravariant functor from
$\mathcal{P}^*([\![0,n]\!])$, i.e., that $T(K\subset I)=T(K\subset J)\circ T(J\subset I)$
for $\emptyset\neq K\subset J \subset I\subset [\![0,n]\!]$.
Let $i_0={\rm min}(I)$, $j_0={\rm min}(J)$ and $k_0={\rm min}(K)$ so that $i_0\leq j_0 \leq k_0$. As $T(I)\subset T(\{i_0,j_0,k_0\})$,
$T(J)\subset T(\{j_0,k_0\})$ and $T(K)\subset T(\{k_0\})$, we may assume that $I=\{i_0,j_0,k_0\}$, $J=\{j_0,k_0\}$ and
$K=\{k_0\}$.
By the recursive formula
\eqref{rec-form-for-T},
we have $T(\{i_0,j_0,k_0\})=\overline{(T(\{i_0,j_0\})\to X)^{-1}(X_{k_0})}$, $T(\{j_0,k_0\})=\overline{(T(\{j_0\})\to X)^{-1}(X_{k_0})}$
and $T(\{k_0\})=\overline{e_{k_0}^{-1}(X_{k_0})}=Y_{k_0}$.
By continuity for the Zariski topology, it is then sufficient to show that
$$\xymatrix@C=1.5pc@R=1.5pc{(T(\{i_0,j_0\})\to X)^{-1}(X_{k_0}) \ar[r] \ar@/_/[dr] & (T(\{j_0\})\to X)^{-1}(X_{k_0}) \ar[d] \\
& X_{k_0}}$$
commutes.
But this is obviously true, as $T(\{i_0,j_0\}) \to T(\{j_0\})$ is a morphism of $X$-schemes.
\end{proof}

\begin{lemma}
\label{lemma:functorialite-de-T}
For $\emptyset\neq I \subset [\![0,n]\!]$, the morphism
$l:\check{Y}_{{\rm min}(I)} \to Y_{{\rm min}(I)}$ maps
$\check{T}(I)$ to $T(I)$, inducing a morphism
$l(I):\check{T}(I) \to T(I)$. As $I$ varies, these morphisms give a natural transformation of functors
$\check{T} \to T$, and thus a morphism $l:(\check{T},\mathcal{P}^*([\![0,n]\!])^{\rm op}) \to
(T,\mathcal{P}^*([\![0,n]\!])^{\rm op})$ in $\Dia(\Sch/k)$ which is the identity on the indexing categories.
\end{lemma}

\begin{proof}
For the first claim, we use induction on $I$. When $I=\{i_0\}$,
there is nothing to prove as $\check{T}(\{i_0\})=\check{Y}_{i_0}$
and $T(\{i_0\})=Y_{i_0}$. Now, assume that $I$ has at least two
elements, and let $I'=I-\{{\rm max}(I)\}$. By the inductive
formula \eqref{rec-form-for-T}, we have
$\check{T}(I)=\overline{(\check{T}(I')\to
\check{X})^{-1}(\check{X}_{{\rm max}(I)})}$ and
$T(I)=\overline{(T(I')\to X)^{-1}(X_{{\rm max}(I)})}$. As
$\check{X}_{{\rm max}(I)}=f^{-1}(X_{{\rm max}(I)})$, we also have
$\check{T}(I)=\overline{(\check{T}(I')\to X)^{-1}(X_{{\rm
max}(I)})}$. As $\check{T}(I')\to T(I')$ is a morphism of
$X$-schemes, it takes $(\check{T}(I')\to X)^{-1}(X_{{\rm
max}(I)})$
inside
$(T(I')\to X)^{-1}(X_{{\rm max}(I)})$, and hence, by continuity for the Zariski topology, $\check{T}(I)$ inside $T(I)$.

For the second part of the lemma, we fix $\emptyset\neq J \subset I\subset [\![0,n]\!]$. We need to show that
$T(J\subset I)\circ l(I)=l(J)\circ \check{T}(J\subset I)$.
This is true when ${\rm min}(I)={\rm min}(J)=i_0$ because then,
$T(I), \, T(J) \subset Y_{i_0}$ and $T(J\subset I)$ is the inclusion morphism, and similarly for $\check{T}$.
So we may assume that $i_0={\rm min}(I)<j_0={\rm min}(J)$.
Using the inclusions $T(I)\subset T(\{i_0,j_0\})$,
$T(J)\subset T(\{j_0\})$ and the similar ones for $\check{T}$, we are furthermore reduced to the case $I=\{i_0,j_0\}$ and $J=\{j_0\}$.
The claim follows now from the commutative square
$$\xymatrix@C=1.5pc@R=1.5pc{\check{e}_{i_0}^{-1}(\check{X}_{j_0}) \ar[r] \ar[d] & e_{i_0}^{-1}(X_{j_0}) \ar[d] \\
\check{X}_{j_0} \ar[r] & X_{j_0},\!}$$
and continuity for the Zariski topology.
\end{proof}

We end this paragraph with a remark which will be helpful later on in constructing some motives and establishing their properties by induction on $n$.

\begin{remark}
\label{rem:for-doing-induction-X-X'}
Assume that $n\geq 1$. Let $X'=X-X_n$ endowed with the stratification
$\mathcal{S}'=(X'_j)_{0\leq j \leq n-1}$ with
$X'_j=X_j$ for $j\in [\![0,n-1]\!]$. As before, let
$X'_{\geq j}$ denotes the Zariski closure
of $X'_j$ in $X'$.
Let
$Y'_j=Y_j\times_{X_{\geq j}}X'_{\geq j}$ and $Z'_j=Z_j\times_{X_{\geq j}} X'_{\geq j}$ and call $e'_j:Y'_j \to X'_{\geq j}$ and
$c'_j:Z'_j \to Y'_j$ the natural projections.
This gives data as in \textbf{D1)}, \textbf{D2)} and
\textbf{D3)} satisfying the properties in \textbf{P1)} and \textbf{P2)}.

As for $X$, we have a contravariant functor
$T'$ from $\mathcal{P}^*([\![0,n-1]\!])$
to the category of $X'$-schemes which sends $\emptyset \neq I\subset [\![0,n-1]\!]$ to a closed subscheme $T'(I)\subset Y'_{{\rm min}(I)}$.
For $\emptyset\neq I\subset [\![0,n-1]\!]$, $T'(I)$ is a closed subscheme of $Y'_{{\rm min}(I)}$ which is an open subscheme of
$Y_{{\rm min}(I)}$. Moreover, the Zariski closure of
$T'(I)$ in $Y_{{\rm min}(I)}$ is equal to $T(I)$. Thus, we have an objectwise dense open immersion of diagram of schemes
$$j:(T',\mathcal{P}^*([\![0,n-1]\!])^{\rm op}) \to (T\circ \iota_{n}, \mathcal{P}^*([\![0,n-1]\!])^{\rm op})$$
where $\iota_{n}:\mathcal{P}^*([\![0,n-1]\!]) \hookrightarrow \mathcal{P}^*([\![0,n]\!])$ is the obvious inclusion.
Also, remark that $(T,\mathcal{P}^*([\![0,n]\!])^{\rm op})$ is the total diagram associated to the following diagram in
$\Dia(\Sch)$ indexed by $\ucarre$:
\begin{equation}
\label{eq:for-T-and-T'-relationes}
\xymatrix{(T\circ \iota_{n},\mathcal{P}^*([\![0,n-1]\!])^{\rm op}) & ((T\circ \iota_{n})\times_X X_n ,
\mathcal{P}^*([\![0,n-1]\!])^{\rm op}) \ar[l]_-{v_n} \ar[r]^-{(q_n,pr)} & X_n,}
\end{equation}
where $v_n$ and $q_n$ are the projections to the first and second
factor in $(T\circ \iota_n)\times_X X_n$, and $pr$ is the unique
functor from $\mathcal{P}^*([\![0,n-1]\!])^{{\rm op}}$ to the
terminal category $\textbf{e}$.

\end{remark}

\subsubsection{The diagram of schemes $(\mathcal{X},\mathcal{P}_2([\![1,n]\!]))$  and the motive $\theta'_{X,\mathcal{S}}$}
\label{subsub:X} As in \S\ref{subsect:comput-tools}, we let
$\mathcal{P}_2([\![1,n]\!])\subset \mathcal{P}([\![1,n]\!])^2$
denotes the subset of pairs $(I_0,I_1)$ such that $I_0\cap
I_1=\emptyset$. We define a functor (i.e., an non-decreasing map)
$$\varsigma_n:\mathcal{P}_2([\![1,n]\!]) \to \mathcal{P}^*([\![0,n]\!])^{\rm op},$$
as follows. For $(I_0,I_1)\in \mathcal{P}_2([\![1,n]\!])$, let
$J=[\![0,n]\!]-I_0$ and $i_{max}={\rm max}(\{0\}\bigsqcup I_1)$.
We set $\varsigma_n(I_0,I_1)=[\![i_{max},n]\!]\cap J$. As
$\{0\}\bigsqcup I_1\subset J$, $i_{max}\in J$ and thus
$\varsigma_n(I_0,I_1)$ is non-empty. One sees likewise that
$\varsigma_n$ is non-decreasing.

We let $\mathcal{X}=T\circ \varsigma_n:\mathcal{P}_2([\![1,n]\!])
\to \Sch/k$. We have a natural morphism of diagrams of schemes
$\varsigma_n:(\mathcal{X},\mathcal{P}_2([\![1,n]\!])) \to
(T,\mathcal{P}^*([\![0,n]\!])^{\rm op})$.

\begin{remark}
\label{rem:for-doing-induction-X-X'-II}
With the notation of Remark \ref{rem:for-doing-induction-X-X'}, we also have an object
$(\mathcal{X}',\mathcal{P}_2([\![1,n-1]\!]))$ of $\Dia(\Sch/k)$ obtained by composing $T'$ with the non-decreasing map
$\varsigma_{n-1}:\mathcal{P}_2([\![1,n-1]\!]) \to \mathcal{P}^*([\![0,n-1]\!])^{\rm op}$.
We have an objectwise dense open immersion of diagrams of schemes
$$j:(\mathcal{X}',\mathcal{P}_2([\![1,n-1]\!])) \to (\mathcal{X}\circ \iota^{0}_n,\mathcal{P}_2([\![1,n-1]\!])),$$
where $\iota^{0}_n:\mathcal{P}_2([\![1,n-1]\!]) \hookrightarrow
\mathcal{P}_2([\![1,n]\!])$ is the non-decreasing map that sends $(I_0,I_1)$ to $(I_0\bigsqcup \{n\},I_1)$.
Moreover, $(\mathcal{X},\mathcal{P}_2([\![1,n]\!]))$ is the total diagram associated to the following diagram in $\Dia(\Sch)$
indexed by $\ucarre$:
\begin{equation}
\label{eq:pour-induction-mathcal-X}
\small{\xymatrix@C=1.2pc{(\mathcal{X}\circ
\iota_n^0,\mathcal{P}_2([\![1,n\!-\!1]\!])) &((\mathcal{X}\circ
\iota^0_n)\!\times_X \! X_n,\mathcal{P}_2([\![1,n\!-\!1]\!]))
\ar[r]^-{q_n} \ar[l]_-{v_n} &
(X_n,\mathcal{P}_2([\![1,n\!-\!1]\!])),}}
\end{equation}
modulo the identification of $\mathcal{P}_2([\![1,n]\!])$ with
$\mathcal{P}_2([\![1,n-1]\!])\times \ucarre$.

\end{remark}

We now define
inductively a motive $\theta'_{X,\mathcal{S}}\in \DM(\mathcal{X},
\mathcal{P}_2([\![1,n]\!]))$, which is a commutative unitary algebra.
When $n=0$, we simply take $\un_{X_0}$.
When $n\geq 1$, we use Remark \ref{rem:for-doing-induction-X-X'-II}
and assume that the motive $\theta'_{X',\mathcal{S}'}\in \DM(\mathcal{X}',\mathcal{P}_2([\![1,n-1]\!]))$
is constructed.

We will abuse notation and
denote $(\mathcal{X},\ucarre)$ the object of
$\Dia(\Dia(\Sch))$ given by
\eqref{eq:pour-induction-mathcal-X}, i.e.,
such that $\mathcal{X}(1,0)=\mathcal{X}\circ \iota_n^0$,
$\mathcal{X}(0,0)=\mathcal{X}(1,0)\times_X X_n$ and
$\mathcal{X}(0,1)=(X_n,\mathcal{P}_2([\![1,n-1]\!]))$.
Let
$o$ be the non-decreasing map $(-,0):\underline{\mathbf{1}}\to
\ucarre$. It induces a morphism
$o:(\mathcal{X}\circ o,\underline{\mathbf{1}}) \to (\mathcal{X},\ucarre)$ in
$\Dia(\Dia(\Sch))$.
We also have a natural morphism
$b:(\mathcal{X}\circ o,\underline{\mathbf{1}}) \to
\mathcal{X}(1,0)=\mathcal{X}\circ \iota_n^0$ in
$\Dia(\Dia(\Sch))$.
Over $1\in \underline{\mathbf{1}}$, it is the identity of
$\mathcal{X}\circ \iota^0_n$. Over $0\in \underline{\mathbf{1}}$,
it is the objectwise
closed immersion $v_n:(\mathcal{X}\circ \iota^0_n)\times_X X_n \to
\mathcal{X}\circ \iota^0_n$. Passing to total diagrams, we obtain a diagram in $\Dia(\Sch)$ as follows:
$$\xymatrix@C=1.5pc@R=1.5pc{ & (\mathcal{X}\circ o, \mathcal{P}_2([\![1,n-1]\!])\times \underline{\mathbf{1}}) \ar[d]^-b  \ar[r]^-o & (\mathcal{X},\mathcal{P}_2([\![1,n-1]\!])\times \ucarre). \\
(\mathcal{X}',\mathcal{P}_2([\![1,n-1]\!]))  \ar[r]^-j & (\mathcal{X}\circ \iota_n^0,\mathcal{P}_2([\![1,n-1]\!])) &}$$

With these notation, we set
\begin{equation}
\label{eq:motive-theta-prime-induct-defn}
\theta'_{X,\mathcal{S}}=\omega^0_{\{(0,1)\}|(\mathcal{X},\smallucarre)} \left( o_*b^*j_*\theta'_{X',\mathcal{S}'} \right).
\end{equation}
In the formula above,
$\omega^0_{\{(0,1)\}|(\mathcal{X},\smallucarre)}$ is really
$\omega^0_{\mathcal{P}_2([\![1,n-1]\!]) \times \{(0,1)\}| (\mathcal{X},\mathcal{P}_2([\![1,n-1]\!])\times \smallucarre )}$
(see Remark \ref{notational-remark-for-omega-on-diagrams}).
This is again a commutative unitary algebra in $\DM(\mathcal{X},\mathcal{P}_2([\![1,n]\!]))$. Over the sub-diagram
$\mathcal{X}(1,0)=\mathcal{X}\circ \iota^0_n$, the motive $\theta'_{X,\mathcal{S}}$ is given by $j_*\theta'_{\mathcal{X}',
\mathcal{S}'}$. Over the sub-diagram
$\mathcal{X}(0,0)=(\mathcal{X}\circ \iota^0_n)\times_X X_n$, the motive $\theta'_{X,\mathcal{S}}$ is given by
$v_n^*j_*\theta'_{X',\mathcal{S}'}$. And finally,
over the constant diagram of schemes $\mathcal{X}(0,1)=(X_n,\mathcal{P}_2([\![1,n-1]\!]))$, the motive $\theta'_{X,\mathcal{S}}$ is given by $\omega^0_{X_n} q_{n*}v_n^*j_*\theta'_{X',\mathcal{S}'}$.

\begin{proposition}
\label{prop:comparaison-theta-et-theta-prime}
Denote by $f:(\mathcal{X},\mathcal{P}_2([\![1,n]\!])) \to  (X,\mathcal{P}_2(
[\![1,n]\!]))$ the natural morphism.
There is a canonical isomorphism of commutative unitary algebras
$\theta_{X,\mathcal{S}}\simeq f_*\theta'_{X,\mathcal{S}}$, where
$\theta_{X,\mathcal{S}}$ is the motive constructed in \emph{Proposition
\ref{prop:cube-strat-F-omega-j-star}}.
\end{proposition}

\begin{proof}
We will construct the isomorphism $\theta_{X,\mathcal{S}}\simeq f_*\theta'_{X,\mathcal{S}}$ inductively on $n$. Keep the above notation and denote $f':(\mathcal{X}',\mathcal{P}_2([\![1,n-1]\!])) \to
(X',\mathcal{P}_2([\![1,n-1]\!]))$ the natural morphism.

When $n=0$, $\mathcal{X}=X$ and $\theta_{X,\mathcal{S}}=\theta'_{X,\mathcal{S}}=\un_{X}$. In the sequel, we assume that $n\geq 1$ and put $m=n-1$. By the induction hypothesis, we have an isomorphism $\theta_{X',\mathcal{S}'}\simeq
f'_*\theta'_{X',\mathcal{S}'}$. We will use the construction of
$\theta_{X,\mathcal{S}}$ out of $\theta_{X',\mathcal{S}'}$ given in the proof of Proposition
\ref{prop:cube-strat-F-omega-j-star}. With the notation of that proof, we have a commutative diagram in $\Dia(\Sch/k)$ as follows:
$$\small{\xymatrix@C=.7pc@R=1.5pc{(\mathcal{X}',\mathcal{P}_2([\![1,m]\!])) \ar[r]^-j \ar[d]^-{f'}& (\mathcal{X}\!\circ\! \iota^0_n,\mathcal{P}_2([\![1,m]\!])) \ar[d]^-{f} & \ar[l]_-b (\mathcal{X}\!\circ \! o,\mathcal{P}_2([\![1,m]\!])\!\times \! \underline{\mathbf{1}})  \ar[r]^-o \ar[d]^-g & (\mathcal{X},\mathcal{P}_2([\![1,m]\!])\!\times\! \ucarre) \ar[d]^-g \ar@/^5pc/[dd]^-f  \\
(X',\mathcal{P}_2([\![1,m]\!])) \ar[r]^-j & (X,\mathcal{P}_2([\![1,m]\!])) & \ar[l]_-b (\mathcal{A}_n\!\circ\! o,\mathcal{P}_2([\![1,m]\!])\!\times \!\underline{\mathbf{1}})  \ar[r]^-o & (\mathcal{A}_n, \mathcal{P}_2([\![1,m]\!])\!\times \!\ucarre ) \ar[d]^-e \\
&&&  (X,\mathcal{P}_2([\![1,m]\!])\!\times \!\ucarre).\!}}$$
Now recall that $\theta_{X,\mathcal{S}}=e_*\omega^0_{\{(0,1)\}|(\mathcal{A}_n,\smallucarre)}o_*b^*j_*\theta_{X',\mathcal{S}'}$. Using the induction hypothesis and the commutation of the first square in the above diagram, we get
\begin{equation}
\label{eq:chain-iso-pour-theta-theta-prime-comp}
o_*b^*j_*\theta_{X',\mathcal{S}'}\simeq
o_*b^*j_*f'_*\theta'_{X', \mathcal{S}'}
\simeq o_*b^* f_*j_*\theta'_{X',\mathcal{S}'}.
\end{equation}
The second square in the diagram above is cartesian. Moreover,
$f_{|\mathcal{P}_2([\![1,m]\!])\times \underline{\mathbf{1}}}$ is objectwise projective. Using \cite[Th.~2.4.22]{ayoub-these-I}, we see that the base change morphism
$b^* f_* \to g_*b^*$ is invertible. Thus, we may continue the chain of isomorphisms \eqref{eq:chain-iso-pour-theta-theta-prime-comp} with
$$\simeq o_*g_*b^*j_*\theta'_{X',\mathcal{S}'}\simeq
g_*o_* b^*j_*\theta'_{X'\mathcal{S}'}.$$
As $g$ restricted to $\mathcal{P}_2([\![1,m]\!])\times \{(0,1)\}$
is an isomorphism, we see immediately that
$$\omega^0_{\{(0,1)\}|(\mathcal{A}_n,\smallucarre)} g_* \simeq
g_*\omega^0_{\{(0,1)\}|(\mathcal{X},\smallucarre)}.$$
Thus, we have canonical isomorphisms
$$\theta_{X,\mathcal{S}}\simeq e_*g_*
\omega^0_{\{(0,1)\}|(\mathcal{X},\smallucarre)}\left(
o_* b^*j_*\theta'_{X'\mathcal{S}'}\right)\simeq f_*\theta'_{X,\mathcal{S}}.$$
This proves the proposition.
\end{proof}

  From Lemma \ref{lemma:functorialite-de-T}, we have a morphism of
diagrams of schemes
$l:(\check{\mathcal{X}},\mathcal{P}_2([\![1,n]\!])) \to
(\mathcal{X},\mathcal{P}_2([\![1,n]\!]))$. Moreover, the following
square
$$\xymatrix@C=1.5pc@R=1.5pc{(\check{\mathcal{X}},\mathcal{P}_2([\![1,n]\!])) \ar[r]^-l \ar[d]_-{\check{f}} & (\mathcal{X},\mathcal{P}_2([\![1,n]\!])) \ar[d]^-f \\
(\check{X},\mathcal{P}_2([\![1,n]\!])) \ar[r]^-l & (X,\mathcal{P}_2([\![1,n]\!]))}$$
is commutative.

\begin{proposition}
\label{prop:functorialite-motif-theta-prime}
There is a morphism of motives
$l^*\theta'_{X,\mathcal{S}} \to \theta'_{\check{X}, \check{\mathcal{S}}}$ which is invertible when $f:\check{X}\to X$ is smooth and $\check{Y}_i=\check{X}\times_X Y_i$ for $i\in [\![0,n]\!]$.
Moreover,
the following diagram of $\DM(\check{X},\mathcal{P}_2([\![1,n]\!]))$:
$$\xymatrix@C=1.5pc@R=1.5pc{l^* f_*\theta'_{X,\mathcal{S}} \ar[r] \ar[d]_-{\sim} & \check{f}_* l^* \theta'_{X,\mathcal{S}} \ar[r] & \check{f}_* \theta'_{\check{X},\check{\mathcal{S}}} \ar[d]^-{\sim}\\
l^* \theta_{X,\mathcal{S}} \ar[rr] & & \theta_{\check{X},\check{\mathcal{S}}}}$$
commutes; the arrow in the bottom being the morphism
of \emph{Proposition
\ref{prop:cube-strat-F-omega-j-star}}.
\end{proposition}

\begin{proof}
The proof is by induction. When $n=0$, the statement is obvious. We assume that $n\geq 1$ and that a morphism
$l'^*\theta'_{X',\mathcal{S}'}\to \theta'_{\check{X}',\check{\mathcal{S}}'}$ has been constructed with the expected properties.
We consider the commutative diagram in $\Dia(\Sch/k)$:
$$\xymatrix@C=1.5pc@R=1.5pc{\check{\mathcal{X}}' \ar[r]^-j\ar[d]_-{l'}  & \check{\mathcal{X}}\circ \iota^0_n \ar[d]^-l & \check{\mathcal{X}}\circ o \ar[l]_-b \ar[r]^-o \ar[d]^-l &
\check{\mathcal{X}} \ar[d]^-l \\
\mathcal{X}' \ar[r]^-j & \mathcal{X}\circ \iota_n^0 & \mathcal{X}\circ o \ar[l]_-b \ar[r]^-o & \mathcal{X}.\!}$$
This gives us natural transformations
$$l^*o_*b^*j_* \to o_* l^* b^* j_*\simeq o_*b^*l^*j_*\to
o_*b^*j_*l^*.$$
Note that the first and third morphisms above are invertible when
$f:\check{X}\to X$ is smooth and $\check{Y}_i=\check{X}\times_X Y_i$ for $i\in [\![0,n]\!]$; this follows from the base change theorem by smooth morphisms \cite[Prop.~4.5.48]{ayoub-these-II}. On the other hand, we have a natural transformation
$$l^*\omega^0_{\{(0,1)\}|(\mathcal{X},\smallucarre)} \to
\omega^0_{\{(0,1)\}|(\check{\mathcal{X}},\smallucarre)} l^*$$
constructed in the same way as the natural transformation in Proposition
\ref{prop:additional-prop-omega-0-x}, (ii).
When
$f:\check{X}\to X$ is smooth and $\check{Y}_i=\check{X}\times_X Y_i$ for $i\in [\![0,n]\!]$, this natural transformation is invertible as it follows immediately from Proposition
\ref{prop:omega-diagram-raisonnable} and the last statement in Proposition \ref{prop:additional-prop-omega-0-x}, (ii).
We now obtain our morphism by taking the composition
$$\xymatrix@C=1.3pc@R=1.3pc{l^*\omega^0_{\{(0,1)\}|(\mathcal{X},\smallucarre)} o_*b^*j_* \theta'_{X',\mathcal{S}'}\ar[r] &
\omega^0_{\{(0,1)\}|(\check{\mathcal{X}},\smallucarre)} l^* o_*b^*j_* \theta'_{X',\mathcal{S}'}\ar[d] \\
& \omega^0_{\{(0,1)\}|(\check{\mathcal{X}},\smallucarre)}  o_*b^*j_* l^* \theta'_{X',\mathcal{S}'}\ar[r] &
\omega^0_{\{(0,1)\}|(\check{\mathcal{X}},\smallucarre)}  o_*b^*j_* \theta'_{\check{X}',\check{\mathcal{S}}'} }$$
and recalling that the object on the left is $l^*\theta'_{X,\mathcal{S}}$ and the object on the right is $\theta'_{\check{X},\check{\mathcal{S}}}$.

The verification that the diagram of the statement is commutative is also done by induction, using the inductive definition of the isomorphisms $f_*\theta'_{X,\mathcal{S}}\simeq \theta_{X,\mathcal{S}}$ and $\check{f}_*\theta'_{\check{X},\check{\mathcal{S}}}\simeq \theta_{\check{X},\check{\mathcal{S}}}$. The details of the proof are left to the reader.
\end{proof}

\subsubsection{The diagram of schemes $\mathcal{T}$}
\label{subsub:calT}
Recall from \S\ref{subsub:setting-for-comput}
that for $\emptyset\neq I \subset [\![0,n]\!]$, there is a
stratification $\mathcal{R}(I)$ on $Y_{{\rm min}(I)}$ induced by
the set of  irreducible components of $Y_{{\rm min}(I)}^{\infty}$
whose image in $X$ is an irreducible component of some $X_{\geq
j}$ with $j\in I$. Moreover, the subscheme $T(I)\subset Y_{{\rm
min}(I)}$ is $\mathcal{R}(I)$-constructible. We let $A(I)$ denote
the set of irreducible closed $\mathcal{R}(I)$-constructible
subsets of $T(I)$. The set $A(I)$ is ordered by
inclusion. There is an non-decreasing bijection from the set of $\mathcal{R}(I)$-strata contained in $T(I)$ which is given by taking closures. Clearly, every irreducible component of $T(I)$ is in
$A(I)$. In particular, the elements of $A(I)$ form a covering of
the scheme $T(I)$ by closed subsets. Note also that if $D_1$ and
$D_2$ are in $A(I)$ and $D$ is a connected component of $D_1\cap
D_2$, then $D\in A(I)$.

\begin{proposition}
\label{prop:for-defining-s-J-I-and-funct}
Let $\emptyset \neq J \subset I \subset [\![0,n]\!]$
and $D\in A(I)$. Then there is a smallest element ${\rm s}_{J\subset I}(D) \in A(J)$ containing the image of $D$ by $T(I)\to T(J)$.
Moreover, the mappings ${\rm s}_{J\subset I}$ make $A$ into a contravariant functor from $\mathcal{P}^*([\![0,n]\!])$ to the category of ordered sets.

\end{proposition}

\begin{proof}
If $T_1$ and $T_2$ are two elements in $A(J)$ containing $(T(I)\to T(J))(D)$, then the connected component of $T_1\cap T_2$ containing
$(T(I)\to T(J))(D)$ is also in $A(J)$. This proves the existence of
${\rm s}_{I\subset J}(D)$.

Next, we show that the maps ${\rm s}_{J\subset I}$ make $A$ into a contravariant functor.
Let $\emptyset \neq K\subset J$ be a third subset of $[\![0,n]\!]$.
As ${\rm s}_{K\subset J}\, {\rm s}_{J\subset I}(D)$ contains the image of $D$ by the morphism $T(I)\to T(K)$, we have by the minimality of
${\rm s}_{K\subset I}(D)$ that
\begin{equation}
\label{eq-prop:rms-est-transitive}
{\rm s}_{K\subset I}(D) \subset {\rm s}_{K\subset J}{\rm s}_{J\subset I}(D).
\end{equation}
Let $J'=\{{\rm min}(J)\}\bigcup K$. Then $J'\subset J$ with ${\rm min}(J')={\rm min}(J)$,
and every
$\mathcal{R}(J')$-constructible subset of $Y_{{\rm min}(J)}$ is also
$\mathcal{R}(J)$-constructible.
By the minimality of
${\rm s}_{J\subset I}(D)$ we thus get an inclusion
${\rm s}_{J\subset I}(D)\subset {\rm s}_{J'\subset I}(D)$.
It follows that ${\rm s}_{K\subset J}\,{\rm s}_{J\subset I}(D)\subset {\rm s}_{K\subset J'}\,{\rm s}_{J'\subset I}(D)$. Thus, it suffices to show that
$${\rm s}_{K\subset I}(D)={\rm s}_{K\subset J'}\,{\rm s}_{J'\subset I}(D).$$
In other words, we may assume that $J=\{j_0\}\bigsqcup K$ for a $0\leq j_0<{\rm min}(K)$. In this case, $T(J)\to T(K)$ is dominant and, by Property \textbf{P1)},
${\rm s}_{K\subset J}$ takes an element of $A(J)$ to its image by $T(J)\to T(K)$.

Again by Property \textbf{P1)}, the inverse image along $T(J) \to
T(K)$ of an $\mathcal{R}(K)$-constructible subset is
$\mathcal{R}(J)$-constructible. In particular, $(T(J)\to
T(K))^{-1}({\rm s}_{K\subset I}(D))$ is
$\mathcal{R}(J)$-constructible. The same is true for any of its
irreducible components. Denote by $P$ one of these irreducible
components containing $(T(I)\to T(J))(D)$. Then, $P\in A(J)$ and
${\rm s}_{J \subset I}(D)\subset P$. It follows that ${\rm
s}_{K\subset I}(D)$ contains the image of ${\rm s}_{J\subset
I}(D)$ in $T(K)$, and hence ${\rm s}_{K\subset J}\,{\rm
s}_{J\subset I}(D)\subset {\rm s}_{K\subset I}(D)$. This proves
the proposition.
\end{proof}

\begin{lemma}
\label{lemma:dominance-elements-A-I}
Let $\emptyset \neq I \subset [\![0,n]\!]$. The image in $X$ of an element $E\in A(I)$
is an irreducible component of $X_{\geq {\rm max}(I)}$.
\end{lemma}

\begin{proof}
Let $i_0={\rm min}(I)$. When $I=\{i_0\}$, $E=Y_{i_0}$ and there is nothing to prove. Also when $n\in I$, the claim is clear as
the image of $E$ in $X$ is an irreducible $\mathcal{S}$-constructible subset contained in $X_n$.

We now assume that ${\rm card}(I)\geq 2$ and ${\rm max}(I)\leq n-1$.
If $D$ is an irreducible component of $Y_{i_0}^{\infty}$ containing $E$, then $D\subset \overline{e_{i_0}^{-1}(X_j)}$ for some $j\in I-\{{\rm min}(I)\}$. This shows that $E$ is not contained in
$e_{i_0}^{-1}(X_{\geq {\rm max}(I)+1})$. As the image of $E$ in $X$
is an $\mathcal{S}$-constructible, closed and irreducible subset of $X_{\geq {\rm max}(I)}$,
it must contain a connected component of
$X_{{\rm max}(I)}$. Thus, it is an irreducible component of
$X_{\geq {\rm max}(I)}$.
\end{proof}

\begin{proposition}
Let $\emptyset \neq I \subset [\![0,n]\!]$. Taking the image by
the morphism $\check{T}(I) \to T(I)$ yields a mapping
$\check{A}(I) \to A(I)$. As $I$ varies, these mappings define a
natural transformation $\check{A} \to A$ between contravariant
functors from $\mathcal{P}^*([\![0,n]\!])$ to the category of ordered sets.
\end{proposition}

\begin{proof}
The image by $\check{T}(I)\to T(I)$ of an element in $\check{A}(I)$
is indeed an element of $T(I)$ as $\check{Y}_{{\rm min}(I)} \to Y_{{\rm min}(I)}$ maps an
$\check{\mathcal{R}}(I)$-stratum to an $\mathcal{R}(I)$-stratum.

Next, let $\emptyset \neq J \subset I \subset [\![0,n]\!]$. We need to check that the square
$$\xymatrix@C=1.5pc@R=1.5pc{\check{A}(I) \ar[r] \ar[d]_-{\check{{\rm s}}_{J\subset I}} & A(I) \ar[d]^-{{\rm s}_{J\subset I}} \\
\check{A}(J) \ar[r] & A(J) }$$ is commutative. Let $\check{D}\in
\check{A}(I)$ and call $D\in A(I)$ its image by $\check{T}(I) \to
T(I)$. Then $(\check{T}(J)\to T(J))(\check{s}_{J\subset
I}(\check{D}))$ is an $\mathcal{R}(J)$-constructible, closed and
irreducible subset containing $(T(I)\to T(J))(D)$. By the
minimality of ${\rm s}_{J\subset I}$, we get the inclusion
$${\rm s}_{J\subset I}(D) \subset (\check{T}(J)\to T(J)) (\check{s}_{J\subset I}(\check{D})).$$

On the other hand, using again that
$\check{Y}_{{\rm min}(J)} \to Y_{{\rm min}(J)}$ maps an $\check{\mathcal{R}}(J)$-stratum to an $\mathcal{R}(J)$-stratum,
we see that
$$(\check{Y}_{{\rm min}(J)}\to Y_{{\rm min}(J)})^{-1}({\rm s}_{J\subset I}(D))$$
is $\check{\mathcal{R}}(J)$-constructible. Let $P$ be an
irreducible component of this subset which contains
$(\check{T}(I)\to \check{T}(J))(\check{D})$. Then $P$ is also
$\check{\mathcal{R}}(J)$-constructible and thus contains
$\check{{\rm s}}_{J\subset I}(\check{D})$. This gives the opposite
inclusion $(\check{T}(J)\to T(J)) (\check{s}_{J\subset
I}(\check{D}))\subset {\rm s}_{J\subset I}(D)$.
\end{proof}

We also record the following lemma and corollary for later use:

\begin{lemma}
\label{lem:fibers-of-s-J-subset-I}
Let $\emptyset \neq J \subset I \subset [\![0,n]\!]$.
We assume that ${\rm min}(I)={\rm min}(J)=i_0$.
Let $F \in A(J)$. Then
\begin{equation}
\label{eq-lem:fibers-of-s-J-subset-I-1}
F\bigcap \left(\bigcup_{i\in I-J} \overline{e_{i_0}^{-1}(X_i)}\right)
\end{equation}
is a sncd in $F$. It induces a stratification which we denote by
$\mathcal{R}_F(J|I)$. Then, for an element $E\in A(I)$, we have
$F={\rm s}_{J\subset I}(E)$ if and only if $E$ is
$\mathcal{R}_F(J|I)$-constructible.
\end{lemma}

\begin{proof}
There is a unique family of irreducible components $(D_{\alpha})_{\alpha\in A}$ of $Y_{i_0}^{\infty}$ such that $E$ is a connected component of $\bigcap_{\alpha\in A} D_{\alpha}$.
As $E$ is
$\mathcal{R}(I)$-constructible, there is a map $t:A \to I-\{i_0\}$ such that $e_{i_0}(D_{\alpha})$ is an irreducible component of $X_{\geq t(\alpha)}$ for all $\alpha\in A$.

Now, assume that $F={\rm s}_{J\subset I}(E)$.
For $\alpha\in A$ such that $t(\alpha)\in J$, we must have
$F\subset D_{t(\alpha)}$. Indeed, the connected component $C$ of
$F\cap D_{t(\alpha)}$ containing $E$ is an $\mathcal{R}(J)$-constructible subset of $T(J)$ containing $E$. By the minimality of
$F={\rm s}_{J\subset I}(E)$, we must have $F=C$. It follows that
$E$ is a connected component of
$$F\bigcap \left(\bigcap_{\alpha\in A} D_{\alpha}\right)=F\bigcap\left(\bigcap_{\alpha\in t^{-1}(I-J)} D_{\alpha}\right).$$
This proves that $E$ is $\mathcal{R}_F(J|I)$-constructible.

Conversely, if ${\rm s}_{J\subset I}(E)\subsetneq F$, we can find an irreducible component $D$ of $Y_{i_0}^{\infty}$, dominating an irreducible component of $X_{\geq j_0}$ with $j_0\in J-\{i_0\}$, and such that
$E\subset F\cap D \subsetneq F$. But then, $F\cap D$ does not contain any non-empty $\mathcal{R}_F(J|I)$-constructible subset.
Thus, $E$ cannot be $\mathcal{R}_F(J|I)$-constructible.
\end{proof}

\begin{corollary}
\label{cor:fibers-of-s-J-subset-I}
Let $\emptyset \neq J \subset I \subset [\![0,n]\!]$ such that
${\rm min}(I)={\rm min}(J)$.
Let $F,\, F'\in A(J)$ and assume that $F\subset F'$. Let
$E\in {\rm s}_{J\subset I}^{-1}(F)$. Then, there is a smallest
element
$E'\in {\rm s}_{J\subset I}^{-1}(F')$ such that $E\subset E'$.
This defines an non-decreasing map ${\rm s}_{J\subset I}^{-1}(F) \to
{\rm s}_{J\subset I}^{-1}(F')$. We obtain in this way a functor
from $A(J)$ to the category of ordered sets sending
$F\in A(J)$ to ${\rm s}_{J\subset I}^{-1}(F)$. Moreover,
$\int_{A(J)} {\rm s}_{J\subset I}^{-1}(-)$ is canonically isomorphic to
$A(I)$.

\end{corollary}

\begin{proof}
The first statement (i.e., the existence of $E'$) follows from
Lemma \ref{lem:fibers-of-s-J-subset-I} by the same argument as in the proof of Proposition \ref{prop:for-defining-s-J-I-and-funct}. The other statements are easy and will be left to the reader.
\end{proof}

Given $\emptyset\neq I \subset [\![0,n]\!]$, elements of
$A(I)$ will be denoted by greek letters, $\alpha$, $\beta$, etc,
and the corresponding irreducible closed subschemes of $T(I)$ will be denoted by $\mathcal{T}(I,\alpha)$, $\mathcal{T}(I,\beta)$, etc.
The assignment
\begin{equation}
\label{eq:new-diag-cal-T-defn}
\mathcal{T}(I):\alpha \rightsquigarrow \mathcal{T}(I,\alpha)
\end{equation}
is a contravariant functor from the ordered set $A(I)$
to the category of $X$-schemes.
Thus, for each $I\in \mathcal{P}^*([\![0,n]\!])$, we have a diagram of schemes $(\mathcal{T}(I),A(I))$.
Moreover, the assignment
\begin{equation}
\label{eq:new-diag-cal-T-defn-II}
\mathcal{T}:I\rightsquigarrow
(\mathcal{T}(I),A(I))
\end{equation}
is also a contravariant functor and gives a diagram in
$\Dia(\Sch/k)$.
The inclusions $\mathcal{T}(I,\alpha)\hookrightarrow T(I)$
induce tautological morphisms
\begin{equation}
\label{TnT}
(\mathcal{T}(I),A(I))\to
T(I),
\end{equation}
that are natural in $I$. Moreover,
the morphism $l:\check{X} \to X$ induces morphisms of diagrams of
schemes $(\check{\mathcal{T}}(I),
\check{A}(I)) \to
(\mathcal{T}(I),A(I))$ that are natural in
$I$, and thus give a morphism in $\Dia(\Dia(\Sch/k))$.

\subsubsection{The diagram of schemes $\mathcal{Y}$ and the motive
$\theta''_{X,\mathcal{S}}$}\label{subsub:Y}

For $(I_0,I_1)\in \mathcal{P}_2([\![1,n]\!])$, let $J=[\![0,n]\!]- I_0$ and order
$\{0\}\bigsqcup I_1=\{i_0< \dots < i_s\}$.
Then $i_0=0$ and we let $i_{s+1}=n$.
We define
a diagram of schemes $\mathcal{Y}(I_0,I_1)$ as follows.
First, we construct a sequence of diagrams of schemes
$\mathcal{Y}_1(I_0,I_1), \dots, \mathcal{Y}_{s+1}(I_0,I_1)$ with
morphisms $p_j(I_0,I_1):\mathcal{Y}_j(I_0,I_1)\to \mathcal{T}(J\cap [\![i_{j-1},i_j]\!])$ and then set $\mathcal{Y}(I_0,I_1)=\mathcal{Y}_{s+1}(I_0,I_1)$.
Let $\mathcal{Y}_1(I_0,I_1)=\mathcal{T}(J\cap [\![i_0,i_1]\!])$ and take the identity morphism for $p_1(I_0,I_1)$. Now assume that
$\mathcal{Y}_j(I_0,I_1)$ and $p_j(I_0,I_1)$
are defined for some $j\leq s$.
The composition
$$\mathcal{Y}_j(I_0,I_1)\to \mathcal{T}(J\cap [\![i_{j-1},i_j]\!]) \to Y_{i_j}$$
makes $\mathcal{Y}_j(I_0,I_1)$ into a diagram of projective $Y_{i_j}$-schemes. In particular, we may consider the diagram $\pi_0(\mathcal{Y}_j(I_0,I_1)/Y_{i_j})$
obtained by taking objectwise the Stein factorizations of the various projections to $Y_{i_j}$. We then define
$$\mathcal{Y}_{j+1}(I_0,I_1)=\pi_0(\mathcal{Y}_j(I_0,I_1)/Y_{i_j}) \times_{Y_{i_j}} \mathcal{T}(J\cap [\![i_j, i_{j+1}]\!])$$
and take for $p_{j+1}(I_0,I_1)$ the projection to the second factor.

By construction, we have a morphism $p(I_0,I_1):\mathcal{Y}(I_0,I_1)\to \mathcal{T}(\varsigma_n(I_0,I_1))$ in $\Dia(\Sch/k)$.
The indexing category $\mathcal{C}(I_0,I_1)$ of $\mathcal{Y}(I_0,I_1)$ is
$$A(J\cap [\![i_0,i_1]\!])\times \dots \times
A(J\cap [\![i_{s-1},i_s]\!]) \times
A(J\cap [\![i_s,n]\!]).$$
The following gather some properties related to this construction.

\begin{proposition}
\label{prop:basic-property-diag-cal-Y-gen}

\begin{enumerate}

\item[(a)] The assignment $\mathcal{Y}:(I_0,I_1) \rightsquigarrow \mathcal{Y}(I_0,I_1)$ extends naturally to a functor from $\mathcal{P}_2([\![1,n]\!])$ to $\Dia(\Sch/k)$.
Moreover, the $p(I_0,I_1)$'s define a morphism of diagrams $p:\mathcal{Y}\to \mathcal{T}\circ \varsigma_n$.

\item[(b)]
Given an object $(\alpha_j)_{j=0,\dots, s}$ of $\mathcal{C}(I_0,I_1)$,
the $k$-scheme $\mathcal{Y}(I_0,I_1,(\alpha_j)_j)$ has only quotient singularities.
The morphism
$\mathcal{Y}(I_0,I_1,(\alpha_j)_j)\to \mathcal{T}(\varsigma_n(I_0,I_1),\alpha_s)$ is finite and each connected component of
$\mathcal{Y}(I_0,I_1,(\alpha_j)_j)$ is dominated by a connected component of $Z_{i_s}\times_{Y_{i_s}} \mathcal{T}(\varsigma_n(I_0,I_1),\alpha_s)$ where $Z_{i_s}$ is the scheme given in \emph{\textbf{D3)}}.
\end{enumerate}

\end{proposition}

\begin{proof}
For (a), consider two pairs $(I_0,I_1)\subset (I_0',I_1')$ in $\mathcal{P}_2([\![1,n]\!])$ and set $J=[\![0,n]\!]- I_0$ and
$J'=[\![0,n]\!]- I_0'$. Also order $\{0\}\bigsqcup I_1=\{i_0<\dots < i_s\}$ and $\{0\}\bigsqcup I_1'=\{i'_0<\dots< i'_{s'}\}$ and set $i_{s+1}=i'_{s'+1}=n$. Let $\tau:[\![0,s+1]\!]\hookrightarrow [\![0,s'+1]\!]$ be the map such that $i'_{\tau(j)}=i_j$ for all $0\leq j \leq s+1$.
We construct by induction on $j\in [\![1,s+1]\!]$ a morphism
$\mathcal{Y}_j(I_0,I_1)\to
\mathcal{Y}_{\tau(j)}(I_0',I_1')$. Assume this is done for $j\leq s$. Remark that
$$\mathcal{Y}_{j+1}(I_0,I_1) = \pi_0(\mathcal{Y}_{j}(I_0,I_1)/Y_{i_{j}})\times_{Y_{i_j}}
\mathcal{T}(J\cap [\![i'_{\tau(j)},i'_{\tau(j+1)}]\!]).$$
We use a second induction, now on $\tau(j)\leq l \leq \tau(j+1)$, to construct morphisms of diagrams
$$(\pi_0(\mathcal{Y}_{j}(I_0,I_1)/Y_{i_{j}})\times_{Y_{i_j}}
\mathcal{T}(J\cap [\![i'_{\tau(j)},i'_{l}]\!])) \to \mathcal{Y}_{l}(I_0',I_1').$$
For $l=\tau(j+1)$, we obtain the morphism
$\mathcal{Y}_{j+1}(I_0,I_1)\to
\mathcal{Y}_{\tau(j+1)}(I_0',I_1')$.
We leave the details to the reader.

Let $1\leq t\leq s$ and assume that
each connected component of
$\mathcal{Y}_{t}(I_0,I_1,(\alpha_j)_{0\leq j \leq t-1})$
is dominated by a connected component $F$ of $Z_{i_{t-1}}\times_{Y_{i_{t-1}}} \mathcal{T}(J\cap [\![i_{t-1},i_t]\!],\alpha_{t-1})$.
To show the corresponding property for $\mathcal{Y}_{t+1}$, it is thus sufficient to show that every connected component of
$\pi_0(F/Y_{i_t})\times_{Y_{i_t}} \mathcal{T}(J\cap [\![i_t,i_{t+1}]\!],\alpha_t)$ is dominated by a connected component of
$Z_{i_t}\times_{Y_{i_t}}
\mathcal{T}(J\cap [\![i_t,i_{t+1}]\!],\alpha_t)$. By \textbf{P2)},
$\pi_0(F/Y_{i_t})$ is dominated by a connected component of
$Z_{i_t}$. This proves the second assertion in (b) by induction. That
$\mathcal{Y}(I_0,I_1,(\alpha_j)_j)$ has quotient singularities is now clear as the latter is normal and has a (possibly ramified) Galois covering by a connected component of
$Z_{i_s}\times_{Y_{i_s}} \mathcal{T}(\varsigma_n(I_0,I_1),\alpha_s)$, which is a smooth scheme.
\end{proof}

There is a commutative triangle in $\Dia(\Dia(\Sch/k))$
$$\xymatrix@C=1.5pc@R=1.5pc{(\mathcal{Y},\mathcal{P}_2([\![1,n]\!])) \ar[r]^-h \ar@/_/[dr]_-{(h,\varsigma_n)} &  (\mathcal{X},\mathcal{P}_2([\![1,n]\!])) \ar[d]^-{\varsigma_n}\\
& (T,\mathcal{P}^*([\![0,n]\!])^{\rm op}) }$$
where, for $(I_0,I_1)\in \mathcal{P}_2([\![1,n]\!])$,
$h$ is the composition
$$\mathcal{Y}(I_0,I_1)\to \mathcal{T}(\varsigma_n(I_0,I_1)) \to T(\varsigma_n(I_0,I_1)).$$

\begin{remark}
\label{rem:for-doing-induction-X-X'-III}
We assume that $n\geq 1$ and we use the notation as in Remarks
\ref{rem:for-doing-induction-X-X'}
and \ref{rem:for-doing-induction-X-X'-II}.
For $\emptyset\neq I\subset [\![0,n-1]\!]$, we denote by $A'(I)$ the set of irreducible closed subsets of
$T'(I)$ which are $\mathcal{R}'(I)$-constructible. It follows from
Lemma \ref{lemma:dominance-elements-A-I} that the map
$A'(I) \to A(I)$, which takes $Z\in A'(I)$ to its Zariski closure in $T(I)$, is a bijection. Hence, we have an objectwise dense open immersion
$\mathcal{T}'(I) \to \mathcal{T}(I)$.
Similarly, let
$(I_0,I_1)\in \mathcal{P}_2([\![1,n-1]\!])$.
Set $J=[\![0,n-1]\!]-I_0=[\![0,n]\!]-
(I_0\bigsqcup \{n\})$ and order
$\{0\}\bigsqcup I_1=\{i_0<\dots < i_s\}$.
By induction on $1\leq j \leq s$, it is easy to see that
$\mathcal{Y}'_j(I_0,I_1)\simeq \mathcal{Y}_j(I_0\bigsqcup \{n\},I_1)\times_X X'$
(with $\mathcal{Y}'_j(I_0,I_1)$ the diagram constructed as above using $X'$, $Y_i'$, etc).
This gives an objectwise dense open immersion
$j:\mathcal{Y}' \to \mathcal{Y}\circ \iota_n^0$.

\end{remark}

In the sequel, we abuse notation and denote by $\mathcal{Y}$ the total diagram of schemes associated to $\mathcal{Y}\in \Dia(\Dia(\Sch/k))$.
We will define a commutative unitary algebra $\theta''_{X,\mathcal{S}}\in \DM(\mathcal{Y})$ using induction on $n$. When $n=0$, $\mathcal{Y}$ is the family of connected components of $X$, and we take $\theta''_{X,\mathcal{S}}=\un_{\mathcal{Y}}$.

Assume $n\geq 1$ and that $\theta''_{X',\mathcal{S}'}$
has been constructed (with the notation of Remark
\ref{rem:for-doing-induction-X-X'-III}).
Consider the following diagram in $\Dia(\Dia(\Sch/k))$:
$$\xymatrix{\mathcal{Y}' \ar[r]^-j & \mathcal{Y}\circ \iota^0_n & \mathcal{Y} \circ o \ar[l]_-b \ar[r]^-o & \mathcal{Y},}$$
which we also view as a diagram in $\Dia(\Sch/k)$ by passing to total diagrams.  Recall
$o:\mathcal{P}_2([\![1,n-1]\!])\times \underline{\mathbf{1}}\hookrightarrow\mathcal{P}_2([\![1,n-1]\!])\times \ucarre = \mathcal{P}_2([\![1,n]\!])$,
which is induced by the inclusion
$(-,0):\underline{\mathbf{1}}\hookrightarrow \ucarre$. The morphism $b$ is given on the indexing categories by the projection to the first factor of $\mathcal{P}_2([\![1,n-1]\!])\times \underline{\mathbf{1}}$. Its restriction to
$\mathcal{P}_2([\![1,n-1]\!])\times \{1\}$ is the identity morphism.
Its restriction to $\mathcal{P}_2([\![1,n-1]\!])\times \{0\}$ is the morphism $\mathcal{Y}\circ \iota_n \to \mathcal{Y}\circ \iota_n^0$ induced by the natural transformation $\iota_n\to \iota_n^0$ (where $\iota_n:\mathcal{P}_2([\![1,n-1]\!]) \hookrightarrow \mathcal{P}_2([\![1,n]\!])$ is the inclusion).
With this notation, we set:
\begin{equation}
\label{eq:recursive-defn-theta-double-prime}
\theta''_{X,\mathcal{S}}=\omega^0_{\{(0,1)\}|(\mathcal{Y},\smallucarre)}\left(o_*b^*j_*\theta''_{X',\mathcal{S}'}\right).
\end{equation}
In the above formula,
$\omega^0_{\{(0,1)\}|(\mathcal{Y},\smallucarre)}$ is really
$\omega^0_{\mathcal{P}_2([\![1,n-1]\!])\times \{(0,1)\}|(\mathcal{Y},\mathcal{P}_2([\![1,n-1]\!])\times \smallucarre)}$
(see Remark
\ref{notational-remark-for-omega-on-diagrams}).
This is again a commutative unitary algebra.

\begin{proposition}
\label{prop:comparison-theta-prime-and-h-theta-double}
There is a canonical isomorphism of commutative unitary algebras $\theta'_{X,\mathcal{S}} \simeq h_*\theta''_{X,\mathcal{S}}$, with
$h:\mathcal{Y} \to \mathcal{X}$ the natural morphism.
\end{proposition}

\begin{proof}
We argue by induction on $n$. When $n=0$, the claim is clear.
Assume that $n\geq 1$ and let $h':\mathcal{Y}' \to \mathcal{X}'$ denote the natural morphism of diagrams of schemes. By induction, we may assume that the isomorphism $\theta'_{X',\mathcal{S}'}\simeq h'_*\theta''_{X',\mathcal{S}'}$ is constructed. We split the proof in four parts.

\smallskip

\noindent \underbar{Part A}: We have a commutative diagram in
$\Dia(\Sch/k)$:
\begin{equation}
\label{eq-prop:comparison-theta-prime-and-h-theta-double}
\xymatrix@C=1.5pc@R=1.5pc{\mathcal{Y}' \ar[r]^-j \ar[d]_-{h'} & \mathcal{Y}\circ \iota_n^0 \ar[d]^-h &  \mathcal{Y}\circ o\ar[d]^-h \ar[l]_-b \ar[r]^o &\mathcal{Y}\ar[d]^-h  \\
\mathcal{X}' \ar[r]^-j & \mathcal{X}\circ \iota_n^0 & \mathcal{X}\circ o \ar[l]_-b \ar[r]^-o & \mathcal{X}.\!}
\end{equation}
This gives natural transformations
$$o_*b^*j_*h'_*\simeq o_*b^*h_*j_*\to o_*h_*b^* j_* \simeq h_*o_*b^*j_*.$$
Recall that
$\theta'_{X,\mathcal{S}}=\omega^0_{\{(0,1)\}|(\mathcal{X},\smallucarre)} o_*b^*j_*\theta'_{X',\mathcal{S}'}$.
Our morphism $\theta'_{X,\mathcal{S}} \to h_*\theta''_{X,\mathcal{S}}$ is then the composition
$$\xymatrix@C=1.5pc@R=1.5pc{\omega^0_{ \{(0,1)\}| (\mathcal{X},\smallucarre)} o_*b^*j_*\theta'_{X',\mathcal{S}'} \ar[r]^-{\sim} &
\omega^0_{\{(0,1)\}| (\mathcal{X}, \smallucarre)} o_*b^*j_*h'_*\theta''_{X',\mathcal{S}'} \ar[d] \\
& \omega^0_{ \{(0,1)\}| (\mathcal{X}, \smallucarre)} h_*o_*b^*j_*\theta''_{X',\mathcal{S}'} \ar[r] &
h_* \omega^0_{\{(0,1)\}|(\mathcal{Y},\smallucarre)}o_*b^*j_*\theta''_{X',\mathcal{S}'}.}$$
(The last morphism is constructed in the same way as in Proposition
\ref{prop:additional-prop-omega-0-x}, (iii).)
To prove the proposition, we need to check that the following natural transformations are invertible:
\begin{enumerate}

\item the base change morphism $b^*h_* \to h_*b^*$ associated to the middle commutative square in \eqref{eq-prop:comparison-theta-prime-and-h-theta-double},

\item $\omega^0_{\{(0,1)\}| (\mathcal{X}, \smallucarre)} h_*
\to h_* \omega^0_{\{(0,1)\}|(\mathcal{Y},\smallucarre)}$.

\end{enumerate}
The first natural transformation will be treated in the next two parts. The second one, will be treated in the last part.

\smallskip

\noindent \underbar{Part B}: Here we begin the verification that
the base change morphism $b^*h_* \to h_* b^*$
is invertible. It suffices to show that this natural transformation is invertible when applying $((I_0,I_1),0)^*$ and
$((I_0,I_1),1)^*$ for $(I_0,I_1)\in \mathcal{P}_2([\![1,n-1]\!])$.
Using Corollary
\ref{cor:very-useful-texnical-cor},
we see that it suffices to show that the base change morphisms associated to the squares
$$ \xymatrix@C=1.7pc@R=1.7pc{\mathcal{Y}(I_0,I_1) \ar[r]^-{b}  \ar[d]^-{h(I_0,I_1)}  & \mathcal{Y}(I_0\bigsqcup \{n\},I_1) \ar[d]^-{h(I_0\bigsqcup\{n\},I_1)} \\
\mathcal{X}(I_0,I_1) \ar[r]_-{b} & \mathcal{X}(I_0\bigsqcup \{n\},I_1) } \qquad
\xymatrix@C=1.7pc@R=1.7pc{\mathcal{Y}(I_0\bigsqcup \{n\},I_1) \ar[r]^-{b}  \ar[d]^-{h(I_0\bigsqcup\{n\},I_1)}  & \mathcal{Y}(I_0\bigsqcup \{n\},I_1) \ar[d]^-{h(I_0\bigsqcup\{n\},I_1)} \\
\mathcal{X}(I_0\bigsqcup\{n\},I_1) \ar[r]_-{b} & \mathcal{X}(I_0\bigsqcup \{n\},I_1) ,\!}$$
are invertible. As the horizontal arrows in the second square are  identities, we only need to consider the first square.
For this, remark that $\mathcal{X}(I_0,I_1)=\mathcal{X}(I_0\bigsqcup \{n\},I_1)\times_X X_n$. Thus, we may factor this square as follows
\begin{equation}
\label{eq-prop:comparison-theta-prime-and-h-theta-double-2}
\xymatrix@C=1.5pc@R=1.5pc{\mathcal{Y}(I_0,I_1) \ar[r]^-c \ar@/_/[dr]_-h \ar@/^1.9pc/[rr]^{b} & \mathcal{Y}(I_0\bigsqcup \{n\}, I_1) \times_X X_n \ar[r]^-{b_1} \ar[d]^-{h_1} & \mathcal{Y}(I_0\bigsqcup \{n\}, I_1) \ar[d]^-h \\
& \mathcal{X}(I_0\bigsqcup \{n\},I_1)\times_X X_n \ar[r]^-b & \mathcal{X}(I_0\bigsqcup \{n\},I_1), \!}
\end{equation}
where, to simplify notation, we wrote $h$ for $h(I_0,I_1)$ and $h(I_0\bigsqcup \{n\}, I_1)$. Using this commutative diagram
\eqref{eq-prop:comparison-theta-prime-and-h-theta-double-2}, we may factor the base change morphism $b^*h_* \to h_*b^*$ as follows:
$$b^*h_* \to h_{1*} b_1^* \to h_{1*} c_*c^* b_1^* \simeq h_*b^*.$$
Applying Proposition
\ref{prop-most-gen-proj-base-change} to the cartesian square in
\eqref{eq-prop:comparison-theta-prime-and-h-theta-double-2},
we get that the base change morphism
$b^*h_* \to h_{1*} b_1^*$ is invertible. Thus, it remains to show that
the unit morphism $\id \to c_*c^*$ is invertible.
This will be treated in the next part.

\smallskip

\noindent \underbar{Part C}: Let $J=[\![0,n-1]\!]- I_0$ and order
$\{0\}\bigsqcup I_1=\{i_0<\dots< i_s\}$. By the construction of $\mathcal{Y}$, we have
a cartesian square in $\Dia(\Sch/k)$:
$$\xymatrix@C=1.5pc@R=1.5pc{\mathcal{Y}(I_0,I_1) \ar[d]_-{c} \ar[r] & \mathcal{T}(K') \ar[d]^-{c'} \\
\mathcal{Y}(I_0\bigsqcup \{n\},I_1) \times_X X_n \ar[r] & \mathcal{T}(K)\times_X X_n,\!}$$
where $K=J\cap [\![i_s,n-1]\!]$ and $K'=K\bigsqcup \{n\}$.

Recall that $\mathcal{T}(K)\times_X X_n$ is indexed by
the ordered set $A(K)$ of irreducible, closed and $\mathcal{R}(K)$-constructible subsets of $T(K)$.
By Corollary \ref{cor:fibers-of-s-J-subset-I}, there is
a functor
${\rm s}_{K\subset K'}^{-1}(-):A(K) \to \Dia$
such that $A(K')\simeq \int_{A(K)}{\rm s}_{K\subset K'}^{-1}(-)$.
Moreover, with $\upsilon_{\alpha}:{\rm s}_{K\subset K'}^{-1}(\alpha) \hookrightarrow A(K')$ the inclusion, the assignment
\begin{equation}
\label{eq:assignment-totol-diagram-pour-mathcal-T}
\alpha \in A(K) \quad \rightsquigarrow \quad
(\mathcal{T}(K')\circ \upsilon_{\alpha} , {\rm s}_{K\subset K'}^{-1}(\alpha))
\end{equation}
is a functor from $A(K)$ to $\Dia(\Sch/k)$.
Also, the total diagram associated to
\eqref{eq:assignment-totol-diagram-pour-mathcal-T} coincides with
$\mathcal{T}(K')$.
Thus, $c'$ and hence $c$ satisfy the conditions on $(f,\rho)$ in Corollary \ref{cor:very-useful-texnical-cor}.

Now, as usual, it suffices to check that the natural transformation
$((\alpha_j)_j)^* \to ((\alpha_j)_j)^* c_*c^*$ is invertible
for $(\alpha_j)_{0\leq j\leq s}$ in the indexing category $\mathcal{C}(I_0\bigsqcup\{n\},I_1)$
of the diagram
$\mathcal{Y}(I_0\bigsqcup \{n\},I_1)$.
By Corollary
\ref{cor:very-useful-texnical-cor},
the base change morphism associated to the
cartesian square
$$\xymatrix@C=1.5pc@R=1.5pc{\mathcal{Y}(I_0,I_1,((\alpha_j)_{0\leq j\leq s-1},\upsilon_{\alpha_s})) \ar[d]_-{c((\alpha_j)_j)} \ar[r] & \mathcal{Y}(I_0,I_1) \ar[d]^-c \\
\mathcal{Y}(I_0\bigsqcup \{n\},I_1,(\alpha_j)_j) \times_X X_n \ar[r] & \mathcal{Y}(I_0\bigsqcup \{n\},I_1) \times_X X_n}$$
is invertible. Hence, it suffices to check that
$\id \to c((\alpha_j)_j)_*c((\alpha_j)_j)^*$ is invertible.
On the other hand, the morphism
$\mathcal{Y}(I_0\bigsqcup \{n\},I_1)((\alpha_j)_j) \times_X X_n \to \mathcal{T}(K,\alpha_s)\times_X X_n$
is finite and the cohomological direct image along this map is conservative. This reduces us to check that
$\id \to c'(\alpha)_*c'(\alpha)^*$ is invertible for any $\alpha \in A(K)$.

Recall that $c'(\alpha)$ is the natural morphism $(\mathcal{T}(K')\circ \upsilon_{\alpha},{\rm s}_{K\subset K'}^{-1}(\alpha)) \to \mathcal{T}(K,\alpha)\times_X X_n$. We are now in the situation of
Lemma \ref{lemma:form-locality-closed-cover} where $X$ is given by
$\mathcal{T}(K,\alpha)\times_X X_n$
with the stratification induced by the family of its irreducible components. By that lemma, $\id \to c'(\alpha)_*c'(\alpha)^*$ is indeed an isomorphism. This finishes the verification that
$\id \to c_*c^*$ is invertible.

\smallskip

\noindent \underbar{Part D}: In this part, we finish the proof of
the proposition by showing that the natural transformation
\begin{equation}
\label{eq-prop:comparison-theta-prime-and-h-theta-double-7}
\omega^0_{ \{(0,1)\}|(\mathcal{X}, \smallucarre)} h_*
\to h_* \omega^0_{ \{(0,1)\}|(\mathcal{Y}, \smallucarre)}
\end{equation}
is invertible.
It suffices to show that
\eqref{eq-prop:comparison-theta-prime-and-h-theta-double-7}
is invertible after applying $(I_0,I_1)^*:\DM(\mathcal{X})\to \DM(\mathcal{X}(I_0,I_1))$. There are two cases depending on whether $n\in I_1$ or $n\not \in I_1$.

First, let's assume that $n\not \in I_1$. Then, by
Proposition
\ref{prop:omega-diagram-raisonnable} and Corollary
\ref{cor:very-useful-texnical-cor}, we have
$$(I_0,I_1)^*\omega^0_{\{(0,1)\}, (\mathcal{X},\smallucarre)} h_*
\simeq (I_0,I_1)^* h_*\simeq
h(I_0,I_1)_* (\mathcal{Y}(I_0,I_1)\to \mathcal{Y})^*,$$
where $h(I_0,I_1)$ is the projection of $\mathcal{Y}(I_0,I_1)$ to $\mathcal{X}(I_0,I_1)$.
Similarly,
$$(I_0,I_1)^*
h_* \omega^0_{\{(0,1)\}|(\mathcal{Y},\smallucarre)}
\simeq h(I_0,I_1)_* (\mathcal{Y}(I_0,I_1)\to \mathcal{Y})^*
\omega^0_{\{(0,1)\}|(\mathcal{Y}, \smallucarre)}$$
$$\simeq
h(I_0,I_1)_* (\mathcal{Y}(I_0,I_1)\to \mathcal{Y})^*.$$
Moreover, modulo these isomorphisms, our natural transformation is the identity.

Next, we assume that $n\in I_1$. Using again
Proposition
\ref{prop:omega-diagram-raisonnable} and Corollary
\ref{cor:very-useful-texnical-cor}, we see that
$$(I_0,I_1)^*\omega^0_{\{(0,1)\}, (\mathcal{X},\smallucarre)} h_*
\simeq \omega^0_{\mathcal{X}(I_0,I_1)} (I_0,I_1)^* h_*
\simeq
\omega^0_{\mathcal{X}(I_0,I_1)} h(I_0,I_1)_* (\mathcal{Y}(I_0,I_1)\to \mathcal{Y})^*,$$
and, similarly,
$$(I_0,I_1)^*
h_* \omega^0_{\{(0,1)\}|(\mathcal{Y},\smallucarre)}
\simeq h(I_0,I_1)_* (\mathcal{Y}(I_0,I_1)\to \mathcal{Y})^*
\omega^0_{\{(0,1)\}|(\mathcal{Y},\smallucarre)}$$
$$\simeq
h(I_0,I_1)_* \omega^0_{\mathcal{Y}(I_0,I_1)} (\mathcal{Y}(I_0,I_1)\to \mathcal{Y})^*.$$
Hence, we are left to check that the natural transformation
$$\omega^0_{\mathcal{X}(I_0,I_1)} h(I_0,I_1)_* \to
h(I_0,I_1)_*\omega^0_{\mathcal{Y}(I_0,I_1)}$$
is invertible.
This follows from
Propositions \ref{prop:limits-with-respect-ordered-set} and
\ref{prop:omega-diagram-raisonnable}
as $h(I_0,I_1)$ is objectwise a finite morphism.
\end{proof}

We have a morphism $l:(\check{\mathcal{Y}},\mathcal{P}_2([\![1,n]\!])) \to (\mathcal{Y},\mathcal{P}_2([\![1,n]\!]))$ in $\Dia(\Dia(\Sch/k))$ which we may view as a morphism of diagrams of schemes by passing to the total diagrams. Moreover, the following square
$$\xymatrix@C=1.5pc@R=1.5pc{(\check{\mathcal{Y}},\mathcal{P}_2([\![1,n]\!])) \ar[r]^-l \ar[d]_-{\check{h}} & (\mathcal{Y},\mathcal{P}_2([\![1,n]\!])) \ar[d]^-h \\
(\check{\mathcal{X}},\mathcal{P}_2([\![1,n]\!]) \ar[r]^-l & (\mathcal{X},\mathcal{P}_2([\![1,n]\!]) }$$
is commutative.

\begin{proposition}
\label{prop:fur-functoriality-construction-theta-double-prime}
There is a morphism of motives $l^*\theta''_{X,\mathcal{S}} \to \theta''_{\check{X},\check{\mathcal{S}}}$ which is
invertible when $f:\check{X} \to X$ is smooth and $\check{Y}_i=\check{X}\times_X Y_i$ for $i\in [\![0,n]\!]$. Moreover, the following diagram of $\DM(\check{\mathcal{X}},\mathcal{P}_2([\![1,n]\!]))$:
$$\xymatrix@C=1.5pc@R=1.5pc{l^*h_* \theta''_{X,\mathcal{S}} \ar[r] \ar[d]_-{\sim} & \check{h}_*l^* \theta''_{X,\mathcal{S}} \ar[r] & \check{h}_*\theta'' \ar[d]^-{\sim} \\
l^*\theta'_{X,\mathcal{S}} \ar[rr] & & \theta'_{\check{X},\check{\mathcal{S}}} }$$
commutes; the arrow in the bottom being the morphism of
\emph{Proposition
\ref{prop:functorialite-motif-theta-prime}}.
\end{proposition}

\begin{proof}
The proof is completely analogous to that of Proposition
\ref{prop:functorialite-motif-theta-prime}. We leave it to the reader.
\end{proof}

\subsubsection{The motive $\beta_{X,\mathcal{S}}$}
\label{par:mot-beta-x-s}
In this paragraph, we construct a motive $\beta_{X,\mathcal{S}}$ over the diagram of schemes $(T,\mathcal{P}^*([\![0,n]\!])^{\rm op})$
using only operations of inverse images and cohomological direct images. We then show that $\theta''_{X,\mathcal{S}}$ can be identified with the inverse image of $\beta_{X,\mathcal{S}}$ along $(h,\varsigma_n)$.

First, we introduce a
notation. Let $\mathcal{C}$ be a category having a final object $\star$. Given an object $(\mathcal{W},\mathcal{A})$ of $\Dia(\mathcal{C})$, we denote by
$(\mathcal{W}^+,\mathcal{A}\times \underline{\mathbf{1}})$ the total diagram associated to the functor $\underline{\mathbf{1}} \to \Dia(\mathcal{C})$ sending $0$ to $(\mathcal{W},\mathcal{A})$, $1$ to
$(\star,\mathcal{A})$ and the arrow $0\to 1$
to the unique morphism $(\mathcal{W},\mathcal{A}) \to (\star,\mathcal{A})$, which is the identity on the indexing categories.
We are mainly interested in the case where the category $\mathcal{C}$ is $\Sch/k$ or $\Dia(\Sch/k)$; in both cases, the final object is given by $\Spec(k)$. In particular we have two diagrams of schemes $(T^+,\mathcal{P}^*([\![0,n]\!])^{\rm op} \times \underline{\mathbf{1}})$ and
$(\mathcal{X}^+, \mathcal{P}_2([\![1,n]\!])\times \underline{\mathbf{1}})$. Also we have two objects of $\Dia(\Dia(\Sch/k))$, namely
$(\mathcal{T}^+,\mathcal{P}^*([\![0,n]\!])^{\rm op} \times \underline{\mathbf{1}})$ and $(\mathcal{Y}^+, \mathcal{P}_2([\![1,n]\!])\times \underline{\mathbf{1}})$.

We now define a commutative unitary algebra $\beta^+_{X,\mathcal{S}}\in \DM(T^+,\mathcal{P}^*([\![0,n]\!])^{\rm op}\times \underline{\mathbf{1}})$ by induction on $n$. When $n=0$, we take for $\beta^+_{X,\mathcal{S}}$ the unit motive on the diagram $\{X \to \Spec(k)\}$.
When $n\geq 1$, we use the notation in
Remark \ref{rem:for-doing-induction-X-X'} and assume that $\beta^+_{X',\mathcal{S}'}$ has been constructed.

As before, denote by $\iota_n:\mathcal{P}^*([\![0,n-1]\!]) \hookrightarrow
\mathcal{P}^*([\![0,n]\!])$ the obvious inclusion. Also,
let $o:\mathcal{P}^*([\![0,n-1]\!])^{\rm op}\times \underline{\mathbf{1}} \hookrightarrow \mathcal{P}^*([\![0,n]\!])^{\rm op}$
denote the non-decreasing map sending $(I,0)$ to $I\bigsqcup\{n\}$ and $(I,1)$ to $I$.
We have a diagram in $\Dia(\Sch/k)$:
\begin{equation}
\label{new-diagram-for-defining-beta-x-s-1}
\xymatrix@C=1.7pc{T' \ar[r]^-j & T\circ \iota_n & T\circ o \ar[l]_-b \ar[r]^-o & T,}
\end{equation}
where $j$ is an objectwise dense open immersion and
$b$ is as follows. On the indexing categories, it is given by the projection to the first factor. Over $\mathcal{P}^*([\![0,n-1]\!])\times \{1\}$, it is objectwise an identity morphism, and
over $\mathcal{P}^*([\![0,n-1]\!])\times \{0\}$,
it is the objectwise closed immersion
$(T\circ \iota_n)\times_X X_n \to (T\circ \iota_n)$.

We deduce from \eqref{new-diagram-for-defining-beta-x-s-1} a new diagram in $\Dia(\Sch/k)$:
\begin{equation}
\label{new-diagram-for-defining-beta-x-s-2}
\xymatrix@C=1.7pc{T'^+ \ar[r]^-{j^+} & T^+\circ (\iota_n\times \id_{\underline{\mathbf{1}}}) & T^+\circ (o\times \id_{\underline{\mathbf{1}}}) \ar[l]_-{b^+} \ar[r]^-{o^+} & T^+.}
\end{equation}
On the other hand, we define a morphism of diagrams of schemes
$e_n:T^+ \to T^+$ as follows.
On the indexing categories, we take the identity except on
$(\{n\},0)\in \mathcal{P}^*([\![0,n]\!])^{\rm op}\times \underline{\mathbf{1}}$ which is sent to $(\{n\},1)$.
Also, we take for $e_n(I,u)$ the identity when $(I,u)\neq (\{n\},0)$ and the projection $T(\{n\})=X_n \to \Spec(k)$ when $(I,u)=(\{n\},0)$.
We now define:
$$\beta^+_{X,\mathcal{S}}=e_n^* (o^+)_* (b^+)^* (j^+)_* \beta^+_{X',\mathcal{S}'}.$$
This is again a commutative unitary algebra.

We claim that over the sub-diagram
$T^+_{|\mathcal{P}^*([\![0,n]\!])^{\rm op}\times \{1\}}\simeq (\Spec(k),\mathcal{P}^*([\![0,n]\!])^{\rm op})$, the motive
$\beta^+_{X,\mathcal{S}}$ is given by the unit motive.
Arguing by induction, we are left to show that
$$\xymatrix@C=1.6pc{\un_{(\Spec(k),\mathcal{P}^*([\![0,n]\!])^{\rm op})} \ar[r] &  o_{*} \un_{(\Spec(k),\mathcal{P}^*([\![0,n-1]\!])^{\rm op}\times \underline{\mathbf{1}})} }$$
is invertible. It suffices to show this after applying $I^*$ for
$I\in \mathcal{P}^*([\![0,n]\!])$.
When $I$ is different from $\{n\}$, this is clear. When $I=\{n\}$,
we need to show that $\un_{\Spec(k)}\simeq {\rm holim}_{\mathcal{P}^*([\![0,n-1]\!])^{\rm op}\times \underline{\mathbf{1}}}
\un$.
This follows from \cite[Prop.~2.1.41]{ayoub-these-I} due to the presence of an initial object, namely
$([\![1,n-1]\!],0)$.

Now, let $\beta_{X,\mathcal{S}}=(T \to T^+)^*\beta^+_{X,\mathcal{S}}$.
This is the motive which is of interest to us.
The motive
$\beta^+_{X,\mathcal{S}}$ is only a technical devise needed for the functorial construction of $\beta_{X,\mathcal{S}}$. Clearly,
$\beta_{X,\mathcal{S}}$ is a commutative unitary algebra and it is related to
$\beta_{X',\mathcal{S}'}$ as follows.
Over the sub-diagram
$T\circ o$ of $T$,
$\beta_{X,\mathcal{S}}$ is given by
$b^*j_*\beta_{X',\mathcal{S}}$,
whereas, over $T(\{n\})=X_n$, it is given by $e_n(\{n\})^*\un_{\Spec(k)}\simeq \un_{X_n}$. We have the following result.

\begin{lemma}
\label{lemma:applic-of-comput-direct-im-sncd-I}
Let $i_0={\rm min}(I)$, and $s_I:T(I)\hookrightarrow Y_{i_0}$ and $t_{i_0}:e_{i_0}^{-1}(X_{i_0}) \hookrightarrow Y_{i_0}$ be the inclusions.
Then $I^*\beta_{X,\mathcal{S}}\in \DM(T(I))$ is canonically isomorphic to
$s_{I}^* t_{i_0 *}\un_{e_{i_0}^{-1}(X_{i_0})}$.

\end{lemma}

\begin{proof}
Write $I=\{i_0<\dots< i_m\}$. For $0\leq j \leq m$, we set
$I_j=\{i_0, \dots, i_j\}$ and $T^0(I_j)=(T(I_j) \to X)^{-1}(X_{i_j})$, a dense open subset of $T(I_j)$.
One sees immediately from the definition of $\beta_{X,\mathcal{S}}$ that $I^*\beta_{X,\mathcal{S}}\in \DM(T(I))$ is given by
$$(T^0(I_m)\hookrightarrow T(I_m))_*(T^0(I_{m}) \hookrightarrow T(I_{m-1}))^*\dots \hspace{5cm}$$
$$\hspace{3cm}(T^0(I_1) \hookrightarrow T(I_1))_*(T^0(I_1) \hookrightarrow  T(I_0))^*(T^0(I_0)\hookrightarrow T(I_0))_*\un_{e_{i_0}^{-1}(X_{i_0})}.$$
For $0\leq j \leq m$, call $M_j\in \DM(T(I_j))$
the motive $I_j^*\beta_{X,\mathcal{S}}$. Thus, we have
$$M_{j+1}=(T^0(I_{j+1})\hookrightarrow T(I_{j+1}))_*(T^0(I_{j+1}) \hookrightarrow
T(I_j))^*M_j.$$
By induction on $j$, we may assume that $M_j\simeq s_{I_j}^*t_{i_0*} \un$.
Our claim follows then from Proposition
\ref{prop:direct-image-compl-sncd}. Indeed,
$T(I_{j+1})$ is $\mathcal{R}(I_{j+1})$-constructible and $T^0(I_{j+1})\subset T(I_{j+1})$ is the complement of a closed subset contained in $e_{i_0}^{-1}(X_{\geq i_{j+1}+1})$.
\end{proof}

Now we view $\mathcal{T}^+$ as an object of $\Dia(\Sch/k)$ by passing to total diagrams. We
define a motive $\beta'^+_{X,\mathcal{S}}\in
\DM(\mathcal{T}^+)$ by induction on $n$ as follows. For $n=0$, we take for $\beta'^+_{X,\mathcal{S}}$ the unit motive.
For $n\geq 1$, we assume that the motive
$\beta'^+_{X,\mathcal{S}} \in \DM(\mathcal{T}'^+)$ has been constructed.
We have a diagram in $\Dia(\Dia(\Sch/k))$:
$$\xymatrix@C=1.7pc{\mathcal{T}' \ar[r]^-j & \mathcal{T}\circ \iota_n & \mathcal{T}\circ o \ar[l]_-{b} \ar[r]^-o & \mathcal{T}}$$
which gives:
$$\xymatrix@C=1.7pc{\mathcal{T}'^+ \ar[r]^-{j^+} &
\mathcal{T}^+\circ (\iota_n \times \id_{\underline{\mathbf{1}}}) & \mathcal{T}^+\circ (o\times \id_{\underline{\mathbf{1}}}) \ar[l]_-{b^+} \ar[r]^-{o^+} & \mathcal{T}^+,}$$
that we consider as a diagram in $\Dia(\Sch/k)$ by passing to total diagrams. We also have a morphism $e_n:\mathcal{T}^+\to \mathcal{T}^+$ in $\Dia(\Dia(\Sch/k))$ constructed in exactly the same manner as $e_n:T^+ \to T^+$. With these notation, we set
$$\beta'^+_{X,\mathcal{S}}=e_n^*(o^+)_*(b^+)^*(j^+)_*
\beta'^+_{X',\mathcal{S}'}.$$
As before, we can show that the restriction of $\beta'^+_{X,\mathcal{S}}$ to the sub-diagram $\mathcal{T}^+_{|\mathcal{P}^*([\![0,n]\!])^{\rm op}\times \{1\}} \simeq (\Spec(k),\mathcal{P}^*([\![0,n]\!])^{\rm op})$ is isomorphic to the unit motive.

Also, we set $\beta'_{X,\mathcal{S}}=(\mathcal{T}\to\mathcal{T}^+)^*\beta'^+_{X,\mathcal{S}}$. This is a commutative unitary algebra of
$\DM(\mathcal{T})$. It can be related to $\beta'_{X',\mathcal{S}'}$ as follows. Over the sub-diagram $\mathcal{T}\circ o$,
$\beta'_{X,\mathcal{S}}$ is given by $b^*j_*\beta'_{X',\mathcal{S}'}$, whereas, over $\mathcal{T}(\{n\})$, it is given by the unit motive
$\un_{\mathcal{T}(\{n\})}$.

\begin{lemma}
\label{lemma:applic-of-comput-direct-im-sncd-II}
Let $I\in \mathcal{P}^*([\![0,n]\!])$ and $\alpha\in A(I)$.
Denote $i_0={\rm min}(I)$, $s_{I,\alpha}:\mathcal{T}(I,\alpha) \hookrightarrow
Y_{i_0}$ the inclusion. Then, $(I,\alpha)^*\beta'_{X,\mathcal{S}} \in \DM(\mathcal{T}(I,\alpha))$ is canonically isomorphic to $s_{I,\alpha}^* t_{i_0*}\un_{e_{i_0}^{-1}(X_{i_0})}$.

\end{lemma}

\begin{proof}
The proof is similar to that of
Lemma \ref{lemma:applic-of-comput-direct-im-sncd-I}.
Write $I=\{i_0<\dots<i_m\}$ and set $I_j=\{i_0,\dots, i_j\}$ for
$0\leq j \leq m$. Let $\alpha_j\in A(I_j)$ be the image of
$\alpha$ by ${\rm s}_{I_j \subset I}:A(I) \to A(I_j)$.
Also let $\mathcal{T}^0(I_j,\alpha_j)$ be the inverse image of
$X_{i_j}$ by the morphism
$\mathcal{T}(I_j,\alpha_j) \to X$.
It follows from the construction of $\beta'_{X,\mathcal{S}}$ that $(I,\alpha)^*\beta'_{X,\mathcal{S}}$ is given by
$$(\mathcal{T}^0(I_m,\alpha_m) \hookrightarrow \mathcal{T}(I_m,\alpha_m))_* (\mathcal{T}^0(I_m,\alpha_m) \hookrightarrow \mathcal{T}(I_{m-1},\alpha_{m-1}))^* \dots \hspace{3.5cm}$$
$$\hspace{.5cm}(\mathcal{T}^0(I_1,\alpha_1) \hookrightarrow \mathcal{T}(I_1,\alpha_1))_* (\mathcal{T}^0(I_1,\alpha_1) \hookrightarrow \mathcal{T}(I_{0},\alpha_{0}))^* (\mathcal{T}^0(I_0,\alpha_0) \hookrightarrow \mathcal{T}(I_0,\alpha_0))_*\un.$$
For $0\leq j \leq m$, call $M_j\in \DM(\mathcal{T}(I_j,\alpha_j))$ the motive $(I_j,\alpha_j)^*\beta'_{X,\mathcal{S}}$. Thus, we have
$$M_{j+1}=(\mathcal{T}^0(I_{j+1},\alpha_{j+1}) \hookrightarrow
\mathcal{T}(I_{j+1},\alpha_{j+1}))_*(\mathcal{T}^0(I_{j+1},\alpha_{j+1}) \hookrightarrow \mathcal{T}(I_j,\alpha_j))^*M_j.$$
We now use Proposition
\ref{prop:direct-image-compl-sncd} and induction on $j$ to show that
$M_j\simeq s_{I_j,\alpha_j}^* t_{i_0*}\un$.
\end{proof}

Call $q:(\mathcal{T},\mathcal{P}^*([\![0,n]\!])^{\rm op}) \to
(T,\mathcal{P}^*([\![0,n]\!])^{\rm op})$ the natural projection which we may equally consider as a morphism in $\Dia(\Dia(\Sch/k))$ or
$\Dia(\Sch/k)$.

\begin{proposition}
\label{prop:comparison-beta-et-beta-prime}
There is canonical isomorphism of commutative unitary algebras
$q^*\beta_{X,\mathcal{S}} \simeq \beta'_{X,\mathcal{S}}$.
\end{proposition}

\begin{proof}
Call $q^+:\mathcal{T}^+ \to T^+$ the morphism in $\Dia(\Dia(\Sch/k))$ deduced from $q$. We will construct by induction on $n$ a canonical isomorphism of commutative algebras $(q^+)^*\beta^+_{X,\mathcal{S}}\simeq \beta'^+_{X,\mathcal{S}}$,
and then get the isomorphism $q^*\beta_{X,\mathcal{S}} \simeq \beta'_{X,\mathcal{S}}$ by applying $(\mathcal{T} \to \mathcal{T}^+)^*$
and using the equality $(\mathcal{T}\to \mathcal{T}^+) \circ q=(q^+)\circ (T\to T^+)$.

There is a commutative diagram
$$\xymatrix@C=1.5pc@R=1.5pc{\mathcal{T}'^+ \ar[r]^-{j^+} \ar[d]^-{q'^+} &
\mathcal{T}^+\circ (\iota_n \times \id_{\underline{\mathbf{1}}}) \ar[d]^-{q^+} & \mathcal{T}^+\circ (o\times \id_{\underline{\mathbf{1}}}) \ar[l]_-{b^+} \ar[r]^-{o^+} \ar[d]^-{q^+} & \mathcal{T}^+ \ar[d]^-{q^+} & \mathcal{T}^+ \ar[l]_-{e_n} \ar[d]^-{q^+} \\
T'^+ \ar[r]^-{j^+} & T^+\circ (\iota_n\times \id_{\underline{\mathbf{1}}}) & T^+\circ (o\times \id_{\underline{\mathbf{1}}}) \ar[l]_-{b^+} \ar[r]^-{o^+} & T^+ & T^+ \ar[l]_-{e_n} ,\!}$$
which we consider in $\Dia(\Sch/k)$ by passing to total diagrams of schemes. This gives natural transformations
$$(q^+)^* e_n^* \simeq e_n^* (q^+)^* , \;
(q^+)^* (o^+)_* \to  (o^+)(q^+)^* , \; (q^+)^*(b^+)^* \simeq (b^+)^*(q^+)^*$$
$$\text{and} \quad
(q^+)^* (j^+)_* \to (j^+)_*(q'^+)^*.$$
We get a canonical morphism of commutative unitary algebras
$(q^+)^* \beta^+_{X,\mathcal{S}} \to \beta'^+_{X,\mathcal{S}}$ by taking the composition:
$$\xymatrix@C=1.7pc@R=1.3pc{(q^+)^* e_n^* (o^+)_* (b^+)^* (j^+)_*\beta^+_{X',\mathcal{S}}  \ar[r] &
e_n^*(o^+)_+(b^+)^* (j^+)_* (q'^+)^* \beta^+_{X',\mathcal{S}} \ar[d]^-{\sim} \\
& e_n^*(o^+)_+(b^+)^* (j^+)_* \beta'^+_{X',\mathcal{S}}.}$$

It remains to show that $(q^+)^*\beta^+_{X,\mathcal{S}} \to \beta'^+_{X,\mathcal{S}}$ is invertible. This is obviously the case over the sub-diagram $\mathcal{T}^+_{|\mathcal{P}^*([\![0,n]\!])\times \{1\}}\simeq (\Spec(k),\mathcal{P}^*([\![0,n]\!]))$ as both
sides of the morphism are canonically isomorphic to the unit motive.
We deduce also that
$(q^+)^*\beta^+_{X,\mathcal{S}} \to \beta'^+_{X,\mathcal{S}}$
is invertible over the sub-diagram $\mathcal{Y}(\{n\})\times \{0\}$. Indeed, by construction,
there are canonical isomorphisms
$$(\{n\},0)^*\beta'^+_{X,\mathcal{S}}\simeq ((\{n\},1)^*\beta'^+_{X,\mathcal{S}})_{|\mathcal{T}(\{n\})}\simeq \un_{\mathcal{T}(\{n\})}$$
and similarly for $\beta_{X,\mathcal{S}}^+$.

To end the proof, it remains to show that
$(q^+)^*\beta^+_{X,\mathcal{S}} \to \beta'^+_{X,\mathcal{S}}$
is invertible over the sub-diagram $\mathcal{T}^+\circ (o\times \id_{\underline{\mathbf{1}}})$. But over this sub-diagram,
$(q^+)^*\beta^+_{X,\mathcal{S}}$ and
$\beta'^+_{X,\mathcal{S}}$ are given by
$q^* b^*j_*\beta_{X',\mathcal{S}'}$ and $b^*j_*\beta'_{X',\mathcal{S}}$ respectively. Moreover, our morphism is given by the composition
$$\xymatrix@C=1.6pc{q^*b^*j_*\beta_{X',\mathcal{S}'} \simeq b^*q^* j_*\beta_{X',\mathcal{S}'} \ar[r] &  b^*j_*q'^*\beta_{X',\mathcal{S}'} \simeq
b^*j_*\beta'_{X',\mathcal{S}'}.}$$
Thus, it suffices to show that the base change morphism
$q^*j_* \to j_*q'^*$ is invertible
when applied to the motive $\beta_{X',\mathcal{S}'}$.
It suffices to show this after applying $(I,\alpha)^*$ for
$I\in \mathcal{P}^*([\![0,n-1]\!])$ and $\alpha \in A(I)$.
We are then reduced to showing that the base change morphism associated to the cartesian diagram
$$\xymatrix@C=1.5pc@R=1.5pc{\mathcal{T}'(I,\alpha) \ar[r]^-{q'} \ar[d]_-j  & T'(I) \ar[d]^-j \\
\mathcal{T}(I,\alpha) \ar[r]^-q & T(I)}$$
is invertible when applied to the motive
$I^*\beta_{X',\mathcal{S}'}$. But by Lemma
\ref{lemma:applic-of-comput-direct-im-sncd-I},
$I^*\beta_{X',\mathcal{S}'}\simeq s'^*_It'_{i_0*} \un_{e^{-1}_{i_0}(X_{i_0})}$ where $i_0={\rm min}(I)$, $t'_{i_0}:e^{-1}_{i_0}(X_{i_0}) \hookrightarrow Y'_{i_0}$ and $s'_I:T'(I)\hookrightarrow Y'_{i_0}$.
By Proposition
\ref{prop:direct-image-compl-sncd}, there is an isomorphism
$s_I^* t_{i_0*} \un_{e^{-1}_{i_0}(X_{i_0})} \simeq  j_*(s'^*_It'_{i_0*}\un_{e^{-1}_{i_0}(X_{i_0})})$. Thus, we are reduced to showing that the canonical morphism
$$\xymatrix@C=1.7pc{s_{I,\alpha}^* t_{i_0*}\un_{e^{-1}_{i_0}(X_{i_0})} \ar[r] &  j_*
s'^*_{I,\alpha} t'_{i_0*} \un_{e^{-1}_{i_0}(X_{i_0}) }}$$
is invertible. This too is true by
Proposition
\ref{prop:direct-image-compl-sncd}. This ends the proof of the proposition.
\end{proof}

Let $(p,\varsigma_n):(\mathcal{Y}, \mathcal{P}_2([\![1,n]\!])) \to (\mathcal{T},\mathcal{P}^*([\![0,n]\!])^{\rm op})$ denote the natural projection which we may equally consider as a morphism in $\Dia(\Sch/k)$ or $\Dia(\Dia(\Sch/k))$.

\begin{proposition}
\label{prop:comparison-theta-double-et-beta-prime}
There is a canonical isomorphism of commutative unitary algebras $(p,\varsigma_n)^*\beta'_{X,\mathcal{S}} \simeq \theta''_{X,\mathcal{S}}$.

\end{proposition}

\begin{proof}
Consider the object $(\mathcal{Y}^+,\mathcal{P}_2([\![1,n]\!])\times \underline{\mathbf{1}})$ of $\Dia(\Dia(\Sch/k))$ obtained from $(\mathcal{Y},\mathcal{P}_2([\![1,n]\!]))$. Thus, $(\mathcal{Y}^+)_{|\mathcal{P}_2([\![1,n]\!])\times \{1\}}$ is the constant diagram
$(\Spec(k),\mathcal{P}_2([\![1,n]\!]))$.

We define a motive
$\theta''^+_{X,\mathcal{S}}$ over the total diagram of schemes
associated to $\mathcal{Y}^+$ (which we still denote $\mathcal{Y}^+$) by induction on $n$ as follows.
When $n=0$, we take the unit motive. If $n\geq 1$, we consider the following diagram in $\Dia(\Dia(\Sch/k))$:
$$\xymatrix@C=1.7pc{\mathcal{Y}'^+ \ar[r]^-{j^+} & \mathcal{Y}^+\circ (\iota^0_n\times \id_{\underline{\mathbf{1}}}) & \mathcal{Y}^+\circ (o\times \id_{\underline{\mathbf{1}}}) \ar[l]_-{b^+} \ar[r]^-{o^+} &
\mathcal{Y}^+,}$$
which we view in $\Dia(\Sch/k)$ by passing to total diagrams.
We set
$$\theta''^+_{X,\mathcal{S}}=\omega^0_{\{(0,1)\}\times \underline{\mathbf{1}}|(\mathcal{Y}^+,\smallucarre \times \underline{\mathbf{1}})}\left((o^+)_*(b^+)^*(j^+)_*\theta''^+_{X',\mathcal{S}'}\right).$$
As usual, $\omega^0_{\{(0,1)\}\times \underline{\mathbf{1}}|(\mathcal{Y}^+,\smallucarre \times \underline{\mathbf{1}})}$ is really $\omega^0_{\mathcal{P}_2([\![1,n-1]\!])\times \{(0,1)\}\times \underline{\mathbf{1}}|(\mathcal{Y}^+,
\mathcal{P}_2([\![1,n-1]\!])\times \smallucarre \times \underline{\mathbf{1}})}$.
It is clear that $\theta''_{X,\mathcal{S}}\simeq (\mathcal{Y}\to \mathcal{Y}^+)^*\theta''^+_{X,\mathcal{S}}$. Thus, it is sufficient to construct a canonical isomorphism of commutative unitary algebras $(p^+,\varsigma_n \times \id_{\underline{\mathbf{1}}})^*\beta'^+_{X,\mathcal{S}}\simeq \theta''^+_{X,\mathcal{S}}$, where $(p^+,\varsigma_n \times \id_{\underline{\mathbf{1}}}):\mathcal{Y}^+ \to \mathcal{T}^+$ is the morphism deduced from $(p,\varsigma_n)$.

We argue by induction on $n$. When $n=0$, the claim is clear as both motives $\beta'^+_{X,\mathcal{S}}$ and $\theta''^+_{X,\mathcal{S}}$ are unit motives. We assume that $n\geq 0$ and that the isomorphism
$(p'^+,\varsigma_{n-1}\times \id_{\underline{\mathbf{1}}} )^*\beta'^+_{X',\mathcal{S}'} \simeq \theta''^+_{X',\mathcal{S}'}$
has been constructed. We split the proof into parts.

To simplify notations, we will write $p$, $p'$, $p^+$, and
$p'^+$ instead of $(p,\varsigma_n)$, $(p',\varsigma_{n-1})$,
$(p^+,\varsigma_n \times \id_{\underline{\mathbf{1}}})$ and
$(p'^+,\varsigma_{n-1}\times \id_{\underline{\mathbf{1}}})$.

\smallskip

\noindent \underbar{Part A}: Here we construct a canonical
morphism $(p^+)^* \beta'^+_{X,\mathcal{S}} \to
\theta''^+_{X,\mathcal{S}}$ of commutative unitary algebras. There
is a commutative diagram in $\Dia(\Dia(\Sch/k))$:
$$\xymatrix@C=1.5pc@R=1.5pc{\mathcal{Y}'^+ \ar[r]^-{j^+} \ar[d]^-{p'^+} & \mathcal{Y}^+\circ (\iota^0_n\times \id_{\underline{\mathbf{1}}}) \ar[d]^-{p^+} & \mathcal{Y}^+\circ (o\times \id_{\underline{\mathbf{1}}}) \ar[l]_-{b^+} \ar[r]^-{o^+} \ar[d]^-{p^+} &
\mathcal{Y}^+ \ar[d]^-{p^+} \\
\mathcal{T}'^+ \ar[r]^-{j^+} & \mathcal{T}^+\circ (\iota_n\times \id_{\underline{\mathbf{1}}}) & \mathcal{T}^+\circ (o\times \id_{\underline{\mathbf{1}}}) \ar[l]_-{b^+} \ar[r]^-{o^+} & \mathcal{T}^+,}$$
which we may view in $\Dia(\Sch/k)$ by passing to total diagrams.
We deduce from this natural transformations
$$\xymatrix@C=1.5pc@R=1.5pc{(p^+)^*(o^+)_*(b^+)^* (j^+)_*
\ar[r] & (o^+)_*(p^+)^* (b^+)^* (j^+)_* \ar[d]^-{\sim} \\
& (o^+)_* (b^+)^* (p^+)^* (j^+)_* \ar[r] & (o^+)_* (b^+)^* (j^+)_* (p'^+)^*.\!}$$

On the other hand, we have a $2$-morphism of diagrams of schemes
$\id_{\mathcal{T}^+} \to e_n$ which on the indexing categories is the identity except on $(\{n\},0)$ where it is given by
$(\{n\},0)\to (\{n\},1)$. This gives a natural transformation
$e_n^* \to \id\simeq (\id_{\mathcal{T}_+})^*$. We now consider the
morphism
$$\xi^+:(p^+)^*\beta'^+_{X,\mathcal{S}} \to
(o^+)_*(b^+)^*(j^+)_*\theta''^+_{X',\mathcal{S}'}$$
given by the composition
$${\small \xymatrix@C=1pc @R=1.2pc{(p^+)^* e_n^* (o^+)_*(b^+)^*(j^+)_*\beta'^+_{X',\mathcal{S}'} \ar[r] &
(p^+)^* (o^+)_*(b^+)^*(j^+)_*\beta'^+_{X',\mathcal{S}'}\ar[d] &\\
& (o^+)_*(b^+)^*(j^+)_*(p'^+)^*\beta'^+_{X',\mathcal{S}'}
\ar[r]^-{\sim} &
(o^+)_*(b^+)^*(j^+)_*\theta''^+_{X',\mathcal{S}'}.}}$$
As $\beta'^+_{X,\mathcal{S}}$ is the unit motive over
$\{n\}\times \underline{\mathbf{1}}\subset \mathcal{P}^*([\![0,n]\!])^{\rm op}\times \underline{\mathbf{1}}$,
the natural morphism
$$\xymatrix@C=1.5pc{\omega^0_{\{(0,1)\}\times \underline{\mathbf{1}}|(\mathcal{Y}^+,\smallucarre\times \underline{\mathbf{1}})} (p^+)^*\beta'^+_{X,\mathcal{S}} \ar[r] &
(p^+)^*\beta'^+_{X,\mathcal{S}}}$$
is invertible. Hence, there exists a unique morphism
$(p^+)^* \beta'^+_{X,\mathcal{S}} \to \theta''^+_{X,\mathcal{S}}$
making the following triangle
$$\xymatrix@C=1.5pc@R=1.5pc{ \omega^0_{\{(0,1)\}\times \underline{\mathbf{1}}|(\mathcal{Y}^+, \smallucarre\times \underline{\mathbf{1}})} \left( (p^+)^*\beta'^+_{X,\mathcal{S}} \right) \ar[r]^-{\sim} \ar[d]   &
(p^+)^*\beta'^+_{X,\mathcal{S}} \ar@/^/[dl]\\
\omega^0_{\{(0,1)\}\times \underline{\mathbf{1}}|(\mathcal{Y}^+,
\smallucarre\times \underline{\mathbf{1}})} \left((o^+)_*(b^+)^*(j^+)_*\theta''^+_{X',\mathcal{S}'}  \right) & }$$
commutative.
Thus, to end the proof, it remains to check that
$\omega^0_{\{(0,1)\}\times \underline{\mathbf{1}}|(\mathcal{Y}^+,
\smallucarre\times \underline{\mathbf{1}})}(\xi^+)$
is invertible.
This will be done in the next three steps.

\smallskip

\noindent \underbar{Part B}: Here we remark that
$\xi^+_{|\mathcal{P}_2([\![1,n]\!])\times \{1\}}$ is invertible.
We have seen that the restriction of $(\beta'^+_{X,\mathcal{S}})$
to $\mathcal{P}^*([\![0,n]\!])^{\rm op}\times \{1\}$ was
canonically isomorphic to the unit motive. It follows that
$((p^+)^*\beta'^+_{X,\mathcal{S}})_{|\mathcal{P}_2([\![1,n]\!])\times
\{1\}}\simeq \un_{(\Spec(k),\mathcal{P}_2([\![1,n]\!])}$.

Similarly, the restriction of $\theta''^+_{X,\mathcal{S}}$ to
$\mathcal{P}_2([\![1,n]\!])\times \{1\}$ is the unit motive. As in the case of $\beta^+_{X,\mathcal{S}}$, we prove this by induction on $n$.
We are then reduced to showing that
$\un \simeq {\rm holim}_{\smallucarre}\un$ which is obviously true.

We leave it to the reader to check that
$\xi^+_{|\mathcal{P}_2([\![1,n]\!])\times \{1\}}$ is the identity of the unit of $\DM(\Spec(k),\mathcal{P}_2([\![1,n]\!]))$ modulo the above isomorphisms.
Denote $\xi:p^*\beta'_{X,\mathcal{S}} \to \theta''_{X,\mathcal{S}}$ the restriction of $\xi^+$ along the inclusion
$\mathcal{Y} \to \mathcal{Y}^+$. It remains to show that $\omega^0_{\{(0,1)\}|(\mathcal{Y},\smallucarre)}(\xi)$ is invertible.

\smallskip

\noindent \underbar{Part C}: Here we show that $\xi$ is invertible
after restricting to the sub-digram $(\mathcal{Y}\circ o,
\mathcal{P}_2([\![1,n-1]\!]) \times \underline{\mathbf{1}}) \hookrightarrow
(\mathcal{Y},\mathcal{P}_2([\![1,n]\!]))$. The restrictions of the
motives $p^* \beta'_{X,\mathcal{S}}$ and
$o_*b^*j_*\theta''_{X',\mathcal{S}'}$ to this sub-diagram are
given by $p^* b^* j_*\beta'_{X',\mathcal{S}'}$ and $b^*j_*
\theta''_{X',\mathcal{S}'}$ respectively. Moreover, our morphism
is given by the composition
$$\xymatrix@C=1.7pc{p^* b^* j_*\beta'_{X',\mathcal{S}'} \ar[r]^-{\sim} &
b^*p^* j_*\beta'_{X',\mathcal{S}'}
\ar[r] &
b^* j_* p'^*  \beta'_{X',\mathcal{S}'} \ar[r]^-{\sim} & b^* j_* \theta''_{X',\mathcal{S}'}.}$$
Thus, it suffices to show that the base change morphism
$p^*j_* \beta'_{X',\mathcal{S}'}
\to j_* p'^* \beta'_{X',\mathcal{S}'}$
is invertible. As usual, it suffices to check this over each constituent of $\mathcal{Y}\circ \iota_n^0$. Thus, fix $(I_0,I_1)\in \mathcal{P}_2([\![1,n-1]\!])$ and let $I_0'=I_0\bigsqcup \{n\}$, $J=[\![0,n-1]\!]- I_0=[\![0,n]\!]- I_0'$,
$\{0\}\bigsqcup I_1=\{i_0<\dots<i_s\}$ and $K=J\cap [\![i_s,n-1]\!]=J\cap [\![i_s,n]\!]$.
We need to show, for $(\alpha_j)_{0\leq j \leq s}\in \mathcal{C}(I_0\bigsqcup \{n\},I_1)$, that
the base change morphism $p^*j_* (K,\alpha_s)^*\beta'_{X',\mathcal{S}'}
\to j_*p'^* (K,\alpha_s)^*\beta'_{X',\mathcal{S}'}$ associated to
the cartesian square
$$\xymatrix@C=1.5pc@R=1.5pc{\mathcal{Y'}(I_0,I_1,(\alpha_j)_j) \ar[r]^-{p'} \ar[d]_-j & \mathcal{T}'(K,\alpha_s) \ar[d]^-j \\
\mathcal{Y}(I_0',I_1,(\alpha_j)_j) \ar[r]^-p & \mathcal{T}(K,\alpha_s)}$$
is invertible.

By Lemma \ref{lemma:applic-of-comput-direct-im-sncd-II},
$(K,\alpha_s)^*\beta'_{X',\mathcal{S}'}$ is canonically isomorphic to
$s'^*_{K,\alpha_s} t'_{i_s*}\un_{e_{i_s}^{-1}(X'_{i_s})}$
where $t'_{i_s*}:e_{i_s}^{-1}(X'_{i_s}) \hookrightarrow Y'_{i_s}$ and
$s'_{K,\alpha_s}:\mathcal{T}'(K,\alpha_s) \hookrightarrow Y_{i_s}$
are the inclusions. Using
Proposition \ref{prop:direct-image-compl-sncd}
applied on $Y_{i_s}$, one gets that $j_*(K,\alpha_s)^*\beta'_{X',\mathcal{S}'} \simeq s_{K,\alpha_s}^* t_{i_s*} \un_{e_{i_s}^{-1}(X_{i_s})}$.
Now, the scheme
$$P=\mathcal{Y}(I_0\bigsqcup [\![i_s+1,n]\!],I_1, ((\alpha_j)_{0\leq j\leq s-1},{\rm s}_{\{i_s\}\subset K}(\alpha_s)))$$
is a finite cover of $Y_{i_s}$ such that each of its connected component is dominated by a connected component of the cover $Z_{i_s}$ of \textbf{D3)}.
Moreover, $\mathcal{Y}(I_0',I_1,(A_j)_j)=P\times_{Y_{i_s}} \mathcal{T}(K,A_s)$. Our claim follows now from Corollary
\ref{cor:new-for-compar-beta-prime-et-theta-dp}.

\smallskip

\noindent \underbar{Part D}: Here we describe the morphism $\xi$
over a sub-diagram $\mathcal{Y}(I_0,I_1)$ with $(I_0,I_1)\in
\mathcal{P}_2([\![1,n]\!])$ such that $n\in I_1$, and show that it
is a universal morphism from an Artin motive to a
cohomological motive.

Let $I_1'=I_1- \{n\}$ and $J=[\![0,n]\!]- I_0$, and order
$\{0\}\bigsqcup I_1'=\{i_0<\dots < i_s\}$.
Also, let $K=[\![i_s,n]\!]\cap J$.
With these notations, we have a commutative diagram
$$\xymatrix@C=1.5pc@R=1.5pc{ p^*(o_*b^*j_*\beta'_{X',\mathcal{S}'})_{|\mathcal{T}(\{n\})} \ar[r] \ar[d]_-{\sim} & p^*\{\mathcal{T}(K) \to \mathcal{T}(\{n\})\}_* (o_*b^*j_*\beta'_{X',\mathcal{S}'})_{|\mathcal{T}(K)} \ar[d] \\
(p^*o_*b^*j_*\beta'_{X',\mathcal{S}'})_{|\mathcal{Y}(I_0,I_1)} \ar[r] \ar[d]_-{\xi_{|\mathcal{Y}(I_0,I_1)}} & \{\mathcal{Y}(I_0,I_1') \to \mathcal{Y}(I_0,I_1)\}_*
(p^*o_*b^*j_*\beta'_{X',\mathcal{S}'})_{|\mathcal{Y}(I_0,I_1')}\ar[d]_-{\sim}^-{\xi_{|\mathcal{Y}(I_0,I_1')}} \\
(o_*b^*j_*\theta''_{X',\mathcal{S}'})_{|\mathcal{Y}(I_0,I_1)}\ar[r]^-{\sim}
& \{\mathcal{Y}(I_0,I_1') \to \mathcal{Y}(I_0,I_1)\}_* (o_*b^* j_*
\theta''_{X',\mathcal{S}'})_{|\mathcal{Y}(I_0,I_1')}.}$$
That the
bottom horizontal arrow is invertible, is an easy consequence of
Axiom \textbf{DerAlg} $\mathbf{4'}$ of
\cite[Rem.~2.3.14]{ayoub-these-I}. That the first vertical arrow
on the left is invertible, is obvious. That the second vertical
arrow on the right is invertible follows from the Part C of the
proof.

On the other hand, we know that $(o_*b^*j_*\beta'_{X',\mathcal{S}'})_{|\mathcal{T}(n)}\simeq \un_{\mathcal{T}(\{n\})}$. Also, by
Lemma \ref{lemma:applic-of-comput-direct-im-sncd-I} and Proposition
\ref{prop:comparison-beta-et-beta-prime}, we have
$$(o_*b^*j_*\beta'_{X',\mathcal{S}'})_{|\mathcal{T}(K)}
=\{\mathcal{T}(K)\to Y_{i_s}\}^*(t_{i_s*} \un_{e_{i_s}^{-1}(X_{i_s})}).$$
It follows that $\xi_{|\mathcal{Y}(I_0,I_1)}$ is isomorphic to the natural morphism
$$\zeta:\xymatrix@C=1.7pc{\un_{\mathcal{Y}(I_0,I_1)} \ar[r] & }\! \{\mathcal{Y}(I_0,I_1') \to \mathcal{Y}(I_0,I_1)\}_* p^* \{\mathcal{T}(K)\to Y_{{\rm min}(K)}\}^*(t_{i_s*} \un_{e_{i_s}^{-1}(X_{i_s})}).$$
To finish the proof of the proposition, we need to show that
$\omega^0_{\mathcal{Y}(I_0,I_1)}(\zeta)$ is invertible.
By Proposition \ref{prop:additional-prop-omega-0-x}, (iii), the natural transformation
$$\omega^0_{\mathcal{Y}(I_0,I_1)}
\{\mathcal{Y}(I_0,I_1') \to \mathcal{Y}(I_0,I_1)\}_*
\omega^0_{\mathcal{Y}(I_0,I_1')} \!\xymatrix@C=1.7pc{\ar[r] &}\!
\omega^0_{\mathcal{Y}(I_0,I_1)}
\{\mathcal{Y}(I_0,I_1') \to \mathcal{Y}(I_0,I_1)\}_*$$
is invertible. Moreover,
using Lemma
\ref{lemma:omega-0-restriction-direct-image-sncd} below,
we see that the natural morphism
$$\un_{\mathcal{Y}(I_0,I_1')} \!\xymatrix@C=1.7pc{\ar[r] &}\!  \omega^0_{\mathcal{Y}(I_0,I_1')} p^* \{\mathcal{T}(K) \to Y_{{\rm min}(K)}\}^*(t_{i_s*} \un_{e^{-1}_{i_s}(X_{i_s})})$$
is invertible. Hence, we are left to check that
$$\un_{\mathcal{Y}(I_0,I_1)} \!\xymatrix@C=1.5pc{\ar[r] & }\!  \omega^0_{\mathcal{Y}(I_0,I_1)}
\{\mathcal{Y}(I_0,I_1') \to \mathcal{Y}(I_0,I_1)\}_*
\un_{\mathcal{Y}(I_0,I_1')}$$
is invertible. This follows from
Proposition
\ref{prop:main-computation-omega0-pi-0} as $\mathcal{Y}(I_0,I_1)$ is objectwise the Stein factorization of the
$X_n$-scheme $\mathcal{Y}(I_0,I_1')$ which is smooth and projective.
Indeed, by \textbf{D3)}, the latter admits a finite \'etale cover by a smooth and projective $X_n$-scheme.
\end{proof}

\begin{lemma}
\label{lemma:omega-0-restriction-direct-image-sncd}
Let $W$ be a quasi-projective $k$-scheme having only quotient singularities, and
$j:W_0 \hookrightarrow W$ the inclusion of the complement of a sncd in $W$. Let $i:Z \to W$ be any morphism from a quasi-projective $k$-scheme. Then, the natural morphism
$\un_Z \to \omega^0_{Z}(i^*j_*\un_{W_0})$
is invertible.
\end{lemma}

\begin{proof}
We may assume that $W=W'/G$ where $W'$ is a smooth $k$-scheme and
$G$ is a finite group acting on $W$. We can also assume that the inverse image of any irreducible component of $W-W_0$ is a smooth divisor in $W'$. Denote $e:W' \to W$ be the quotient map and
$j':W'_0=e^{-1}(W_0)\hookrightarrow W'$ the inclusion. Then $e_*j'_*\un_{W'_0}$ admits an action of $G$ and $j_*\un_{W_0}$ is the image of the projector $\frac{1}{|G|}\sum_{g\in G}g$
(cf.~\cite[Lem.~2.1.165]{ayoub-these-I}). Thus, it suffices to show that $i^*e_*\un \to \omega^0_Z(i^*e_*j'_*\un)$ is an isomorphism. Using base-change for finite
morphisms (cf.\cite[Cor.~1.7.18]{ayoub-these-I}) and Proposition \ref{prop:additional-prop-omega-0-x}, (iii,\,c), we reduce to prove the lemma for $W'$, $W'_0$ and $Z'=Z\times_W W'$. In other words, we may assume that $W$ is smooth.

Denote $D_1, \dots, D_r$ the irreducible components of the divisor
$W- W_0$. For $\emptyset \neq I\subset [\![1,r]\!]$,
let $D_I=\bigcap_{i\in I} D_i$. Denote $s_I:D_I \hookrightarrow
W$ the inclusion and $\mathcal{N}_I$ the normal sheaf to $s_I$.
Let $C=Cone(\un_{W} \to j_*\un_{W_0})$. It suffices to show that
$\omega^0_Z(i^*C)=0$. We know, using \cite[Prop.~1.4.9 and Th.~1.6.19]{ayoub-these-I}, that $C$ is in the triangulated subcategory
of $\DM(W)$ generated by
$s_{I*}{\rm Th}^{-1}(\mathcal{N}_I)\un_{D_I}$
for $\emptyset \neq I \subset [\![1,r]\!]$. Denote $t_I:i^{-1}(D_I)
\hookrightarrow Z$ the inclusion. Then
$i^*C$ is in the triangulated subcategory of $\DM(Z)$ generated by
$t_{I*}{\rm Th}^{-1}(t_I^*\mathcal{N}_I)\un_{i^{-1}(D_I)}$ for
$\emptyset \neq I \subset [\![1,r]\!]$. The lemma follows as
$\omega^0_Z(t_{I*}{\rm Th}^{-1}(t_I^*\mathcal{N}_I)\un_{i^{-1}(D_I)})=0$.
\end{proof}

As before, let
$(h,\varsigma_n):(\mathcal{Y},\mathcal{P}_2([\![1,n]\!])) \to
(T,\mathcal{P}^*([\![0,n]\!])^{\rm op})$ be the natural
projection. From Propositions
\ref{prop:comparison-beta-et-beta-prime} and
\ref{prop:comparison-theta-double-et-beta-prime} there exists a
canonical isomorphism of commutative unitary algebras
$$(h,\varsigma_n)^* \beta_{X,\mathcal{S}} \overset{\sim}{\to}
\theta''_{X,\mathcal{S}}.$$

Recall that we have is a commutative square in $\Dia(\Dia(\Sch/k))$
$$\xymatrix@C=1.5pc@R=1.5pc{\check{\mathcal{Y}} \ar[r]^-l \ar[d]_-{(\check{h},\varsigma_n)} &
\mathcal{Y} \ar[d]^-{(h,\varsigma_n)}\\
\check{T} \ar[r]^-l & T }$$
which we view in $\Dia(\Sch/k)$ by passing to total diagrams.
The proof of the following proposition is omitted:

\begin{proposition}
\label{prop:fur-functoriality-construction-beta}
There is a morphism of motives $l^*\beta_{X,\mathcal{S}}\to\beta_{\check{X},
\check{\mathcal{S}}}$ which is
invertible when $f:\check{X} \to X$ is smooth and $\check{Y}_i=\check{X}
\times_X Y_i$ for $i\in [\![0,n]\!]$. Moreover, the following diagram of
$\DM(\check{\mathcal{Y}})$:
$$\xymatrix@C=1.5pc@R=1.5pc{l^* (h,\varsigma_n)^* \beta_{X,\mathcal{S}} \ar[r]^-{\sim} \ar[d]_-{\sim}
 & (\check{h},\varsigma_n)^* l^* \beta_{X,\mathcal{S}} \ar[r] & (\check{h},
\varsigma_n)^*\beta_{\check{X},\check{\mathcal{S}}} \ar[d]^-{\sim} \\
l^*\theta''_{X,\mathcal{S}} \ar[rr] & & \theta''_{\check{X},\check{
\mathcal{S}}} }
$$
commutes;
the arrow in the bottom being the morphism of
\emph{Proposition
\ref{prop:fur-functoriality-construction-theta-double-prime}}.
\end{proposition}

\subsubsection{Conclusion}
\label{subsubsection:conclusion}
Let $\Upsilon:\mathcal{Y} \to X$ be the natural projection.
Putting together Propositions
\ref{prop:comparaison-theta-et-theta-prime},
\ref{prop:comparison-theta-prime-and-h-theta-double},
\ref{prop:comparison-beta-et-beta-prime} and
\ref{prop:comparison-theta-double-et-beta-prime}, we obtain the
canonical isomorphism $\theta_{X,\mathcal{S}}\simeq f_*h_*(h,\varsigma_n)^*
\beta_{X,\mathcal{S}}$. On the other hand,
$\Upsilon = p_{\mathcal{P}_2([\![1,n]\!])} \circ f\circ h$,
where $p_{\mathcal{P}_2([\![1,n]\!])}$ is the morphism of diagrams of schemes $(X,\mathcal{P}_2([\![1,n]\!])) \to (X,\textbf{e})$ induced by the projection of $\mathcal{P}_2([\![1,n]\!])$ to $\mathbf{e}$.
Moreover,
$(p_{\mathcal{P}_2([\![1,n]\!])})_*$ is
the homotopy limit
along $\mathcal{P}_2([\![1,n]\!])$.
Combining this with Corollary
\ref{cor:E=colim} gives:

\begin{theorem}
\label{thm:final-form-for-applic-main-thm}
With the above notation, we have:

\begin{enumerate}

\item[(a)]
There is a canonical isomorphism of commutative unitary
algebras
$$\EE_X \simeq \Upsilon_* (h,\varsigma_n)^* \beta_{X,\mathcal{S}}.$$

\item[(b)] There is a canonical morphism
$l^*\beta_{X,\mathcal{S}} \to
\beta_{\check{X},\check{\mathcal{S}}}$
of commutative unitary algebras
which is invertible when
$f:\check{X} \to X$ is smooth and $\check{Y}_i=\check{X}\times_X
Y_i$ for $i\in [\![0,n]\!]$. Moreover, the following diagram
commutes:
$$\xymatrix@C=1.2pc@R=1.5pc{l^*\EE_X \ar[d]_-{\sim} \ar[rrr] & & & \EE_{\check{X}} \ar[d]^-{\sim} \\
l^*\Upsilon_*(h,\varsigma_n)^* \beta_{X,\mathcal{S}} \ar[r] &
\check{\Upsilon}_* l^*(h,\varsigma_n)^* \beta_{X,\mathcal{S}}
\ar[r]^-{\sim} & \check{\Upsilon}_*(\check{h},\varsigma_n)^* l^*
\beta_{X,\mathcal{S}} \ar[r] & \check{\Upsilon}_*
(\check{h},\varsigma_n)^* \beta_{\check{X},\check{\mathcal{S}}}
}$$
\end{enumerate}

\end{theorem}

Fix a complex embedding $k\subset \C$ and denote by $\beta_{X,\mathcal{S}}^{an}=\An^*(\beta_{X,\mathcal{S}})$ the Betti realization of the motive
$\beta_{X,\mathcal{S}}$.
This is an object of $\mathbf{D}(T(\C),\mathcal{P}^*([\![0,n]\!])^{\rm op})$. The following is a consequence of
Theorem \ref{thm:final-form-for-applic-main-thm}:

\begin{corollary}
\label{cor:realization-EE-X-gen}
There is a canonical isomorphism of commutative unitary algebras
$$\An^*(\EE_{X})\simeq {\rm R}\Upsilon^{an}_*(h^{an},\varsigma_n)^* \beta^{an}_{X,\mathcal{S}},$$
where ${\rm R}\Upsilon^{an}_*$ is the derived direct image of complexes of sheaves. Moreover, the diagram
$$\xymatrix@C=1.3pc@R=1.5pc{(l^{an})^* \An^*(\EE_X) \ar[r]^-{\sim} \ar[d] & (l^{an})^* {\rm R}\Upsilon^{an}_* (h^{an},\varsigma_n)^* \beta^{an}_{X,\mathcal{S}} \ar[r] & {\rm R}\check{\Upsilon}^{an}_* (\check{h}^{an},\varsigma_n)_* l^* \beta^{an}_{X,\mathcal{S}} \ar[d] \\
\An^*(\EE_{\check{X}}) \ar[rr]^-{\sim} & & {\rm R}\check{\Upsilon}^{an}_* (\check{h}^{an},\varsigma_n)_* \beta^{an}_{\check{X},\check{\mathcal{S}}}}$$
is commutative.
\end{corollary}

\begin{proof}
The only point that remains to be checked is the commutation of the Betti realization with the cohomological direct image along $\Upsilon$, i.e., that the natural transformation
$\An^* \Upsilon_* \to {\rm R}\Upsilon^{an}_* \An^*$
is invertible when applied to compact motives. For this, we use the factorization of $\Upsilon$ into its geometric and categorical parts.
The commutation with the cohomological direct image along the geometric part follows from \cite{realiz-oper}.
We are then reduced to showing that
$\An^*$ commutes with homotopical limits along the indexing category of the diagram $\mathcal{Y}$. This follows from Lemma \ref{lemma:permanence-universal-for-holim} and Proposition
\ref{prop:limits-with-respect-ordered-set}.
\end{proof}

\label{page-for-digamma-resolution-top-spaces}
In the analytic context, we will need a direct construction of $\beta_{X,\mathcal{S}}^{an}$ which we now describe. This construction is possible as the
inverse image functors for sheaves on topological spaces
are exact, and thus do not need to be left derived as it is the case for motives.

Fix a functorial flasque resolution
$\digamma_{\!\dagger}$, for each topological space $\dagger$, that
is pseudo-monoidal and natural with respect to morphisms of
topological spaces. The latter condition means that a continuous
mapping $f:\dagger' \to \dagger$ induces a natural transformation
of pseudo-monoidal functors $f^*\circ \digamma_{\!\dagger} \to
\digamma_{\!\dagger'}\circ f^*$; moreover, these natural
transformations are compatible with the composition of continuous
mappings in the obvious way. One can take as
$\digamma_{\!\dagger}$ the monadic Godement resolution, for
instance.
It is clear that the resolution $\digamma_{\!\dagger}$
carries over to diagrams of topological spaces objectwise.
In the sequel, we
write just ``$\digamma$'', with the diagram of topological spaces
understood.

Clearly,
$\beta_{X,\mathcal{S}}^{an}$ is the restriction
to the sub-diagram $T^{an}\hookrightarrow T^{an+}$
of a complex of sheaves $\beta_{X,\mathcal{S}}^{an+}$ which is defined
inductively by the formula
\begin{equation}
\label{eq:for-defining-beta-an-X-S-directly-1}
\beta_{X,\mathcal{S}}^{an+}=(e^{an}_n)^*{\rm R} (o^{an+})_*(b^{an+})^*{\rm R}(j^{an+})_*\beta^{an+}_{X',\mathcal{S}'}.
\end{equation}
Of course, we are using the notation from
Remark \ref{rem:for-doing-induction-X-X'}, and the diagrams
\eqref{new-diagram-for-defining-beta-x-s-1} and
\eqref{new-diagram-for-defining-beta-x-s-2}.
Using the fixed resolution $\digamma$, we can take $(j^{an+})_*\circ \digamma$ for the derived functor ${\rm R}(j^{an+})_*$.

Now, assume that the restriction of $\beta^{an+}_{X',\mathcal{S}'}$ to $(pt,\mathcal{P}^*([\![1,n]\!])^{\rm op}\times \{1\}) \subset T^+$ is constant, i.e., equal to $K_{(pt,\mathcal{P}^*([\![1,n]\!])^{\rm op}\times \{1\})}$ where $K$ is a complex of $\Q$-vector spaces quasi-isomorphic to $\Q[0]$.
We claim that the natural morphism
\begin{equation}
\label{eq:for-defining-beta-an-X-S-directly-3}
(e^{an}_n)^*(o^{an+})_*(b^{an+})^*(j^{an+})_*\digamma\beta^{an+}_{X',\mathcal{S}'} \to (e^{an}_n)^*{\rm R} (o^{an+})_* (b^{an+})^*(j^{an+})_*\digamma\beta^{an+}_{X',\mathcal{S}'}
\end{equation}
is a quasi-isomorphism.
Over the sub-diagram $T^{an+}\circ (o\times \id_{\underline{\mathbf{1}}})$, this is clear as $(o^{an+})_*$ is the identity functor there.
As $(\beta^{an+}_{X',\mathcal{S}'})_{|\mathcal{P}^*([\![1,n-1]\!])^{\rm op}\times \{1\}}$ is the constant sheaf associated to $K$,
then
\eqref{eq:for-defining-beta-an-X-S-directly-3}
is given over $(pt,\{(\{n\},1)\})$ by
\begin{equation}
\label{eq:for-defining-beta-an-X-S-directly-4}
{\rm lim}_{\mathcal{P_2}([\![1,n-1]\!])^{\rm op}\times \underline{\mathbf{1}}} \, \digamma_{\!pt}K \to
{\rm holim}_{\mathcal{P_2}([\![1,n-1]\!])^{\rm op}\times \underline{\mathbf{1}}} \, \digamma_{\!pt}K.
\end{equation}
The latter is a quasi-isomorphism as both sides are quasi-isomorphic to
$\digamma_{\!pt}K$. (The left hand side is in fact isomorphic to
the complex $\digamma_{\!pt}K$.) Finally,
over $T^{an}(\{n\})=T^{an+}(\{n\},0)$, the morphism
\eqref{eq:for-defining-beta-an-X-S-directly-3}
is the pull-back of
\eqref{eq:for-defining-beta-an-X-S-directly-4}
along the projection of $T^{an}(\{n\})$ to the point. Hence, it is also a quasi-isomorphism.

It follows from the above that
$\beta_{X,\mathcal{S}}^{an+}$ can be defined inductively using the simpler formula
\begin{equation}
\label{eq:for-defining-beta-an-X-S-directly-5}
\beta_{X,\mathcal{S}}^{an+}=(e^{an}_n)^*(o^{an+})_* (b^{an+})^*(j^{an+})_*\digamma\beta^{an+}_{X',\mathcal{S}'}.
\end{equation}
Remark that if $(\beta^{an+}_{X',\mathcal{S}'})_{|\mathcal{P}^*([\![1,n-1]\!])^{\rm op}\times \{1\}}$ is the constant sheaf associated to $K$, then
$(\beta^{an+}_{X,\mathcal{S}})_{|\mathcal{P}^*([\![1,n]\!])^{\rm op}\times \{1\}}$ is
the constant sheaf associated to $\digamma K$. By an easy induction, we see that $(\beta^{an+}_{X,\mathcal{S}})_{|\mathcal{P}^*([\![1,n]\!])^{\rm op}\times \{1\}}$ is the constant sheaf associated to
$\digamma^n\Q$.

Now, in the formula
\eqref{eq:for-defining-beta-an-X-S-directly-5},
$(e^{an}_n)^*$ has the effect to replace the complex of sheaves
$(\{n\},0)^*(o^{an+})_* (b^{an+})^*(j^{an+})_*\digamma\beta^{an+}_{X',\mathcal{S}'}$
on $T^{an+}(\{n\},0)=T^{an}(\{n\})$ by
$(\digamma^n\Q)_{T^{an}(\{n\})}$. This shows that $\beta_{X,\mathcal{S}}^{an}$ is obtained from
$\beta^{an}_{X',\mathcal{S}'}$ as follows. First, consider
the complex of sheaves
$(b^{an+})^*(j^{an+})_*\digamma\beta^{an+}_{X',\mathcal{S}'}$ on $T^{an}\circ o$.
Then extend it to $T^{an}$
by adding the sheaf $(\digamma^n\Q)_{T(\{n\})}$ over $T(\{n\})$. In fact, it doesn't change much if one puts $\Q_{T(\{n\})}$ instead of $(\digamma^n\Q)_{T(\{n\})}$. This is possible, i.e., we still get an object of $\mathbf{K}(\Shv(T^{an}))$, by using the canonical map $\Q \to \digamma^n\Q$ to define the restriction maps along arrows in
${\rm Ouv}(T,\mathcal{P}^*([\![1,n]\!])^{\rm op})$.

For $\emptyset \neq I \subset [\![1,n]\!]$, denote $T^0(I)$ the inverse image of $X_{{\rm max}(I)}$ in $T(I)$.
It is now clear that
$\beta^{an}_{X,\mathcal{S}}$ is given over $\emptyset \neq I \subset [\![1,n]\!]$ by the following complex of sheaves on $T(I)^{an}$:
$$(T^0(I\cap [\![i_0,n]\!])^{an}\hookrightarrow T(I\cap [\![i_0,n]\!])^{an})_*\digamma (T^0(I\cap [\![i_0,n]\!])^{an}\hookrightarrow T(I\cap [\![i_0,n-1]\!])^{an})^* \cdots $$
$$(T^0(I\cap [\![i_0,i_0+1]\!])^{an}\hookrightarrow T(I\cap [\![i_0,i_0+1]\!])^{an})_*\digamma (T^0(I\cap [\![i_0,i_0+1]\!])^{an}\hookrightarrow T(I\cap [\![i_0,i_0]\!])^{an})^*$$
$$(T^0(I\cap [\![i_0,i_0]\!])^{an}\hookrightarrow T(I\cap [\![i_0,i_0]\!])^{an})_*\digamma \Q_{T^0(\{i_0\})^{an}},$$
with $i_0={\rm min}(I)$.
Simplifying a little bit, we arrive to the following statement
(see the proof of
Lemma \ref{lemma:applic-of-comput-direct-im-sncd-I}):

\begin{lemma}
\label{lemma:rajou-concrete-desc-beta-an-gen}
The complex of sheaves of $\Q$-vector spaces
$\beta^{an}_{X,\mathcal{S}}$ has, up to a canonical quasi-isomorphism, the following description.
Let $\emptyset \neq I \subset [\![1,n]\!]$ and write $I=\{i_0<\cdots<i_m\}$. For $0\leq j \leq m$, we set $I_j=\{i_0,\dots, i_j\}$. Then $\beta^{an}_{X,\mathcal{S}}(I)$ is the following complex
$$(T^0(I_m)^{an}\hookrightarrow T(I_m)^{an})_*\digamma(T^0(I_{m})^{an} \hookrightarrow T(I_{m-1})^{an})^*\dots \hspace{3cm}$$
$$(T^0(I_1)^{an} \hookrightarrow T(I_1)^{an})_*\digamma(T^0(I_1)^{an} \hookrightarrow  T(I_0)^{an})^*(T^0(I_0)^{an}\hookrightarrow T(I_0)^{an})_*\digamma\Q_{T^0(I_0)^{an}}.$$
Moreover, for $\emptyset \neq J \subset I$, the morphism
$\beta^{an}_{X,\mathcal{S}}(J) \to (T(I) \to T(J))_*
\beta^{an}_{X,\mathcal{S}}(I)$ is a composition of units of adjunction and augmentations $\id \to \digamma$.
\end{lemma}

It is a corollary of Theorem
\ref{thm:final-form-for-applic-main-thm} that one can use $\un_{\mathcal{Y}}$ instead of the more complicated $(h,\varsigma_n)^*\beta_{X,\mathcal{S}}$ to compute $\EE_X$, though we need the original version for the proof of Theorem \ref{thm:main-thm}. Precisely:

\begin{corollary}
\label{cor:not-used-for-main-thm}
There is a canonical isomorphism of commutative unitary algebras
$\EE_X \simeq \Upsilon_*\un_{\mathcal{Y}}$. Moreover, the following diagram
$$\xymatrix@C=1.5pc@R=1.5pc{l^*\EE_X \ar[rr] \ar[d]_-{\sim} & & \EE_{\check{X}} \ar[d]^-{\sim} \\
l^* \Upsilon_*\un_{\mathcal{Y}} \ar[r] & \check{\Upsilon}_* l^* \un_{\mathcal{Y}} \ar[r]^-{\sim} & \check{\Upsilon}_* \un_{\check{\mathcal{Y}}}}$$
commutes.

\end{corollary}

\begin{proof}
We only prove the first claim.
There is a canonical morphism $\un_{T} \to \beta_{X,\mathcal{S}}$ (which is the unity of the algebra) that induces a morphism
\begin{equation}
\label{eq-cor:not-used-for-main-thm-1}
\xymatrix@C=1.7pc{\Upsilon_* \un_{\mathcal{Y}} \ar[r] &  \Upsilon_* (h,\varsigma_n)^*\beta_{X,\mathcal{S}}.}
\end{equation}
By Theorem
\ref{thm:final-form-for-applic-main-thm}, it suffices to show that
\eqref{eq-cor:not-used-for-main-thm-1} is invertible.
We split the proof into two steps.

\smallskip

\noindent
\underbar{Part A:}
Here, we prove, by induction on $n$, that
$\Upsilon_*\un_{\mathcal{Y}}$ is an Artin motive.
When $n=0$, this is clear.

Assume $n\geq 0$ and that $\Upsilon'_*\un_{\mathcal{Y}'}$ is known to be an Artin motive over $X'$.
To check that $\Upsilon_*\un_{\mathcal{Y}}$ is an Artin motive, it suffices to show that
$j^*\Upsilon_*\un_{\mathcal{Y}}$ and $u_n^*\Upsilon_*\un_{\mathcal{Y}}$ are Artin motives, with $j:X'\hookrightarrow X$ and $u_n:X_n \hookrightarrow X$ the inclusions.
We have $j^*\Upsilon_*\un_{\mathcal{Y}}\simeq \Upsilon'_*\un_{\mathcal{Y}'}$, which settles the case of $j^*\Upsilon_*\un_{\mathcal{Y}}$ by the induction hypothesis.

It remains to show that
$u_n^* \Upsilon_*\un$ is an Artin motive. Using Proposition
\ref{prop-most-gen-proj-base-change},
we have $u_n^*\Upsilon_*\un \simeq \{\mathcal{Y}\times_X X_n \to X_n\}_*\un$. Moreover, the latter is the homotopy limit of
\begin{equation}
\label{eq-cor:not-used-for-main-thm-rajou-1}
\xymatrix@C=1.5pc{
\{(\mathcal{Y}\circ \iota^0_n)\times_X X_n\to X_n\}_*\un
\ar[r] & \{(\mathcal{Y}\circ \iota_n)\to X_n\}_*\un &
\{(\mathcal{Y}\circ \iota_n^1) \to X_n\}_*\un  \ar[l]}
\end{equation}
with $\iota^0_n$, $\iota_n$ and $\iota^1_n$ the non-decreasing maps from $\mathcal{P}_2([\![1,n-1]\!])$ to $\mathcal{P}_2([\![1,n]\!])$ sending $(I_0,I_1)$ to $(I_0\bigsqcup \{n\},I_1)$, $(I_0,I_1)$ and
$(I_0,I_1\bigsqcup\{n\})$ respectively.
As $\mathcal{Y}\circ \iota_n^1$ is objectwise finite over $X_n$, we deduce that
$\{(\mathcal{Y}\circ \iota_n^1) \to X_n\}_*\un$ is an Artin motive.
Hence, it suffices to show that
the first arrow in
\eqref{eq-cor:not-used-for-main-thm-rajou-1}
is an isomorphism. This would follows if the natural morphisms
$$\un_{\mathcal{Y}(I_0\bigsqcup\{n\},I_1)\times_X X_n} \!\xymatrix@C=1.7pc{\ar[r] &}
\{\mathcal{Y}(I_0,I_1)\to \mathcal{Y}(I_0\bigsqcup\{n\},I_1)\times_X X_n\}_*\un_{\mathcal{Y}(I_0,I_1)}$$
are invertible for all $(I_0,I_1)\in \mathcal{P}_2([\![1,n-1]\!])$.
This can be done as in Part C of the proof of
Proposition \ref{prop:comparison-theta-prime-and-h-theta-double}.
We leave the details to the reader.

\smallskip

\noindent
\underbar{Part B:} Recall that we need to show that
\eqref{eq-cor:not-used-for-main-thm-1} is invertible.
As both sides are Artin motives, it suffices to show that
\begin{equation}
\label{eq-cor:not-used-for-main-thm-701}
\xymatrix@C=1.7pc{\omega^0_{X} (\Upsilon_* \un_{\mathcal{Y}}) \ar[r] & \omega^0_X(\Upsilon_* (h,\varsigma_n)^*\beta_{X,\mathcal{S}}) \simeq
\omega^0_X(\Upsilon_* (p,\varsigma_n)^*\beta'_{X,\mathcal{S}})}
\end{equation}
is invertible.
Using Proposition
\ref{prop:additional-prop-omega-0-x}, (ii) and (iii), we have
canonical isomorphisms
$$\omega^0_X \Upsilon_* (p,\varsigma_n)^* \omega^0_{\mathcal{T}} \beta'_{X,\mathcal{S}} \simeq \omega^0_X\Upsilon_*\omega^0_{\mathcal{Y}} (p,\varsigma_n)^*\beta'_{X,\mathcal{S}} \simeq
\omega^0_X\Upsilon_* (p,\varsigma_n)^*\beta'_{X,\mathcal{S}}.$$
Hence, it suffices to check that
$\un_{\mathcal{T}} \to \omega^0_{\mathcal{T}}(\beta'_{X,\mathcal{S}})$
is invertible. But this follows immediately from
Lemmas
\ref{lemma:omega-0-restriction-direct-image-sncd}
and \ref{lemma:applic-of-comput-direct-im-sncd-I}.
\end{proof}

\section{Compactifications of locally symmetric varieties}
\label{sect:compactif-locally-symmetric}

This section is an exposition of known material that is fundamental for our
construction.

\subsection{Generalities involving algebraic groups and symmetric spaces}
\label{subsect:generalities}
Linear algebraic groups over $\Q$
will always be denoted with boldface roman letters: $\mathbf{G}$,
$\mathbf{H}$, $\mathbf{P}$, etc. Their groups of $\R$-points
$\mathbf{G}(\R)$, $\mathbf{H}(\R)$, $\mathbf{P}(\R)$, etc.\,\,will
be denoted by the corresponding italic letters: $G$, $H$, $P$,
etc. Given a Lie group $G$, we denote by $G^0$ the connected
component of the identity element.

Let $\mathbf{G}$ be a semi-simple linear algebraic group over $\Q$. We assume
that $\mathbf{G}$ is simple over $\Q$, for the general case can be
deduced from that. Let $D$ be a symmetric space (of non-compact type)
such that ${\rm
Aut}(D)=G$ (modulo compact factors). One has that $D$ is a contractible
space.
Given a base point $x\in D$,
$K={\rm Stab}(x)$ is a maximal compact subgroup of $G$ and one has
$D\simeq G/K$. $D$ is said to be {\it hermitian}
when it admits a $G$-invariant complex structure.

An {\it arithmetic subgroup} $\Gamma\subset\mathbf{G}(\Q)$ is a group
commensurable
with $\mathscr{G}(\Z)$ (where $\mathscr{G}$ is group scheme over $\Z$ such that $\mathbf{G} = \mathscr{G}\otimes_{\Z}\Q$). For such $\Gamma$, one considers
the quotient $\Gamma\backslash D$, which has finite volume with respect to
an invariant metric.  When $D$ is
hermitian, $\Gamma\backslash D$ is actually the analytic space of $\C$-points of a quasi-projective $\C$-scheme $X$, as follows from \cite{baily-borel} (see our
\S\ref{subsect:The-Baily-Borel-compactification});
it is called a {\it locally symmetric variety} for obvious
reasons. In fact, the $\C$-scheme $X$ can be defined over a number
field.\footnote{The {\it Shimura variety} associated to
$\mathbf{G}$, where in effect $\Gamma$ is allowed to vary, has
$X(\C)$ as a connected component, and it is defined over a single
number field $k$ (called the {\it reflex field}) (see
\cite[2.2.1]{deligne}). Each connected component generally will
not be defined over $k$, but rather some algebraic extension of
$k$.}

The analytic space $X(\C)$ has various natural
compactifications, some of them algebraic and others only
topological. We describe a few of these below.  We will assume
throughout that $\Gamma$ is {\it neat}, in the sense of
\cite[D\'ef.~17.1]{borel}.
(Any arithmetic group $\Gamma$ contains a neat arithmetic subgroup
that is normal and of finite index.)  This ensures there are no
quotient singularities distorting the stratification of the
compactifications below.

If $\Gamma$ is an arithmetic subgroup of $\mathbf{G}(\Q)$ and
$\mathbf{H_1/H_2}$ is an algebraic subquotient group of
$\mathbf{G}$ (so $\mathbf{H_2}$ is a normal subgroup of
$\mathbf{H_1}$), we
let $\Gamma(\mathbf{H_1}/\mathbf{H_2})$ denote the induced arithmetic subgroup of
$H_1/H_2$, viz., $(\Gamma\cap H_1)/(\Gamma\cap H_2)$. In other words,
we view $\Gamma$ as defining a functor from such pairs
$(\mathbf{H_1},\mathbf{H_2})$ to groups.

Given two arithmetic subgroups $\Gamma, \, \Gamma'\subset \mathbf{G}(\Q)$ and $g\in \mathbf{G}(\Q)$ such that $g\Gamma' g^{-1}\subset \Gamma$, we have an induced map
(essentially a {\it Hecke correspondence})
$\Gamma'\backslash D \to \Gamma \backslash D$ which we usually denote by $g$. When $D$ is hermitian, this map comes from a morphism of $\C$-schemes $g:X' \to X$
(where $X'$ is the $\C$-scheme such that $X'(\C)\simeq \Gamma'\backslash D$). In fact, this morphism is defined over a number field.

\subsection{The Borel-Serre compactifications}
\label{subsect:The Borel-Serre compactifications}
\emph{The main reference for
the material in this subsection is \cite{borel-serre}; the reductive version
was introduced in \cite[\S4]{warped} (see \cite{rbs-1}). For these
compactifications, $D$ does not have to be hermitian.}

Let $\mathbf{P}\subset\mathbf{G}$
be a parabolic $\Q$-subgroup,
$\mathbf{N_P}$ its unipotent radical and $\mathbf{L_P}=\mathbf{P}/
\mathbf{N_P}$. The choice of a base point $x\in D$ induces
a lift of $L_P$ to $L_P(x)\subset P$. It is possible
(see for example
\cite[Prop.~III.1.11]{borel-ji})
to choose $x$, so that $L_P(x)$ is the Lie group of $\R$-points of a $\Q$-subgroup
$\mathbf{L_P}(x)\subset \mathbf{P}$, and we will do so.
$\mathbf{L_P}(x)$ is called a
{\it Levi subgroup} of $\mathbf{P}$, and we have
$\mathbf{P}=\mathbf{N_P}
\mathbf{L_P}(x)$, a semi-direct product.
Let $\mathbf{S_P}$ be the maximal $\Q$-split torus in the
center of $\mathbf{L_P}$. Then one has an almost direct product decomposition
$\mathbf{L_P}=\mathbf{S_P}\mathbf{M_P}$.
We denote by $\mathbf{S_P}(x)$ and $\mathbf{M_P}(x)$ the images of
$\mathbf{S_P}$ and $\mathbf{M_P}$ in the lift $\mathbf{L_P}(x)$.
One obtains the Langlands decomposition of $P$:
\begin{equation}\label{Langlands-decomp}
P=N_P\times (M_P(x)\times A_P),
\end{equation}
a semi-direct product, where $A_P=S_P(x)^0$. There is a maximal $\Q$-split
torus $\mathbf{S}$ of
$\mathbf{G}$ containing $\mathbf{S_P}(x)$ and a set of simple $\Q$-roots (characters)
$\Delta(\mathbf{G},\mathbf{S})$
with respect to $\mathbf{S}$ for which $\mathbf{P}$ is standard
(see \cite[4.1]{borel-serre}, or \S
\ref{subsect:The-Baily-Borel-compactification}
below).  Then the subset
$\Delta_{\mathbf{P}}\subset \Delta(\mathbf{G},\mathbf{S})$,
consisting of those roots $\alpha$ that are non-trivial
on $A_P$,
provides coordinates on $A_P$, which determines a canonical isomorphism
\begin{equation}\label{AP}
A_P\simeq (\R^+)^{\Delta_{\mathbf{P}}}.
\end{equation}
The {\it parabolic $\Q$-rank} of $\mathbf{P}$, denoted $r(\mathbf{P})$, is
${\rm card}(\Delta_{\mathbf{P}})=\dim A_P$.

The symmetric space $D$ admits two useful, topological partial
compactifications,
the Borel-Serre and the reductive
Borel-Serre, which we proceed to describe.
Given a parabolic $\Q$-subgroup $\mathbf{P}\subseteq \mathbf{G}$ (not necessary a proper subgroup of $\mathbf{G}$, i.e., $\mathbf{P}=\mathbf{G}$ is allowed), let $\overline{A}_P$ denote the
``pure corner'' given in terms of \eqref{AP} as $(0,\infty]^{\Delta_{\mathbf{P}}}$,
a torus embedding over $\R$.\footnote{In \cite{borel-serre},
$\overline{A}_P$ is given as $[0,\infty)^{\Delta_{\mathbf{P}}}$, but there the
convention is
that $G$ acts on $D$ on the right. We are using the more common
convention nowadays of a left-action.}
Then, the {\it corner}
for $\mathbf{P}$ is defined to be the partial compactification of $D$:
\begin{equation}\label{corner}
\overline{D}(\mathbf{P}) = D\times^{A_P}\overline{A}_P,
\end{equation}
where $A_P$ acts on $D$ by the {\it geodesic action} \cite[\S 3]
{borel-serre},\footnote{When $\mathbf{G}$ is $\mathbf{SL_2}$, so $D$ is the
upper
half-plane, $\mathbf{P}$ the group of upper-triangular matrices, then $A_P$
is the
subgroup of diagonal matrices with positive diagonal entries. The usual
action of $A_P$ on $D$ is radial, but
the geodesic
action is vertical. Thus, the Borel-Serre construction for $\mathbf{P}$ puts a
line at infinity. (The line is collapsed to a point in the reductive
version; see below.)} which commutes with the usual action of $P$.
Moreover, when $\mathbf{P}\subseteq \mathbf{Q}$, there are canonical inclusions
$A_Q\subseteq A_P$ and
$\Delta_{\mathbf{Q}}\subseteq \Delta_{\mathbf{P}}$ (so $\mathbf{N_Q}\subseteq \mathbf{N_P}$). This yields a canonical
embedding
\begin{equation}\label{gluing}
\overline{D}(\mathbf{Q}) \hookrightarrow \overline{D}(\mathbf{P}),
\end{equation}
of partial compactifications
of $D$. Note that $\overline{D}(\mathbf{G}) = D$. Using \eqref{gluing} for gluing, one obtains from
these $\overline{D}(\mathbf{P})$
the space $\overline{D}^{bs}$, which is shown to be a manifold with corners
for
which \eqref{corner} provides local charts.
The boundary face, or stratum, $e(\mathbf{P})$ of $\overline{D}^{bs}$ that is
associated to
$\mathbf{P}$ is
the lowest-dimensional $A_P$-orbit in $\overline{D}(\mathbf{P})$. In terms of
\eqref{AP},
\begin{equation}\label{eP}
e(\mathbf{P})= D\times^{A_P}\{\infty\}^{\Delta_{\mathbf{P}}}\simeq D/A_P\simeq
N_P\times D_P,
\end{equation}
where $D_P=M_P(x)/(M_P(x)\cap K)$ (cf.~(\ref{Langlands-decomp})).  Thus,
$e(\mathbf{P})$ is contractible,
and it is attached to
$D$ as the set of limits of the full geodesic action of $A_P$.
Then as sets,
\begin{equation}\label{corner-stratification}
\overline{D}(\mathbf{P})=\bigsqcup_{\mathbf{P}\subseteq
\mathbf{Q}}
e(\mathbf{Q}) \quad
\text{and} \quad
\overline{D}^{bs}=\bigsqcup_{\mathbf{P}} e(\mathbf{P}),
\end{equation}
and the above displays the standard stratification of a manifold
with corners. (In the language of \S\ref{sub:direct-image-sncd},
we have $e(\mathbf{P})\preceq e(\mathbf{Q})$ when
$\mathbf{P}\subseteq\mathbf{Q}$.) Thus, $e(\mathbf{P})$ is of
codimension $r(\mathbf{P})$ in $\overline{D}^{bs}$ and the open
stratum is $e(\mathbf{G})=D$. When $\mathbf{P}\subseteq
\mathbf{Q}$, the action of $A_P$ on $\overline{D}(\mathbf{P})$
preserves the stratum $e(\mathbf{Q})$. Moreover, $A_P$ acts on
$e(\mathbf{Q})$ through the quotient $A_P/A_Q$.

The group $\mathbf{G}(\Q)$ acts on $\overline{D}^{bs}$, with
$g\in\mathbf{G}(\Q)$
taking $e(\mathbf{Q})$ to $e(g\mathbf{Q}g^{-1})$.
A neat arithmetic subgroup
$\Gamma\subset\mathbf{G}(\Q)$
acts on $\overline{D}^{bs}$ without fixed points, and the quotient
$\Gamma\backslash\overline{D}^{bs}$ is a {\it compact} manifold with corners.
To emphasize that this is a compactification of $\Gamma\backslash D$, we
also write
$\overline{\Gamma\backslash D}^{bs}$; this is the {\it Borel-Serre
compactification}
of $\Gamma\backslash D$. We have, also as sets, a finite decomposition
into strata (cf.~\eqref{corner-stratification})
\begin{equation}\label{arithmetic-bs}
\overline{\Gamma\backslash D}^{bs}=
\bigsqcup_{\mathbf{P}} e'(\mathbf{P}),
\end{equation}
where $\mathbf{P}$ is taken modulo $\Gamma$-conjugacy, and the
``prime'' in the term for $\mathbf{P}$ indicates the quotient by
$\Gamma(\mathbf{P})$, which coincides with $\{\gamma\in\Gamma;\,
\gamma\text{ stabilizes }
e(\mathbf{P})\}$.
The open stratum in \eqref{arithmetic-bs} is $e'(\mathbf{G})=\Gamma\backslash D$.
The compactness of $\overline{\Gamma\backslash D}^{bs}$
gives the existence of a
neighborhood
of $\overline{e(\mathbf{P})}$ in $\overline{D}^{bs}$ on which
$\Gamma$-equivalence
and $\Gamma(\mathbf{P})$-equivalence coincide.\footnote{Unless
$\mathbf{P}$ is minimal, this neighborhood
cannot be taken
to be of the form $N_P\times D_P\times\{a\in A_P:a^\beta>t\text{ for all }
\beta\in
\Delta(\mathbf{P})\}$, as is stated erroneously in \cite[\S 10]{borel-serre}.
(One
can trace this back to 5.4, (7) of {\it op.~cit.})}

The reductive Borel-Serre compactification of $\Gamma\backslash D$ is the quotient by $\Gamma$ of a certain
stratified quotient space $\overline{D}^{rbs}$
of $\overline{D}^{bs}$,
or equivalently (from the point of view of $\Gamma\backslash D$), a quotient space of
$\overline{\Gamma\backslash D}^{bs}$.
The mapping $\overline{D}^{bs}\to \overline{D}^{rbs}$ is given stratum
by stratum by the canonical projection
$e(\mathbf{P})\to\widehat`{e}(\mathbf{P})$, where
\begin{equation}\label{reductive}
\widehat{e}(\mathbf{P}):=N_P\backslash e(\mathbf{P})\simeq D_P
\end{equation}
for $\mathbf{P}\subseteq \mathbf{G}$ a parabolic $\Q$-subgroup (not necessarily proper). In particular, $\widehat{e}(\mathbf{G})=D$ and
$\overline{D}^{bs} \to \overline{D}^{rbs}$ is the identity on their common open stratum.

It is rather straightforward to determine that with the quotient
topology, $\overline{D}^{rbs}$ is a separated space.
It is clear from \eqref{corner-stratification} that as sets,
\begin{equation}
\label{reductive-stratification}
\overline{D}^{rbs}=\bigsqcup_{\mathbf{P}} \widehat{e}(\mathbf{P}),
\end{equation}
where $\mathbf{P}$ runs over all parabolic $\Q$-subgroups of $\mathbf{G}$.
The quotient by a neat arithmetic subgroup is separated as well, and
$\Gamma\backslash\overline{D}^{rbs}=\overline{\Gamma\backslash D}^{rbs}$
is a compact stratified space. It is called, because of \eqref{reductive}, the
{\it reductive Borel-Serre
compactification} of $\Gamma\backslash D$. Clearly,
\eqref{arithmetic-bs}
and \eqref{reductive-stratification}
imply that as a set,
\begin{equation}
\overline{\Gamma\backslash D}^{rbs}=
\bigsqcup_{ \mathbf{P}} \widehat{e}^{\,\prime}(\mathbf{P}),
\end{equation}
with $\mathbf{P}$ as for (\ref{arithmetic-bs}). Note that
$\widehat{e}^{\,\prime}(\mathbf{G})=\Gamma \backslash D$. More generally,
\begin{equation}
\label{eq:startum-rbs-compact-loc-sym}
\widehat{e}^{\, \prime}(\mathbf{P})=\Gamma(\mathbf{M}_{\mathbf{P}})\backslash \widehat{e}(\mathbf{P}).
\end{equation}
where $\Gamma(\mathbf{M}_{\mathbf{P}})=(\Gamma\cap P)/(\Gamma \cap N_P A_P)$ which coincides with $\Gamma(\mathbf{P}/\mathbf{N}_{\mathbf{P}}\mathbf{S}_{\mathbf{P}}(x))$ as $\Gamma$ is neat.
There is a canonical quotient
mapping $\overline{
\Gamma\backslash D}^{bs}\to \overline{\Gamma\backslash D}^{rbs}\!$, which is a {\it
morphism of compactifications}, i.e., it maps $\Gamma\backslash D$ to itself
by the identity mapping.

The above constructions are hereditary, in that the closure of $e(\mathbf{P})$ (resp. $\widehat{e}(\mathbf{P})$) in
$\overline{D}^{bs}$ (resp. $\overline{D}^{rbs}$)
can be identified with the Borel-Serre (resp. reductive Borel-Serre) compactification
$\overline{e(\mathbf{P})}^{bs}$ (resp. $\overline{\widehat{e}(\mathbf{P})}^{rbs}$) of $e(\mathbf{P})$ (resp. $\widehat{e}(\mathbf{P})$). Note that
$e(\mathbf{P})$ is not a symmetric space unless $\mathbf{P}=
\mathbf{G}$, and $\widehat{e}(\mathbf{P})$ may contain euclidean
factors. Nevertheless, these are spaces to which the Borel-Serre
construction applies \cite[\S 2]{borel-serre}. As sets,
$$
\overline{e(\mathbf{P})}^{bs}=\bigsqcup_{\mathbf{Q}}
e(\mathbf{Q})  \quad \text{and} \quad \overline{\widehat{e}(
\mathbf{P})}^{rbs}=\bigsqcup_{\mathbf{Q}}
\widehat{e}(\mathbf{Q}),
$$
where $\mathbf{Q}$ runs over all parabolic $\Q$-subgroups of $\mathbf{G}$ contained in $\mathbf{P}$.
Inside $\overline{D}^{bs}$, we have
\begin{equation}\label{intersections}
\overline{e(\mathbf{P})}^{bs}\cap\, \overline{e(\mathbf{Q})}^{bs}=\left\{
\begin{array}{cc}
\overline{e(\mathbf{P}\cap\mathbf{Q})}^{bs} & \text{if }
\mathbf{P}\cap\mathbf{Q} \text{ is parabolic,}\\
\emptyset & \text{otherwise.}
\end{array}\right.
\end{equation}
However, in $\overline{\Gamma\backslash D}^{bs}$,
$\overline{e'(\mathbf{P})}^{bs}$ and $\overline{e'(\mathbf{Q})}^{bs}$ have
non-empty intersection if and only if $\mathbf{P}$ and {\it a
$\Gamma$-conjugate of}\, $\mathbf{Q}$ have parabolic intersection.
It is known that when $\mathbf{P}\cap
\mathbf{Q}$ is parabolic,
$\overline{e'(\mathbf{P})}^{bs}\cap\; \overline{e'(\mathbf{Q})}^{bs}$ is the
union of finitely many
connected components, one of which is
$\overline{e'(\mathbf{P}\cap\mathbf{Q})}^{bs}$, and the others are of a similar
nature (see \cite[\S3:\,Appendix]{HZ2}).
Parallel statements hold for $\overline{\widehat{e}(\mathbf{P})}^{rbs}$.

If $\Gamma'\subset \mathbf{G}(\Q)$ is another neat arithmetic subgroup and
$g\in \mathbf{G}(\Q)$ is such $g\Gamma' g^{-1}\subset \Gamma$, the induced
morphism $g:\Gamma' \backslash D \to \Gamma
\backslash D$
extends
to the Borel-Serre and the reductive Borel-Serre compactifications, yielding:
\begin{equation}
\label{eqn:g-extended-to-borel-serre-compact}
g^{bs}:\overline{\Gamma'\backslash D}^{bs}\to\overline{\Gamma\backslash D}^{
bs} \quad \text{and} \quad
g^{rbs}:\overline{\Gamma'\backslash D}^{rbs} \to \overline{\Gamma\backslash
D}^{rbs}.
\end{equation}

\subsection{The Baily-Borel Satake compactification}
\label{subsect:The-Baily-Borel-compactification}
\emph{The main reference for
the material in this subsection is \cite{baily-borel}.}

We assume that $D$ is a hermitian symmetric space. Let
$\mathbf{P}$ be a parabolic $\Q$-subgroup of $\mathbf{G}$ (not necessarily
proper). The
Levi quotient $\mathbf{L_P}$ admits a more refined decomposition
than is given in \S\ref{subsect:The Borel-Serre compactifications},
which we next describe.

Let $\mathbf{S}\subset \mathbf{G}$ be a maximal $\Q$-split torus in
$\mathbf{G}$. Let $\Phi(\mathbf{G},\mathbf{S})$ be the set of $\Q$-roots of
$\mathbf{G}$ with respect to $\mathbf{S}$.
Choose an order on $\mathbf{S}$ and denote the set of positive roots by
$\Phi^+(\mathbf{G},\mathbf{S})$ and the set of simple roots by $\Delta(
\mathbf{G},\mathbf{S})$. By \cite[\S 2.9]{baily-borel},
the root system $\Phi(\mathbf{G},\mathbf{S})$ is of classification
type $BC_r$ or $C_r$,
where $r={\rm rk}_{\Q}(\mathbf{G})$ (recall that $\mathbf{G}$ is assumed to
be $\Q$-simple).

List the simple roots as $\beta_1, \dots, \beta_r$ so that
$\beta_i$ is not orthogonal to $\beta_{i+1}$, and $\beta_r$ is the
short root if $\Phi(\mathbf{G},\mathbf{S})$ is of classification
type $BC_r$ and the long root if $\Phi(\mathbf{G},\mathbf{S})$ is
of classification type $C_r$. The root $\beta_r$ will be called
the {\it distinguished root} or the root at the distinguished end.

There is a unique minimal parabolic $\Q$-subgroup
$\mathbf{P}$ whose unipotent radical $\mathbf{N_P}$ is spanned by
the root spaces of the roots in $\Phi^+(\mathbf{G},\mathbf{S})$.
The parabolic $\Q$-subgroups $\mathbf{Q}$ that contain
$\mathbf{P}$ will be called {\it standard}. They are the ones
expressible in the form $\mathbf{P}_I$ for proper subsets
$I\subset \Delta(\mathbf{G},\mathbf{S})$; this is generated by
$\mathbf{N_P}$ and the centralizer of $\mathbf{S}_I:=\{s\in
\mathbf{S}; \, s^{\beta}=1, \, \beta \in I\}$.  Then
$\mathbf{N}_{\mathbf{P}_I}$ is the product of the root spaces of all
roots not in the span of $I$; this set of roots is denoted
$\Phi^+(\mathbf{G},\mathbf{S})^I$. Every parabolic subgroup $\mathbf{Q}$
of $\mathbf{G}$ is a $\mathbf{G}(\Q)$-conjugate of a unique standard
parabolic subgroup $\mathbf{P}_I$. We then say that $\mathbf{Q}$ is of {\it type}
$I$, or of {\it cotype} $J$, where $J=\Delta(\mathbf{G},\mathbf{S})-I$.

Recall that a subset of $\Delta(\mathbf{G},\mathbf{S})$ is
called {\it connected} if it is not the disjoint union of two
non-empty subsets which are orthogonal with respect to the Killing
form. Given a proper subset $I\subset
\Delta(\mathbf{G},\mathbf{S})$, let $\Delta_{I,h}$ be the
connected component of $I$ containing the distinguished root
$\beta_r$, with the convention that if $\beta_r\not \in I$, then
$\Delta_{I,h}=\emptyset$. We also put
$\Delta_{I,\ell}=I-\Delta_{I,h}$.

The subset $\Delta_{I,h}$ spans a subsystem
$\Phi_{I,h}(\mathbf{G}, \mathbf{S})$ of
$\Phi(\mathbf{G},\mathbf{S})$. The root spaces of elements in
$\Phi_{I,h}(\mathbf{G},\mathbf{S})$ generate a semi-simple
subgroup $\mathbf{M}_{\mathbf{Q},h}$ of $\mathbf{M_Q}$. Similarly,
$\Delta_{I,\ell}$ spans a subsystem
$\Phi_{I,\ell}(\mathbf{G},\mathbf{S})$ and the root spaces of
roots in $\Phi_{I,\ell}(\mathbf{G},\mathbf{S})$ generate a
semi-simple subgroup $\mathbf{M}_{\mathbf{Q},\ell}$ of
$\mathbf{M_Q}$. We have an almost direct product decomposition
$\mathbf{M_Q}=\widetilde{\mathbf{M}}_{\mathbf{Q},\ell}\times
{\mathbf{M}}_{ \mathbf{Q},h}$,\footnote{In the literature, notably
\cite{toroidal-comp}, one finds the subscripts reversed:
``${\ell,\mathbf{Q}}$'' and ``${h,\mathbf{Q}}$'', and use of the
notation $\mathbf{G}_{\ell, \mathbf{Q}}$ and
$\mathbf{G}_{h,\mathbf{Q}}$ instead of
$\widetilde{\mathbf{M}}_{\mathbf{Q},\ell}$ and ${\mathbf{M}}_{
\mathbf{Q},h}$.} where $\widetilde{\mathbf{M}}_{\mathbf{Q},\ell}$
is a reductive group containing $\mathbf{M}_{\mathbf{Q},\ell}$ and
having the same root system. This decomposition can be extended to
any parabolic $\Q$-subgroup $\mathbf{Q}$ of $\mathbf{G}$ (i.e.,
not necessarily standard). Indeed, as any parabolic $\Q$-subgroup
is conjugate to a unique standard one (or equivalently, we can
change $\mathbf{S}$ and $\Phi^+(\mathbf{G},\mathbf{S})$ to make
$\mathbf{Q}$ standard), we can define $\Delta_{\mathbf{Q},h}$,
etc. We get in this way a decomposition
\begin{equation}
\label{refined-langlands}
\mathbf{L_Q}=\mathbf{S_Q}\widetilde{\mathbf{M}}_{\mathbf{Q},\ell}\,\mathbf{M}_{
\mathbf{Q},h}
\end{equation}
(compare with \eqref{Langlands-decomp}).

Given a maximal parabolic $\Q$-subgroup $\mathbf{Q}\subset \mathbf{G}$,
we have the {\it rational boundary component}
\begin{equation}\label{boundary-comp}
e_h(\mathbf{Q}):=\widetilde{M}_{Q,\ell}\backslash \widehat{e}(\mathbf{Q})
\end{equation}
sitting in the boundary of $D$ in its embedding as a bounded
symmetric domain (see \cite[p.~170]{toroidal-comp}). It is isomorphic to
the hermitian symmetric space
$M_{Q,h}/(M_{Q,h}\cap K)$. We let
$$
\overline{D}^{bb}= D\sqcup \left( \bigsqcup_{\mathbf{Q} \; \text{maximal}}
{e_h}(\mathbf{Q})\right).
$$
Suitably topologized, $\overline{D}^{bb}$ is a stratified space,
with $e_h(\mathbf{Q'})$ in the closure of $e_h(\mathbf{Q})$, i.e.,
${e_h}(\mathbf{Q'})\preceq{e_h}(\mathbf{Q})$, if and only if
$\mathbf{Q'}\preceq\mathbf{Q}$\,; the latter is defined to mean
that, $\mathbf{Q'}$ and $\mathbf{Q}$ can be made simultaneously
standard of respective cotypes $\{\beta_{i'}\}$ and $\{\beta_i\}$ with
$i\leq i'$.
(We also write $\mathbf{Q}'\prec \mathbf{Q}$ if $i<i'$.)
The quotient by $\Gamma$,
$$
\overline{\Gamma\backslash D}^{bb}=\Gamma\backslash \overline{D}^{bb},
$$
is the {\it Baily-Borel Satake compactification} of $\Gamma\backslash D$.

It is shown in \cite{baily-borel}
that, in effect, $\overline{\Gamma\backslash D}^{bb}$ is the analytic
variety of
$\C$-points of a
normal $\C$-scheme $\overline{X}^{bb}$; in fact, $\overline{X}^{bb}$ can be
defined
over a number field.
The boundary $\partial \overline{X}^{bb}=\overline{X}^{bb}- X$ is
naturally stratified
with each stratum written as
$X^{bb}_{\mathbf{Q}}$,
with $\mathbf{Q}$ running over the finite set of
$\Gamma$-conjugacy
classes of maximal parabolic $\Q$-subgroups.
More precisely,
\begin{equation}\label{bb-strat}
\overline{X}^{bb}=X\sqcup\left(\bigsqcup_{\mathbf{Q}\text{ max'l, mod }\Gamma}
X^{bb}_\mathbf{Q}\right)\!,
\end{equation}
where $X^{bb}_{\mathbf{Q}}$ is the $\C$-scheme such that
$$X^{bb}_{\mathbf{Q}}(\C)=\Gamma(\mathbf{M}_{\mathbf{Q},h})\backslash e_h(
\mathbf{Q}).$$
In the above, $\Gamma(\mathbf{M}_{\mathbf{Q},h})=(\Gamma\cap Q)/(\Gamma \cap
N_QA_Q\widetilde{M}_{Q,\ell})$. As $\Gamma$ is neat, this arithmetic subgroup
coincides with
$\Gamma(\mathbf{Q}/\mathbf{N}_{\mathbf{Q}}\mathbf{S}_{\mathbf{Q}}\widetilde{
\mathbf{M}}_{\mathbf{Q},\ell})$.

The construction is hereditary, in that the normalization of the closure
$\overline{X}^{bb}_{ \mathbf{Q}}$ of the stratum
$X^{bb}_\mathbf{Q}$ in $\overline{X}^{bb}$ can be identified with
the Baily-Borel Satake compactification of $X^{bb}_\mathbf{Q}$.
Thus, there is a finite and surjective morphism
$$\overline{(X^{bb}_{\mathbf{Q}})}^{bb} \to \overline{X}^{bb}_{ \mathbf{Q}}$$
which is an isomorphism over $X^{bb}_{\mathbf{Q}}$.

Citing \cite[\S3.11]{satake-compactifications} or \cite[\S2]{GT},
we assert:

\begin{proposition}\label{bs-bb} There is a commutative diagram
\begin{equation}\label{diagram-bs-bb}
\xymatrix @C=2.7pc@R=1.5pc
{& \overline{\Gamma \backslash D}^{bs} \ar[d]^-{q} \\
 & \overline{\Gamma \backslash D}^{rbs}\ar[d]^-{p}\\
\Gamma \backslash D \ar[r]^-{j^{bb}} \ar@/^1pc/[uur]^-{j^{bs}\!\!\!\!}
\ar@/^.2pc/[ur]^-{j^{rbs}\!\!\!\!\!\!\!} & \overline{\Gamma \backslash D}^{bb}}
\end{equation}
where $p$ and $q$ are morphisms of compactifications of $\Gamma
\backslash D$.
\end{proposition}

As was the case with the Borel-Serre compactifications, if $\Gamma'\subset
\mathbf{G}(\Q)$ is another neat arithmetic subgroup and
$g\in \mathbf{G}(\Q)$ such that $g\Gamma' g^{-1}\subset \Gamma$, the induced
morphism $g:\Gamma'\backslash D \to \Gamma \backslash D$ extends to the
Baily-Borel Satake compactifications, yielding a morphism of analytic spaces
\begin{equation}
\label{eq:induced-g-on-baily-borel-compact}
g^{bb}:\overline{\Gamma' \backslash D}^{bb}\to\overline{\Gamma\backslash D}^{bb}.
\end{equation}
As both analytic spaces are projective, we deduce a morphism of $\C$-schemes
$g^{bb}:\overline{X'}^{bb} \to \overline{X}^{bb}$ which is finite and surjective.
In fact, this morphism is defined over a number
field.

\subsection{The toroidal compactifications}
\label{subsect:toroidal-compact}
{\it The main reference for the material in this subsection is
\cite{toroidal-comp}.}

We start with a quick summary of the outcome of the construction. There
are usually infinitely many toroidal compactifications
$\overline{X}^{tor}_{\Sigma}$ of $X$, depending on the
choice of some combinatorial data denoted $\Sigma$.  They are
algebraic varieties constructed over $\overline{X}^{bb}$, so that
there is a morphism
$\overline{X}^{tor}_{\Sigma}\to\overline{X}^{bb}$, which is a morphism
of compactifications of $X$. For suitable
choices of $\Sigma$ (again, infinitely many), one has that
$\overline{X}^{tor}_{\Sigma}$ is smooth and projective, and the
boundary $\partial\overline{X}^{tor}_{\Sigma}$ is a divisor with
simple normal crossings.

We specify some of the details.
Let $\mathbf{Q}$ be a maximal parabolic $\Q$-subgroup. Denote by
$\mathbf{U_Q}$ the center of the unipotent radical $\mathbf{N_Q}$
and let $\mathbf{V_Q}=\mathbf{N_Q}/\mathbf{U_Q}$. Then
$\mathbf{U_Q}$ and $\mathbf{V_Q}$ are vector group-schemes (i.e., isomorphic to
direct products of copies of the additive group $\mathbb{G}_a$). The action of $\mathbf{L_Q}$ on $\mathbf{U_{Q}}$ factors through $\mathbf{L_Q}/\mathbf{M}_{\mathbf{Q},h}$. The latter is isomorphic to
the quotient of $\mathbf{S_Q}\widetilde{\mathbf{M}}_{\mathbf{Q},\ell}$ by the finite normal subgroup $\mathbf{S_Q}\widetilde{\mathbf{M}}_{\mathbf{Q},\ell}\cap \mathbf{M}_{\mathbf{Q},h}$.

There is a homogeneous, self-adjoint cone (with vertex removed)
$C_Q\subset U_Q$, invariant under the
action of $A_{Q}\widetilde M_{Q,\ell}$, with the geodesic action of $A_Q$
giving
the cone dilations; it arises in the realization of $D$ as a
Siegel domain with respect to $\mathbf{Q}$
(see \cite[pp.~235--236]{toroidal-comp}).
Denote by $\overline{C}_Q$
the union of $C_Q$ and its rational boundary components, equipped with
the Satake topology (see \cite[pp.~81]{toroidal-comp}).

Let $\mathbf{Q_1}$ and $\mathbf{Q_2}$ be two standard maximal
parabolic $\Q$-subgroups. Then $\mathbf{Q_1}\succeq \mathbf{Q_2}$ if and only if
$\mathbf{M}_{\mathbf{Q_1},\ell}\subseteq\mathbf{M}_{\mathbf{Q_2},
\ell}$ (or equivalently
$\mathbf{M}_{\mathbf{Q_1},h}\supseteq\mathbf{M}_{\mathbf{Q_2},h}$).\footnote{These are inclusions of subquotients of $\mathbf{G}$.}
In that case, $\mathbf{U_{Q_1}}
\subseteq\mathbf{U_{Q_2}}$ and the inclusion is $\widetilde
{\mathbf{M}}_{\mathbf{Q_1},\ell}\,$-equivariant. However, what is relevant is the
embedding, for $\mathbf{Q}_1\succ \mathbf{Q}_2$, of $C_{Q_1}$ in
$\overline{C}_{Q_2}$ as a {\it rational boundary component},
analogous to what we had for the $e_h(\mathbf{Q})$'s in
$\overline{D}^{bb}$ in
\S\ref{subsect:The-Baily-Borel-compactification}.

Given a parabolic $\Q$-subgroup $\mathbf{P}\subset \mathbf{G}$
(not necessarily maximal), we put
$\Gamma(\widetilde{\mathbf{M}}_{\mathbf{P},\ell})=(\Gamma \cap
P)/(\Gamma \cap N_PA_PM_{Q,h})$. As $\Gamma$ is neat, this
coincides with
$\Gamma(\mathbf{P}/\mathbf{N_P}\mathbf{S_P}\mathbf{M}_{\mathbf{P},h})$.
The arithmetic subgroup
$\Gamma(\widetilde{\mathbf{M}}_{\mathbf{Q},\ell})$ acts on $U_Q$.

\begin{definition}
\label{defn:compatible-family-prpcd-s}
A \emph{compatible family of
partial rational polyhedral cone decompositions} (with respect to
$\Gamma$) $\Sigma=\{\Sigma_{\mathbf{Q}}\}$ is a family of rational
polyhedral cone decompositions \emph{(prpcd's)}
$\Sigma_{\mathbf{Q}}$ of $\overline C_Q$,\footnote{If one means
{\it closed} cones, that displays the face relations.  We will
mean throughout their interiors, obtaining a stratification of
$\overline C_Q$ and thus a decomposition in the literal sense.}
one for each maximal parabolic $\Q$-subgroup $\mathbf{Q}$, such
that the following conditions are satisfied.
\begin{enumerate}

\item $\Sigma_{\mathbf{Q}}$ is equivariant with respect to the action of
$\Gamma(\widetilde{\mathbf{M}}_{\mathbf{Q},\ell})$, and there are finitely
many equivalence classes of rational polyhedral cones modulo this action.

\item For $\gamma\in \Gamma$, the isomorphism $\overline C_Q\simeq \overline C_{\gamma Q\gamma^{-1}}$
induced by ${\rm int}(\gamma):C_Q \overset{\sim}{\to} C_{\gamma
Q\gamma^{-1}}$ sends a rational polyhedral cone in
$\Sigma_{\mathbf{Q}}$ to a rational polyhedral cone in
$\Sigma_{\gamma \mathbf{Q}\gamma^{-1}}$.

\item if $\mathbf{Q_1}\succeq\mathbf{Q_2}$,
then $\Sigma_{\mathbf{Q_1}}$ is the trace of
$\Sigma_{\mathbf{Q_2}}$ with respect to the inclusion of
$\overline C_{Q_1}$ in $\overline C_{Q_2}$, i.e., $\Sigma_{\mathbf{Q}_1}$ is
the subset of polyhedral cones in $\Sigma_{\mathbf{Q}_2}$ that are contained
in $\overline{C}_{Q_1}$.

\end{enumerate}

\end{definition}

By \cite{toroidal-comp}, such decompositions always exist.
Fix a compatible family of {\it prpcd}'s $\Sigma=\{\Sigma_{\mathbf{Q}}\}$. One gets
for each $\mathbf{Q}$, from
the corresponding Siegel domain picture of $D$, a tower of schemes
\begin{equation}
\label{tower}
\xymatrix@C=1.7pc{ \mathcal{S}_{\mathbf{Q}} \ar[r] &  \mathcal{A}_{\mathbf{Q}}
\ar[r] & \widetilde{X}^{bb}_{\mathbf{Q}},}
\end{equation}
associated to the tower of algebraic groups
$$
\xymatrix@C=1.7pc{ \mathbf{M}_{\mathbf{Q},h}\mathbf{N_Q} \ar[r] & \mathbf{M}_{
\mathbf{Q},h}\mathbf{N_Q/U_Q}
\ar[r] &  \mathbf{M}_{\mathbf{Q},h}.}
$$
In \eqref{tower}, $\widetilde{X}^{bb}_{\mathbf{Q}}$ is a Galois \'etale cover of $X^{bb}_{\mathbf{Q}}$. It corresponds to the locally symmetric variety $(\Gamma(\mathbf{M_Q})\cap M_{Q,h})\backslash e_h(\mathbf{Q})$. Hence, the group of automorphisms
of $\widetilde{X}^{bb}_{\mathbf{Q}}\to X^{bb}_{\mathbf{Q}}$
is given by the finite quotient
\begin{equation}
\label{eq:galois-group-tilde-X-bb}
\frac{\Gamma(\mathbf{M}_{\mathbf{Q},h})}{\Gamma(\mathbf{M_Q})\cap M_{Q,h}}\simeq \frac{\Gamma(\mathbf{M_Q})}{(\Gamma(\mathbf{M_Q}) \cap M_{Q,h})(\Gamma(\mathbf{M_Q}) \cap \widetilde{M}_{Q,\ell})}\simeq \frac{\Gamma(\widetilde{\mathbf{M}}_{\mathbf{Q},\ell})}{\Gamma(\mathbf{M_Q})\cap \widetilde{M}_{Q,\ell}}.
\end{equation}
Moreover, $\mathcal{A}_{\mathbf{Q}}$ is an abelian scheme over
$\widetilde{X}^{bb}_{
\mathbf{Q}}$ and $\mathcal{S}_{\mathbf{Q}}\to \mathcal{A}_{\mathbf{Q}}$
is a $\mathbf{T_Q}$-torsor,
where $\mathbf{T_Q}=(\Gamma(\mathbf{U_Q})\otimes \mathbb{G}_m)$, which is a split $\Q$-torus.
Furthermore, the arithmetic subgroup
$\Gamma(\widetilde{\mathbf{M}}_{\mathbf{Q},\ell})$ acts on $\widetilde{X}^{bb}_{\mathbf{Q}}$,
$\mathbf{T_Q}$ and $\mathcal{S}_{\mathbf{Q}}$
and the morphisms in
\eqref{tower} are compatible with these actions. Also note that
$\Gamma(\widetilde{\mathbf{M}}_{\mathbf{Q},\ell})$ acts on $\widetilde{X}^{bb}_{\mathbf{Q}}$ by its quotient $\Gamma(\widetilde{\mathbf{M}}_{\mathbf{Q},\ell})/\Gamma(\mathbf{M}_{\mathbf{Q}}) \cap \widetilde{M}_{Q,\ell}$ via the isomorphisms
\eqref{eq:galois-group-tilde-X-bb}. In particular, it permutes transitively the fibers of the \'etale cover $\widetilde{X}^{bb}_{\mathbf{Q}}\to X^{bb}_{\mathbf{Q}}$.

Let $\mathbf{T}_{\mathbf{Q},\Sigma}$ be the
$\Gamma(\widetilde{\mathbf{M}}_{\mathbf{Q},\ell})$-equivariant torus
embedding associated to the {\it prpcd}
$\Sigma_{\mathbf{Q}}$, the rational polyhedral cones in $\Sigma_{\mathbf{Q}}$ corresponding
to $\mathbf{T_Q}$-orbits, and put
\begin{equation}\label{SQ}
{\mathcal{S}}_{\mathbf{Q},\Sigma} =
\mathcal{S}_{
\mathbf{Q}}\times^{\mathbf{T_Q}}\mathbf{T}_{\mathbf{Q},\Sigma}\quad
\text{ and }\quad
\mathcal{B}_{
\mathbf{Q},\Sigma}=\partial{\mathcal{S}}_{\mathbf{Q},
\Sigma}=\mathcal{S}_{\mathbf{Q},\Sigma}-
\mathcal{S}_{\mathbf{Q}}.
\end{equation}

Using reduction theory, one sees that the
${\mathcal{S}}_{\mathbf{Q},\Sigma}$'s can be used to define the
boundary for the compactification $\overline{X}^{tor}_{\Sigma}$,
the {\it toroidal compactification} of $X$ constructed from
$\Sigma$ \cite{toroidal-comp}. One calls $\Sigma$ {\it projective}
(resp.~{\it smooth}), when $\overline{X}^{tor}_{\Sigma}$ is
projective (resp.~smooth). Again by \cite{toroidal-comp}, smooth
projective $\Sigma$ always exist. For a smooth $\Sigma$, the
rational polyhedral cones in the decompositions must be generated
by a subset of a $\Z$-basis of $\Gamma(\mathbf{U}_{\mathbf{Q}})$.
We also say that $\Sigma$ is \emph{simplicial} if the rational
polyhedral cones in the decompositions are simplicial cones, i.e.,
generated by a subset of a basis of the $\R$-vector space $U_Q$.
When $\Sigma$ is simplicial, the toroidal compactification
$\overline{X}^{tor}_{\Sigma}$ has only quotient
singularities. From the construction:

\begin{theorem}
\label{morph-comp}
There is a commutative triangle
$$\xymatrix@C=1.5pc@R=1.5pc{X \ar[r] \ar@/_/[dr] & \overline{X}^{tor}_{\Sigma} \ar[d]^-e \\
& \overline{X}^{bb}}$$
with $e$ a morphism of compactifications of $X$. For a cofinal subset of compatible families of {\it prpcd}'s $\Sigma$,
$\overline{X}^{tor}_{\Sigma}$ is a
smooth and projective compactification of $X$, with a simple normal crossing divisor at infinity.
\end{theorem}

Let $\mathcal{B}^{\circ}_{\mathbf{Q},\Sigma}$
be the complement in
$\mathcal{B}_{\mathbf{Q},\Sigma}$ of the divisors
that correspond to rays in $\Sigma_{\mathbf{Q}}$ which are contained in the boundary of $\overline{C}_{Q}$. Let also
$\mathcal{B}^{c}_{\mathbf{Q},\Sigma}$ be the closure of
$\mathcal{B}^{\circ}_{\mathbf{Q},\Sigma}$ in
$\mathcal{B}_{\mathbf{Q},\Sigma}$. The group
$\Gamma(\widetilde{\mathbf{M}}_{\mathbf{Q},\ell})$
acts on the $\C$-schemes
$\mathcal{B}^{\circ}_{\mathbf{Q},\Sigma}$ and
$\mathcal{B}^{c}_{\mathbf{Q},\Sigma}$.
The next proposition describes, in effect, the fibers of $e$ in
Theorem \ref{morph-comp}.
Again from the construction:

\begin{proposition}
\label{fiber}
For $\mathbf{Q}\subset \mathbf{G}$
a maximal parabolic $\Q$-subgroup, the base-change of
$e:\overline{X}^{tor}_{\Sigma} \to \overline{X}^{bb}$
with respect to the inclusion
$X^{bb}_{\mathbf{Q}} \hookrightarrow \overline{X}^{bb}$ is isomorphic to
$$\xymatrix@C=1.7pc{\Gamma(\widetilde{\mathbf{M}}_{\mathbf{Q},\ell})\backslash
\mathcal{B}^{c}_{
\mathbf{Q},\Sigma}\ar[r] & X^{bb}_{\mathbf{Q}}.}$$
\end{proposition}

For evident reasons, the schemes $\Gamma(\widetilde{\mathbf{M}}_{\mathbf{Q},
\ell})\backslash
\mathcal{B}^{c}_{\mathbf{Q},\Sigma}$, with
$\mathbf{Q}$
maximal, have been called
the Baily-Borel-type ``strata'' of $\partial\overline{X}^{tor}_{\Sigma}$
(though $\mathcal{B}^{c}_{\mathbf{Q},\Sigma}$ generally has crossings).
They admit further refinement, which we now describe.

Let $\mathbf{R}\subset \mathbf{G}$ be a proper parabolic $\Q$-subgroup (not necessarily maximal). Let $\mathbf{Q}$ be the maximal parabolic $\Q$-subgroup containing $\mathbf{R}$ and such that $\mathbf{M}_{\mathbf{Q},h}\simeq \mathbf{M}_{\mathbf{R},h}$.
(For this reason,
one says that $\mathbf{R}$ is {\it
subordinate to} $\mathbf{Q}$, as
in \cite[\S1]{HZ2}.)
Let $\Sigma^{\circ}_{\mathbf{R}}\subset \Sigma_{\mathbf{Q}}$ be the subset of rational polyhedral cones $\sigma$ satisfying the following two conditions:
\begin{enumerate}

\item every extremal ray of $\sigma$ is contained in $C_P$ with $\mathbf{P}$ one of the maximal parabolic $\Q$-subgroups that contain $\mathbf{R}$,

\item for every maximal parabolic $\Q$-subgroup $\mathbf{P}$ containing $\mathbf{R}$, there is at least one extremal ray of
$\sigma$ contained in $C_P$.
\end{enumerate}
Let also $\Sigma^c_{\mathbf{R}}\subset \Sigma_{\mathbf{Q}}$ be the
subset of rational polyhedral cones containing an element of
$\Sigma^{\circ}_{\mathbf{R}}$ in their closure.
Denote by $\mathcal{B}^{\circ}_{\mathbf{R},\Sigma}$ the locally closed subscheme of $\mathcal{S}_{\mathbf{Q},\Sigma}$ that is the union of the strata corresponding to rational polyhedral cones in $\Sigma_{\mathbf{R}}^{\circ}$. Also denote by
$\mathcal{B}^{c}_{\mathbf{R},\Sigma}$ the closed subscheme of $\mathcal{S}_{\mathbf{Q},\Sigma}$ which is the union of the strata corresponding to rational polyhedral cones in $\Sigma_{\mathbf{R}}^{c}$. Clearly, $\mathcal{B}^c_{\mathbf{R},\Sigma}$ is the closure of
$\mathcal{B}^{\circ}_{\mathbf{R},\Sigma}$
in $\mathcal{S}_{\mathbf{Q},\Sigma}$.
(When $\mathbf{R}$ is itself maximal, this agrees with what was defined above).
When $\Sigma_{\mathbf{Q}}$ is fine enough, $\Sigma^c_{\mathbf{R}}$
is the union of $\Sigma^{\circ}_{\mathbf{R'}}$ where $\mathbf{R'}$ runs overs the parabolic $\Q$-subgroups of $\mathbf{G}$ contained in
$\mathbf{R}$ and subordinate to $\mathbf{Q}$.
In this case, we have, as sets,
\begin{equation}
\label{corner-like-strata}
\mathcal{B}^c_{\mathbf{R},\Sigma}\quad = \bigsqcup_{\mathbf{R'} \subseteq \mathbf{R} \; \text{with} \; \mathbf{M}_{\mathbf{R},h}\simeq \mathbf{M}_{\mathbf{R}',h}} \mathcal{B}^{\circ}_{\mathbf{R'},\Sigma}\,.
\end{equation}

\begin{proposition}
\label{normalizers}
Let $\mathbf{Q}\subset \mathbf{G}$ be a maximal parabolic $\Q$-subgroup and $\mathbf{R}\subset \mathbf{Q}$ a parabolic $\Q$-subgroup of $\mathbf{G}$ which is subordinate to $\mathbf{Q}$.

\begin{enumerate}

\item[(i)]
For $\gamma\in\Gamma(\widetilde{\mathbf{M}}_{\mathbf{Q},\ell})$,
we have $\gamma\cdot\mathcal{B}^\circ_{\mathbf{R},\Sigma} = \mathcal{B}^\circ_{\gamma\mathbf{R}
\gamma^{-1},\Sigma}$ and $\gamma\cdot \mathcal{B}^c_{\mathbf{R},\Sigma} = \mathcal{B}^c_{
\gamma\mathbf{R}\gamma^{-1},\Sigma}$.

\item[(ii)]
The stabilizers of $\mathcal{B}^{\circ}_{\mathbf{R},\Sigma}$ and of $\mathcal{B}^c_{\mathbf{R},\Sigma}$
in $\Gamma(\widetilde{\mathbf{M}}_{\mathbf{Q},\ell})$ are given by the same arithmetic group
$\Gamma(\widetilde{\mathbf{M}}_{\mathbf{Q},\ell}\,|\,\mathbf{R})$ in {\rm (\ref{eq:useful-arthmetic-subgr})} below.

\item[(iii)] When $\Sigma_{\mathbf Q}$ is sufficiently fine,
$\mathcal{B}^c_{\mathbf{R},\Sigma}\cap \mathcal{B}^c_{\gamma\mathbf{R}
\gamma^{-1},\Sigma}=\emptyset$ for
$\gamma\in\Gamma(\widetilde{\mathbf{M}}_{\mathbf{Q},\ell})$ not
in the stabilizer of $\mathcal{B}^{c}_{\mathbf{R},\Sigma}$.
\end{enumerate}
\end{proposition}
\begin{proof}
For any parabolic $\Q$-subgroup $\mathbf{R}$ of $\mathbf{G}$ that
is subordinate to $\mathbf{Q}$, we
denote by $\mathbf{\langle\,\widetilde
M}_{\mathbf{Q},\ell}\,|\,\mathbf{R\,\rangle}$ the
image of $\mathbf{R}$ by
the projection of $\mathbf{Q}$ onto (a quotient by a finite group
of) $\mathbf{\widetilde M_{Q,\ell}}$ or, equivalently, the
intersection of $\mathbf{\widetilde M_{Q,\ell}}$ with the image of
$\mathbf{R}$ in $\mathbf {M_Q}$. This is a parabolic
$\Q$-subgroup of $\mathbf{\widetilde M_{Q,\ell}}$ (whose Lie group
of real points was denoted $G_{\ell,R}$ in
\cite[(2.2.12)]{rbs-2}). We then set
\begin{equation}
\label{eq:useful-arthmetic-subgr}
\Gamma(\widetilde{\mathbf{M}}_{\mathbf{Q},\ell}\,|\,\mathbf{R})=
\Gamma(\widetilde{\mathbf{M}}_{\mathbf{Q},\ell})\cap (R/N_QA_QM_{Q,h}),\footnote{
By $N_QA_QM_{Q,h}$ we really mean $N_QA_QM_{Q,h}(x)$, where $M_{Q,h}(x)\subset Q$
is the lift of $M_{Q,h}$ induced by the lift $\mathbf{L_Q}(x)\subset \mathbf{Q}$.
Note that $N_QA_QM_{Q,h}(x)$ does not depends on the choice of $\mathbf{L_Q}(x)$, which justifies the abuse of notation.}
\end{equation}
an arithmetic subgroup of $\mathbf{\langle\,\widetilde
M}_{\mathbf{Q},\ell}\,|\,\mathbf{R\,\rangle}$.\,\, The three
statements above follow directly from \eqref{corner-like-strata}.
\end{proof}

\begin{proposition}
\label{prop:intersection-closure-strata-toroidal-comp}
Assume that $\Sigma=(\Sigma_{\mathbf{Q}})$ is fine enough.
Let $\mathbf{Q}_1\succ \mathbf{Q}_2\succ \dots \succ \mathbf{Q}_s$ be maximal parabolic $\Q$-subgroups of $\mathbf{G}$.
Let $E$ be the set of parabolic $\Q$-subgroups that can be written as
$\bigcap_{i=1}^s \gamma_i \mathbf{Q}_i \gamma_i^{-1}$
for some $s$-tuple $(\gamma_1,\dots, \gamma_s)\in \Gamma^s$.
Then, the locally closed subscheme of $\overline{X}^{tor}_{\Sigma}$ given by
$$\overline{e^{-1}(X^{bb}_{\mathbf{Q}_1})}\bigcap \dots \bigcap
\overline{e^{-1}(X^{bb}_{\mathbf{Q}_{s-1}})} \bigcap e^{-1}(X^{bb}_{\mathbf{Q}_s})$$
corresponds via the isomorphism $e^{-1}(X^{bb}_{\mathbf{Q}_s})\simeq
\Gamma(\widetilde{\mathbf{M}}_{\mathbf{Q}_s,\ell})\backslash \mathcal{B}^c_{\mathbf{Q}_s,\Sigma}$ to
$$\Gamma(\widetilde{\mathbf{M}}_{\mathbf{Q}_s,\ell})\backslash \bigsqcup_{\mathbf{R}\in E} \mathcal{B}^c_{\mathbf{R},\Sigma}.$$

\end{proposition}

The subset $X^{tor}_{\mathbf{R},\Sigma}=\Gamma(\widetilde{\mathbf{M}}_{\mathbf{Q},\ell}\,|\,\mathbf{R})\backslash \mathcal{B}^{\circ}_{
\mathbf{R}}$
of $\overline{X}^{tor}_{\Sigma}$ has been called the corner-like
``$\mathbf{R}$-stratum'' of $\partial\overline{X}^{tor}_{\Sigma}$ (though it, too,
generally has crossings).  It
is defined for all parabolic $\Q$-subgroups $\mathbf{R}$.

As was the case with the other compactifications, the toroidal compactifications are functorial with respect to the action of $\mathbf{G}(\Q)$. Let $\Gamma'\subset \mathbf{G}(\Q)$ be another neat arithmetic subgroup and $g\in \mathbf{G}(\Q)$ such that $g\Gamma'g^{-1}\subset \Gamma$. Given a compatible family of {\it prpcd}'s $\Sigma=\{\Sigma_{\mathbf{Q}}\}$ (with respect to $\Gamma$), we can find a compatible family of {\it prpcd}'s $\Sigma'=\{\Sigma'_{\mathbf{Q}}\}$ (with respect to $\Gamma'$) such that for every
maximal parabolic $\Q$-subgroup $\mathbf{Q}\subset \mathbf{G}$,
the isomorphism ${\rm int}(g):U_Q \to U_{gQg^{-1}}$ sends a rational polyhedral cone of
$\Sigma'_{\mathbf{Q}}$ inside a rational polyhedral cone of $\Sigma_{g\mathbf{Q}g^{-1}}$. If this is the case, the morphism $g:X' \to X$ extends to the toroidal compactifications, yielding
$$g^{tor}:\overline{(X')}^{tor}_{\Sigma'} \to \overline{X}^{tor}_{\Sigma}.$$
This morphism maps the $\mathbf{R}$-stratum $(X')^{tor}_{\mathbf{R},\Sigma'}$ onto the $\mathbf{R}$-stratum
$X^{tor}_{\mathbf{R},\Sigma}$.

\subsection{The hereditary property of toroidal boundary strata}
\label{subsection:hereditary-toroidal}
The hereditary property of the strata of the toroidal compactification is
properly done using the notions of mixed Shimura data and mixed Shimura
varieties \cite{pink-thesis}. Roughly speaking,
given a maximal parabolic $\Q$-subgroup $\mathbf{Q} \subset \mathbf{G}$, it is
possible to relate the closure of $X^{tor}_{\mathbf{Q},\Sigma}$ in
$\overline{X}^{tor}_{\Sigma}$ with the toroidal compactification of the mixed
Shimura variety associated to the non-reductive $\Q$-group
$\mathbf{M}_{\mathbf{Q},h}\mathbf{N}_{\mathbf{Q}}$, a subgroup of $\mathbf{G}$.
However, for our purposes we need only a weaker statement that
does not invoke mixed Shimura varieties at all. We begin with a definition:

\begin{definition}
\label{defn:extended-fam-prpcd-s}
Let $\mathcal{M}$ be the set of pairs $(\mathbf{Q},\mathbf{R})$
where:
\begin{itemize}

\item $\mathbf{Q}\subseteq \mathbf{G}$ is a parabolic $\Q$-subgroup
which is maximal or improper,

\item $\mathbf{R} \subset \mathbf{M}_{\mathbf{Q},h}$ is a maximal parabolic $\Q$-subgroup.

\end{itemize}

\noindent An \emph{extended compatible family of partial rational polyhedral cone decompositions} (with respect to $\Gamma$) $\Sigma=\{\Sigma_{\mathbf{Q},\mathbf{R}}\}_{(\mathbf{Q},\mathbf{R})\in \mathcal{M}}$ is a family of prpcd's $\Sigma_{\mathbf{Q},\mathbf{R}}$ of $\overline{C}_{R}\subset U_R$ such that the following conditions are satisfied.
\begin{enumerate}

\item[(i)] For $\gamma\in \Gamma$ and $(\mathbf{Q},\mathbf{R})\in \mathcal{M}$, the isomorphism ${\rm int}(\gamma):U_R \overset{\sim}{\to} U_{\gamma R\gamma^{-1}}$ sends a rational polyhedral cone of $\Sigma_{\mathbf{Q},\mathbf{R}}$ to a rational polyhedral cone of $\Sigma_{\gamma\mathbf{Q}\gamma^{-1},\gamma\mathbf{R}\gamma^{-1}}$.

\item[(ii)] For $\mathbf{Q}\subseteq \mathbf{G}$ a parabolic $\Q$-subgroup which is maximal or improper,
the family $\Sigma_{(\mathbf{Q})}=\{\Sigma_{\mathbf{Q},\mathbf{R}}\}_{\mathbf{R}}$ is a compatible family of prpcd's with respect to $\Gamma(\mathbf{M}_{\mathbf{Q},h})$ for the $\Q$-group $\mathbf{M}_{\mathbf{Q},h}$ (in the sense of \emph{Definition \ref{defn:compatible-family-prpcd-s}}).

\item[(iii)] Let $(\mathbf{Q}_1,\mathbf{R}_1)$ and $(\mathbf{Q}_2,\mathbf{R}_2)$ be two elements of $\mathcal{M}$ such that
$\mathbf{Q_1}\succeq \mathbf{Q_2}$
and $\mathbf{R}_2=\mathbf{M}_{\mathbf{Q_2},h}\cap \mathbf{R}_1$.
Then the image of a rational polyhedral cone of $\Sigma_{\mathbf{Q_1},\mathbf{R_1}}$ by the natural map\footnote{There is indeed a natural morphism of $\Q$-groups $\mathbf{N_{R_1}} \to \mathbf{N_{R_2}}$ that induces $U_{R_1} \to U_{R_2}$. It is defined as follows. Let $\mathbf{P}$ be the image of $\mathbf{Q_1}\cap \mathbf{Q_2}$ by the projection of $\mathbf{Q_1}$ to (a finite quotient of) $\mathbf{M}_{\mathbf{Q_1},h}$. As $\mathbf{M}_{\mathbf{P},h}\simeq \mathbf{M}_{\mathbf{Q_2},h}$, there is a canonical projection
$\mathbf{P}\to \mathbf{M}_{\mathbf{Q_2},h}$ that maps $\mathbf{P}\cap \mathbf{R_1}$ onto $\mathbf{R_2}$. This gives a natural morphism
$\mathbf{N_{P\cap R_1}}\to \mathbf{N_{R_2}}$. On the other hand, the inclusion of parabolic subgroups $\mathbf{P}\cap \mathbf{R_1}\subset \mathbf{R_1}$ gives the inclusion of nilpotent radicals
$\mathbf{N_{R_1}} \subset \mathbf{N_{P\cap R_1}}$. Our morphism is then the composition $\mathbf{N_{R_1}} \hookrightarrow \mathbf{N_{P\cap R_1}} \to \mathbf{N_{R_2}}$.}
$U_{R_1} \to U_{R_2}$ is contained in a rational polyhedral cone of
$\Sigma_{\mathbf{Q_2},\mathbf{R_2}}$.

\end{enumerate}

\end{definition}

\
We say that an extended compatible family of {\it prpcd}'s $\Sigma=\{\Sigma_{
\mathbf{Q},\mathbf{R}}\}_{(\mathbf{Q},\mathbf{R})\in \mathcal{M}}$
is smooth (resp. simplicial, projective) if for every parabolic $\Q$-subgroup
$\mathbf{Q}\subseteq \mathbf{G}$ which is maximal or improper, the compatible
family of {\it prpcd}'s $\Sigma_{(\mathbf{Q})}=\{\Sigma_{\mathbf{Q},\mathbf{R}}\}_{
\mathbf{R}}$ is smooth (resp. simplicial, projective).

\begin{remark}
Given a collection of {\it prpcd}'s $\{\Sigma^0_{\mathbf{Q},\mathbf{R}}\}_{(\mathbf{Q},
\mathbf{R})\in \mathcal{M}}$ satisfying
the conditions (i) and (ii) of Definition
\ref{defn:extended-fam-prpcd-s},
there is a smooth and projective extended compatible family of {\it prpcd}'s $\{\Sigma_{\mathbf{Q},\mathbf{R}}\}_{(\mathbf{Q},\mathbf{R})\in \mathcal{M}}$ such that
$\Sigma_{\mathbf{Q},\mathbf{R}}$ is finer than $\Sigma^0_{\mathbf{Q},\mathbf{R}}$ for all $(\mathbf{Q},\mathbf{R})\in \mathcal{M}$.
\end{remark}

We now fix an
extended compatible family of {\it prpcd}'s
$\Sigma=\{\Sigma_{\mathbf{Q},\mathbf{R}}\}_{(\mathbf{Q},\mathbf{R})\in
\mathcal{M}}$.
For $\mathbf{Q}\subseteq \mathbf{G}$ a parabolic $\Q$-subgroup
which is maximal or improper,
we may consider
$\overline{(X^{bb}_{\mathbf{Q}})}^{tor}_{\Sigma_{(\mathbf{Q})}}$,
the toroidal compactification of the locally symmetric variety $X^{bb}_{\mathbf{Q}}$ associated to the compatible family
of {\it prpcd}'s $\Sigma_{(\mathbf{Q})}=\{\Sigma_{\mathbf{Q},\mathbf{R}}\}_{\mathbf{R}}$. This is a smooth and projective $\C$-scheme that depends only on the conjugacy class of $\mathbf{Q}$ modulo $\Gamma$. Moreover, we have a canonical morphism $e_{\mathbf{Q}}:\overline{(X^{bb}_{\mathbf{Q}})}^{tor}_{\Sigma_{(\mathbf{Q})}} \to \overline{X}^{bb}_{\mathbf{Q}}$.
(When $\mathbf{Q}=\mathbf{G}$, we recover the projection $e$ from
Theorem \ref{morph-comp}.)

Let $\mathbf{R}\subset \mathbf{M}_{\mathbf{Q},h}$ be a proper
parabolic $\Q$-subgroup, and denote by $\mathbf{P}\subset
\mathbf{Q}$ the inverse image of $\mathbf{R}$ by the projection
$\mathbf{Q}\to\mathbf{M}_{ \mathbf{Q},h}$. Let $\mathbf{R}\subset
\mathbf{R}'$ and $\mathbf{P}\subset \mathbf{P'}$ be the maximal
parabolic $\Q$-subgroups of $\mathbf{M}_{\mathbf{Q},h}$ and
$\mathbf{G}$ respectively, such that $\mathbf{R}$ is subordinate
to $\mathbf{R'}$ and $\mathbf{P}$ is subordinate to $\mathbf{P'}$.
Using the construction in \eqref{tower} for
$\mathbf{R'}\subset\mathbf{M}_{ \mathbf{Q},h}$, we have a morphism
of schemes
$\mathcal{S}_{(\mathbf{Q},\mathbf{R}')}\to\mathcal{A}_{(\mathbf{Q},
\mathbf{R}')}$, where $\mathcal{A}_{(\mathbf{Q},\mathbf{R}')}$ is
an abelian scheme over an \'etale cover of $X^{bb}_{\mathbf{P'}}$
whose fibers are made from $V_{R'}$, and
$\mathcal{S}_{(\mathbf{Q},\mathbf{R'})}$ is a torsor over the
torus
$\mathbf{T}_{(\mathbf{Q},\mathbf{R'})}=\Gamma(\mathbf{U}_{\mathbf{R'}})\otimes
\mathbb{G}_m$ with $\Gamma(\mathbf{U_{R'}})=\Gamma(\mathbf{M}_{\mathbf{Q},h})\cap U_{R'}$. We deduce from the {\it prpcd}
$\Sigma_{\mathbf{Q},\mathbf{R'}}$ a torus embedding
$\mathcal{S}_{(\mathbf{Q},\mathbf{R'}),\Sigma_{(\mathbf{Q})}}$
over $\mathcal{A}_{(\mathbf{Q},\mathbf{R'})}$ with boundary
$\mathcal{B}_{(\mathbf{Q},\mathbf{R'}),\Sigma_{(\mathbf{Q})}}$.
The schemes
$\mathcal{B}^{\circ}_{(\mathbf{Q},\mathbf{R}),\Sigma_{(\mathbf{Q})}}$
and
$\mathcal{B}^c_{(\mathbf{Q},\mathbf{R}),\Sigma_{(\mathbf{Q})}}$
are defined as before. We assume that $\Sigma$ is fine enough and
set, also as before,
$(X^{bb}_{\mathbf{Q}})^{tor}_{\mathbf{R},\Sigma_{(\mathbf{Q})}}=
\Gamma(\widetilde{\mathbf{M}}_{\mathbf{R'},\ell}\,|\,\mathbf{R})\backslash
\mathcal{B}^{\circ}_{(\mathbf{Q},\mathbf{R}),\Sigma_{(\mathbf{Q})}}$;
here the arithmetic group
$\Gamma(\widetilde{\mathbf{M}}_{\mathbf{R'},\ell}\,|\,\mathbf{R})$
is defined as in
\eqref{eq:useful-arthmetic-subgr}, but for parabolic subgroups of
$\mathbf{M}_{\mathbf{Q},h}$ instead of $\mathbf{G}$  and its
arithmetic subgroup $\Gamma(\mathbf{M}_{\mathbf{Q},h})$ (instead
of $\Gamma$). This is the $\mathbf{R}$-stratum in the toroidal
compactification of $X^{bb}_{\mathbf{Q}}$ associated to
$\Sigma_{(\mathbf{Q})}$.

Now, let $\mathbf{Q_1}, \, \mathbf{Q_2}\subseteq \mathbf{G}$ be two
parabolic $\Q$-subgroups which are maximal or improper and such that
$\mathbf{Q_1}\succeq \mathbf{Q_2}$ (i.e.,
$\mathbf{M}_{\mathbf{Q_2},h}\subseteq \mathbf{M}_{\mathbf{Q_1},h}$).
For $i=1,\,2$, let $\mathbf{R_i}\subset \mathbf{M}_{\mathbf{Q_i},h}$ be a proper
parabolic
$\Q$-subgroup such that
$\mathbf{R}_2=\mathbf{M}_{\mathbf{Q_2},h}\cap \mathbf{R_1}$.
Also, let $\mathbf{R_i'}\subset \mathbf{M}_{\mathbf{Q_i},h}$ be the
maximal
parabolic $\Q$-subgroup such that
$\mathbf{R_i}$ is subordinate to $\mathbf{R_i'}$; then
$\mathbf{M}_{\mathbf{R'_1},h}\simeq \mathbf{M}_{\mathbf{R'_2},h}$.
Then, there is a commutative square
$$\xymatrix@C=1.5pc@R=1.5pc{\mathcal{S}_{(\mathbf{Q_1},\mathbf{R_1'})} \ar[r] \ar[d] &
\mathcal{S}_{(\mathbf{Q_2},\mathbf{R_2'})} \ar[d] \\
\mathcal{A}_{(\mathbf{\mathbf{Q}_1},\mathbf{R_1'})} \ar[r] &
\mathcal{A}_{(\mathbf{Q_2},\mathbf{R_2'})} ,\!}$$
and condition (iii) of Definition \ref{defn:extended-fam-prpcd-s}
gives an extension
$\mathcal{S}_{(\mathbf{Q_1},\mathbf{R'_1}),\Sigma_{(\mathbf{Q_1})}} \to
\mathcal{S}_{(\mathbf{Q_2},\mathbf{R_2'}),\Sigma_{(\mathbf{Q_2})}}$
of the top horizontal arrow.
This yields morphisms
$$\xymatrix@C=1.7pc{\mathcal{B}^{\circ}_{(\mathbf{Q_1},\mathbf{R_1}),\Sigma_{(
\mathbf{Q_1})}}
\ar@<.2pc>[r] &
\mathcal{B}^{\circ}_{(\mathbf{Q_2},\mathbf{R_2}),\Sigma_{(\mathbf{Q_2})}}}
\quad \text{and} \quad
\xymatrix@C=1.7pc{\mathcal{B}^c_{(\mathbf{Q_1},\mathbf{R_1}),\Sigma_{(\mathbf{Q_1})}}
\ar@<.2pc>[r] &
\mathcal{B}^c_{(\mathbf{Q_2},\mathbf{R_2}),\Sigma_{(\mathbf{Q_2})}}}$$
which are equivariant for the action of
$\Gamma(\widetilde{\mathbf{M}}_{\mathbf{R_1'},\ell}\,|\,\mathbf{R_1})$.
We are now in position to state the weak hereditary property for the
strata of the toroidal compactifications.

\begin{proposition}
\label{prop:hereditary-for-toroidal}
With notation as above, let $\mathbf{P}$ be the image of
$\mathbf{Q_1} \cap \mathbf{Q_2}$ in $\mathbf{M}_{\mathbf{Q_1},h}$.  Then the
morphism $(X_{\mathbf{Q}_1}^{bb})^{
tor}_{\mathbf{P},\Sigma_{(\mathbf{Q_1})} } \to X^{bb}_{\mathbf{Q_2}}$
from the toroidal construction for parabolic $\Q$-subgroups of
$\mathbf{M}_{\mathbf{Q_1},h}$ extends
(uniquely) to a morphism
\begin{equation}
\label{eq:hereditary-for-toroidal}
\xymatrix@C=1.7pc{\overline{(X_{\mathbf{Q}_1}^{bb})^{tor}_{\mathbf{P},\Sigma_{
(\mathbf{Q_1})} }}\ar[r] &
\overline{(X^{bb}_{\mathbf{Q_2}})}^{tor}_{\Sigma_{
(\mathbf{Q_2})}}}
\end{equation}
where the source is the Zariski closure of the $\mathbf{P}$-stratum
$(X_{\mathbf{Q}_1}^{bb})^{tor}_{\mathbf{P},\Sigma_{(\mathbf{Q_1})} }$ in
the toroidal compactification
$\overline{(X^{bb}_{\mathbf{Q_1}})}^{tor}_{\Sigma_{(\mathbf{Q_1})}}$ of
$X^{bb}_{\mathbf{Q_1}}$.

Moreover, this morphism sends the $\mathbf{R_1}$-stratum
$(X^{bb}_{\mathbf{Q_1}})^{tor}_{\mathbf{R_1},\Sigma_{(\mathbf{Q_1})}}$ to the
$\mathbf{R_2}$-stratum
$(X^{bb}_{\mathbf{Q_2}})^{tor}_{\mathbf{R_2},\Sigma_{(\mathbf{Q_2})}}$,
and the restriction
$(X^{bb}_{\mathbf{Q_1}})^{tor}_{\mathbf{R_1},\Sigma_{(\mathbf{Q_1})}} \to
(X^{bb}_{\mathbf{Q_2}})^{tor}_{\mathbf{R_2},\Sigma_{(\mathbf{Q_2})}}$ of
\eqref{eq:hereditary-for-toroidal} is given by
$$\xymatrix@C=1.7pc{\Gamma(\widetilde{\mathbf{M}}_{\mathbf{R_1'},\ell}\,|\,
\mathbf{R_1}) \backslash
\mathcal{B}^{\circ}_{(\mathbf{Q_1},\mathbf{R_1}),\Sigma_{(\mathbf{Q_1})}} \ar[r]
&  \Gamma(\widetilde{\mathbf{M}}_{\mathbf{R_2'},\ell}\,|\,\mathbf{R_2})
\backslash
\mathcal{B}^{\circ}_{(\mathbf{Q_2},\mathbf{R_2}),\Sigma_{(\mathbf{Q_2})}}.}$$
In particular, it takes the stratum corresponding to a rational polyhedral
cone $\sigma \in \Sigma^{\circ}_{\mathbf{Q_1},\mathbf{R_1}}$ to the
stratum corresponding to the rational polyhedral cone of
$\Sigma^{\circ}_{\mathbf{Q_2},\mathbf{R_2}}$ that contains the image of
$\sigma$ by
$U_{R_1'} \to U_{R_2'}$.
\end{proposition}

\begin{proof}
This is a reformulation of part of \cite[Props.~6.25 and 7.9]{pink-thesis}.
\end{proof}

\begin{remark}
\label{rmk:hereditary-for-toroidal}
When $\mathbf{Q_1} = \mathbf{G}$, the above formula simplifies a little.
Writing $\mathbf{Q}$ instead of $\mathbf{Q_2}$, and $\Sigma$ for
$\Sigma_{(\mathbf{G})}$, we get that
$X^{tor}_{\mathbf{Q},\Sigma}\to X^{bb}_{\mathbf{Q}}$ extends
to a morphism
$$\xymatrix@C=1.7pc{\overline{X^{tor}_{\mathbf{Q},\Sigma}}\ar[r] &
\overline{(X^{bb}_{\mathbf{Q}})}^{tor}_{\Sigma_{(\mathbf{Q})}}}$$
from the Zariski closure of the $\mathbf{Q}$-stratum
$X^{tor}_{\mathbf{Q},\Sigma}$ in
the toroidal compactification $\overline{X}^{tor}_{\Sigma}$ to
the toroidal
compactification of the $\mathbf{Q}$-stratum of $\overline{X}^{bb}$.

\end{remark}

\subsection{Toroidal and Borel-Serre compactifications, taken together}
\label{subsection:toroidal-vs-rbs} It is well-known that, in
general, there are no morphisms of compactifications between the
toroidal and the Borel-Serre compactifications of a locally
symmetric variety. Thus, one is led to consider their least common
modification (see \cite[\S1]{rbs-2}), a compactification of
$\Gamma\backslash D$ we denote by
$\widehat{\overline{\Gamma\backslash D}}_{\Sigma}$, defined as the
closure of the diagonal embedding of $\Gamma\backslash D$ in
$\overline{\Gamma\backslash D}^{bs}\times
\overline{X}^{tor}_{\Sigma}(\C)$.  The projections to the
first and second factors yield morphisms of compactifications
$$\xymatrix@C=1.8pc{\overline{\Gamma\backslash D}^{bs} & \widehat{\overline{\Gamma
\backslash D}}_{\Sigma} \ar[r]^-{pr_2} \ar[l]_-{pr_1} & \overline{X}^{tor}_{
\Sigma}(\C).}$$
In this paragraph we gather some easy facts about the natural stratification of
$\widehat{\overline{\Gamma\backslash D}}_{\Sigma}$.

Let $\mathbf{P}$ be a proper parabolic $\Q$-subgroup of $\mathbf{G}$, and
$\mathbf{Q}$ the maximal parabolic $\Q$-subgroup containing $\mathbf{P}$ and
such that $\mathbf{M}_{\mathbf{P},h}\simeq\mathbf{M}_{\mathbf{Q},h}$.
With the notation of
\S \ref{subsect:The Borel-Serre compactifications},
the canonical retraction
$\overline{D}(\mathbf{P})\to e(\mathbf{P})$ induces a continuous mapping
\begin{equation}
\label{eq:compar-rbs-tor-1}
\Gamma(\mathbf{Q}_h)\backslash \overline{D}(\mathbf{P}) \to \widetilde{X}_{\mathbf{Q}}^{bb}(\C)
\end{equation}
which is equivariant for the action of $\Gamma(\mathbf{M}_{\mathbf{Q},\ell}\,|\,\mathbf{P})$; here,
$\mathbf{Q}_h$ denotes the inverse image of
$\mathbf{M}_{\mathbf{Q},h}$ by the projection of $\mathbf{Q}$ to (the quotient by a finite normal subgroup of) $\mathbf{M_Q}$.
On the other hand, we have
\begin{equation}
\label{eq:compar-rbs-tor-2}
\mathcal{S}_{\mathbf{Q},\Sigma}(\C)\to
\widetilde{X}^{bb}_{\mathbf{Q}}(\C)
\end{equation}
which is also equivariant for the action of $\Gamma(\mathbf{M}_{\mathbf{Q},\ell}\,|\,\mathbf{P})$.
Moreover, there is an open neighborhood $\mathcal{N}_{\mathbf{P},\Sigma} \subset \mathcal{S}_{\mathbf{Q},\Sigma}(\C)$ of
$\mathcal{B}^{\circ}_{\mathbf{P},\Sigma}(\C)$
stable under the action
of $\Gamma(\mathbf{M}_{\mathbf{Q},\ell}\,|\,\mathbf{P})$ and such that
the deleted neighborhood $\mathcal{N}^\circ_{\mathbf{P},\Sigma}=\mathcal{N}_{\mathbf{P},\Sigma}-
\mathcal{B}_{\mathbf{Q},\Sigma}(\C)=\mathcal{N}_{\mathbf{P},\Sigma} \cap \mathcal{S}_{\mathbf{Q}}(\C)$
is naturally identified with
an open subset of $\Gamma(\mathbf{Q}_h)\backslash D$, also stable under $\Gamma(\mathbf{M}_{\mathbf{Q},\ell}\,|\,\mathbf{P})$.

We define $\widehat{B}^{\circ}_{\mathbf{P},\Sigma}$
to be the intersection with $(\Gamma(\mathbf{Q}_h)\backslash e(\mathbf{P}))\times \mathcal{B}^{\circ}_{\mathbf{P},\Sigma}(\C)$ of the closure of the diagonal imbedding of $\mathcal{N}^{\circ}_{\mathbf{P},\Sigma}$ in
$(\Gamma(\mathbf{Q}_h)\backslash \overline{D}(\mathbf{P}))\times \mathcal{S}_{\mathbf{Q},\Sigma}(\C)$.
One checks that $\widehat{B}^{\circ}_{\mathbf{P},\Sigma}$ does not depend on the choice of $\mathcal{N}_{\mathbf{P},\Sigma}$.
We have:

\begin{proposition}
\label{prop:strata-compar-rbs-tor}
There is a natural action of
$\Gamma(\mathbf{M}_{\mathbf{Q},\ell}\,|\,\mathbf{P})$ on
$\widehat{B}^{\circ}_{\mathbf{P},\Sigma}$. If $\Gamma$ is small enough, the diagonal morphism
$$\Gamma(\mathbf{M}_{\mathbf{Q},\ell}\,|\,\mathbf{P})\backslash
\widehat{B}^{\circ}_{\mathbf{P},\Sigma}
\!\xymatrix@C=1.7pc{\ar[r] & }\!
(\Gamma(\mathbf{P})\backslash e(\mathbf{P}))\times (\Gamma(\mathbf{M}_{\mathbf{Q},\ell}\,|\,\mathbf{P})\backslash \mathcal{B}^{\circ}_{\mathbf{P},\Sigma}(\C))$$
identifies $\Gamma(\mathbf{M}_{\mathbf{Q},\ell}\,|\,\mathbf{P})\backslash
\widehat{B}^{\circ}_{\mathbf{P},\Sigma}$ with the intersection of
$e'(\mathbf{P})\times X^{tor}_{\mathbf{P},\Sigma}(\C)$ with $\widehat{\overline{\Gamma\backslash D}}$.

\end{proposition}

For this reason,
$\Gamma(\mathbf{M}_{\mathbf{Q},\ell}\,|\,\mathbf{P})\backslash
\widehat{B}^{\circ}_{\mathbf{P},\Sigma}$ will be called the corner-like $\mathbf{P}$-stratum of $\widehat{\overline{\Gamma \backslash D}}_{\Sigma}$.
We make note of the following assertion for later use:

\begin{lemma}
\label{lemma:two-proper-maps}
We have two cartesian squares
$$\xymatrix@C=1.5pc@R=1.5pc{\widehat{B}^{\circ}_{\mathbf{P},\Sigma} \ar[r] \ar[d] & \Gamma(\mathbf{M}_{\mathbf{Q},\ell}\,|\,\mathbf{P}) \backslash  \widehat{B}^{\circ}_{\mathbf{P},\Sigma} \ar[d]\\
 \mathcal{B}^{\circ}_{\mathbf{P},\Sigma}(\C) \ar[r] &
\Gamma(\mathbf{M}_{\mathbf{Q},\ell}\,|\,\mathbf{P}) \backslash  \mathcal{B}^{\circ}_{\mathbf{P},\Sigma}(\C)  } \qquad \quad
\xymatrix@C=1pc{\widehat{B}^{\circ}_{\mathbf{P},\Sigma} \ar[r]\ar[d] &  \Gamma(\mathbf{M}_{\mathbf{Q},\ell}\,|\,\mathbf{P}) \backslash \widehat{B}^{\circ}_{\mathbf{P},\Sigma} \hspace{1.3cm}\ar@<-1.7pc>[d]
\\
\Gamma(\mathbf{Q}_h) \backslash e(\mathbf{P})  \ar[r] & \Gamma(\mathbf{Q}) \backslash e(\mathbf{P}) =e'(\mathbf{P})}$$
where the bottom horizontal arrows are proper maps. In particular,
the top horizontal arrows are also proper maps.

\end{lemma}

\begin{proof}
That the squares are cartesian follows from the fact that $\Gamma(\mathbf{M}_{\mathbf{Q},\ell}\,|\,\mathbf{P})$
acts properly discontinuously on $\mathcal{B}^{\circ}_{\mathbf{P},\Sigma}(\C)$ and
$\Gamma(\mathbf{Q}_h) \backslash e(\mathbf{P})$.
That the bottom arrows are proper maps follows from
Proposition \ref{prop:strata-compar-rbs-tor}.
\end{proof}


\section{Application to the reductive Borel-Serre compactification}

\label{sect:main-comput}
In this section, we state and prove the main result of the paper.

\subsection{The Main Theorem: statement and complements}
We keep the notation and assumptions of Section
\ref{sect:compactif-locally-symmetric}. Recall that
$\mathbf{G}$ is a simple $\Q$-group, and $D$ is a hermitian symmetric domain
with $\Aut(D)\simeq G$ modulo compact factors. Our main result is:

\begin{theorem}
\label{thm:main-thm}
\begin{itemize}

\item[(a)]
Let $\Gamma\subset \mathbf{G}(\Q)$ be an arithmetic subgroup
and $X$ the $\C$-scheme such that
$X(\C)\simeq \Gamma\backslash D$.
Denote $p:\overline{\Gamma\backslash D}^{rbs}\to \overline{\Gamma
\backslash D}^{bb}$ the natural projection.
Then, there exists
a canonical isomorphism of commutative unitary algebras
$$\varphi:\An^*(\EE_{\overline{X}^{bb}})\overset{\sim}{\to} {\rm R}p_*\Q\,;$$
here $\EE_{\overline{X}^{bb}}$ is the Artin motive defined in
\emph{Corollary \ref{defn-cor:E-of-X}}, which is a unitary algebra by
\emph{Proposition \ref{E-unitary}}.

\item[(b)]
Let $\Gamma, \; \Gamma'\subset \mathbf{G}(\Q)$ be arithmetic subgroups
and denote by $X$ and $X'$ the $\C$-schemes such that
$X(\C)\simeq \Gamma\backslash D$ and $X'(\C)\simeq \Gamma'\backslash D$. Also, denote $p:\overline{\Gamma\backslash D}^{rbs}\to \overline{\Gamma
\backslash D}^{bb}$ and
$p':\overline{\Gamma'\backslash D}^{rbs}\to \overline{\Gamma'
\backslash D}^{bb}$ the natural projections.

Let $g\in \mathbf{G}(\Q)$ such that
$g\Gamma'g^{-1}\subset\Gamma$. We have induced morphisms $g^{rbs}$ and
$g^{bb}$ from the compactifications of $\Gamma'\backslash D$ to the compactifications of $\Gamma\backslash D$ as in
\eqref{eqn:g-extended-to-borel-serre-compact}
and \eqref{eq:induced-g-on-baily-borel-compact}.
Moreover, $g^{bb}$ is induced by a morphism of $\C$-schemes which we also denote by $g^{bb}$.
With these notations, we have a commutative diagram in
$\mathbf{D}(\Shv(\overline{\Gamma'\backslash D}^{bb}))$:
$$\xymatrix@C=1.5pc@R=1.5pc{(g^{bb})^* \An^* \EE_{\overline{X}^{bb}} \ar[d]^-{\sim}_-{\varphi}
\ar[r]^-{\sim} & \An^* (g^{bb})^* \EE_{\overline{X}^{bb}} \ar[r] & \An^* \EE_{
\overline{X'}^{bb}} \ar[d]^-{\varphi}_-{\sim} \\
(g^{bb})^* {\rm R} p_* \Q_{\overline{\Gamma\backslash D}^{rbs}} \ar[r] &
{\rm R} p'_*(g^{rbs})^*\Q_{\overline{\Gamma\backslash D}^{rbs}} \ar[r]^-{\sim}
& {\rm R} p'_*\Q_{\overline{\Gamma'\backslash D}^{rbs}},\!}$$
where $(g^{bb})^*{\rm R}p_* \to {\rm R}p'_* (g^{rbs})^*$ is the base change
morphism and $(g^{bb})^* \EE_{\overline{X}^{bb}} \to \EE_{\overline{X'}^{bb}}$
is the morphism in \emph{Corollary \ref{cor:functoriality-EE-X-010}}.
\end{itemize}
\end{theorem}

\begin{remark}
The claim that $\varphi$ is an isomorphism of unitary algebras
implies in particular that the square
$$\xymatrix@C=1.5pc@R=1.5pc{\An^*(\un_{\overline{X}^{bb}})  \ar[r]^-{\sim} \ar[d] & \Q_{\overline{\Gamma\backslash D}^{bb}} \ar[d] \\
\An^*(\EE_{\overline{X}^{bb}}) \ar[r]^-{\varphi}_-{\sim} & {\rm R}p_*\Q_{
\overline{\Gamma\backslash D}^{rbs}}}$$
commutes. Indeed, the vertical arrows are the unit morphisms of the algebras
$\An^*(\EE_{\overline{X}^{bb}})$ and ${\rm R}p_*\Q_{\overline{\Gamma\backslash
D}^{rbs}}$.
\end{remark}

\begin{remark}
\label{rem:compat-with-Hecke-corresp}
The isomorphism in Theorem \ref{thm:main-thm}, (a)
is compatible with the action of Hecke correspondences.
These are a composite of a pullback and a trace.
By Theorem \ref{thm:main-thm}, (b), we are thus reduced
to check the compatibility
with the trace map associated to arithmetic subgroups
$\Gamma, \; \Gamma'\subset \mathbf{G}(\Q)$
and $g\in \mathbf{G}(\Q)$ such that
$g\Gamma'g^{-1}\subset\Gamma$.
Again by Theorem \ref{thm:main-thm}, (b),
we can assume that $g=1$. For simplicity, we also assume
that $\Gamma'$ is a normal
subgroup of $\Gamma$; the more general case reduces to that one.
Using the adjunction $((1^{bb})^*,1^{bb}_*)$
one has a canonical morphism
$\EE_{\overline{X}^{bb}} \to 1^{bb}_*\EE_{\overline{X'}^{bb}}$, and similarly for the relative cohomology of the reductive Borel-Serre compactification under its projection to the Baily-Borel Satake compactification. Using
Theorem \ref{thm:main-thm}, (b),
one deduces that these morphisms are compatible with $\varphi$.
Now, the group $G=\Gamma'/\Gamma$
acts on the target of
$\EE_{\overline{X}^{bb}} \to 1^{bb}_*\EE_{\overline{X'}^{bb}}$ and identifies the source with the image of the projector ${\rm card}(G)^{-1}\sum_{h\in G}h$ (cf.~Lemma \ref{lem:Phi-invariant-EE-X}).
The trace map
$tr:1^{bb}_*\EE_{\overline{X'}^{bb}} \to \EE_{\overline{X}^{bb}}$
is a multiple (by ${\rm card}(G)$) of the projection of
$1^{bb}_*\EE_{\overline{X'}^{bb}}$ to its direct factor
$\EE_{\overline{X}^{bb}}$
(and similarly for the relative cohomology of the reductive Borel-Serre compactification).
This proves that the isomorphism $\varphi$
is compatible with the trace maps.

\end{remark}

\begin{remark}
In \cite{GT}, Goresky and Tai constructed a morphism from the singular
cohomology of the reductive Borel-Serre compactification $\overline{\Gamma
\backslash D}^{rbs}$ to the Betti cohomology of a toroidal compactification
$\overline{X}^{tor}_{\Sigma}$,
for fine enough compatible families of {\it prpcd}'s $\Sigma$.  This came out of a
study of the
least common modification of the two compactifications of $\Gamma
\backslash D$,
and it is induced by a continuous mapping.  We can use
Theorem \ref{thm:main-thm} to recover a version of their result.
Indeed, assume that $\Sigma$ is chosen so that $\overline{X}^{tor}_{\Sigma}$
is projective and smooth. Denote by $e:\overline{X}^{tor}_{\Sigma}\to
\overline{X}^{bb}$ the natural projection. As $e$ is dominant, we have, by
Corollary
\ref{cor:functoriality-EE-X-010}, a natural morphism $e^*\EE_{\overline{X}^{
bb}} \to \EE_{\overline{X}^{tor}_{\Sigma}}\simeq \un_{\overline{X}^{tor}_{
\Sigma}}$. By adjunction, we deduce a natural morphism
$\EE_{\overline{X}^{bb}} \to e_*\un_{\overline{X}^{tor}_{\Sigma}}$. Applying
the Betti realization, and using
Theorem \ref{thm:main-thm}, we deduce a natural morphism
${\rm R} p_*\Q_{\overline{\Gamma \backslash D}^{rbs}} \to {\rm R}e_* \Q_{
\overline{X}^{tor}_{\Sigma}(\C)}$. Taking the
cohomological direct images along the projection of $\overline{\Gamma
\backslash D}^{bb}$ to the point, we obtain a natural morphism
${\rm H}^*(\overline{\Gamma \backslash D}^{rbs}) \to {\rm H}^*(\overline{X}^{
tor}_{\Sigma}(\C))$. We expect this to agree with the morphism from \cite{GT}.

\end{remark}

\begin{remark}
We indicate somewhat heuristically how the determination in Theorem \ref{thm:main-thm} (of $\omega^0_{\overline X^{bb}}j^{bb}_*\un_X$ when
$\Gamma$ is neat) is consistent with the notion of punctual lowest
weight in a Hodge theoretical sense. We refer to
\eqref{diagram-bs-bb} in Proposition \ref{bs-bb} for notation. The diagram gives
$${\rm R} j^{bb}_*\Q_{(\Gamma\backslash
D)}\simeq {\rm R}(pqj^{bs})_*\Q_{(\Gamma\backslash
D)}\simeq {\rm R}(pq)_*\Q_{(\overline{\Gamma\backslash D)}^{bs}},
$$
as $j^{bs}$ is a homotopy equivalence.

Let $\mathbf Q$ be a maximal $\Q$-parabolic subgroup of $\mathbf G$. Over $(\Gamma\backslash D)^{bb}_Q$
(the underlying topological space of $X_{\mathbf{Q}}^{bb}$ from \eqref{bb-strat}),
we
have that $q$ is a fibration, with
\begin{equation}\label{rbs-fiber}
q^{-1}(x)\simeq (\overline{\Gamma(\widetilde M_{Q,\ell})\backslash \widetilde D_{Q,\ell}})^{rbs}
\end{equation}
whenever $x\in (\Gamma\backslash D)^{bb}_Q$ (see
\cite[Prop.\,2.3.8]{rbs-2}). Likewise, for such $x$ one has
\begin{equation}\label{bs-fiber}
(pq)^{-1}(x)\simeq\overline{\Gamma( N_Q\widetilde
M_{Q,\ell})\backslash (N_Q\times \widetilde D_{Q,\ell})}^{bs},
\end{equation}
which has the homotopy type of $\Gamma( N_Q\widetilde
M_{Q,\ell})\backslash (N_Q\times \widetilde D_{Q,\ell})$. (In the
preceding, $\widetilde D_{Q,\ell}$ denotes the symmetric space of
$\widetilde M_{Q,\ell}$\,.) In particular, the latter is a
$(\Gamma(N_Q)\backslash N_Q)$-fibration over $\Gamma(\widetilde
M_{Q,\ell})\backslash \widetilde D_{Q,\ell}$.

We can take the sheaves of smooth differential forms on $\Gamma(
N_Q\widetilde M_{Q,\ell})\backslash (N_Q\times \widetilde
D_{Q,\ell})$ with coefficients in the exterior algebra
$\bigwedge^{\!*} \mathfrak{n}_Q^{\vee}$, where $\mathfrak{n}_Q$ is
the Lie algebra of $N_Q$, as the $\C$-datum of a mixed Hodge
complex for ${\rm H}^*(\Gamma( N_Q\widetilde M_{Q,\ell})\backslash
(N_Q\times \widetilde D_{Q,\ell}))$ (cf.
\cite[\S5.2]{HZ2}).\footnote{In fact, allowing $x$ to vary
produces a variation of mixed Hodge structure on
$(\Gamma\backslash D)^{bb}_Q$.} The weights are those that come
from the definition of a Shimura variety \cite[\S2.1]{deligne}:
forms on $\Gamma(\widetilde M_{Q,\ell})\backslash \widetilde
D_{Q,\ell}$ with $\C$-coefficients comprise $W_0$ -- indeed, these
forms appear only combinatorially in the toroidal setting (cf.
Definition \ref{defn:compatible-family-prpcd-s}), and have trivial
contribution to the mixed Hodge structure; and $\bigwedge^{\!i}
\mathfrak{n}^{\vee}_Q$ has only positive weights when $i>0$. Thus,
the lowest weight is given by $\Q_{(\Gamma(\widetilde
M_{Q,\ell})\backslash \widetilde D_{Q,\ell})}$. However, we
preferred to see (71) here, which involves more than the quotient
of (72) by $N_Q$, insufficient over {\it its} boundary. However,
factoring $q$ through the {\it excentric} Borel-Serre
compactification $\overline{\Gamma \backslash D}^{ebs}$ (see
\cite[(2.3.5)]{rbs-2}) produces
$$\overline{\Gamma((N_Q/U_Q)\widetilde
M_{Q,\ell})\backslash ((N_Q/U_Q)\times \widetilde D_{Q,\ell})}^{bs}.$$

\end{remark}

In the statement of Theorem \ref{thm:main-thm} we used the
notation ${\rm R}p_*$ for the derived operation of cohomological direct image
of sheaves. As we mainly consider derived operations on sheaves, we will drop
from now on
the ``${\rm R}$''; this convention was already used for the operations
on motives in Sections \ref{sect:triang-cat-of-mot} and
\ref{sect:artin-zero-slice-weight}.

\begin{definition}
\label{defn:reductive-bs-motive}
We keep the notation from \emph{Theorem \ref{thm:main-thm}}.
Let $\pi:\overline{X}^{bb} \to \Spec(\C)$ be the projection to the point.
The motive $\pi_*\left(\EE_{\overline{X}^{bb}}\right)$
is called the \emph{reductive Borel-Serre motive} of $X$ and will be denoted
$\M^{rbs}(X)$.
\end{definition}

\begin{remark}
\label{rem:Mrbs-defined-numberfield}
As was the case for the scheme $X$, the motive $\M^{rbs}(X)$ can be defined
over a number
field. Indeed, let $k\subset \C$ be a field of definition of $\overline{
X}^{bb}$, which we may take to be a finite extension of $\Q$. Let
$\overline{X}^{bb}_{/k}$ be a $k$-scheme such that $\overline{X}^{bb}\simeq
\overline{X}^{bb}_{/k} \otimes_k \C$. Also, denote by $\pi_{/k}:\overline{X}^{
bb}_{/k} \to \Spec(k)$ the projection to the point. Then, the motive
$\M^{rbs}(X_{/k})=(\pi_{/k})_* \left(\EE_{\overline{X}^{bb}_{/k}}\right)$
satisfies $\M^{rbs}(X_{/k})\otimes_k \C \simeq \M^{rbs}(X)$, where
$-\otimes_k \C$ denotes the inverse image of motives along $\Spec(\C) \to
\Spec(k)$. For this reason, $\M^{rbs}(X_{/k})$ is called a \emph{reductive
Borel-Serre motive} over $k$.

\end{remark}
\smallskip

In the following statement, we identify $\mathbf{D}(\Q)$ with
$\mathbf{D}(\Shv(pt))$, where $pt$ is the topological space consisting of
one point. With this understood, the Betti realization on $\DM(\C)$ takes
values in
$\mathbf{D}(\Q)$.

\begin{corollary}
\label{cor-of-main-thm}
There is an isomorphism of commutative unitary algebras
$$\varphi:\An^*(\M^{rbs}(X)) \overset{\sim}{\to}
{\rm H}^*(\overline{\Gamma \backslash D}^{rbs})$$
from the Betti realization of the motive $\M^{rbs}(X)$ to the singular
cohomology of the topological space
$\overline{\Gamma\backslash D}^{rbs}$.
Moreover, for $g\in \mathbf{G}(\Q)$ such that $g\Gamma'g^{-1}\subset \Gamma$,
there is a morphism of commutative unitary algebras $\M^{rbs}(g):\M^{rbs}(X)
\to \M^{rbs}(X')$, which makes the following square in $\mathbf{D}(\Q)$ commutative:
$$\xymatrix@C=1.5pc@R=1.5pc{\An^*(\M^{rbs}(X)) \ar[r]^-{\varphi}_-{\sim} \ar[d]_-{\M^{rbs}(g)}
& {\rm H}^*(\overline{\Gamma\backslash D}^{rbs}) \ar[d]^-{{\rm H}^*(g^{rbs})}\\
\An^*(\M^{rbs}(X')) \ar[r]^-{\varphi}_-{\sim} & {\rm H}^*(\overline{\Gamma'
\backslash D}^{rbs}).\!}$$

\end{corollary}

\begin{proof}
The morphism $\varphi:\An^*(\M^{rbs}(X)) \to {\rm H}^*(\overline{\Gamma
\backslash D}^{rbs})$ is the composition
$$\xymatrix@C=1.3pc{\An^*(\M^{rbs}(X)) = \An^* \pi_*\EE_{\overline{X}^{bb}}
\ar[r] & \pi^{an}_*\An^*\EE_{\overline{X}^{bb}} \ar[r]^-{\sim} &
\pi^{an}_*p_*\Q_{\overline{\Gamma \backslash D}^{rbs}} \simeq
{\rm H}^*(\overline{\Gamma\backslash D}^{rbs})}$$
where the isomorphism $\pi^{an}_*\An^*\EE_{\overline{X}^{bb}} \simeq
\pi^{an}_*p_*\Q_{\overline{\Gamma \backslash D}^{rbs}}$ is the one induced
by the isomorphism in
Theorem \ref{thm:main-thm}, (a).
That $\An^*\pi_*\EE_{\overline{X}^{bb}} \to \pi^{an}_*\An^*\EE_{\overline{X}^{bb}}$ is invertible follows from the commutation of the Betti realization with the cohomological direct images, the motive $\EE_{\overline{X}^{bb}}$ being compact.

We now pass to the second part of the corollary.
Call $\pi$ and $\pi'$ the projections of the schemes $X$ and $X'$ to $\Spec(\C)$. Note that we have $\pi'=\pi\circ g^{bb}$.
We define our
$\M^{rbs}(g)$ as the following composition
$$\xymatrix@C=1.2pc{\pi_*\EE_{\overline{X}^{bb}} \ar[r] & \pi_*(g^{bb})_*(g^{bb})^* \EE_{\overline{X}^{bb}} \ar[r] & \pi_*(g^{bb})_* \EE_{\overline{X'}^{bb}} \simeq \pi'_* \EE_{\overline{X'}^{bb}}}$$
where the morphism in the middle is the one described in
Corollary \ref{cor:functoriality-EE-X-010}.
That the square
of the statement commutes follows from part (b) of
Theorem \ref{thm:main-thm}. We leave the details to the reader.
\end{proof}

\begin{remark}\label{huber}
Let $k\subset \C$ be a number field as in Remark
\ref{rem:Mrbs-defined-numberfield}. We may apply Huber's mixed
realization functor $R_{\mathcal{MR}}:\mathbf{DM}_{gm}(k) \to
D_{\mathcal{MR}}$ \cite[Th.~2.3.3]{huber-realisation} to the dual
of ${\rm a}_{tr}(\M^{rbs}(X_{/k}))$, where ${\rm
a}_{tr}:\DM(k)\simeq \mathbf{DM}(k)$ is the equivalence given by
Proposition \ref{prop:compare}. (Note that ${\rm
a}_{tr}(\M^{rbs}(X_{/k}))$ is a geometric motive as
$\M^{rbs}(X_{/k})$ is compact by Proposition
\ref{prop:additional-prop-omega-0-x}, (vii) and
\cite[Cor.~2.2.21]{ayoub-these-I}.) We get in this way an object
of the derived category of mixed realizations which we simply
denote by $R^{\,rbs}_{\mathcal{MR}}(X_{/k})$. The singular
component of $R^{\,rbs}_{\mathcal{MR}}(X_{/k})$ corresponding to
the canonical embedding $k\hookrightarrow \C$ is Huber's singular
realization of the dual of ${\rm a}_{tr}(\M^{rbs}(X))$ which is
canonically isomorphic to $\An^*(\M^{rbs}(X))$. (Unfortunately,
the comparison between Huber's singular realization
\cite{huber-realisation,huber-corrigendum} and the Betti
realization \cite{realiz-oper} we have used in this paper is not treated
in the literature, though we expect it be straightforward.) Hence,
by Corollary \ref{cor-of-main-thm}, the cohomology groups of
$\overline{\Gamma\backslash D}^{rbs}$ are naturally mixed
realizations in the sense of \cite[Def.~11.1.1]{huber-monograph}.
In particular, ${\rm H}^*(\overline{\Gamma\backslash D}^{rbs})$
carries a mixed Hodge structure (presumably the same as what one
would get when \cite{rbs-mixed-hodge-str} is corrected) and ${\rm
H}^*(\overline{\Gamma\backslash D}^{rbs})\otimes \Q_{\ell}$ is
naturally a representation of $Gal(\overline{\Q}/k)$ for every
prime number $\ell$. All this is compatible with the action of
Hecke correspondences (see Remark
\ref{rem:compat-with-Hecke-corresp}).

\end{remark}

In the remainder of this section, we explain how to reduce Theorem
\ref{thm:main-thm}
to the case where the arithmetic subgroups are neat.

\begin{proposition}
\label{prop:neat=superfluous}
If \emph{Theorem \ref{thm:main-thm}} holds for neat arithmetic subgroups
of $\mathbf{G}(\Q)$, then it holds for all arithmetic subgroups.

\end{proposition}


\begin{proof}
We assume that Theorem \ref{thm:main-thm} is proven for $\Gamma$ neat, and
we show how to extend
it for arithmetic subgroups of $\mathbf{G}(\Q)$ which are not necessarily
neat. In fact, we will deal only with part (a) and leave part (b) to the reader.

Let $\Gamma_0\subset \mathbf{G}(\Q)$ be any arithmetic subgroup.
We may find a normal subgroup $\Gamma\lhd \Gamma_0$ of finite index which is
neat.
The finite group $\Gamma\backslash \Gamma_0$ acts on the topological spaces
$\Gamma \backslash D$, $\overline{\Gamma\backslash D}^{rbs}$ and $\overline{
\Gamma \backslash D}^{bb}$, and their quotients with respect to these actions
are $\Gamma_0\backslash D$, $\overline{\Gamma_0\backslash D}^{rbs}$ and
$\overline{\Gamma_0\backslash D}^{bb}$ respectively.
We let $e:\Gamma'\backslash D \to \Gamma \backslash D$, $e^{bb}:\overline{\Gamma\backslash D}^{bb} \to
\overline{\Gamma_0\backslash D}^{bb}$ and
$e^{rbs}:\overline{\Gamma\backslash D}^{rbs} \to
\overline{\Gamma_0\backslash D}^{rbs}$ be the quotient maps.

Also, if $X$ and $\overline{X}^{bb}$ are the $\C$-schemes such that $X(\C)\simeq \Gamma\backslash D$ and $\overline{X}^{bb}(\C)=\overline{\Gamma\backslash D}^{bb}$, then
$\Gamma\backslash \Gamma_0$ acts on $X$ and $\overline{X}^{bb}$, and their
quotients with respect to these actions are
respectively $X_0$ and $\overline{X_0}^{bb}$, the $\C$-schemes such that $X_0(\C)\simeq \Gamma_0\backslash D$ and $\overline{X_0}^{bb}(\C)=\overline{\Gamma_0\backslash D}^{bb}$.
We also denote by
$e^{bb}:\overline{X}^{bb} \to \overline{X_0}^{bb}$ the morphism of
$\C$-schemes that is given by
$e^{bb}:\overline{\Gamma\backslash D}^{bb} \to
\overline{\Gamma_0\backslash D}^{bb}$ on the varieties of $\C$-points.

Now, denote by
$p:\overline{\Gamma\backslash D}^{rbs} \to
\overline{\Gamma\backslash D}^{bb}$ and
$p_0:\overline{\Gamma_0\backslash D}^{rbs} \to
\overline{\Gamma_0\backslash D}^{bb}$ the natural projections.
With the notation of Theorem \ref{thm:main-thm}, (b), an element
$g\in \Gamma_0$ acts on
$e^{bb}_*\EE_{\overline{X}^{bb}}$ by the composition
$$\xymatrix@C=1.5pc{e^{bb}_*\EE_{\overline{X}^{bb}}
\ar[r]^-{\sim} & e^{bb}_* (g^{bb})_*(g^{bb})^* \EE_{\overline{X}^{bb}}
\ar[r]^-{\sim} & e^{bb}_* g^{bb}_* \EE_{\overline{X}^{bb}} \simeq e^{bb}_*
\EE_{\overline{X}^{bb}}.}$$
For the last isomorphism, we used that $e^{bb}\circ g^{bb}=e^{bb}$.
It is easy to check that this gives a representation of $\Gamma \backslash
\Gamma_0$ on $e^{bb}_*\EE_{\overline{X}^{bb}}$.
Applying Lemma \ref{lem:Phi-invariant-EE-X}, we have that the sub-object of
$(\Gamma \backslash \Gamma_0)$-invariants
is canonically isomorphic to $\EE_{\overline{X_0}^{bb}}$.
Similarly, $g\in \Gamma_0$ acts on $e^{rbs}_* \Q_{\overline{\Gamma
\backslash D}^{rbs}}$ by the composition
$$\xymatrix@C=1.2pc{e^{rbs}_* \Q_{\overline{\Gamma\backslash D}^{rbs}} \ar[r] &
e^{rbs}_*(g^{rbs})_*(g^{rbs})^* \Q_{\overline{\Gamma\backslash D}^{rbs}}
\simeq e^{rbs}_*g^{rbs}_* \Q_{\overline{\Gamma\backslash D}^{rbs}} \simeq
e^{rbs}_* \Q_{\overline{\Gamma\backslash D}^{rbs}}.}$$
For the last isomorphism, we used that $e^{rbs}\circ g^{rbs}=e^{rbs}$.
It is easy to check that this gives a representation of
$\Gamma\backslash \Gamma_0$ on
$e^{rbs}_* \Q_{\overline{\Gamma\backslash D}^{rbs}}$. Moreover, the sub-object
of $(\Gamma\backslash \Gamma_0)$-invariants is canonically isomorphic to
$\Q_{\overline{\Gamma_0\backslash D}^{rbs}}$.

By Theorem \ref{thm:main-thm}, (b),
we have a commutative diagram
$$\xymatrix@C=1pc@R=1.5pc{\An^*\EE_{\overline{X}^{bb}} \ar[r]^-{\sim} \ar[d]_-{\varphi}^-{\sim} & (g^{bb})_*(g^{bb})^*\An^*\EE_{\overline{X}^{bb}} \ar[r]^-{\sim} \ar[d]^-{\varphi}_-{\sim}  & (g^{bb})_* \An^*(g^{bb})^*\EE_{\overline{X}^{bb}} \ar[r]^-{\sim} & (g^{bb})_*\An^* \EE_{\overline{X}^{bb}}\ar[d]^-{\varphi}_-{\sim}  \\
p_*\Q_{\overline{\Gamma \backslash D}^{rbs}} \ar[r]^-{\sim} & (g^{bb})_*(g^{bb})^* p_*\Q_{\overline{\Gamma \backslash D}^{rbs}} \ar[r]^-{\sim} & (g^{bb})_*p_*(g^{rbs})^*\Q_{\overline{\Gamma \backslash D}^{rbs}} \ar[r]^-{\sim} & (g^{bb})_* p_*\Q_{\overline{\Gamma \backslash D}^{rbs}}.\!}$$
If we apply $e^{bb}_*$
to the first horizontal line, we get
the action of $g\in \Gamma_0$ on
the complex of sheaves $\An^* e^{bb}_*\EE_{\overline{X}^{bb}}$ modulo the
isomorphisms
$e^{bb}_*\An^*\EE_{\overline{X}^{bb}}\simeq\An^*e^{bb}_*\EE_{\overline{
X}^{bb}}$ and
$e^{bb}_*g^{bb}_*\An^*\EE_{\overline{X}^{bb}} \simeq e^{bb}_*\An^*\EE_{
\overline{X}^{bb}}\simeq \An^*e^{bb}_*\EE_{\overline{X}^{bb}}$.
Also, if we apply $e^{bb}_*$ to the second horizontal line, we get the
action of $g\in \Gamma_0$ on
$p_{0*}e^{rbs}_*\Q_{\overline{\Gamma\backslash D}^{rbs}}$
modulo the isomorphisms
$e^{bb}_*p_*\Q_{\overline{\Gamma\backslash D}^{rbs}}\simeq p_{0*}
e^{rbs}_*\Q_{\overline{\Gamma \backslash D}^{rbs}}$ and
$e^{bb}_*g^{bb}_*p_*\Q_{\overline{\Gamma \backslash D}^{rbs}}\simeq
e^{bb}_*p_*\Q_{\overline{\Gamma \backslash D}^{rbs}} \simeq p_{0*}
e^{rbs}_*\Q_{\overline{\Gamma \backslash D}^{rbs}}$. This shows that the
isomorphism $\An^* e^{bb}_*\EE_{\overline{X}^{bb}} \overset{\sim}{\to}
p_{0*} e^{rbs}_*\Q_{\overline{\Gamma \backslash D}^{rbs}}$, given
by the composition
$$
\xymatrix@C=1.4pc{\An^*e^{bb}_*\EE_{\overline{X}^{bb}}
\ar[r]^-{\sim} & e^{bb}_*\An^*\EE_{\overline{X}^{bb}} \ar[r]^-{\varphi}_{
\sim} & e^{bb}_* p_*\Q_{\overline{\Gamma \backslash D}^{rbs}} \ar[r]^-{
\sim} & p_{0*} e^{rbs}_*
\Q_{\overline{\Gamma \backslash D}^{rbs}},}
$$
is $(\Gamma \backslash \Gamma_0)$-equivariant. Passing to the sub-objects
of $(\Gamma \backslash \Gamma_0)$-invariants, yields an isomorphism
\begin{equation}
\label{eq:varphi-when-Gamma-nonneat}
\varphi:\An^*\EE_{\overline{X_0}^{bb}} \overset{\sim}{\to}
p_{0*} \Q_{\overline{\Gamma_0\backslash D}^{rbs}}.
\end{equation}
Moreover, this is an isomorphism of unitary algebras as $\Gamma\backslash
\Gamma_0$
acts by unitary algebra automorphisms on $e^{bb}_*\EE_{\overline{X}^{bb}}$ and
$e^{rbs}_*\Q_{\overline{\Gamma \backslash D}^{rbs}}$.
We leave it to the reader to check that
\eqref{eq:varphi-when-Gamma-nonneat}
is independent of the choice of a neat normal subgroup $\Gamma\subset\Gamma_0$.
\end{proof}

\subsection{Proof of the Main Theorem}
Keep the notation in Theorem \ref{thm:main-thm}.
We denote by $r$ the $\Q$-rank of the simple $\Q$-group $\mathbf{G}$.
As in \S
\ref{subsect:The-Baily-Borel-compactification}, we list the simple roots:
$\beta_1, \dots, \beta_r$ so that
$\beta_i$ is not orthogonal to $\beta_{i+1}$ and $\beta_r$ is the
distinguished root. We will identify
$[\![1,r]\!]$ with $\Delta(\mathbf{G},\mathbf{S})$,
by sending $1\leq i\leq r$ to $\beta_i$.
For $I \subset [\![1,r]\!]$, we let
$\mathbf{P}_I$ denote the standard parabolic $\Q$-subgroup
of type $I$ and cotype $[\![1,r]\!]-I$ (see
\S \ref{subsect:toroidal-compact}).
Note that
$\mathbf{P}_{[\![1,r]\!]}=\mathbf{G}$, which for convenience will be designated as the parabolic
$\Q$-subgroup of cotype $\{0\}$ (rather than $\emptyset$).

\subsubsection{Setting the stage}\label{Stage}
The Baily-Borel Satake compactification $\overline{X}^{bb}$ of $X$
admits a natural stratification $(X^{bb}_i)_{i\in [\![0,r]\!]}$
such that $X^{bb}_i$ is the union of the strata
$X^{bb}_{\mathbf{Q}}$, where $\mathbf{Q}\subseteq \mathbf{G}$
varies among parabolic $\Q$-subgroups that are of cotype $\{i\}$.
Thus, the connected components of $X^{bb}_i$ are locally symmetric varieties
of the same dimension. In particular, the open stratum
$X^{bb}_0=X^{bb}_{\mathbf{G}}$ is simply $X$. As $\Gamma$ is neat,
the schemes $X^{bb}_i$ are smooth. For $i\in [\![0,r]\!]$, denote
by $X^{bb}_{\geq i}$ the Zariski closure of $X^{bb}_i$. Then, as
sets, we have $X^{bb}_{\geq i}=\bigsqcup_{j\in [\![i,r]\!]}
X^{bb}_j$.  Thus, we are in the situation of \textbf{D1)} of \S
\ref{subsub:setting-for-comput}. Note also that each irreducible
component of $X^{bb}_{\geq i}$ is of the form
$\overline{X}^{bb}_{\mathbf{Q}}$. The normalization of the latter is
$\overline{(X^{bb}_{\mathbf{Q}})}^{bb}$, the Baily-Borel Satake
compactification of $X^{bb}_{\mathbf{Q}}$.

The data in \textbf{D2)} of \S \ref{subsub:setting-for-comput}
are realized using the toroidal compactifications (see \S
\ref{subsect:toroidal-compact}) of the connected components of $X^{bb}_i$. However,
to ensure Properties
\textbf{P1)} and \textbf{P2)} of \S \ref{subsub:setting-for-comput}, some
care is needed in the choice of the compatible families of {\it prpcd}'s
for the
locally symmetric varieties $X^{bb}_{\mathbf{Q}}$. First, we introduce the
following notation: if $\mathbf{Q}\subset \mathbf{G}$ is a parabolic
$\Q$-subgroup which is maximal or improper, we denote by
$\widetilde{\Gamma}(\mathbf{M}_{\mathbf{Q},h})$ the arithmetic subgroup of
$\mathbf{M}_{\mathbf{Q},h}$ equal to
$\Gamma(\mathbf{M_Q})\cap M_{Q,h}$. This is a normal subgroup of
finite index in $\Gamma(\mathbf{M}_{\mathbf{Q},h})$.\footnote{We recall that $\Gamma(\mathbf{M_Q})=\Gamma(\mathbf{Q}/\mathbf{N_Q}\mathbf{S_Q})$ and $\Gamma(\mathbf{M}_{\mathbf{Q},h})=\Gamma(\mathbf{Q}/\mathbf{N_Q}\mathbf{S_Q}\widetilde{\mathbf{M}}_{\mathbf{Q},\ell})$, where $\Gamma$ is viewed as a functor on pairs as in \S \ref{subsect:generalities}.}

In the sequel, we fix an extended compatible family of {\it prpcd}'s $\Sigma=
\{\Sigma_{\mathbf{Q},\mathbf{R}}\}$ (with respect to $\Gamma$) in the sense of
Definition \ref{defn:extended-fam-prpcd-s} satisfying the following properties:
\begin{enumerate}

\item $\Sigma=\{\Sigma_{\mathbf{Q},\mathbf{R}}\}$ is projective and simplicial.

\item For every parabolic $\Q$-subgroup $\mathbf{Q}\subseteq \mathbf{G}$
which is maximal or improper, the compatible family of {\it prpcd}'s
$\Sigma_{(\mathbf{Q})}=(\Sigma_{\mathbf{Q},\mathbf{R}})_{\mathbf{R}}$
is a smooth and projective family with respect to the arithmetic
subgroup
$\widetilde{\Gamma}(\mathbf{M}_{\mathbf{Q},h})$.

\end{enumerate}

Clearly, there exist such extended compatible families of {\it prpcd}'s
and they
form a cofinal subset (with respect to refinement) of the set of all
extended compatible families of
{\it prpcd}'s.
We will also assume that our $\Sigma$ is fine enough so that the statements
in Propositions \ref{normalizers} and
\ref{prop:intersection-closure-strata-toroidal-comp}
hold wherever they are needed.

For $\mathbf{Q}\subseteq \mathbf{G}$ a parabolic $\Q$-subgroup
which is maximal or improper, we denote by
$Y^{tor}_{\mathbf{Q}}=\overline{
(X^{bb}_{\mathbf{Q}})}^{tor}_{\Sigma_{(\mathbf{Q})}}$ the toroidal
compactification of the locally symmetric variety
$X^{bb}_{\mathbf{Q}}$ associated to the compatible family of {\it
prpcd}'s
$\Sigma_{(\mathbf{Q})}=\{\Sigma_{\mathbf{Q},\mathbf{R}}\}_{
\mathbf{R}}$. This is a projective $\C$-scheme having only
quotient singularities. As for the stratum $X^{bb}_{\mathbf{Q}}$,
the scheme $Y^{tor}_{\mathbf{Q}}$ depends only on the conjugacy
class of $\mathbf{Q}$ modulo $\Gamma$. Moreover, we have a
canonical projective morphism
\begin{equation}\label{eq:e-subQ}
e_{\mathbf{Q}}: Y^{tor}_{\mathbf{Q}} \to
\overline{X}^{bb}_{\mathbf{Q}}.
\end{equation}

As in \S \ref{subsect:toroidal-compact}, denote by
$\widetilde{X}^{bb}_{\mathbf{Q}}$
the $\C$-scheme whose analytic variety of $\C$-points is
$\widetilde{\Gamma}(\mathbf{M}_{\mathbf{Q},h}) \backslash
e_h(\mathbf{Q})$. This an
\'etale cover of $X^{bb}_{\mathbf{Q}}$ with Galois group
$\widetilde{\Gamma}(\mathbf{M}_{\mathbf{Q},h}) \backslash
\Gamma(\mathbf{M}_{\mathbf{Q},h})$.
Let
$Z^{tor}_{\mathbf{Q}}=\overline{(\widetilde{X}^{bb}_{\mathbf{Q}})}^{tor}_{
\Sigma_{(\mathbf{Q})}}$
be the toroidal compactification of the locally symmetric variety
$\widetilde{X}^{bb}_{\mathbf{Q}}$ associated to the same compatible
family of
{\it prpcd}'s $\Sigma_{(\mathbf{Q})}$. Then
$Z^{tor}_{\mathbf{Q}}$ is a smooth and projective scheme and
there is a morphism $c_{\mathbf{Q}}:Z^{tor}_{\mathbf{Q}}
\to Y^{tor}_{\mathbf{Q}}$ which is a finite Galois cover. Also,
if $\Sigma_{
(\mathbf{Q})}$ is fine enough, the
inverse image by $c_{\mathbf{Q}}$ of an irreducible divisor in the boundary of
$Y^{tor}_{\mathbf{Q}}$ is a smooth divisor, i.e., a disjoint union of
irreducible divisors in $Z^{tor}_{\mathbf{Q}}$.

For $i\in [\![0,r]\!]$, we let $Y^{tor}_i$ and $Z^{tor}_i$ be the
disjoint union of the $Y^{tor}_{\mathbf{Q}}$ and
$Z^{tor}_{\mathbf{Q}}$ respectively, for $\mathbf{Q}\subseteq
\mathbf{G}$ of cotype $\{i\}$, taken up to conjugation by elements
of $\Gamma$. We have natural morphisms $e_i:Y^{tor}_i \to
X^{bb}_{\geq i}$ and $c_i:Z^{tor}_i \to Y^{tor}_i$ which gives
\textbf{D2)} and \textbf{D3)}.

\begin{lemma}
The stratified scheme $\overline{X}^{bb}$ and the families of morphisms
$(e_i)_{i\in [\![0,r]\!]}$ and $(c_i)_{i\in [\![0,r]\!]}$ satisfy
\emph{Properties
\textbf{P1)}} and \emph{\textbf{P2)}} of
\emph{\S \ref{subsub:setting-for-comput}}.
\end{lemma}

\begin{proof}
Everything is a direct consequence of Proposition
\ref{prop:hereditary-for-toroidal} except the property concerning
the Stein factorization in \textbf{P2)}, which we now prove. Let
$\mathbf{Q}\subseteq \mathbf{G}$ be a parabolic $\Q$-subgroup
which is maximal or improper. A stratum $E\subset
Y^{tor}_{\mathbf{Q}}$ corresponds to a rational polyhedral cone
$\sigma\in \Sigma_{\mathbf{Q},\mathbf{R}}$. Let $F$ be a connected
component of $c_{\mathbf{Q}}^{-1}(E)$. Then $F$ is a
$\Gamma(\widetilde{\mathbf{M}}_{\mathbf{R},\ell})$-translate of the stratum of
$Z_{\mathbf{Q}}^{tor}$ that corresponds to $\sigma$, so we may
assume that $F$ corresponds
also to $\sigma \in \Sigma_{\mathbf{Q},\mathbf{R}}$. Moreover, the
image of $E$ in $\overline{X}^{bb}$ is the stratum
$X^{bb}_{\mathbf{P}}$ where $\mathbf{P}\subset \mathbf{G}$ is the
maximal or improper parabolic $\Q$-subgroup such that
$\mathbf{M}_{\mathbf{P},h}\simeq \mathbf{M}_{\mathbf{R},h}$.  Let
$F'$ be the closure of $F$ in $(e_{\mathbf{Q}}\circ
c_{\mathbf{Q}})^{-1}(X^{bb}_{\mathbf{P}})$.

That $F'$ is projective over $X^{bb}_{\mathbf{P}}$ is clear.
When $\Sigma_{\mathbf{Q},\mathbf{R}}$ is fine enough,
$F'$ is isomorphic to an irreducible, closed and constructible subset
of $\widetilde{\mathcal{B}}^c_{(\mathbf{Q},\mathbf{R}),\Sigma_{(\mathbf{Q})}}$. This isomorphism is induced by
the canonical projection of $\widetilde{\mathcal{B}}^c_{(\mathbf{Q},\mathbf{R}),\Sigma_{(\mathbf{Q})}}$ to the corner-like $\mathbf{R}$-stratum of the
toroidal compactification
$Z_{\mathbf{Q}}^{tor}$ of the locally symmetric variety
$\widetilde{X}^{bb}_{\mathbf{Q}}$. Here
$\widetilde{\mathcal{B}}^c_{(\mathbf{Q},\mathbf{R}),\Sigma_{(\mathbf{Q})}}$
is for
$\widetilde{X}^{bb}_{\mathbf{Q}}$ what
$\mathcal{B}^c_{(\mathbf{Q},\mathbf{R}),\Sigma_{(\mathbf{Q})}}$
is for $X^{bb}_{\mathbf{Q}}$.
It follows that $F$ is a torsor over
$\widetilde{\mathcal{A}}_{\mathbf{Q},\mathbf{R}}$ under a split $\Q$-torus and
$F'$ is a relative smooth torus-embedding. Here again, $\widetilde{\mathcal{A}}_{
\mathbf{Q},\mathbf{R}}$ is for $\widetilde{X}^{bb}_{\mathbf{Q}}$ what
$\mathcal{A}_{\mathbf{Q},\mathbf{R}}$ is for $X^{bb}_{\mathbf{Q}}$, i.e.,
$\widetilde{\mathcal{A}}_{\mathbf{Q},\mathbf{R}}$ is
an abelian scheme over
$\widetilde{X}^{bb}_{\mathbf{Q},\mathbf{R}}=(\widetilde{\widetilde{X}^{bb}_{\mathbf{Q}}})^{bb}_{\mathbf{R}}$, a Galois \'etale cover of the $\mathbf{R}$-stratum of
the Baily-Borel
compactification of $\widetilde{X}^{bb}_{\mathbf{Q}}$.
It follows that $F'$ is smooth and projective over $X^{bb}_{\mathbf{P}}$ and
its Stein factorization is given by
$\widetilde{X}^{bb}_{\mathbf{Q},\mathbf{R}} \to X^{bb}_{\mathbf{P}}$.
The variety of $\C$-points of
$\widetilde{X}^{bb}_{\mathbf{Q},\mathbf{R}}$
is the quotient of
$e_h(\mathbf{P})$ by the action of the arithmetic subgroup
\begin{equation}
\label{eq:arith-gr-pour-X-Q-R}
\frac{\widetilde{\Gamma}(\mathbf{M}_{\mathbf{Q},h})\cap R}{\Gamma(\mathbf{M}_{\mathbf{Q},h})\cap N_RS_R} \cap M_{R,h}=\frac{\Gamma(\mathbf{M}_{\mathbf{Q}})\cap M_{Q,h} \cap R}{\Gamma(
\mathbf{M}_{\mathbf{Q}}) \cap M_{Q,h} \cap N_{R}S_R} \cap M_{R,h}.
\end{equation}
As $\Gamma(\mathbf{M_P})\cap M_{P,h}$ is clearly contained in
\eqref{eq:arith-gr-pour-X-Q-R}, we see that
$\widetilde{X}^{bb}_{\mathbf{Q},\mathbf{R}}$ is dominated by
$\widetilde{X}^{bb}_{\mathbf{P}}$.
This proves the lemma.
\end{proof}

Next, with the notation of
Theorem \ref{thm:main-thm}, (b),
let $\Sigma'=\{\Sigma'_{\mathbf{Q},\mathbf{R}}\}_{(\mathbf{Q},\mathbf{R})\in \mathcal{M}}$ be an
extended compatible family of {\it prpcd}'s with respect to $\Gamma'$ which we assume to satisfy properties (1) and (2) as in the case of $\Sigma$. After a refinement, if necessary, we may assume that for $(\mathbf{Q},\mathbf{R})\in \mathcal{M}$ the natural isomorphism ${\rm int}(g):U_R \overset{\sim}{\to} U_{gRg^{-1}}$ sends a rational polyhedral cone of $\Sigma'_{\mathbf{Q},\mathbf{R}}$ inside a rational polyhedral cone of $\Sigma_{g\mathbf{Q}g^{-1},g\mathbf{R}g^{-1}}$. We let $e'_i:Y'^{\,tor}_i\to X'^{\,bb}_{\geq i}$
and $c'_i:Z'^{\,tor}_i \to Y'^{\,tor}_i$ denote the morphisms constructed as before.
Then, $g\in \mathbf{G}(\Q)$ induces morphisms
$g:Y'^{\,tor}_i \to Y^{tor}_i$ and $g:Z'^{\,tor}_i \to Z^{tor}_i$ making the diagram analogous to
\eqref{diag:for-naturality-comput-E-X-rajoute}
commutative.
One also checks that the properties at the end of
\S \ref{subsub:setting-for-comput} are satisfied.

We are now in position to apply the results of
\S \ref{subsection:compute-E-X}.
We respectively denote by $T^{tor}$, $\mathcal{X}^{tor}$, $\mathcal{T}^{tor}$ and
$\mathcal{Y}^{tor}$
the diagrams of schemes $T$ (\S\ref{subsub:T}), $\mathcal{X}$ (\S\ref{subsub:X}), $\mathcal{T}$ (\S\ref{subsub:calT})
and $\mathcal{Y}$ (\S\ref{subsub:Y})
associated to the stratified scheme $\overline{X}^{bb}$ ($X$ in
\S \ref{subsection:compute-E-X}) and the
morphisms $e_i:Y^{tor}_i \to X^{bb}_{\geq i}$ and $c_i:Z^{tor}_i\to Y^{tor}_i$.
Likewise, denote by $T'^{\,tor}$, $\mathcal{X}'^{\,tor}$ and
$\mathcal{Y}'^{\,tor}$
the corresponding diagrams of schemes for $\overline{X'}^{bb}$;
these play the role of $\check{T}$, $\check{\mathcal{X}}$ and
$\check{\mathcal{Y}}$ in \S
\ref{subsection:compute-E-X}. We also write $\beta_{\overline{X}^{bb}}$ and
$\beta_{\overline{X'}^{\,bb}}$
instead of $\beta_{\overline{X}^{bb},(X^{bb}_i)_i}$ and
$\beta_{\overline{X'}^{\,bb},(X'^{\,bb}_i)_i}$ (from \S\ref{par:mot-beta-x-s}).
These are motives over $T^{tor}$ and $T'^{\,tor}$ respectively.


\subsubsection{The diagram of schemes $\widetilde{T}^{tor}$}
\label{subsub:wTtor} For $\emptyset \neq I \subset [\![0,r]\!]$,
let $\mathscr{P}(I)$ be the set of pairs $(\mathbf{Q},\mathbf{R})$
with $\mathbf{Q}$ a parabolic $\Q$-subgroup of cotype $\{{\rm
min}(I)\}$ and
$\mathbf{R}$ a parabolic $\Q$-subgroup of
$\mathbf{M}_{\mathbf{Q},h}$ conjugate (as a sub-quotient of
$\mathbf{G}$) to the image of $\mathbf{P}_{[\![1,r]\!]- I}$ in
$\mathbf{M}_{\mathbf{P}_{[\![1,r]\!]- \{{\rm min}(I)\}},h}$. Given
such $(\mathbf{Q},\mathbf{R})$, let $\mathbf{R'}\supset\mathbf{R}$
be the maximal or improper parabolic $\Q$-subgroup of
$\mathbf{M}_{\mathbf{Q},h}$ such that
$\mathbf{M}_{\mathbf{R},h}\simeq \mathbf{M}_{\mathbf{R'},h}$ (i.e., $\mathbf{R}$ is subordinate to $\mathbf{R'}$).
We denote by
$\mathbf{E}_{\mathbf{Q},\mathbf{R}}$ the inverse image of
$\mathbf{R}$ by the natural projection from $\mathbf{Q}$ to (the
quotient by a finite normal subgroup of)
$\mathbf{M}_{\mathbf{Q},h}$. This is a parabolic $\Q$-subgroup of
cotype $I$. It determines the pair $(\mathbf{Q},\mathbf{R})$ as
follows: $\mathbf{Q}$ is the unique maximal or improper parabolic $\Q$-subgroup of
cotype $\{{\rm min}(I)\}$ that contains $\mathbf{E_{Q,R}}$, and
$\mathbf{R}$ is the image of $\mathbf{E_{Q,R}}$ in
$\mathbf{M}_{\mathbf{Q},h}$.\footnote{$\mathscr{P}(I)$ is also the
set of parabolic $\Q$-subgroups $\mathbf{E}$ of cotype $I$,
 for we can associate to such
$\mathbf{E}$ the unique pair $(\mathbf{Q_E},\mathbf{R_E})$ such
that $\mathbf{E}=\mathbf{E_{Q_E,R_E}}$. We feel that our choice is
better suited to the geometry, being adapted to the diagram of
schemes $\widetilde{T}^{tor}(I)$ (constructed below), whose
connected components are naturally indexed by the elements of
$\mathscr{P}(I)$.} Clearly,
$\mathbf{E}_{\mathbf{Q},\mathbf{R}}=\mathbf{N_Q}\mathbf{S_Q}
\widetilde{\mathbf{M}}_{\mathbf{Q},\ell}\mathbf{R}$.\footnote{Strictly
speaking, $\mathbf{N_Q}$ is a subgroup of $\mathbf{G}$ and
$\mathbf{S_Q}\widetilde{\mathbf{M}}_{\mathbf{Q},\ell}\mathbf{R}$
is a subgroup of the Levi quotient $\mathbf{L_Q}$. However, we can
choose a lift $\mathbf{L_Q}(x)\subset\mathbf{Q}$ (i.e., a Levi
subgroup), as in
\S\ref{subsect:The Borel-Serre compactifications}, and define $\mathbf{E_{Q,R}}(x)=\mathbf{N_Q}
\mathbf{S_Q}(x)\widetilde{\mathbf{M}}_{\mathbf{Q},\ell}(x)
\mathbf{R}(x) \subset\mathbf{G}$. But $\mathbf{E_{Q,R}}(x)$ is in
fact independent of the choice of $x$.} Similarly, we put
$\mathbf{K}_{\mathbf{Q},\mathbf{R}}=\mathbf{N_Q}\mathbf{S_Q}
\widetilde{\mathbf{M}}_{\mathbf{Q},\ell}\mathbf{R'_h}$, where
$\mathbf{R}'_h$ is the inverse image of
$\mathbf{M}_{\mathbf{R}',h}$ by the projection of $\mathbf{R}'$ to (the quotient by a finite normal subgroup of)
$\mathbf{M_{R'}}$. We obtain a commutative diagram
\begin{equation}
\label{eq:emodK}
\begin{matrix}
\mathbf{K_{\mathbf{Q,R}}} & \hookrightarrow & \mathbf{E_{Q,R}} & \hookrightarrow & \mathbf{Q} \\
\downarrow && \downarrow && \downarrow \\
\mathbf{R}'_h & \hookrightarrow  & \mathbf{R} & \hookrightarrow  &
\mathbf{M}_{\mathbf{Q},h}
\end{matrix}
\end{equation}
with cartesian squares. In particular, $\mathbf{K}_{\mathbf{Q},
\mathbf{R}}$ is a normal subgroup of
$\mathbf{E}_{\mathbf{Q},\mathbf{R}}$ that is determined by
$\mathbf{R}'$, and
\begin{equation}
\label{eq:EmodK-RmodR'h}
\mathbf{E}_{\mathbf{Q},\mathbf{R}}/\mathbf{K}_{\mathbf{Q},\mathbf{R}}
\simeq\mathbf{R}/\mathbf{R}'_h.
\end{equation}

Let $(\mathbf{Q_1},\mathbf{R_1})$ and $(\mathbf{Q_2},\mathbf{R_2})$
be two elements of $\mathscr{P}(I)$. We set
$$[(\mathbf{Q_1},\mathbf{R_1}),(\mathbf{Q_2},\mathbf{R_2})] =
\{ \gamma\in \mathbf{G}(\Q):
\gamma\mathbf{Q_1}\gamma^{-1}=\mathbf{Q_2}\text{ and }\gamma\mathbf{R_1}
\gamma^{-1}= \mathbf{R_2}\}.$$
This is the set of $\gamma$'s for which
$\gamma \mathbf{E_{Q_1,R_1}}\gamma^{-1}=\mathbf{E_{Q_2,R_2}}$.  For
$\gamma,\gamma'\in[(\mathbf{Q_1},\mathbf{R_1}),(\mathbf{Q_2},\mathbf{R_2})]$,
we write
$\gamma\sim\gamma'$
when there exists $\delta_1\in \mathbf{K_{Q_1,R_1}}(\Q)$
such that $\gamma'=\gamma\delta_1$
(equivalently, when  there exists $\delta_2\in \mathbf{K_{Q_2,R_2}}(\Q)$
such that $\gamma'=\delta_2\gamma$).
This defines an equivalence relation on $[(\mathbf{Q_1},\mathbf{R_1}),
(\mathbf{Q_2},\mathbf{R_2})]$ that is compatible with
multiplication in $\mathbf{G}(\Q)$. We make the set
$\mathscr{P}(I)$ into a groupoid by setting
$$
\hom_{\mathscr{P}(I)}((\mathbf{Q_1},\mathbf{R_1}),(\mathbf{Q_2},\mathbf{R_2}))=
[(\mathbf{Q_1},\mathbf{R_1}),(\mathbf{Q_2},\mathbf{R_2})]/\sim.
$$
As $\mathbf{E_{Q,R}}$ is parabolic, it is its own normalizer.
Thus
$[(\mathbf{Q},\mathbf{R}),(\mathbf{Q},\mathbf{R})]=\mathbf{E_{Q,R}}(\Q)$,
and by construction
\begin{equation}\label{eq:end=EmodK}
{\rm end}_{\mathscr{P}(I)}(\mathbf{Q},\mathbf{R})=\mathbf{E_{Q,R}}(\Q)/
\mathbf{K_{Q,R}}(\Q).\footnote{Though we have the isomorphism \eqref{eq:EmodK-RmodR'h},
the canonical morphism $\mathbf{E_{Q,R}}(\Q)/
\mathbf{K_{Q,R}(\Q)} \to \mathbf{R}(\Q)/\mathbf{R}'_h(\Q)$ need not be an isomorphism.}
\end{equation}
The group $\mathbf{G}(\Q)$ acts on $\mathscr{P}(I)$ by conjugation: an element
$b\in \mathbf{G}(\Q)$ determines an endofunctor ${\rm int}(b)$ of
$\mathscr{P}(I)$, which sends a pair $(\mathbf{Q},\mathbf{R})$ to
$(b\mathbf{Q}b^{-1},b\mathbf{R}b^{-1})$ and a morphism
$\gamma\in \hom_{\mathscr{P}(I)}((\mathbf{Q_1},\mathbf{R_1}),(\mathbf{Q_2},
\mathbf{R_2}))$ to
$b\gamma b^{-1}$.

We will rather be interested in the sub-groupoid
$\mathscr{P}_{\Gamma}(I) \subset \mathscr{P}(I)$. Objects in
$\mathscr{P}_{ \Gamma}(I)$ are the same as in $\mathscr{P}(I)$.
However,
$\hom_{\mathscr{P}_{\Gamma}(I)}((\mathbf{Q_1},\mathbf{R_1}),(\mathbf{Q_2},
\mathbf{R_2}))$ is the set of equivalence classes of $\gamma\in
\Gamma$ such that $\gamma\mathbf{Q_1}\gamma^{-1}=\mathbf{Q_2}$ and
$\gamma\mathbf{R_1}\gamma^{-1}=\mathbf{R_2}$. Immediate from the
construction, one sees:

\begin{lemma}
\label{lemme:P-sub-Gamma}

\begin{enumerate}

\item[(a)] $\mathscr{P}_{\Gamma}(\{i\})$ is a discrete category whose
objects are pairs $(\mathbf{Q},\mathbf{M}_{\mathbf{Q},h})$ with $\mathbf{Q}$ a
parabolic $\Q$-subgroup of $\mathbf{G}$ of cotype $\{i\}$. Two
pairs $(\mathbf{Q},\mathbf{M}_{\mathbf{Q},h})$ and
$(\mathbf{Q'},\mathbf{M}_{\mathbf{Q'},h})$ are linked by an arrow if and only if
$\mathbf{Q}$ and $\mathbf{Q'}$ are conjugate by $\Gamma$. In
particular, $\mathscr{P}_{\Gamma}(\{0\})$ is the terminal
category, with only one object and one arrow.

\item[(b)] For $(\mathbf{Q},\mathbf{R})\in \mathscr{P}_{\Gamma}
(I)$, we have ${\rm
end}_{\mathscr{P}_{\Gamma}(I)}(\mathbf{Q},\mathbf{R})=
\Gamma(\mathbf{E_{Q,R}}/\mathbf{K_{Q,R}}) \simeq
\Gamma(\mathbf{R}/\mathbf{R}'_h)$, where we have set
$\Gamma(\mathbf{R}/\mathbf{R}'_h)=(\Gamma(\mathbf{M}_{\mathbf{Q},h})\cap
R)/ (\Gamma(\mathbf{M}_{\mathbf{Q},h})\cap R'_h)$.
The connected components of the groupoid $\mathscr{P}_{\Gamma}(I)$
are parametrized by the $\Gamma$-conjuguacy classes of parabolic
$\Q$-subgroups of cotype $I$.

\item[(c)] When $g\Gamma'g^{-1}\subset\Gamma$, the automorphism ${\rm int}(g):
\mathscr{P}(I) \to \mathscr{P}(I)$ takes
$\mathscr{P}_{\Gamma'}(I)$ into $\mathscr{P}_{\Gamma}(I)$.

\end{enumerate}

\end{lemma}

\begin{remark}\label{g-gprime}
We establish the convention that whenever ``\,$\Gamma'$\,''
appears in the sequel, it occurs in the context of
$g\Gamma'g^{-1}\subset\Gamma$.
\end{remark}

Next, let $\emptyset \neq J\subset I \subset [\![0,r]\!]$. Given
$(\mathbf{Q},\mathbf{R})\in \mathscr{P}(I)$, there is a unique
$(\mathbf{F},\mathbf{H})\in \mathscr{P}(J)$ such that
$\mathbf{E_{Q,R}}\subset \mathbf{E_{F,H}}$. We then have
$\mathbf{K_{Q,R}}\subset \mathbf{K_{F,H}}$. Also, we have an inclusion
$$
[(\mathbf{Q_1},\mathbf{R_1}),(\mathbf{Q_2},\mathbf{R_2})]\subset
[(\mathbf{F_1},\mathbf{H_1}),(\mathbf{F_2},\mathbf{H_2})]
$$
when there are two such sets of data.  This defines a mapping
$$
\hom_{\mathscr{P}(I)}((\mathbf{Q_1},\mathbf{R_1}),(\mathbf{Q_2},
\mathbf{R_2})) \to \hom_{\mathscr{P}(J)}
((\mathbf{F_1},\mathbf{H_1}),(\mathbf{F_2},\mathbf{H_2})).
$$
Thus, we have a functor
\begin{equation}\label{eq:tJI}
\mathbf{t}_{J\subset I}:\mathscr{P}(I) \to \mathscr{P}(J),
\end{equation}
which takes a pair $(\mathbf{Q},\mathbf{R})$ to the unique
$(\mathbf{F},\mathbf{H})$ such that $\mathbf{E}_{\mathbf{Q},\mathbf{R}}
\subset \mathbf{E}_{\mathbf{F},\mathbf{H}}$.

It is clear that $\mathbf{t}_{J\subset I}$ takes the sub-groupoid
$\mathscr{P}_{\Gamma}(I)$ of $\mathscr{P}(I)$ into
$\mathscr{P}_{\Gamma}(J)$. We also write $\mathbf{t}_{J\subset
I}:\mathscr{P}_{\Gamma}(I) \to \mathscr{P}_{\Gamma}(J)$ for the
induced functor. We leave the verification of the following lemma
to the reader. With $\mathcal{P}^*$ as in \S\ref{subsub:T}:

\begin{lemma}
The family of functors $\{\mathbf{t}_{J\subset
I}:\mathscr{P}_{\Gamma}(I) \to \mathscr{P}_{\Gamma}(J)\}_{J\subset
I}$ defines a functor $\mathscr{P}_{\Gamma}$ from
$\mathcal{P}^*([\![0,r]\!])^{\rm op}$ to the category of groupoids.
Moreover, the family $\{{\rm int}(g):\mathscr{P}_{\Gamma'}(I) \to
\mathscr{P}_{\Gamma}(I)\}_{I}$ defines a natural transformation
${\rm int}(g):\mathscr{P}_{\Gamma'}\to \mathscr{P}_{\Gamma}$.

\end{lemma}

\begin{remark}
We gather here some facts about groupoids and their representations. Let
$\mathcal{G}$ be a small groupoid and $\mathcal{C}$ a category. A
\emph{representation} of $\mathcal{G}$ in $\mathcal{C}$ is a functor $F:\mathcal{G} \to \mathcal{C}$.
By the quotient $\mathcal{G}\backslash F$, we mean the colimit (if it exists)
of the functor $F$.
In the case of $\mathcal{G} = \mathscr{P}_{\Gamma}(I)$ and
$\mathcal{C}=\Sch/\C$,
to give a representation is equivalent to
giving a representation of $\Gamma$ on a scheme $W$ and specifying for every
$(\mathbf{Q},\mathbf{R})\in \mathscr{P}_{\Gamma}(I)$
an open and closed subscheme $W(\mathbf{Q},\mathbf{R})$ such that:
\begin{itemize}

\item $W=\coprod_{(\mathbf{Q},\mathbf{R})\in {\rm ob}(\mathscr{P}_{\Gamma}
(I))} W(\mathbf{Q},\mathbf{R})$,

\item the automorphism
$\gamma:W \to W$ takes $W(\mathbf{Q},\mathbf{R})$ to
$W(\gamma\mathbf{Q}\gamma^{-1},\gamma\mathbf{R}\gamma^{-1})$
for every $\gamma\in \Gamma$,

\item the action of $\Gamma(\mathbf{E}_{\mathbf{Q},\mathbf{R}})$ on
$W(\mathbf{Q},\mathbf{R})$ factors through $\Gamma(\mathbf{E_{Q,R}}/
\mathbf{K_{Q,R}})$.
\end{itemize}
When the $W(\mathbf{Q},\mathbf{R})$'s are to be connected (as is
the case for the $\mathcal{B}_I(\mathbf{Q},\mathbf{R})$'s below), they
are uniquely determined. Indeed, $W(\mathbf{Q},\mathbf{R})$ is
then the unique connected component of $W$ having
$\Gamma(\mathbf{E_{Q,R}})$ for stabilizer. In the sequel, we will
often say that $\mathscr{P}_{\Gamma}(I)$ acts on a scheme $W$
without specifying the components $W(\mathbf{Q},\mathbf{R})$
(especially when these schemes are connected).
\end{remark}

Next, we define a diagram of schemes
$\mathcal{B}_I$ indexed by the groupoid $\mathscr{P}_{\Gamma}(I)$, i.e., a representation of that groupoid, as follows.
Given an object $(\mathbf{Q},\mathbf{R})$
of $\mathscr{P}_{\Gamma}(I)$, we let
$\mathcal{B}_I(\mathbf{Q},\mathbf{R})=\mathcal{B}^c_{
(\mathbf{Q},\mathbf{R}),\Sigma_{(\mathbf{Q})}}$ as in Proposition
\ref{prop:hereditary-for-toroidal}. Let $\mathbf{R'}\subset \mathbf{M}_{
\mathbf{Q},h}$ be the maximal or improper parabolic $\Q$-subgroup
containing $\mathbf{R}$ and such that $\mathbf{M}_{\mathbf{R},h}\simeq
\mathbf{M}_{\mathbf{R'},h}$.
The scheme $\mathcal{B}_I(\mathbf{Q},\mathbf{R})$
admits an action of the arithmetic group
$[\Gamma(\mathbf{M}_{\mathbf{Q},h})](\widetilde{\mathbf{M}}_{\mathbf{R'},\ell}\,|\,\mathbf{R})$, given by
\eqref{eq:useful-arthmetic-subgr} with $\Gamma(\mathbf{M}_{\mathbf{Q},h})$ instead of $\Gamma$.
Recall that the latter was defined as
$[\Gamma(\mathbf{M}_{\mathbf{Q},h})](\widetilde{\mathbf{M}}_{\mathbf{R'},\ell})\cap R/N_{R'}A_{R'}M_{R',h}$ with
$[\Gamma(\mathbf{M}_{\mathbf{Q},h})](\widetilde{\mathbf{M}}_{\mathbf{R'},\ell})$
the image of $\Gamma(\mathbf{M}_{\mathbf{Q},h})\cap R'$
by the projection from $R'$ to (the quotient by a finite normal subgroup of)
$\widetilde{M}_{R',\ell}$. (As $\Gamma$ is neat, one may replace
$A_{R'}$ with $S_{R'}$.)
Thus
$[\Gamma(\mathbf{M}_{\mathbf{Q},h})](\widetilde{\mathbf{M}}_{\mathbf{R'},\ell}\,|\,\mathbf{R})$
is simply the image of $\Gamma(\mathbf{E_{Q,R}})$ by the projection
from $E_{Q,R}$ to (the quotient by a finite normal subgroup of)
$\widetilde{M}_{R',\ell}$. This shows that
\begin{equation}
\label{eq:groupoid-acts-B-Q-R}
[\Gamma(\mathbf{M}_{\mathbf{Q},h})](\widetilde{\mathbf{M}}_{\mathbf{R'},\ell}\,|\,\mathbf{R})
\simeq\Gamma(\mathbf{E_{Q,R}}/\mathbf{K_{Q,R}}).
\end{equation}
In other words, the group ${\rm end}_{\mathscr{P}_{\Gamma}(I)}(\mathbf{Q},
\mathbf{R})$ acts on $\mathcal{B}_I(\mathbf{Q},\mathbf{R})$.

Moreover, given two objects $(\mathbf{Q_1},\mathbf{R_1})$
and $(\mathbf{Q_2},\mathbf{R_2})$ of $\mathscr{P}_{\Gamma}(I)$ and
$\gamma\in \Gamma$ such that
$\gamma\mathbf{Q_1}\gamma^{-1}=\mathbf{Q_2}$ and
$\gamma\mathbf{R_1}\gamma^{-1}=\mathbf{R_2}$, there is an induced isomorphism
(also denoted $\gamma$)
$\gamma:\mathcal{B}_I(\mathbf{Q_1},\mathbf{R_1})\to\mathcal{B}_I(\mathbf{Q_2},
\mathbf{R_2})$.
Indeed, $\gamma$ induces an isomorphism $\gamma:X^{bb}_{\mathbf{Q_1}}
\to X^{bb}_{\mathbf{Q_2}}$ which is compatible with the isomorphism
of $\Q$-groups ${\rm int}(\gamma):\mathbf{M}_{\mathbf{Q_1},h}\simeq
\mathbf{M}_{\mathbf{Q_2},h}$.
Our claim follows, as the construction of the toroidal compactification is
canonical with respect to the group, the arithmetic subgroup and the family
of prpcd's. From \eqref{eq:groupoid-acts-B-Q-R}, we see that
$\gamma:\mathcal{B}_I(\mathbf{Q_1},\mathbf{R_1})\to \mathcal{B}_I(\mathbf{Q_2},
\mathbf{R_2})$
depends only on the class of $\gamma$ in $\hom_{\mathscr{P}_{\Gamma}(I)}
((\mathbf{Q_1},\mathbf{R_1}),(\mathbf{Q_2},\mathbf{R_2}))$.

\begin{lemma}
\label{lemma:universal-cover-cal-B-I}
The assignment $(\mathbf{Q},\mathbf{R})\rightsquigarrow\mathcal{B}_I(
\mathbf{Q},\mathbf{R})=\mathcal{B}^c_{(\mathbf{Q},\mathbf{R}),\Sigma_{
(\mathbf{Q})}}$ defines a covariant functor
$\mathcal{B}_I:\mathscr{P}_{\Gamma}(I)\to \Sch/\C$.
Moreover, there is a morphism of diagrams of schemes
$$(\mathcal{B}_I,\mathscr{P}_{\Gamma}(I))\to T^{tor}(I)$$
that identifies
$T^{tor}(I)^0$
with the quotient $\mathscr{P}_{\Gamma}(I)\backslash \mathcal{B}_I$.
\end{lemma}

\begin{proof}
We only explain the last claim in the statement.
Recall from \eqref{eq:final-def-T(I)} that
$T^{tor}(I)=\bigcap_{i\in I} \overline{e_{{\rm min}(I)}^{-1}(X^{bb}_i)}$
where $e_{{\rm min}(I)}$ is the projection of $Y^{tor}_{{\rm min}(I)}$ onto $\overline{X}^{bb}_{\geq {\rm min}(I)}$.
Recall also that $T^{tor}(I)^0$ is the inverse image of
$X^{bb}_{{\rm max}(I)}$ along the natural morphism $T^{tor}(I) \to \overline{X}^{bb}$. This is a dense open subscheme of $T^{tor}(I)$ which is given by
$$\left(\bigcap_{i\in I-\{{\rm max}(I)\}}
\overline{e_{{\rm min}(I)}^{-1}(X^{bb}_i)}\right) \bigcap
e_{{\rm min}(I)}^{-1}\left(X^{bb}_{{\rm max}(I)}\right).$$
The claim follows now from Proposition
\ref{prop:intersection-closure-strata-toroidal-comp}.
\end{proof}

For $(\mathbf{Q},\mathbf{R})\in \mathscr{P}_{\Gamma}(I)$, we
denote by $T^{tor}(\mathbf{Q},\mathbf{R})$ the connected component
of $T^{tor}(I)$ that is dominated by
$\mathcal{B}_I(\mathbf{Q},\mathbf{R})$. Of course,
$T^{tor}(\mathbf{Q},\mathbf{R})$ depends only on the connected
component of $(\mathbf{Q},\mathbf{R})$ in
$\mathscr{P}_{\Gamma}(I)$.

We now construct the diagram of schemes $\widetilde{T}^{tor}$. Let
$\emptyset \neq I\subset [\![0,r]\!]$. We bring in the
stratification $\mathcal{R}(I)$ on $Y^{tor}_{{\rm min}(I)}$ from
\S\ref{subsub:setting-for-comput}. A subset $V\subset T^{tor}(I)$
is called an $\mathcal{R}(I)$-\emph{star} if there exists an
$\mathcal{R}(I)$-stratum $E$, called the center of $V$, such that
$V$ is the union of the $\mathcal{R}(I)$-strata $F$ satisfying
$E\subset \overline{F}$.\footnote{This notion makes sense for
every stratified topological space.} We write $V=V(E)$; $E$ is
uniquely determined by $V$, equaling the smallest
$\mathcal{R}(I)$-stratum (with respect to $\preceq$) in $V$.
It is clear that an
$\mathcal{R}(I)$-star is an open $\mathcal{R}(I)$-constructible
subset of $T^{tor}(I)$, and that the latter is covered by
$\mathcal{R}(I)$-stars. Moreover, if the extended compatible
family of {\it prpcd}'s
$\Sigma=\{\Sigma_{\mathbf{Q},\mathbf{R}}\}$ is fine enough, which
we assume, the intersection $V(E_1)\cap V(E_2)$ of two
$\mathcal{R}(I)$-stars, if non-empty, is the $\mathcal{R}(I)$-star
$V(E_{1,2})$, where $E_{1,2}$ is the smallest stratum whose
closure contains both $E_1$ and $E_2$.

It follows from Lemma \ref{lemma:dominance-elements-A-I} that an
$\mathcal{R}(I)$-stratum $F$ in $T^{tor}(I)$ meets the open subset
$T^{tor}(I)^0$, and the intersection $F\cap T^{tor}(I)^0$ is dense
in $F$. For an $\mathcal{R}(I)$-star $V\subset T^{tor}(I)$, the
intersection $V^0 = V\cap T^{tor}(I)^0$ will be called, by abuse
of language, an \emph{$\mathcal{R}(I)$-star} in $T^{tor}(I)^0$. If
$\Sigma$ is fine enough, the inverse image of $V^0$ in
$\mathcal{B}_I$ is a disjoint union of copies $V^0_{\alpha}$ of
$V^0$ which are permuted by the groupoid $\mathscr{P}_{\Gamma}(I)$.
For each copy $V^0_{\alpha}$, choose a copy $V_{\alpha}$ of $V$.
Now, let $V_1, \, V_2 \subset T^{tor}(I)$ be two
$\mathcal{R}(I)$-stars. Assume that $V_3=V_1\cap V_2$ is not
empty, and hence an $\mathcal{R}(I)$-star. Then to each connected
component $V_{1,\alpha}$ corresponds a unique connected component
$V_{2,\alpha}$ such that $V^0_{3,\alpha}=V^0_{1,\alpha} \cap
V^0_{2,\alpha}$ is not empty and hence isomorphic to $V_3^0$.
Gluing the various $V_{1,\alpha}$ and $V_{2,\alpha}$ along
$V_{3,\alpha}$ yields a scheme $\widetilde{T}^{tor}(I)$
on which the groupoid
$\mathscr{P}_{\Gamma}(I)$ acts naturally.
Given
$(\mathbf{Q},\mathbf{R})\in \mathscr{P}_{\Gamma}(I)$, we let $\widetilde{T}^{tor}(\mathbf{Q},\mathbf{R})$ denote the connected
component of $\widetilde{T}^{tor}(I)$ that contains
$\mathcal{B}_I(\mathbf{Q},\mathbf{R})$ as a dense open
subset.

>From the construction,
we have a cartesian square of diagrams of
schemes
$$\xymatrix@C=1.5pc@R=1.5pc{(\mathcal{B}_I,\mathscr{P}_{\Gamma}(I)) \ar[r] \ar[d] &
(\widetilde{T}^{tor}(I),\mathscr{P}_{\Gamma}(I)) \ar[d]^-{u_I} \\
T^{tor}(I)^0 \ar[r] & T^{tor}(I).}$$
Thus, $\widetilde{T}^{tor}(I)$ is a Zariski
locally trivial covering of $T^{tor}(I)$ which extends the
covering $\mathcal{B}_I$ of $T^{tor}(I)^0$.
Using Lemma \ref{lemma:universal-cover-cal-B-I}, we thus have an isomorphism
\begin{equation}
\label{eq:rajoute-quotient-tilde-T-tor=T-tor}
\mathscr{P}_{\Gamma}(I)\backslash \widetilde{T}^{tor}(I)\simeq
{T}^{tor}(I)
\end{equation}
induced by $u_I$. Moreover, we have:

\begin{proposition}
\label{prop:diagram-schemes-I-P-Gamma-tilde-T}

\begin{enumerate}

\item[(a)] The assignment $I \rightsquigarrow (\widetilde{T}^{tor}(I),
\mathscr{P}_{\Gamma}(I))$ extends canonically to a contravariant functor from
$\mathcal{P}^*([\![0,r]\!])$ to $\Dia(\Sch/\C)$.
Moreover, we have a natural morphism in $\Dia(\Dia(\Sch/\C))$:
$$u:\xymatrix@C=1.7pc{(\widetilde{T}^{tor},\mathcal{P}^*([\![0,r]\!])^{\rm op})\ar[r] &
 (T^{tor},\mathcal{P}^*([\![0,r]\!])^{\rm op})}$$
which is the identity on the indexing categories.

\item[(b)] There are canonical
morphisms of diagrams of schemes
$$g:(\widetilde{T}'^{\,tor}(I),\mathscr{P}_{\Gamma'}(I)) \to
(\widetilde{T}^{tor} (I),\mathscr{P}_{\Gamma}(I)),$$
which are given by ${\rm int}(g)$ on the indexing categories and
which are natural in $I\in \mathcal{P}^*([\![0,r]\!])$.
Moreover, we have a commutative square in $\Dia(\Dia(\Sch/\C))$:
$$\xymatrix{(\widetilde{T}'^{\,tor},\mathcal{P}^*([\![0,r]\!])^{\rm op})
\ar[r]^-g \ar[d]_-{u'} & (\widetilde{T}^{tor},\mathcal{P}^*([\![0,r]\!])^{
\rm op}) \ar[d]^-{u} \\
(T'^{tor},\mathcal{P}^*([\![0,r]\!])^{\rm op}) \ar[r]^-g & (T^{tor},
\mathcal{P}^*([\![0,r]\!])^{\rm op}).\!}$$

\end{enumerate}

\end{proposition}

\begin{proof}
We show part (a) and leave the verification of
(b) to the reader. For $\emptyset \neq J \subset I$, we
need to define a morphism of diagrams of schemes
$\widetilde{T}^{tor}(J\subset I)$. On the indexing categories, this
morphism is given by the functor $\mathbf{t}_{J\subset I}$ we have
already defined \eqref{eq:tJI}. We also want this morphism to be compatible with the morphism
$T^{tor}(J\subset I)$ we already defined in
\S \ref{subsub:T}, i.e., that
$u_J\circ \widetilde{T}^{tor}(J\subset I)=T^{tor}(J\subset I)\circ u_I$.

First, note that the
morphism $\widetilde{T}^{tor}(I) \to T^{tor}(I)$, together with
$\mathcal{R}(I)$, gives rise to a stratification
$\widetilde{\mathcal{R}}(I)$ of $\widetilde{T}^{tor}(I)$: a subset
of $\widetilde{T}^{tor}(I)$ is an
$\widetilde{\mathcal{R}}(I)$-stratum if and only if it is a
connected component of the inverse image of an $\mathcal{R}
(I)$-stratum of $T^{tor}(I)$.
Moreover,
$u_I:\widetilde{T}^{tor}(I)\to T^{tor}(I)$ takes an
$\widetilde{\mathcal{R}}(I)$-stratum isomorphically to its image, an
$\mathcal{R}(I)$-stratum of $T^{tor}(I)$. In
\S \ref{subsub:calT}, we introduced the ordered set $A(I)$ of
irreducible, closed and $\mathcal{R}(I)$-constructible subsets of $T^{tor}(I)$. Similarly, let $\widetilde{A}(I)$ be the set of
irreducible, closed and $\widetilde{\mathcal{R}}(I)$-constructible subsets of $\widetilde{T}^{tor}(I)$. (Clearly, every element of $\widetilde{A}(I)$ is the closure of a unique $\widetilde{\mathcal{R}}(I)$-stratum, so
there is a non-decreasing bijection between $\widetilde{A}(I)$ and the set of
$\widetilde{\mathcal{R}}(I)$-strata in $\widetilde{T}^{tor}(I)$.)
As for $A(I)$, elements of $\widetilde{A}(I)$ will be denoted using greek letters $\alpha$, $\beta$, etc, and the corresponding closed subsets will be denoted by
$\widetilde{\mathcal{T}}^{tor}(I,\alpha)$, $\widetilde{\mathcal{T}}^{tor}(I,\beta)$, etc.

Now, for the morphism $T^{tor}(J\subset I)$,
there is a non-decreasing map ${\rm s}_{J\subset I}:A(I) \to A(J)$
such that $T^{tor}(J\subset I)$ maps $\mathcal{T}^{tor}(I,\alpha)$ inside
$\mathcal{T}^{tor}(J,{\rm s}_{J\subset I}(\alpha))$ for all $\alpha \in A(I)$
(see Proposition \ref{prop:for-defining-s-J-I-and-funct}). We will construct a non-decreasing map $\widetilde{\rm s}_{J\subset I}:\widetilde{A}(I) \to \widetilde{A}(J)$ which is compatible with ${\rm s}_{J\subset I}$, i.e., for every $\beta\in \widetilde{A}(I)$ and
$\alpha\in A(I)$ such that $u_I(\widetilde{\mathcal{T}}^{tor}(I,\beta))=\mathcal{T}^{tor}(I,\alpha)$, we have
$u_J(\widetilde{\mathcal{T}}^{tor}(J,\widetilde{\rm s}_{J\subset I}(\beta)))=
\mathcal{T}^{tor}(J,{\rm s}_{J\subset I}(\alpha))$.
As $u_I$ and $u_J$ are Zariski locally trivial covers and induces isomorphisms between strata, it is clear that
${\rm s}_{J\subset I}$ determines a unique morphism
$\widetilde{T}^{tor}(J\subset I)$, compatible with $T^{tor}(J\subset I)$ and which maps $\widetilde{\mathcal{T}}^{tor}(I,\alpha)$ inside
$\widetilde{\mathcal{T}}^{tor}(J,\widetilde{\rm s}_{J\subset I}(\alpha))$
for all $\alpha \in \widetilde{A}(I)$.

The
$\widetilde{ \mathcal{R}}(I)$-strata of $\widetilde{T}^{tor}(I)$
are in a one-to-one correspondence with the rational polyhedral
cones in $\coprod_{(\mathbf{Q},\mathbf{R})\in {\rm
ob}(\mathscr{P}_{\Gamma}(I))}
\Sigma^{\circ}_{\mathbf{Q},\mathbf{R}}$. Let $\sigma\in
\Sigma^{\circ}_{\mathbf{Q},\mathbf{R}}$, and
$(\mathbf{F},\mathbf{H})=\mathbf{t}_{J\subset
I}(\mathbf{Q},\mathbf{R})$. Denote by $\mathbf{R'}$ the maximal
or improper parabolic $\Q$-subgroup of $\mathbf{M}_{\mathbf{Q},h}$ to which
$\mathbf{R}$ is subordinate. Also, let $\mathbf{H'}$ (resp.
$\mathbf{H''}$) denote the maximal or improper
parabolic $\Q$-subgroup of $\mathbf{M}_{\mathbf{F},h}$ to which
$\mathbf{H}$ (resp. the image of $\mathbf{R}$ in
$\mathbf{M}_{\mathbf{F},h}$) is subordinate. Let $\sigma'$ be the
unique rational polyhedral cone of
$\Sigma_{\mathbf{F},\mathbf{H''}}$ that contains the image of
$\sigma$ under $U_{R'}\to U_{H''}$. The morphism
$\widetilde{\rm s}_{J\subset I}$ is determined as follows. It takes the closure of the stratum
corresponding to $\sigma\in
\Sigma^{\circ}_{\mathbf{Q},\mathbf{R}}$ into the closure (in
$Y^{tor}_{\mathbf{F}}$) of the stratum corresponding to the
rational polyhedral cone $\sigma''\in
\Sigma^{\circ}_{\mathbf{F},\mathbf{H}}$ that is open in
$\overline{\sigma'} \cap U_{H'}$.

Clearly, $\widetilde{\rm s}_{J\subset I}$ is equivariant for the action of the groupoid $\mathscr{P}_{\Gamma}(J)$;
the action on the domain being the restriction along
the functor $\mathbf{t}_{J\subset I}$ of the action of
$\mathscr{P}_{\Gamma}(I)$. This shows that
$\widetilde{T}^{tor}(J\subset I)$ is a morphism of diagrams.
Also, $\widetilde{\rm s}_{J\subset I}$ and ${\rm s}_{J\subset I}$ are clearly compatible.
Finally, let $\emptyset \neq K \subset J \subset I$. From the construction and the corresponding property for ``${\rm s}$'', we can show that $\widetilde{\rm s}_{K\subset I}=\widetilde{\rm s}_{K\subset J}\circ
\widetilde{\rm s}_{J\subset I}$. (We leave the details of this to the reader.) It follows that $\widetilde{T}^{tor}(K\subset I)=\widetilde{T}^{tor}(K\subset J)\circ \widetilde{T}^{tor}(J\subset I)$; this finishes the proof of the proposition.
\end{proof}


\subsubsection{The diagram of schemes $\mathcal{V}^{tor}$}
\label{subsub:V}
For $(I_0,I_1)\in \mathcal{P}_2([\![1,r]\!])$, let $J=[\![0,r]\!]-
I_0$, and $\{0\}\bigsqcup I_1=\{i_0<\dots< i_s\}$. We define a
diagram of schemes $\mathcal{V}^{tor}(I_0,I_1)$ as follows. We
recursively construct diagrams of schemes $\mathcal{V}^{
tor}_1(I_0,I_1), \dots, \mathcal{V}^{tor}_{s+1}(I_0,I_1)$ and
morphisms $v_j(I_0,I_1):\mathcal{V}^{tor}_j(I_0,I_1)\to
\widetilde{T}^{tor} (J\cap [\![i_{j-1},i_j]\!])$ ($i_{s+1}$ is
taken to be $r$), and then set
$\mathcal{V}^{tor}(I_0,I_1)=\mathcal{V}^{tor}_{s+1}(I_0,I_1)$ and
$v(I_0,I_1)=v_{s+1}(I_0,I_1)$.

We start by taking $\mathcal{V}^{tor}_1(I_0,I_1)=\widetilde{T}^{tor}(J\cap
[\![i_0,i_1]\!])$ and $v_1(I_0,I_1)$ the identity mapping.
Assume that $\mathcal{V}_j^{tor}(I_0,I_1)$ and $v_j(I_0,I_1)$ have been
defined for some $j\leq s$. The composition
\begin{equation}\label{eq:V-to-Y}
\mathcal{V}^{tor}_j(I_0,I_1) \to \widetilde{T}^{tor}(J\cap
[\![i_{j-1},i_j]\!]) \to Y^{tor}_{i_j}
\end{equation}
makes $\mathcal{V}_j^{tor}(I_0,I_1)$ into a diagram of
$Y^{tor}_{i_j}$-schemes. In particular, we may consider the
diagram of $Y^{tor}_{i_j}$-schemes
$\pi_0(\mathcal{V}^{tor}_j(I_0,I_1)/Y^{tor}_{i_j})$, obtained from
$\mathcal{V}_j^{tor}(I_0,I_1)$ by taking objectwise the Stein
factorization\footnote{Here we use the notion of a Stein
factorization in a broad sense. Given a morphism of schemes $a:P
\to S$, we may consider the $\mathcal{O}_S$-algebra $\mathcal{A}$
of integral elements in $a_*\mathcal{O}_P$. When this algebra is coherent (which is the case here),
$\Spec(\mathcal{A})$ is a finite $S$-scheme which we call the
Stein factorization of $a$.} of the projection to $Y^{tor}_{i_j}$.
We then define
\begin{equation}\label{Vtor-iteration}
\mathcal{V}^{tor}_{j+1}(I_0,I_1)=\pi_0(\mathcal{V}^{tor}_j(I_0,I_1)/Y^{tor}_{
i_j})\times_{Y^{tor}_{i_j}}\widetilde{T}^{tor}(J\cap [\![i_j,i_{j+1}]\!])
\end{equation}
and take $v_{j+1}(I_0,I_1)$ to be the projection to the second
factor. By construction, we obtain a morphism of diagrams
$v(I_0,I_1):\mathcal{V}^{tor}(I_0,I_1)\to
\widetilde{T}^{tor}(\varsigma_r( I_0,I_1))$.
Adapting the argument in the proof of
Proposition \ref{prop:basic-property-diag-cal-Y-gen}, one can see
that the assignment $(I_0,I_1)\rightsquigarrow
\mathcal{V}^{tor}(I_0,I_1)$ extends in a canonical way to a
functor $\mathcal{V}^{tor}$ from $\mathcal{P}_2([\![1,r]\!])$ to
$\Dia(\Sch/\C)$. Moreover, the $v(I_0,I_1)$'s give a morphism in
$\Dia(\Dia( \Sch/\C))$:
$$(v,\varsigma_r):\xymatrix@C=1.7pc{(\mathcal{V}^{tor},\mathcal{P}_2([\![1,r]\!]))
\ar[r] &  (\widetilde{T}^{tor},\mathcal{P}^*([\![0,r]\!])^{\rm
op}).}$$
Taking compositions with $\widetilde{T}^{tor}\to
T^{tor}$ and $\widetilde{T}^{tor}\to
\overline{X}^{bb}$ yields morphisms
$$(w,\varsigma_r):
\xymatrix@C=1.7pc{(\mathcal{V}^{tor},\mathcal{P}_2([\![1,r]\!])) \ar[r] &  (T^{tor},
\mathcal{P}^*([\![0,r]\!])^{\rm op})}\quad
\text{and}\quad \Xi:\mathcal{V}^{tor} \!\xymatrix@C=1.7pc{\ar[r] &}\! \overline{X}^{bb}.$$

\begin{proposition}
\label{prop:comparison-cal-Y-and-cal-V}
With $\beta$ as in
{\em\S\ref{par:mot-beta-x-s}:}

\begin{itemize}

\item [(a)] There are canonical isomorphisms of commutative unitary algebras
$$\EE_{\overline{X}^{bb}}\simeq
\Xi_*(w,\varsigma_r)^* \beta_{\overline{X}^{bb}} \qquad \text{and}
\qquad \An^*(\EE_{\overline{X}^{bb}})\simeq
\Xi^{an}_*(w^{an},\varsigma_r)^* \beta^{an}_{\overline{X}^{bb}}.$$

\item [(b)] Moreover, the following diagram
$$\xymatrix@C=1.5pc@R=1.5pc{g^*(\EE_{\overline{X}^{bb}}) \ar[d]_-{\sim} \ar[rr] & & \EE_{
\overline{X'}^{bb}} \ar[d]^-{\sim} \\
g^*\Xi_* (w,\varsigma_r)^* \beta_{\overline{X}^{bb}} \ar[r]  & \Xi'_*(w',
\varsigma_r)^* g^*\beta_{\overline{X}^{bb}} \ar[r] & \Xi'_* (w',\varsigma_r)^*
\beta_{\overline{X'}^{bb}}\!}$$
is commutative, and likewise for the corresponding diagram in the analytic
context.
\end{itemize}
\end{proposition}

\begin{proof}
We prove only the motivic statements. The proof in the analytic
context goes exactly the same way.

We need to introduce another diagram of schemes
$\widetilde{\mathcal{Y}}^{tor}$, one that interpolates between
$\mathcal{Y}^{tor}$ and $\mathcal{V}^{tor}$. First, we bring in the diagram $\widetilde{\mathcal{T}}^{tor}$ introduced in the proof of
Proposition \ref{prop:diagram-schemes-I-P-Gamma-tilde-T}.
Recall that for $\emptyset \neq I\subset [\![0,r]\!]$,
we have a diagram
$\widetilde{\mathcal{T}}^{tor}(I)$ sending
$\alpha \in \widetilde{A}(I)$
to $\widetilde{\mathcal{T}}^{tor}(I,\alpha)$,
a closed, irreducible and
$\widetilde{\mathcal{R}}(I)$-constructible
subset of $\widetilde{T}^{tor}(I)$.
For $(\mathbf{Q},\mathbf{R})\in \mathscr{P}_{\Gamma}(I)$, we denote
$\widetilde{A}(\mathbf{Q},\mathbf{R})\subset \widetilde{A}(I)$
the subset of $\alpha\in \widetilde{A}(I)$ such that $\widetilde{\mathcal{T}}^{tor}(I,\alpha)\subset \widetilde{T}^{tor}(\mathbf{Q},\mathbf{R})$. For such $\alpha$, we write
$\widetilde{\mathcal{T}}^{tor}((\mathbf{Q},\mathbf{R}),\alpha)$
for $\widetilde{\mathcal{T}}^{tor}(I,\alpha)$.
In this way, we may consider $\widetilde{\mathcal{T}}^{tor}(I)$ as an object of $\Dia(\Dia(\Sch/\C))$ sending
$(\mathbf{Q},\mathbf{R})\in \mathscr{P}_{\Gamma}(I)$ to
the diagram $(\widetilde{\mathcal{T}}^{tor}(\mathbf{Q},\mathbf{R}),\widetilde{A}(\mathbf{Q},\mathbf{R}))$. Moreover, this gives a functor from $\mathcal{P}^*([\![0,r]\!])^{\rm op}$ to
$\Dia(\Dia(\Sch/\C))$. As usual, passing to total diagrams, we may view $\widetilde{\mathcal{T}}^{tor}$ as an object of
$\Dia(\Sch/\C)$.

In the same way that $\widetilde{T}^{tor}$ is used in defining
$\mathcal{V}^{tor}$, and $\mathcal{T}^{tor}$ was used in defining
$\mathcal{Y}^{tor}$ (in \S\ref{subsub:Y}), we can use
$\widetilde{\mathcal{T}}^{tor}$ to define a diagram $\widetilde{
\mathcal{Y}}^{tor}$. Specifically, for $(I_0,I_1)\in
\mathcal{P}_2([\![1,r]\!])$, let $J=[\! [0,r]\!]- I_0$, and
$\{0\}\bigsqcup I_1=\{i_0<\dots< i_s\}$ as before. There is a
sequence of diagrams $\widetilde{\mathcal{Y}}^{tor}_1(I_0,I_1),
\dots, \widetilde{\mathcal{Y}}^{tor}_{s+1}(I_0,I_1)$.
It is defined inductively by the formula
\begin{equation}\label{eq:wideYtor-iteration}
\widetilde{\mathcal{Y}}^{tor}_{j+1}(I_0,I_1)=\pi_0(\widetilde{
\mathcal{Y}}^{tor}_j(I_0,I_1)/Y^{tor}_{i_j})\times_{Y^{tor}_{i_j}}
\widetilde{\mathcal{T}}^{tor}(J\cap [\![i_j,i_{j+1}]\!])
\end{equation}
(where
$i_{s+1}$ is taken to be $r$)
and the initial condition
$\widetilde{\mathcal{Y}}^{tor}_1(I_0,I_1)=\widetilde{\mathcal{T}}^{tor}
(J\cap [\![i_0,i_1]\!])$.
We then set
$\widetilde{\mathcal{Y}}^{tor}(I_0,I_1)=
\widetilde{\mathcal{Y}}^{tor}_{s+1}(I_0,I_1)$.
There is a morphism of diagrams
$\widetilde{p}(I_0,I_1):\widetilde{\mathcal{Y}}^{tor}(I_0,I_1)\to
\widetilde{ \mathcal{T}}^{tor}(\varsigma_r(I_0,I_1))$.
Adapting again the
argument in the proof of Proposition
\ref{prop:basic-property-diag-cal-Y-gen}, one can show that the
assignment $(I_0,I_1)\rightsquigarrow
\widetilde{\mathcal{Y}}^{tor}(I_0,I_1)$ extends naturally to a
functor $\widetilde{\mathcal{Y}}^{tor}$ from
$\mathcal{P}_2([\![1,r]\!])$ to $\Dia( \Sch/\C)$ and that we have
a morphism of diagrams $\widetilde{p}:\widetilde{
\mathcal{Y}}^{tor} \to \widetilde{\mathcal{T}}^{tor}\circ
\varsigma_r$.

The morphisms from $\widetilde{\mathcal{T}}^{tor}$ to $\mathcal{T}^{tor}$ and $\widetilde{T}^{tor}$ induce canonical morphisms form $\widetilde{\mathcal{Y}}^{tor}$ to
$\mathcal{Y}^{tor}$ and $\mathcal{V}^{tor}$, yielding the
following commutative diagram in $\Dia(\Dia(\Sch/\C))$:
\begin{equation}
\label{eq:diadia}
\xymatrix{(\widetilde{\mathcal{Y}}^{tor},\mathcal{P}_2([\![1,r]\!]))
\ar[r]^-{\rho_1} \ar[d]_-{\rho_2} &
(\mathcal{V}^{tor},\mathcal{P}_2([\![1,r]\!]))
\ar[d]^-{(w,\varsigma_r)} \ar@/^1.3pc/[ddr]^-{\Xi} & \\
(\mathcal{Y}^{tor},\mathcal{P}_2([\![1,r]\!]))
\ar[r]_-{(h,\varsigma_r)} \ar@/_1.2pc/[drr]_-{\Upsilon} &
(T^{tor},\mathcal{P}^*([\![0,r]\!])^{\rm op}) \ar[dr]^-e & \\
& & \overline{X}^{bb}.\!}
\end{equation}
Using Theorem \ref{thm:final-form-for-applic-main-thm} (and
Corollary \ref{cor:realization-EE-X-gen} for the analytic
version), it suffices to check that the morphism $\id \to
\rho_{i*} \rho_i^*$ is invertible for $i\in \{1,2\}$.
Indeed, we then get a chain of isomorphisms
$$\Xi_*(w,\varsigma_r)^*\simeq e_*(w,\varsigma_r)_*(w,\varsigma_r)^*
\simeq e_*(w,\varsigma_r)_*\rho_{1*}\rho_1^*(w,\varsigma_r)^*$$
$$\simeq e_*(h,\varsigma_r)_*\rho_{2*}\rho_2^*(h,\varsigma_r)^*
\simeq e_*(h,\varsigma_r)_*(h,\varsigma_r)^*\simeq \Upsilon_*(h,\varsigma_r)^*.$$
We deal with the morphisms
$\id \to
\rho_{1*} \rho_1^*$ and $\id \to
\rho_{2*} \rho_2^*$
separately.

\smallskip

\noindent \underbar{Case 1, part A}:
Using Corollary
\ref{cor:very-useful-texnical-cor}, we need to verify that $\id
\to \rho_{1}(I_0,I_1)_*\rho_1(I_0,I_1)^*$ is invertible for every
$(I_0,I_1)\in \mathcal{P}_2([\![1,r]\!])$. As usual, we let
$J=[\![0,r]\!]- I_0$ and $\{0\}\bigsqcup I_1=\{i_0<\dots< i_s\}$.
For $1\leq t \leq s$, we let
$Z^{(t)}=\pi_0(\mathcal{V}^{tor}_t(I_0,I_1)/Y_{i_t}^{tor})$
and
$\mathcal{Z}^{(t)}=\pi_0(\widetilde{\mathcal{Y}}^{tor}_t(I_0,I_1)/Y_{i_t}^{tor})$
(compare with \eqref{Vtor-iteration} and
\eqref{eq:wideYtor-iteration}).
We denote
$\varrho_t:\mathcal{Z}^{(t)}\to Z^{(t)}$ the natural morphism.
In the next part, we will show that the morphisms $\id \to \varrho_{t*}\varrho_t^*$
are universally invertible, i.e., the same is true for any base-change of $\varrho_t$ by morphisms of diagrams of schemes.
The case $t=s$ is used to prove our claim as follows.
There is a commutative diagram
$$\xymatrix@C=1.5pc@R=1.5pc{\widetilde{\mathcal{Y}}^{tor}(I_0,I_1) \ar[rrr] \ar[dd] \ar@/_/[drr]^-{\rho_1(I_0,I_1)} & & & \mathcal{Z}^{(s)} \ar[d]^-{\varrho_s} \\
& & \mathcal{V}^{tor}(I_0,I_1)  \ar[r] \ar[d]
& Z^{(s)} \ar[d] \\
\widetilde{\mathcal{T}}^{tor}(J\cap [\![i_s,r]\!])
\ar[rr]^-{\widetilde{q}(J\cap [\![i_s,r]\!])} & &
\widetilde{T}^{tor}(J\cap [\![i_s,r]\!]) \ar[r] &
Y^{tor}_{i_s},}$$
in which the two rectangular squares are
cartesian. It is rather straightforward that the latter can be
completed, to a diagram of the form
$$\xymatrix@C=1pc@R=1pc{\bullet\ar[r] \ar[d] \ar@/_.1pc/[rd] & \bullet\ar[r] \ar[d]& \bullet
\ar[d] \\
\bullet \ar[r] \ar[d]& \bullet \ar[r] \ar[d]
& \bullet \ar[d] \\
\bullet \ar[r] & \bullet \ar[r] & \bullet}$$
in which all the
rectangular squares are cartesian.
It follows that
$\rho_1(I_0,I_1)$
can be written as a composition of base-changes of $\varrho_s$ and
$\widetilde{q}(J\cap [\![i_s,r]\!])$ (in fact, in two ways).
Using that $\id \to \varrho_{s*}\varrho_s^*$
is universally invertible, we are reduced to showing that
$\id \to \widetilde{q}(J\cap [\![i_s,r]\!])_*
\widetilde{q}(J\cap [\![i_s,r]\!])^*$ is universally invertible.
By Corollary
\ref{cor:very-useful-texnical-cor}, we need to show for
$(\mathbf{Q_s},\mathbf{R_s})\in \mathscr{P}_{\Gamma}(J\cap [\![i_s,r]\!])$
that
$\id\to
\widetilde{q}(\mathbf{Q_s},\mathbf{R_s})_*\widetilde{q}(\mathbf{Q_s},
\mathbf{R_s})^*$ is invertible with
$\widetilde{q}(\mathbf{Q_s},\mathbf{R_s}):
(\widetilde{\mathcal{T}}^{tor}(\mathbf{Q_s},\mathbf{R_s}),\widetilde{A}(\mathbf{Q_s},\mathbf{R_s}))\to \widetilde{T}^{tor}(\mathbf{Q_s},\mathbf{R_s})$ the natural morphism. The proof of
Lemma \ref{lemma:form-locality-closed-cover}
can be easily extended to
show this.

\smallskip

\noindent \underbar{Case 1, Part B}:
Here we show that
$\id \to \varrho_{t*}\varrho_t^*$
is universally invertible (with $1\leq t \leq s$). Using
Corollary
\ref{cor:very-useful-texnical-cor}, we only need to check that
$$\id \to \varrho_t((\mathbf{Q_j},\mathbf{R_j})_{0\leq j\leq t-1})_*
\varrho_t((\mathbf{Q_j},\mathbf{R_j})_{0\leq j\leq t-1})^*$$
is universally invertible for all $(\mathbf{Q_j},\mathbf{R_j})_{0\leq j \leq t-1}\in \prod_{j=0}^{t-1}\mathscr{P}_{\Gamma}(J\cap [\![i_j,i_{j+1}]\!])$, the indexing category of $Z^{(t)}$.
Recursively, one sees that, objectwise, $\varrho_t((\mathbf{Q_j},\mathbf{R_j})_{0\leq j\leq t-1})$ induces an isomorphism from each connected
component of the domain to a connected component of the target.
Indeed, given a stratum $S$ of
$\mathcal{B}^c_{(\mathbf{Q_j,R_j}),\Sigma_{(\mathbf{Q_j})}}$, the
Stein factorizations of the projections of $S$ and $\mathcal{B}^c_{(\mathbf{Q_j,R_j}),\Sigma_{(\mathbf{Q_j})}}$ to $X^{bb}_{\mathbf{Q_{j+1}}}$ are the same, and coincide with the Stein factorization
of $\mathcal{A}_{(\mathbf{Q_j,R_j})} \to
X^{bb}_{\mathbf{Q_{j+1}}}$. A
similar statement holds if we replace $X^{bb}_{\mathbf{Q_j}}$ by
$\widetilde{X}^{bb}_{\mathbf{Q_j}}$, or by any other \'etale cover
of $X^{bb}_{\mathbf{Q_j}}$ dominated by
$\widetilde{X}^{bb}_{\mathbf{Q_j}}$.
Moreover, given a connected component $E$ of
$Z^{(t)}((\mathbf{Q_j},\mathbf{R_j})_{0\leq j\leq t-1})$,
$\varrho_t^{-1}(E)$ is canonically isomorphic to the constant diagram
${\big (}E,\prod_{j=0}^{t-1} \widetilde{A}(J\cap [\![i_j,i_{j+1}]\!]){\big )}$.
This is also proven inductively, and we leave the details to the reader. Now, the result follows from Lemma
\ref{lemma:texnic-comparison-cal-Y-and-cal-V-1} below.

\smallskip

\noindent
\underbar{Case 2}:
Here we show that
$\id \to \rho_{2*} \rho_2^*$ is invertible.
Using Corollary \ref{cor:very-useful-texnical-cor}, we are reduced to checking
that
$\id \to \rho_{2}(\dagger)_*\rho_2(\dagger)^*$ is
invertible for every object $\dagger$ of the indexing category of
$\mathcal{Y}^{tor}$.
Thus, we fix $(I_0,I_1)\in \mathcal{P}_2([\![1,r]\!])$
and let $J=[\![0,r]\!]- I_0$ and
$\{0\}\bigsqcup I_1=\{i_0<\dots< i_s\}$.
Let $(\alpha_j)_{0\leq j \leq s}$ be an object of the indexing category of
$\mathcal{Y}^{tor}(I_0,I_1)$,
that is of $\prod_{j=0}^s
A(J\cap [\![i_j,i_{j+1}]\!])$
(with $i_{s+1}=r$).
We need to show that
\begin{equation}
\label{eq-prop:comparison-cal-Y-and-cal-V-071}
\id\to \rho_2((\alpha_j)_{0\leq j\leq s})_*
\rho_2((\alpha_j)_{0\leq j\leq s})^*
\end{equation}
is invertible.

We show by induction on $1\leq t \leq s+1$ that the groupoid
$\prod_{j=0}^{t-1}\mathscr{P}_{\Gamma}(J\cap [\![i_{j},i_{j+1}]\!])$
(with $i_{s+1}=r$) acts freely on the set of connected components
of $\widetilde{\mathcal{Y}}^{tor}_t((\alpha_j)_{0\leq j \leq t-1})$,
and that the natural morphism
$\varrho'_{t}:\widetilde{\mathcal{Y}}^{tor}_t((\alpha_j)_{0\leq j \leq
t-1})\to \mathcal{Y}^{tor}_t((\alpha_j)_{0\leq j \leq t-1})$ induces an
isomorphism
$$\left(\prod_{j=0}^{t-1}\mathscr{P}_{\Gamma}(J\cap [\![i_{j},i_{j+1}]\!])\right)\backslash \;
\widetilde{\mathcal{Y}}^{tor}_t((\alpha_j)_{0\leq j \leq t-1})\simeq
\mathcal{Y}^{tor}_t((\alpha_j)_{0\leq j \leq t-1}).$$
By Lemma \ref{lemma:texnic-comparison-cal-Y-and-cal-V-2} below, this would imply that the morphisms
$\id \to \varrho'_{t*}\varrho'^*_t$, so in particular
\eqref{eq-prop:comparison-cal-Y-and-cal-V-071}, are invertible.

For $t=1$,
note that
$\mathscr{P}_{\Gamma}(J\cap [\![i_0,i_1]\!])$ acts freely on
the set of connected components of
$\mathcal{T}^{tor}(J\cap [\![i_0,i_1]\!])$. Indeed, this set can be identified with
$\coprod_{(\mathbf{G},\mathbf{R})\in \mathscr{P}_{\Gamma}(J\cap [\![i_0,i_1]\!])} \Sigma^{\circ}_{\mathbf{G},\mathbf{R}}$.
Using \eqref{eq:rajoute-quotient-tilde-T-tor=T-tor},
we see that our claim is true for $t=1$.

Now assume that our claim is true for some $1\leq t \leq s$.
Fix a connected component $E$ of
$\mathcal{Y}^{tor}_t((\alpha_j)_{0\leq j \leq t-1})$.  Let $\mathbf{Q}$ be a
maximal or improper parabolic $\Q$-subgroup of $\mathbf{G}$
for which $E$ dominates $\overline{X}^{bb}_{\mathbf{Q}}$.
Then $\pi_0(E/Y^{tor}_{\mathbf{Q}})$ is isomorphic to the toroidal compactification
$\overline{({^{\diamond}\!X}^{bb}_{\mathbf{Q}})}^{tor}_{\Sigma_{(\mathbf{Q})}}$ of a locally symmetric variety ${^{\diamond}\!X}^{bb}_{\mathbf{Q}}$ which is a finite \'etale cover of $X^{bb}_{\mathbf{Q}}$ dominated by $\widetilde{X}^{bb}_{\mathbf{Q}}$.
Denote
$F=\mathcal{T}^{tor}(J\cap [\![i_t,i_{t+1}]\!], \alpha_t)$
and $\widetilde{F}$ its inverse image in $\widetilde{\mathcal{T}}^{tor}(J\cap [\![i_t,i_{t+1}]\!])$. Using induction, we are reduced to showing that
$\mathscr{P}_{\Gamma}(J\cap [\![i_t,i_{t+1}]\!])$
acts freely on the set of connected components of
$\pi_0(E/Y^{tor}_{\mathbf{Q}})\times_{Y^{tor}_{\mathbf{Q}}} \widetilde{F}$ and that we have an isomorphism
$$\mathscr{P}_{\Gamma}(J\cap [\![i_t,i_{t+1}]\!])\backslash
\left(\pi_0(E/Y^{tor}_{\mathbf{Q}}) \times_{Y^{tor}_{\mathbf{Q}}}
\widetilde{F}\right) \simeq \pi_0(E/Y^{tor}_{\mathbf{Q}}) \times_{Y^{tor}_{\mathbf{Q}}} F.$$
In fact, these properties are already true for $\widetilde{F}$ and the projection $\widetilde{F} \to F$.
This is proved in the same way as for the case $t=1$.
\end{proof}

\begin{lemma}
\label{lemma:texnic-comparison-cal-Y-and-cal-V-1}
Let $\emptyset \neq I \subset [\![0,r]\!]$ and
$(\mathbf{Q},\mathbf{R})\in \mathscr{P}_{\Gamma}(I)$.
We denote by $\varrho$ the projection of
$\widetilde{A}(\mathbf{Q},\mathbf{R})$ to $\mathbf{e}$.
Let $S$ be a noetherian scheme and $M\in \DM(S)$. Then, the canonical
morphism $M \to \varrho_*\varrho^*M$
is invertible.

\end{lemma}

\begin{proof}
By the adjunction formula (cf.~\cite[Lem.~2.1.146]{ayoub-these-I}), we have natural isomorphisms
$$\underline{\rm Hom}(\varrho_{\sharp} \varrho^*\un_S, M)\simeq
\varrho_*\underline{\rm Hom}(\varrho^*\un_S,\varrho^*M)\simeq \varrho_*\varrho^*M$$
for all $M\in \DM(S)$.
Hence, it suffices to show that $\varrho_{\sharp}\varrho^*\un_S \simeq \un_S$. On the other hand, there is a canonical functor
$c:\mathbf{H} \to \DM(S)$, where $\mathbf{H}=\mathbf{Ho}(\mathbf{\Delta}^{\rm op}{\rm Set})$ is the homotopy category of simplicial sets. Given a simiplicial set $X_{\scriptscriptstyle{\bullet}}$, we may form the simplicial abelian group
$\Z X_{\scriptscriptstyle{\bullet}}$ given in degree $d\geq 0$ by the free $\Z$-module generated by the elements of $X_d$. Then $c$ takes
$X_{\scriptscriptstyle{\bullet}}$ to
the $T$-spectrum $\Sigma^{\infty}_T({\rm N}(\Z X))$
where ${\rm N}(\Z X)$ is the Moore complex associated to $\Z X_{\scriptscriptstyle{\bullet}}$ which we consider as a constant sheaf on $\Smooth/S$.
For instance, for the simplicial set $pt$ having one element in each degree, we have $c(pt) = \un_S$.
As the functor $c$ commutes with homotopy colimits,
it then suffices to show that
$\varrho_{\sharp}\varrho^* pt\simeq pt$.
Now, there is a Quillen equivalence between the model category ${\rm Top}$ of topological spaces and that of simplicial sets. In particular,
$\mathbf{H}\simeq \mathbf{Ho}({\rm Top})$,
and it suffices to show that $\varrho_{\sharp}\varrho^* pt\simeq pt$
in $\mathbf{Ho}({\rm Top})$. (Here, of course, $pt$ stands for the topological space with one element.)

We need to compute
the homotopy colimit in the category of topological spaces of the constant functor $pt:\widetilde{A}(\mathbf{Q},\mathbf{R}) \to {\rm Top}$. Recall the bijection between
$\widetilde{A}(\mathbf{Q},\mathbf{R})$ and $\Sigma^{\circ}_{\mathbf{Q},\mathbf{R}}$: it sends an element $\alpha\in
\widetilde{A}(\mathbf{Q},\mathbf{R})$ to the rational polyhedral cone
$\sigma\in \Sigma^{\circ}_{\mathbf{Q},\mathbf{R}}$ that corresponds to the stratum of $\mathcal{B}^{\circ}_{(\mathbf{Q},\mathbf{R}),\Sigma_{(\mathbf{Q})}}$ whose closure in $\widetilde{T}^{tor}(\mathbf{Q},\mathbf{R})$ is $\widetilde{\mathcal{T}}^{tor}(\mathbf{Q},\mathbf{R},\alpha)$. Clearly, sending $\alpha$ to the closure of $\sigma$ in
$U_{R}$ yields a functor $L:\widetilde{A}(\mathbf{Q},\mathbf{R}) \to {\rm Top}$. As $L(\alpha)$ is a contractible topological space for all $\alpha$'s, it suffices to compute the homotopy colimit of the functor $L$. Now, it is easy to see that the diagram $L$ is Reedy cofibrant in the sense of \cite[Ch.~15]{hirschhorn}. Hence, its homotopy colimit is given by its categorical colimit which is
$\overline{C}_R=\bigcup_{\sigma\in \Sigma^{\circ}_{\mathbf{Q},\mathbf{R}}} \overline{\sigma}$, equipped with the Satake topology.
The latter has the homotopy type of its interior which is contractible being a convex subset of $U_R$. This finishes the proof of the lemma.
\end{proof}

The other lemma needed to complete the proof of Proposition \ref{prop:comparison-cal-Y-and-cal-V} is:

\begin{lemma}
\label{lemma:texnic-comparison-cal-Y-and-cal-V-2}
Let $\mathcal{G}$ be a small groupoid and
$\mathcal{P}$ a representation
of $\mathcal{G}$ in the category of locally noetherian schemes. Assume that
$\mathcal{G}$ acts freely on the set of connected components of $\mathcal{P}$, i.e., for each $\alpha\in \mathcal{G}$ the stabilizer in ${\rm end}_{\mathcal{G}}(\alpha)$ of each connected component of $\mathcal{P}(\alpha)$
is trivial. Denote by
$\pi:\mathcal{P} \to \mathcal{G}\backslash \mathcal{P}$
the canonical projection. Then $\id \to \pi_*\pi^*$ is invertible.
\end{lemma}

\begin{proof}
If $C$ is a connected component of $\mathcal{G}\backslash \mathcal{P}$, denote
by $\pi_C:\mathcal{P}\times_{\mathcal{G}\backslash \mathcal{P}}C\to C$ the
canonical projection. It suffices to show that $\id \to \pi_{C*}\pi_C^*$ is
invertible for every $C$. In other words, we may assume that $\mathcal{G}
\backslash \mathcal{P}$ is connected.
In that case, there is a connected component $\mathcal{G}_0$ of $\mathcal{G}$
such that $\mathcal{P}(\alpha)=\emptyset$ if $\alpha\in {\rm ob}(\mathcal{G})
- {\rm ob}(\mathcal{G}_0)$.
Replacing $\mathcal{G}$ by $\mathcal{G}_0$, we may further assume that
$\mathcal{G}$ is connected. In this case, $\mathcal{G}$ is equivalent to the
category $\bullet_G$ associated to an actual group $G$. (Recall that
$\bullet_{G}$ has only one object, denoted $\bullet$, whose endomorphisms are
given by the elements of $G$.) Thus, it
suffices to consider the case of a group $G$ that is acting on the scheme
$|G|\times Q$, where $|G|$ denotes the discrete set underlying $G$,
and $Q$ a connected noetherian scheme.

Let $\widetilde{G}$ be the category with ${\rm
ob}(\widetilde{G})=G$ and $\hom_{\widetilde{G}}(g,h)=\{g^{-1}h\}$.
Clearly, $\widetilde{G}$ is a groupoid, and is equivalent to the
category $\mathbf{e}$. We also have the
functor $\widetilde{G} \to \bullet_G$, which sends every object
$g$ to $\bullet$ and and is the identity mapping on the set of
arrows. Also, let $(\widetilde{Q},\widetilde{G})$ be the diagram
of schemes sending $g$ in $G = {\rm ob}(\widetilde{G})$ to
$\{g\}\times Q$. We have a morphism in $\Dia(\Sch)$:
$$p:(\widetilde{Q},\widetilde{G}) \to (|G|\times Q,\bullet_G).$$
We claim that $\id \to p_*p^*$ is invertible.
Indeed, we are in the situation of Corollary
\ref{cor:very-useful-texnical-cor}
with $\mathcal{I}=\bullet_G$ and
$((\mathcal{Y},\mathcal{J}),\mathcal{I})$ the diagram taking
$\bullet$ to the diagram $g\in |G| \rightsquigarrow \{g\}\times Q$.
Thus, we are reduced to showing that
$\id \to p'_*p'^*$ is invertible for
$p':(\{{\rm -}\}\times Q,|G|) \to |G|\times Q$ the obvious morphism.
Our claim is now clear.
To end the proof of the lemma, it remains to see that
$\id \to \widetilde{\pi}_*\widetilde{\pi}^*$ is invertible with
$\widetilde{\pi}:(\widetilde{Q},\widetilde{G}) \to Q$ the canonical
projection. But this is clear, as $\widetilde{G}$ is equivalent to $\mathbf{e}$.
\end{proof}

\subsubsection{The diagram $\mathcal{W}^{tor}$: a condensed model
of $\mathcal{V}^{tor}$} By Corollary
\ref{cor:making-the-diag-connected}, we may replace
$\mathcal{V}^{tor}$ by its diagram of connected components
$\mathcal{V}^{\flat, \, tor}$ and the conclusion of Proposition
\ref{prop:comparison-cal-Y-and-cal-V} will still hold. More
precisely, let $\mathcal{V}^{\flat, \, tor}$ be the diagram which
takes an object $\dagger$ of the indexing category of
$\mathcal{V}^{tor}$ to the discrete diagram
$(\mathcal{V}^{\flat,\, tor}( \dagger),\Pi(\dagger))$ of connected
components of $\mathcal{V}^{tor}(\dagger)$. Let $\Xi^{\flat}$ be
the projection of $\mathcal{V}^{\flat, \, tor}$ to
$\overline{X}^{bb}$ and $(w^{\flat},\varsigma_r)$ its projection
to $(T^{tor},\mathcal{P}^*([\![0,r]\!])^{\rm op})$. Then there is
a canonical isomorphism of commutative unitary algebras
$\EE_{\overline{X}^{bb}}\simeq
\Xi^{\flat}_*(w^{\flat},\varsigma_r)^*\beta_{ \overline{X}^{bb}}$,
and similarly in the analytic context. Moreover, there is a
commutative diagram analogous to the one in Proposition
\ref{prop:comparison-cal-Y-and-cal-V}, (b). We will show that the
total diagram associated to $\mathcal{V}^{\flat,\, tor}$ is
equivalent (in the $2$-category of diagrams) to a much smaller
diagram $\mathcal{W}^{tor}$.  We can then reformulate Proposition
\ref{prop:comparison-cal-Y-and-cal-V} in terms of $\mathcal{W}^{
tor}$.

We begin by verifying the following:

\begin{lemma}
\label{lemme:criterion-for-non-emptyness-cal-V}
Let $(I_0,I_1)\in \mathcal{P}_2([\![1,r]\!])$ and put
$J=[\![0,r]\!]- I_0$ and
$\{0\}\bigsqcup I_1=\{i_0<\dots< i_s\}$. Let
$(\mathbf{Q_j},\mathbf{R_j})_{0\leq j \leq s}$ be an object of
$\prod_{j=0}^s \mathscr{P}_{\Gamma}(J\cap [\![i_j,i_{j+1}]\!])$
(with $i_{s+1}=r$). Then
$\mathcal{V}^{tor}((\mathbf{Q_j},\mathbf{R_j})_{0\leq j \leq s})\neq\emptyset$
if and only if there exists a family $(\gamma_j)_{0\leq j\leq s}$
of elements in $\Gamma$ such that
$\bigcap_{j=0}^s \gamma_j\mathbf{E_{Q_j,R_j}}\gamma_j^{-1}$ is a parabolic
$\Q$-subgroup of $\mathbf{G}$.
\end{lemma}

\begin{proof}
Recall from the construction in \S\ref{subsub:V} that
$\mathcal{V}^{tor}((\mathbf{Q_j},
\mathbf{R_j})_{0\leq j \leq s})$
is the last term in a finite sequence of diagrams
$\{\mathcal{V}^{tor}_t((\mathbf{Q_j},\mathbf{R_j})_{0\leq j\leq t-1})\}_{1
\leq t
\leq s+1}$. We show by induction on $t$ that:

${\rm S}_t$) {\it $\mathcal{V}^{tor}_t((\mathbf{Q_j},
\mathbf{R_j})_{0\leq j \leq t-1})\neq\emptyset$ if and only if
$\mathbf{H_t}(\gamma_0,\dots, \gamma_{t-1})=\bigcap_{j=0}^{t-1}
\gamma_j\mathbf{E_{Q_j,R_j}} \gamma_j^{-1}$ is parabolic for some
$\gamma_0,...,\gamma_{t-1}\in \Gamma$.}

The statement ${\rm S}_1$ is trivial, as
$\mathbf{E_{Q_0,R_0}}=\mathbf{R_0}$ is parabolic and
$\mathcal{V}_1^{tor}(\mathbf{Q_0},\mathbf{R_0})=Y^{tor}_0$ is not
empty. We assume that ${\rm S}_t$ is true for some $1\leq t \leq
s$ and we prove ${\rm S}_{t+1}$. Let $\mathbf{F_t}$ be the maximal
parabolic $\Q$-subgroup containing
$\gamma_{t-1}\mathbf{E_{Q_{t-1},R_{t-1}}}\gamma_{t-1}^{-1}$ and to
which the latter is subordinate. From the formula
$$\mathcal{V}_{t+1}^{tor}((\mathbf{Q_j},\mathbf{R_j})_{0\leq j \leq t})=\pi_0(
\mathcal{V}^{tor}_t((\mathbf{Q_j},\mathbf{R_j})_{0\leq j \leq t-1})/Y_{i_t}^{
tor})\times_{Y_{i_t}^{tor}}
\widetilde{T}^{tor}(\mathbf{Q_t},\mathbf{R_t}),
$$
we deduce that the following conditions are equivalent:
\begin{enumerate}

\item[(i)] $\mathcal{V}_{t+1}^{tor}((\mathbf{Q_j},\mathbf{R_j})_{0\leq j\leq t})\neq \emptyset$,

\item[(ii)] $\mathcal{V}^{tor}_t((\mathbf{Q_j},\mathbf{R_j})_{0\leq j \leq t-1})\neq \emptyset$ and $Y^{tor}_{\mathbf{F_t}}=Y^{tor}_{\mathbf{Q_t}}$.

\end{enumerate}
Indeed, if $\mathcal{V}^{tor}_t((\mathbf{Q_j},\mathbf{R_j})_{0\leq j \leq t-1})$ is not empty,
$\mathcal{V}^{tor}_t((\mathbf{Q_j},\mathbf{R_j})_{0\leq j \leq t-1})\to Y^{tor}_{i_t}$ is proper
and surjective over the connected component $Y^{tor}_{\mathbf{F_t}}$ of $Y^{tor}_{i_t}$. On the other hand,
the image of $\widetilde{T}^{tor}(\mathbf{Q_t}, \mathbf{R_t}) \to Y^{tor}_{i_t}$ is contained
in the connected component $Y^{tor}_{\mathbf{Q_t}}$ of $Y^{tor}_{i_t}$.
By the induction hypothesis, the condition (ii) is also equivalent to:
\begin{enumerate}

\item[(iii)] $\mathbf{H_t}(\gamma_0,\dots, \gamma_{t-1})$ is parabolic and
$\mathbf{F_t}=\gamma_t \mathbf{Q_t} \gamma_t^{-1}$
for some $\gamma_0,...,\gamma_{t}\in \Gamma$.

\end{enumerate}
Now, $\mathbf{F_t}$ and $\gamma_t \mathbf{Q_t}\gamma_t^{-1}$ are
parabolic of the same type and $\mathbf{F_t}$ contains
$\mathbf{H_t}(\gamma_0,\dots, \gamma_{t-1})$. Thus, we may rewrite
(iii) in a slightly different but equivalent form:

\begin{enumerate}

\item[(iii$'$)] $\mathbf{H_t}(\gamma_0,\dots, \gamma_{t-1})$ is parabolic and
is contained in $\gamma_t \mathbf{Q_t} \gamma_t^{-1}$ for some
$\gamma_0,...,\gamma_{t}$.

\end{enumerate}
To prove the statement ${\rm S}_{t+1}$, we verify that
(iii$'$) is equivalent to:
\begin{enumerate}

\item[(iii$''$)] $\mathbf{H_{t+1}}(\gamma_0,\dots, \gamma_{t})$ is parabolic
for some $\gamma_0,...,\gamma_{t}\in \Gamma$.

\end{enumerate}
The implication (iii$''$) $\Rightarrow$ (iii$'$) is clear. Indeed,
if $\mathbf{H_{t+1}}(\gamma_0,\dots, \gamma_{t})$ is parabolic,
then $\mathbf{H_t}(\gamma_0,\dots, \gamma_{t-1})$ and $\gamma_t
\mathbf{Q_t} \gamma_t^{-1}$ are also parabolic. As they are of
cotype $J\cap [\![i_0,i_t]\!]$ and $\{i_t\}$ respectively, we also
have $\mathbf{H_t}(\gamma_0,\dots, \gamma_{t-1}) \subset \gamma_t
\mathbf{Q_t} \gamma_t^{-1}$. The converse implication
(iii$'$) $\Rightarrow$ (iii$''$) follows from Lemma
\ref{sub-lemme:for-criterion-for-non-emptyness-cal-V} below.
\end{proof}

\begin{lemma}
\label{sub-lemme:for-criterion-for-non-emptyness-cal-V}
Let $\mathbf{P_1}$ and $\mathbf{P_2}$ be two parabolic $\Q$-subgroups of cotypes $\emptyset \neq I_1, \, I_2 \subset [\![1,r]\!]$ and assume that ${\rm max}(I_1)={\rm min}(I_2)={s}$. Let $\mathbf{Q}$ be
the maximal parabolic $\Q$-subgroup containing $\mathbf{P_1}$ and of cotype $\{s\}$, i.e., $\mathbf{P_1}$ is subordinate to $\mathbf{Q}$.
Then $\mathbf{P_1} \cap \mathbf{P_2}$ is parabolic if and only if
$\mathbf{P_2}\subset \mathbf{Q}$.
\end{lemma}

\begin{proof}
If $\mathbf{P_1}\cap \mathbf{P_2}$ is parabolic, then $\mathbf{P_2} \subset \mathbf{Q}$ as $I_2$ contains $s$. Conversely, assume that
$\mathbf{P_2}\subset \mathbf{Q}$.
Denote $\mathbf{P_1'}$ and $\mathbf{P_2'}$ the images of $\mathbf{P_1}$ and $\mathbf{P_2}$ by the projection of $\mathbf{Q}$ to (the quotient by a finite normal subgroup of) $\mathbf{M_Q}$.
It suffices to show that $\mathbf{P_1'}\cap \mathbf{P_2'}$ is a
parabolic subgroup of $\mathbf{M_Q}$. Looking at the cotypes of $\mathbf{P_1}$ and $\mathbf{P_2}$,
we see that
$\widetilde{\mathbf{M}}_{\mathbf{Q},\ell}\subset \mathbf{P_2'}$
and $\mathbf{M}_{\mathbf{Q},h}\subset \mathbf{P_1'}$.
As $\mathbf{M_Q}=\widetilde{\mathbf{M}}_{\mathbf{Q},\ell} \cdot \mathbf{M}_{\mathbf{Q},h}$, it follows that
$$\mathbf{P_1'}=(\mathbf{P_1'}\cap  \widetilde{\mathbf{M}}_{\mathbf{Q},\ell})\cdot \mathbf{M}_{\mathbf{Q},h} \qquad
\text{and} \qquad
\mathbf{P_2'}=\widetilde{\mathbf{M}}_{\mathbf{Q},\ell} \cdot (
\mathbf{P_2'}\cap \mathbf{M}_{\mathbf{Q},h}).$$
Thus, $\mathbf{P'_1}\cap \mathbf{P_2'}=
(\mathbf{P_1'}\cap  \widetilde{\mathbf{M}}_{\mathbf{Q},\ell})\cdot
(\mathbf{P_2'}\cap \mathbf{M}_{\mathbf{Q},h})$.
This proves the lemma as the latter factors are parabolic subgroups of
$\widetilde{\mathbf{M}}_{\mathbf{Q},\ell}$ and $\mathbf{M}_{\mathbf{Q},h}$ respectively.
\end{proof}

Though the following construction resembles the one at the
beginning of \S4.2.2, it is not an extension of that. For
$(I_0,I_1)\in \mathcal{P}_2([\![1,r]\!])$, denote by
$\mathscr{Q}(I_0,I_1)$ the set of pairs $(\mathbf{Q},\mathbf{E})$
of parabolic $\Q$-subgroups of $\mathbf{G}$ such that
$\mathbf{E}\subset\mathbf{Q}$, and $\mathbf{E}$ and $\mathbf{Q}$
are of type $I_0$ and cotype $I_1$ respectively. For
$(\mathbf{Q},\mathbf{E})\in \mathscr{Q}(I_0,I_1)$, let
$\mathbf{B_{Q,E}}$ be
the intersection of $\mathbf{Q}$ with the maximal
parabolic $\Q$-subgroup to which $\mathbf{E}$ is subordinate. (When
$(\mathbf{Q},\mathbf{E})=(\mathbf{G},\mathbf{G})$ we take this subgroup to be $\mathbf{G}$ itself.)
This is a parabolic $\Q$-subgroup of $\mathbf{G}$ containing $\mathbf{E}$ and of cotype $I_1\cup \{{\rm max}([\![1,r]\!]-I_0)\}$ (with the convention that $\{{\rm max}(\emptyset)\}=\emptyset$).
We denote by $\mathbf{H_{Q,E}}\subset \mathbf{B_{Q,E}}$ the inverse image of $\mathbf{M}_{\mathbf{B_{Q,E}},h}\subset \mathbf{M_{B_{Q,E}}}$ by the projection of $\mathbf{B_{Q,E}}$ to (the quotient by a finite normal subgroup of) $\mathbf{M_{B_{Q,E}}}$. This is a normal subgroup of $\mathbf{E}$.

Given two pairs $(\mathbf{Q_1},\mathbf{E_1})$ and
$(\mathbf{Q_2},\mathbf{E_2})$ in $\mathscr{Q}(I_0,I_1)$, denote by
$[(\mathbf{Q_1},\mathbf{E_1}),(\mathbf{Q_2},\mathbf{E_2})]$ the
subset of $\mathbf{G}(\Q)$ consisting of elements $\gamma$ such
that $\gamma\mathbf{E_1}\gamma^{-1}=\mathbf{E_2}$ (and thus also,
$\gamma\mathbf{Q_1}\gamma^{-1}=\mathbf{Q_2}$). For
$\gamma, \, \gamma'\in [(\mathbf{Q_1},\mathbf{E_1}),(\mathbf{Q_2},\mathbf{E_2})]$,
we write $\gamma\sim \gamma'$ when
there exists $\delta_1\in \mathbf{H_{Q_1,E_1}}(\Q)$ such that
$\gamma'=\gamma \delta_1$ (equivalently,
when there exists $\delta_2\in \mathbf{H_{Q_2,E_2}}(\Q)$ such that
$\gamma'=\delta_2\gamma$).
This defines an equivalence relation on
$[(\mathbf{Q_1},\mathbf{E_1}),(\mathbf{Q_2},\mathbf{E_2})]$ that
is compatible with multiplication in $\mathbf{G}(\Q)$. We make the set $\mathscr{Q}(I_0,I_1)$ into a groupoid by setting
$$\hom_{\mathscr{Q}(I_0,I_1)}((\mathbf{Q_1},\mathbf{E_1}),(\mathbf{Q_2},
\mathbf{E_2}))=
[(\mathbf{Q_1},\mathbf{E_1}),(\mathbf{Q_2},\mathbf{E_2})]/\sim.$$
We also let $\mathscr{Q}_{\Gamma}(I_0,I_1)$ be the sub-groupoid of
$\mathscr{Q}(I_0,I_1)$ having the same objects, but where
morphisms are the equivalence classes of elements of $\Gamma$.
Given a pair $(\mathbf{Q},\mathbf{E})$ in $\mathscr{Q}(I_0,I_1)$,
we have (cf.~\eqref{eq:end=EmodK} and Lemma \ref{lemme:P-sub-Gamma})
\begin{equation}
\label{eq:endo-dans-groupoid-scr-Q}
{\rm end}_{\mathscr{Q}_{\Gamma}(I_0,I_1)}(\mathbf{Q},\mathbf{E})=\Gamma(
\mathbf{E}/\mathbf{H_{Q,E}}).
\end{equation}

Given another $(I_0',I_1')\in \mathcal{P}_2([\![1,r]\!])$
such that $(I_0,I_1)\subset (I_0',I_1')$,
there is a functor
\begin{equation}
\label{eq:naturalite-grpd-scr-Q}
\xymatrix@C=1.7pc{\mathscr{Q}_{\Gamma}(I_0,I_1) \ar[r] &
\mathscr{Q}_{\Gamma}(I_0',I_1')}
\end{equation}
which sends a pair
$(\mathbf{Q},\mathbf{E})\in\mathscr{Q}_{\Gamma}(I_0,I_1)$ to the
unique pair $(\mathbf{Q'},\mathbf{E'})\in \mathscr{Q}(I_0',I_1')$
satisfying $\mathbf{E'}\supset\mathbf{E}$. The functoriality of
this assignment is clear, as $\mathbf{H_{Q,E}}\subset
\mathbf{H_{Q',E'}}$. Thus, there is a covariant functor
$\mathscr{Q}_{\Gamma}$ from $\mathcal{P}_2([\![1,r]\!])$ to the
category of groupoids.

As $g\Gamma'g^{-1}\subset \Gamma$, conjugation by
the element $g\in \mathbf{G}(\Q)$ induces a morphism of groupoids
${\rm int}(g):\mathscr{Q}_{\Gamma'}(I_0,I_1) \to
\mathscr{Q}_{\Gamma}(I_0,I_1)$. This is natural in $(I_0,I_1)$,
so it defines a morphism of diagrams of groupoids.

\begin{lemma}
\label{lemma:nat-appl-scr-Q-P-grpd}
For $(I_0,I_1)\in
\mathcal{P}_2([\![1,r]\!])$, let $J=[\![0,r]\!]- I_0$ and
$\{0\}\bigsqcup I_1=\{i_0<\dots< i_s\}$.

\begin{enumerate}

\item[(a)] There is a natural morphism of groupoids
\begin{equation}
\label{eq:nat-appl-scr-Q-scr-P-groupoids}
\mathbf{d}(I_0,I_1):\mathscr{Q}_{\Gamma}(I_0,I_1)
\!\xymatrix@C=1.7pc{\ar[r] &}\!
\prod_{j=0}^s \; \mathscr{P}_{\Gamma}(J\cap [\![i_j,i_{j+1}]\!])
\end{equation}
(with $i_{s+1}=r$).
It takes
$(\mathbf{Q},\mathbf{E}) \in \mathscr{Q}_{\Gamma}(I_0,I_1)$
to the family $(\mathbf{Q_j},\mathbf{R_j})_{0\leq j \leq s}$
where:
\begin{itemize}

\item $\mathbf{Q_j}$ is the maximal or improper parabolic $\Q$-subgroup of $\mathbf{G}$ of cotype $\{i_j\}$ that
contains $\mathbf{Q}$.

\item $\mathbf{R_j}$
is the image in $\mathbf{M}_{\mathbf{Q_j},h}$ of the unique
parabolic $\Q$-subgroup $\mathbf{E_j}$ of cotype $J\cap
[\![i_{j},i_{j+1}]\!]$ containing $\mathbf{E}$.

\end{itemize}
Moreover, $\mathbf{Q}=\bigcap_{j=0}^s \mathbf{Q_j}$ and
$\mathbf{E}=\bigcap_{j=0}^s\mathbf{E_j}$.

\item[(b)] The morphism
\begin{equation}
\label{eq-lemma:nat-appl-scr-Q-P-grpd-1}
{\rm end}_{\mathscr{Q}_{\Gamma}(I_0,I_1)}(\mathbf{Q},\mathbf{E}) \!\xymatrix@C=1.7pc{\ar[r] &}\!
\prod_{j=0}^s \; {\rm end}_{\mathscr{P}_{\Gamma}(J\cap [\![i_j,i_{j+1}]\!])}(\mathbf{Q_j},\mathbf{R_j})
\end{equation}
is injective and its image has finite index.

\item[(c)]
The functors $\mathbf{d}(I_0,I_1)$ are natural in
$(I_0,I_1)$ and yield a morphism of diagrams of groupoids from $\mathscr{Q}$
to the diagram of indexing groupoids of $\mathcal{V}^{tor}$.

\item[(d)] The following square commutes:
$$\xymatrix@C=3pc{\mathscr{Q}_{\Gamma'}(I_0,I_1) \ar[r]^-{\mathbf{d}(I_0,I_1)}
\ar[d]_-{{\rm int}(g)}  & \prod_{j=0}^s \mathscr{P}_{\Gamma'}(J\cap
[\![i_j,i_{j+1}]\!]) \ar[d]^-{{\rm int}(g)}\\
\mathscr{Q}_{\Gamma}(I_0,I_1) \ar[r]^-{\mathbf{d}(I_0,I_1)} &
\prod_{j=0}^s \mathscr{P}_{\Gamma}(J\cap
[\![i_j,i_{j+1}]\!]).}$$

\end{enumerate}

\end{lemma}

\begin{proof}
We prove only parts (a) and (b), and leave the naturality questions to the reader.

That $\mathbf{E}=\bigcap_{j=0}^s\mathbf{E_j}$ is clear (as is
$\mathbf{Q}=\bigcap_{j=0}^s \mathbf{Q_j}$), so there is a
diagonal embedding $\mathbf{E}\hookrightarrow \mathbf{E_0}\times
\cdots \times\mathbf{E_s}$.
For $\gamma\in \mathbf{G}(\Q)$, $\mathbf{d}(I_0,I_1)$ takes
$(\gamma\mathbf{Q}\gamma^{-1},\gamma\mathbf{E}\gamma^{-1})$ to
$(\gamma\mathbf{Q_j}\gamma^{-1},\gamma\mathbf{R_j}\gamma^{-1})_{0\leq
j \leq s}$.
Thus, to show that \eqref{eq:nat-appl-scr-Q-scr-P-groupoids} is a
morphism of groupoids, it suffices to check that
$\Gamma(\mathbf{H_{Q,E}})$ is in the kernel of
$\Gamma(\mathbf{E}) \to \prod_{j=0}^s \Gamma(\mathbf{E_j}/\mathbf{K_j})$,
where $\mathbf{K_j}=\mathbf{K_{Q_j,R_j}}$ is as in
\S\ref{subsub:wTtor}. In fact, we can show more, namely that there
is an induced isomorphism of algebraic $\Q$-groups:
\begin{equation}
\label{eq:Dmod=Rmod}
\mathbf{E}/\mathbf{H_{Q,E}}\, \simeq
\mathbf{E_0}/\mathbf{K_0} \times \cdots \times \mathbf{E_s}/\mathbf{K_s},
\end{equation}
which will also imply the stated properties of
\eqref{eq-lemma:nat-appl-scr-Q-P-grpd-1}.

Denote $\mathbf{Q_{s+1}}$ the maximal or improper parabolic
$\Q$-subgroup of $\mathbf{G}$ to which $\mathbf{E}$ is subordinate.
Thus, we have $\mathbf{B_{Q,E}}=\bigcap_{j=0}^{s+1}\mathbf{Q_j}$.
To prove \eqref{eq:Dmod=Rmod}, we use that the type of $\mathbf{B}=\mathbf{B_{Q,E}}$ decomposes into a disjoint union of (possibly empty) intervals:
$$]\!]i_0,i_1[\![\; \bigsqcup \dots \bigsqcup \; ]\!]i_{s-1},i_s[\![\; \bigsqcup\; ]\!]i_s, m[\![ \;\bigsqcup\; ]\!]m,r]\!],$$
with
$m={\rm max}(J)$.
This yields an almost direct product decomposition
\begin{equation}
\label{eq-lemma:nat-appl-scr-Q-P-grpd-316}
\widetilde{\mathbf{M}}_{\mathbf{B},\ell}=\widetilde{\mathbf{M}}_{\mathbf{B},\ell}^{(0)} \times \dots \times
\widetilde{\mathbf{M}}_{\mathbf{B},\ell}^{(s)}\,,
\end{equation}
with $\widetilde{\mathbf{M}}^{(j)}_{\mathbf{B},\ell}\simeq
\widetilde{\mathbf{M}}_{\mathbf{R'_j},\ell}$ as sub-quotients of $\mathbf{G}$ for all $0\leq j \leq s$.
Here, as in
\S\ref{subsub:wTtor},
$\mathbf{R_j'}$ denote the maximal or improper parabolic $\Q$-subgroup of $\mathbf{M}_{\mathbf{Q_j},h}$ to which $\mathbf{R_j}$ is subordinate.\footnote{In fact, $\mathbf{R_j'}$ is improper unless $j=s$ and $J\cap [\![i_s,n]\!]=\{i_s\}$.}
Let $\mathbf{F}\simeq
\mathbf{E}/\mathbf{H_{Q,E}}$ be the image of $\mathbf{E}$ by the projection of $\mathbf{B}$ to (the quotient by a finite normal subgroup of) $\widetilde{\mathbf{M}}_{\mathbf{B},\ell}$.
The decomposition
\eqref{eq-lemma:nat-appl-scr-Q-P-grpd-316} induces a decomposition
of $\mathbf{F}$ into an almost direct product
$\mathbf{F}=\mathbf{F}^{(0)}\times \dots \times \mathbf{F}^{(s)}$.
For each $0\leq j \leq s$, $\mathbf{F}^{(j)}$ corresponds to
the image of
$\mathbf{R_j}$ in $\widetilde{\mathbf{M}}_{\mathbf{R'_j},\ell}$
modulo the identification
$\widetilde{\mathbf{M}}_{\mathbf{B},\ell}^{(j)}\simeq
\widetilde{\mathbf{M}}_{\mathbf{R'_j},\ell}$. That image is
naturally isomorphic to $\mathbf{E_j}/\mathbf{K_j}$ by
\eqref{eq:EmodK-RmodR'h}.
This proves the lemma.
\end{proof}

\begin{remark}
The statement of Lemma
\ref{lemme:criterion-for-non-emptyness-cal-V}, can be expressed in
terms of $\mathbf{d}$.  For an object $\dagger$ in
$\prod_{j=0}^s\mathscr{P}_{\Gamma}(J\cap [\![i_j,i_{j+1}]\!])$,
$\mathcal{V}^{tor}(\dagger)$ is non-empty if and only if $\dagger$
is in the essential image of $\mathbf{d}(I_0,I_1)$, i.e.,
isomorphic to an object lying in the image of
$\mathbf{d}(I_0,I_1)$.
\end{remark}

\begin{lemma}
\label{lemme-clef-pour-simplifier-V-tor}
Fix $(I_0,I_1)\in \mathcal{P}_2([\![1,r]\!])$, and let
$J=[\![1,r]\!]- I_0$ and $\{0\}\bigsqcup I_1=\{i_0<\dots <i_s\}$ as usual.

\begin{enumerate}

\item[(a)]
Let $(\mathbf{Q},\mathbf{E})\in \mathscr{Q}_{\Gamma}(I_0,I_1)$
and denote $(\mathbf{Q_j},\mathbf{R_j})_{0\leq j\leq s}$
its image by the functor $\mathbf{d}(I_0,I_1)$.
The group
$\prod_{j=0}^s{\rm end}_{\mathscr{P}_{\Gamma}(J\cap [\![i_j,i_{j+1}]\!])}
(\mathbf{Q_j},\mathbf{R_j})$ permutes transitively the connected components of
$\mathcal{V}^{tor}((\mathbf{Q_j},\mathbf{R_j})_{0\leq j\leq s})$.
Moreover, the latter has a distinguished connected component
$\mathcal{V}^{\star,\, tor}((\mathbf{Q_j},\mathbf{R_j})_{0\leq
j\leq s})$ whose stabilizer is ${\rm
end}_{\mathscr{Q}_{\Gamma}(I_0,I_1)}(\mathbf{Q},\mathbf{E})$,
considered, via the monomorphism
\eqref{eq-lemma:nat-appl-scr-Q-P-grpd-1}, as a subgroup of
$\prod_{j=0}^s{\rm end}_{\mathscr{P}_{\Gamma}(J\cap
[\![i_j,i_{j+1}]\!])} (\mathbf{Q_j},\mathbf{R_j})$.

\item[(b)] Let $(\mathbf{Q^{\sharp}},\mathbf{E^{\sharp}})$ be another object of $\mathscr{Q}_{\Gamma}(I_0,I_1)$ and denote
$(\mathbf{Q_j^{\sharp}},\mathbf{R^{\sharp}_j})_{0\leq j \leq s}$ its image by ${\rm d}(I_0,I_1)$. Let $(\gamma_j)_{0\leq j \leq s}\in \prod_{j=0}^s \hom_{\mathscr{P}_{\Gamma}(J\cap [\![i_j,i_{j+1}]\!])}((\mathbf{Q_j},\mathbf{R_j}),(\mathbf{Q^{\sharp}_j},\mathbf{R^{\sharp}_j}))$.
Assume that the isomorphism
$$\mathcal{V}^{tor}((\mathbf{Q_j},\mathbf{R_j})_{0\leq j \leq s}) \overset{\sim}{\to} \mathcal{V}^{tor}((\mathbf{Q_j^{\sharp}},\mathbf{R_j^{\sharp}})_{0\leq j \leq s}),$$
induced by $(\gamma_j)_{0\leq j \leq s}$,
takes $\mathcal{V}^{\star, \, tor}((\mathbf{Q_j},\mathbf{R_j})_{0\leq j \leq s})$ onto
$\mathcal{V}^{\star, \, tor}((\mathbf{Q_j^{\sharp}},\mathbf{R_j^{\sharp}})_{0\leq j \leq s})$. Then $(\gamma_j)_{0\leq j \leq s}$ is in the image of
$\hom_{\mathscr{Q}_{\Gamma}(I_0,I_1)}((\mathbf{Q},\mathbf{E}),
(\mathbf{Q^{\sharp}},\mathbf{E^{\sharp}}))$ by
$\mathbf{d}(I_0,I_1)$.
\end{enumerate}

\end{lemma}

\begin{proof}
As in the proof of Lemma
\ref{lemma:nat-appl-scr-Q-P-grpd},
we let $\mathbf{E_j}=\mathbf{E_{Q_j,R_j}}$ and
$\mathbf{K_j}=\mathbf{K_{Q_j,R_j}}$ for $0\leq j \leq s$.
We extend the family $(\mathbf{Q_j})_{0\leq j \leq s}$ by taking $\mathbf{Q_{s+1}}$ the maximal or improper parabolic $\Q$-subgroup of $\mathbf{G}$ to which $\mathbf{E}$ is subordinate.
As such, $\mathbf{E_j}$ is
subordinate to $\mathbf{Q_{j+1}}$ for all $0\leq j \leq s$.
(We also use similar notation for $(\mathbf{Q^{\sharp}},\mathbf{E^{\sharp}})$.)

For each $1\leq t \leq s+1$, let $\mathbf{Q(t)} = \mathbf{Q_0}\cap
\cdots \cap \mathbf{Q_{t}}$. Thus, we have
$\mathbf{Q(s+1)}=\mathbf{B_{Q,E}}$.
We also let:
$$
^{\diamond}\Gamma^{(t)}_h=\Gamma(\mathbf{M}_{\mathbf{Q(t)}})\cap M_{Q(t),h}
\quad\text{ and }\quad \Gamma^{(t)}_h=\Gamma(\mathbf{M}_{\mathbf{Q(t)},h});
$$
$$^{\diamond}\Gamma^{(t)}_{\ell}=\Gamma(\mathbf{M}_{\mathbf{Q(t)}})\cap \widetilde{M}_{Q(t),\ell} \quad \text{ and } \quad
\Gamma^{(t)}_{\ell}=\Gamma(\widetilde{\mathbf{M}}_{\mathbf{Q(t)},\ell}).
$$
Then we have canonical
isomorphisms (of finite groups):
$$\frac{\Gamma^{(t)}_h}{^{\diamond}\Gamma^{(t)}_h}\simeq\frac{\Gamma(\mathbf{M_{Q(t)}})}{{^{\diamond}\Gamma^{(t)}_{\ell}}\cdot{^{\diamond}\Gamma^{(t)}_h}}
\simeq
\frac{\Gamma^{(t)}_{\ell}}{^{\diamond}\Gamma^{(t)}_{\ell}}.
$$
Moreover, for each $1\leq t \leq s+1$, let $\mathbf{E(t)}=\mathbf{E_0}\cap \cdots \cap \mathbf{E_{t-1}}$. Then $\mathbf{E(t)}\subset \mathbf{Q(t)}$ and they are both subordinate to $\mathbf{Q_{t}}$. Also, let
$\Gamma^{(t)}_{\mathbf{Q,E}}$ be the intersection
of $\Gamma^{(t)}_{\ell}$ with the image of
$E(t)\subset Q(t)$ by the projection of $Q(t)$ to (the quotient by a finite normal subgroup of) $\widetilde{M}_{Q(t),\ell}$.
In particular, $\Gamma^{(s+1)}_{\mathbf{Q,E}}=\Gamma(\mathbf{E}/\mathbf{H_{Q,E}})={\rm end}_{\mathscr{Q}_{\Gamma}(I_0,I_1)}(\mathbf{Q},\mathbf{E})$.
By Lemma
\ref{lemma:nat-appl-scr-Q-P-grpd}, there is a monomorphism
$\Gamma^{(t)}_{\mathbf{Q,E}} \hookrightarrow \prod_{j=0}^{t-1}\Gamma(\mathbf{E_j}/\mathbf{K_j})$ with finite index.

We show the following properties by induction on $1\leq t \leq s+1$:
\begin{enumerate}

\item[(a$'$)] The group $\prod_{j=0}^{t-1}{\rm end}_{\mathscr{P}_{\Gamma}(J\cap [\![i_j,i_{j+1}]\!])}(\mathbf{Q_j},\mathbf{R_j})$ acts transitively on the set of connected components of
$\mathcal{V}_t^{tor}((\mathbf{Q_j},\mathbf{R_j})_{0\leq j \leq t-1})$.
The latter has a distinguished connected component
$\mathcal{V}^{\star,\, tor}_t((\mathbf{Q_j},\mathbf{R_j})_{0\leq j \leq t-1})$
whose stabilizer is
$\Gamma^{(t)}_{\mathbf{Q,E}}$.
Moreover,
$\pi_0(\mathcal{V}_t^{\star,\, tor}((\mathbf{Q_j},\mathbf{R_j})_{0\leq j \leq t-1})/Y_{i_t}^{tor})$
is canonically isomorphic to the toroidal compactification
$\overline{({^{\diamond}\!X}^{bb}_{\mathbf{Q_t}})}^{tor}_{\Sigma_{(\mathbf{Q_t})}}$
of the scheme $^{\diamond}\!X^{bb}_{\mathbf{Q_t}}$ whose variety of
$\C$-points is ${^{\diamond}\Gamma^{(t)}_h}\backslash
e_h(\mathbf{Q_t})$.

\item[(b$'$)] If
$(\gamma_j)_{0\leq j \leq t-1}:\mathcal{V}^{tor}_t((\mathbf{Q_j},\mathbf{R_j})_{0\leq j \leq t-1}) \overset{\sim}{\to}
\mathcal{V}^{tor}_t((\mathbf{Q_j^{\sharp}},\mathbf{R_j^{\sharp}})_{0\leq j \leq t-1})$
preserves the distinguished connected components in (a$'$), then
there is $\widetilde{\gamma}(t)\in \Gamma$ such that
$\widetilde{\gamma}(t)\,\mathbf{E_j}\,\widetilde{\gamma}(t)^{-1}=\mathbf{E_j^{\sharp}}$ and the class of
$\widetilde{\gamma}(t)$ in
$[(\mathbf{Q_j},\mathbf{R_j}),(\mathbf{Q_j^{\sharp}},\mathbf{R_j^{\sharp}})]/\sim$ is equal to $\gamma_j$
for all $0\leq j \leq t-1$.

\end{enumerate}

When $t=1$, these properties are clear. Indeed, $\mathbf{Q_0}=\mathbf{G}$, $\mathbf{R_0}=\mathbf{E_0}$ and the scheme
$\mathcal{V}^{tor}_1(\mathbf{Q_0},\mathbf{R_0})$ is connected. Also
$\Gamma^{(1)}_{\mathbf{Q,E}}=\Gamma(\widetilde{\mathbf{M}}_{\mathbf{Q_1},\ell}\,|\,\mathbf{R_0})={\rm
end}_{\mathscr{P}_{\Gamma}(J\cap
[\![i_0,i_1]\!])}(\mathbf{Q_0},\mathbf{R_0})$. Thus, (a$'$) and also
(b$'$) hold in this case.
Next we assume that these properties are proven for some $0\leq t
\leq s$, and we prove them for $t+1$.

For the first claim in (a$'$), with
$\mathcal{V}_t^{\star,\, tor}$ already defined, it suffices to
check that
${\rm end}_{\mathscr{P}_{\Gamma}(J\cap [\![i_t,i_{t+1}]\!])}(\mathbf{Q_t},\mathbf{R_t})$
acts transitively on the set of connected components of
\begin{equation}
\label{eq-lemme-clef-pour-simplifier-V-tor-179}
\pi_0(\mathcal{V}_t^{\star,\, tor}((\mathbf{Q_j},\mathbf{R_j})_{0\leq j
\leq t-1})/Y_{\mathbf{Q_t}}^{tor}) \times_{Y^{tor}_{\mathbf{Q_t}}} \widetilde{T}^{tor}(\mathbf{Q_t},\mathbf{R_t}).
\end{equation}
As the left factor above is connected, it suffices to show that
${\rm end}_{\mathscr{P}_{\Gamma}(J\cap [\![i_t,i_{t+1}]\!])}(\mathbf{Q_t},\mathbf{R_t})$
acts transitively on the fibers of the morphism
$\widetilde{T}^{tor}(\mathbf{Q_t},\mathbf{R_t})\to Y^{tor}_{\mathbf{Q_t}}$. This follows from the isomorphism
\eqref{eq:rajoute-quotient-tilde-T-tor=T-tor}.

Next, we specify the connected component $\mathcal{V}^{\star, \,
tor}_{t+1}((\mathbf{Q_j},\mathbf{R_j})_{0\leq j \leq t})$ of the
scheme \eqref{eq-lemme-clef-pour-simplifier-V-tor-179}. Let
${^{\diamond}T}^{tor}(\mathbf{Q_t},\mathbf{R_t})$ be the closure
in
$\overline{({^{\diamond}\!X}^{bb}_{\mathbf{Q_t}})}^{tor}_{\Sigma_{(\mathbf{Q_t})}}$
of the stratum
$({^{\diamond}\!X}^{bb}_{\mathbf{Q_t}})^{tor}_{\mathbf{R_t},\Sigma_{(\mathbf{Q_t})}}$.
Define
${^{\diamond}\widetilde{T}^{tor}}(\mathbf{Q_t},\mathbf{R_t})$ to
be the analogue for
${^{\diamond}T}^{tor}(\mathbf{Q_t},\mathbf{R_t})$ of what
$\widetilde{T}^{tor}(\mathbf{Q_t},\mathbf{R_t})$ is for
$T^{tor}(\mathbf{Q_t},\mathbf{R_t})$, as in Proposition
\ref{prop:diagram-schemes-I-P-Gamma-tilde-T}. Then, there is a Zariski locally trivial cover
${^{\diamond}\widetilde{T}^{tor}}(\mathbf{Q_t}, \mathbf{R_t}) \to
{^{\diamond}T^{tor}}(\mathbf{Q_t},\mathbf{R_t})$
with automorphism group
${^{\diamond}\Gamma^{(t)}_h}
(\mathbf{M}_{\mathbf{R'_t},\ell}\,|\,\mathbf{R_t})$.
(Recall that $\mathbf{R'_t}$ is the maximal or improper parabolic $\Q$-subgroup
of $\mathbf{M}_{\mathbf{Q_t},h}$ to which $\mathbf{R_t}$ is subordinate.) The commutative square
\begin{equation}
\label{eq:for-computing-stab-V-star-raj-001}
\xymatrix@C=1.5pc@R=1.5pc{{^{\diamond}\widetilde{T}^{tor}}(\mathbf{Q_t},\mathbf{R_t}) \ar[d] \ar[r] & \widetilde{T}^{tor}(\mathbf{Q_t},\mathbf{R_t}) \ar[d] \\
{^{\diamond}T^{tor}}(\mathbf{Q_t},\mathbf{R_t}) \ar[r] &
T^{tor}(\mathbf{Q_t},\mathbf{R_t})}
\end{equation}
yields a closed immersion
${^{\diamond}\widetilde{T}^{tor}}(\mathbf{Q_t},\mathbf{R_t})\to
{^{\diamond}T^{tor}}(\mathbf{Q_t},\mathbf{R_t})
\times_{T^{tor}(\mathbf{Q_t},\mathbf{R_t})}
\widetilde{T}^{tor}(\mathbf{Q_t},\mathbf{R_t})$. The target of the latter morphism is a
closed subscheme of
\eqref{eq-lemme-clef-pour-simplifier-V-tor-179}, and we define
$\mathcal{V}^{\star,\, tor}_{t+1}(
(\mathbf{Q_j},\mathbf{R_j})_{0\leq j \leq t})$ to be the image of
${^{\diamond}\widetilde{T}^{tor}}(\mathbf{Q_t},\mathbf{R_t})$.

>From the construction, we have
$\mathcal{V}^{\star,\, tor}_{t+1}(
(\mathbf{Q_j},\mathbf{R_j})_{0\leq j \leq t}) \simeq {^{\diamond}\widetilde{T}^{tor}}(\mathbf{Q_t},\mathbf{R_t})$.
On the other hand,
$\pi_0({^{\diamond}\widetilde{T}^{tor}}(\mathbf{Q_t},\mathbf{R_t})/Y_{\mathbf{Q_{t+1}}}^{tor})$
is canonically isomorphic to the toroidal compactification
$\overline{({^{\diamond}\!X}^{bb}_{\mathbf{Q_{t+1}}})}^{tor}_{\Sigma_{(\mathbf{Q_{t+1}})}}$
of ${^{\diamond}\!X}^{bb}_{\mathbf{Q_{t+1}}}$, the scheme whose
variety of $\C$-points is the quotient of $e_h(\mathbf{Q_{t+1}})$ by the arithmetic group
$\widetilde{{^{\diamond}\Gamma}^{(t)}_{h}}(\mathbf{M}_{\mathbf{R'_{t}},h})$. To obtain the last assertion in (a$'$), we need to
identify the latter arithmetic group with ${^{\diamond}\Gamma}_h^{(t+1)}$, but this is immediate from the definitions.

To verify (a$'$), it remains to compute the stabilizer $S\subset \prod_{j=0}^{t}{\rm end}_{\mathscr{P}_{\Gamma}(J\cap [\![i_j,i_{j+1}]\!])}(\mathbf{Q_j},\mathbf{R_j})$ of the connected component
$\mathcal{V}_{t+1}^{\star,\, tor}((\mathbf{Q_j},\mathbf{R_j})_{0\leq j
\leq t})$. That $S$ contains $\Gamma^{(t+1)}_{\mathbf{Q,E}}$ is easy to see. We show the reverse inclusion. Let $\gamma\in S$.
It decomposes uniquely as a product
$\gamma=\gamma_0\cdots \gamma_t$ with
$\gamma_j\in {\rm end}_{\mathscr{P}_{\Gamma}(J\cap [\![i_j,i_{j+1}]\!])}(\mathbf{Q_j},\mathbf{R_j})$.
We set $\gamma(t)=\gamma_0\cdots \gamma_{t-1}$ so that
$\gamma=\gamma(t)\cdot \gamma_t$.
The morphism
$$\mathcal{V}_{t+1}^{\star,\, tor}((\mathbf{Q_j},\mathbf{R_j})_{0\leq j
\leq t}) \to \pi_0(\mathcal{V}_{t}^{\star,\, tor}((\mathbf{Q_j},\mathbf{R_j})_{0\leq j
\leq t-1})/Y_{\mathbf{Q_t}})$$
being equivariant for the action of $\gamma(t)$, we deduce from the induction hypothesis that $\gamma(t)\in \Gamma^{(t)}_{\mathbf{Q,E}}$.
Moreover, as $\gamma$ acts on the commutative square
\eqref{eq:for-computing-stab-V-star-raj-001}, we deduce that
$\gamma(t)$ stabilizes
${^{\diamond}T^{tor}}(\mathbf{Q_t},\mathbf{R_t})$,
the closure of the $\mathbf{R_t}$-stratum in
$\overline{({^{\diamond}\!X}^{bb}_{\mathbf{Q_t}})}^{tor}_{\Sigma_{(\mathbf{Q_t})}}$.
This shows that $\gamma(t)$ maps to an element of the subgroup
${^{\diamond}\Gamma}^{(t)}_h\backslash \big{(}{^{\diamond}\Gamma}^{(t)}_h\cdot (\Gamma^{(t)}_h\cap R_t)\big{)}$
by the composition
$$\Gamma^{(t)}_{\mathbf{Q,E}} \to \Gamma^{(t)}_{\ell} \to {^{\diamond}\Gamma}^{(t)}_{\ell} \backslash \Gamma^{(t)}_{\ell}
\simeq {^{\diamond}\Gamma}^{(t)}_{h} \backslash \Gamma^{(t)}_{h}.$$
In other words, there exists a lift $\widetilde{\gamma}(t)\in \Gamma(\mathbf{E(t)})$ of $\gamma(t)$ whose class in $\Gamma^{(t)}_h$ lies in the subgroup ${^{\diamond}\Gamma}^{(t)}_h\cdot (\Gamma^{(t)}_h\cap R_t)$. Now, from the construction, every element of ${^{\diamond}\Gamma}^{(t)}_h$ is the class of an element of $\Gamma(\mathbf{E(t)})$ which has the class of the neutral element in $\Gamma^{(t)}_{\mathbf{Q,E}}$.
Thus, replacing our lift if necessary, we may assume that
the class of $\widetilde{\gamma}(t)$ in $\Gamma^{(t)}_h$ lies in the subgroup $\Gamma^{(t)}_h\cap R_t$. We then have
$\widetilde{\gamma}(t)\in \Gamma(\mathbf{E(t+1)})$, and we let
$\gamma'$ be its image in $\Gamma^{(t+1)}_{\mathbf{Q,E}}$. Clearly, we have $\gamma'(t)=\gamma(t)$. (Here we are using, as for $\gamma$, the decomposition $\gamma'=\gamma'(t)\cdot \gamma'_t\,$.) Replacing $\gamma$ by $\gamma\cdot \gamma'^{-1}$, we may assume that
$\gamma(t)=1$, i.e., $\gamma$ lies in the factor
${\rm end}_{\mathscr{P}_{\Gamma}(J\cap [\![i_t,i_{t+1}]\!])}(\mathbf{Q_t},\mathbf{R_t})$. With this new assumption, consider again the action on the square
\eqref{eq:for-computing-stab-V-star-raj-001}:
$\gamma$ acts by $\gamma_t$ on
$\widetilde{T}^{tor}(\mathbf{Q_t},\mathbf{R_t})$, and by identity
on ${^{\diamond}T^{tor}}(\mathbf{Q_t},\mathbf{R_t})$ and $T^{tor}(\mathbf{Q_t},\mathbf{R_t})$. As the vertical arrows in
\eqref{eq:for-computing-stab-V-star-raj-001} are Zariski
locally trivial covers
of automorphism groups
${^{\diamond}\Gamma^{(t)}_h}
(\mathbf{M}_{\mathbf{R'_t},\ell}\,|\,\mathbf{R_t})$ and
$[\Gamma(\mathbf{M}_{\mathbf{Q_t},h})](\mathbf{M}_{\mathbf{R'_t},h}\,|\,\mathbf{R_t})$ respectively,
we see that $\gamma_{t}$ is necessarily in the subgroup
${^{\diamond}\Gamma^{(t)}_h}
(\mathbf{M}_{\mathbf{R'_t},\ell}\,|\,\mathbf{R_t}) \subset
{\rm end}_{\mathscr{P}_{\Gamma}(J\cap [\![i_t,i_{t+1}]\!])}(\mathbf{Q_t},\mathbf{R_t})$.
But clearly, $\{1\}\times
{^{\diamond}\Gamma^{(t)}_h}
(\mathbf{M}_{\mathbf{R'_t},\ell}\,|\,\mathbf{R_t}) \subset \Gamma^{(t+1)}_{\mathbf{Q,E}}$. This finishes the proof of (a$'$).

For (b$'$), we argue as for the determination of the stabilizer
$S$; here each $\gamma_j:(\mathbf{Q_j},\mathbf{R_j}) \to
(\mathbf{Q_j^{\sharp}},\mathbf{R_j^{\sharp}})$ is a morphism
between two distinct objects. We set $\gamma(t)=\gamma_0\cdots
\gamma_{t-1}$. Using induction, we may find
$\widetilde{\gamma}(t)$ as in (b$'$). Using that $\gamma$ induces
a morphism from the commutative square
\eqref{eq:for-computing-stab-V-star-raj-001} to the similar one
associated to $(\mathbf{Q^{\sharp}},\mathbf{E^{\sharp}})$, we
deduce that $\gamma(t)$ maps the $\mathbf{R_t}$-stratum in
$\overline{({^{\diamond}\!X}^{bb}_{\mathbf{Q_t}})}^{tor}_{\Sigma_{(\mathbf{Q_t})}}$
to the $\mathbf{R^{\sharp}_t}$-stratum in
${\overline{({^{\diamond}\!X}^{bb}_{\mathbf{Q_t^{\sharp}}})}^{tor}_{\Sigma}}_{\!\!\!\!\scriptscriptstyle{(\mathbf{Q_t^{\sharp}})}}$.
As in the case of an endomorphism, this can be used to construct
an element $\widetilde{\gamma}'(t+1)\in \Gamma$ satisfying all the
properties of (b$'$) (for $t+1$), except, possibly, that the class
of $\widetilde{\gamma}'(t+1)$ in
$[(\mathbf{Q_t},\mathbf{R_t}),(\mathbf{Q_t^{\sharp}},\mathbf{R_t^{\sharp}})]/\sim$
is equal to $\gamma_t$. Then, multiplying each $\gamma_j$ by the
inverse of (the class of) $\widetilde{\gamma}'(t+1)$, we reduce to
the case where $\mathbf{Q_j}=\mathbf{Q^{\sharp}_j}$ and
$\mathbf{R_j}=\mathbf{R_j^{\sharp}}$ for all $0\leq j \leq t$. We
are then in the case of an endomorphism, and we may use (a$'$) to
finish the proof.
\end{proof}

For $(I_0,I_1)\in \mathcal{P}_2([\![1,r]\!])$
and $(\mathbf{Q},\mathbf{E})\in \mathscr{Q}_{\Gamma}(I_0,I_1)$,
we set
$$\mathcal{W}^{tor}(\mathbf{Q},\mathbf{E})=\mathcal{V}^{\star,\, tor}(\mathbf{d}(I_0,I_1)(\mathbf{Q},\mathbf{E})).$$
The scheme $\mathcal{W}^{tor}(\mathbf{Q},\mathbf{E})$ can be described as follows.
Write $\mathbf{d}(I_0,I_1)=(\mathbf{Q_j},\mathbf{R_j})_{0\leq j \leq s}$ and let ${^{\diamond}\!X}^{bb}_{\mathbf{Q_s}}$ be the scheme
such that
${^{\diamond}\!X}^{bb}_{\mathbf{Q_s}}(\C)={^{\diamond}\Gamma}(\mathbf{M}_{\mathbf{Q(s)},h})\backslash e_h(\mathbf{Q_s})$,
where
$$
{^{\diamond
}\Gamma}(\mathbf{M}_{\mathbf{Q(s)},h})=
\Gamma(\mathbf{M}_{\mathbf{Q(s)}})\cap M_{Q(s),h}.
$$
(This group
was denoted ${^{\diamond}\Gamma}^{(s)}_h$ in the proof of Lemma
\ref{lemme-clef-pour-simplifier-V-tor}.) Let
${^{\diamond}\mathcal{B}}^c_{(\mathbf{Q_s},\mathbf{R_s}),
\Sigma_{(\mathbf{Q_s})}}$
be the scheme used to construct the $\mathbf{R_s}$-stratum in the
toroidal compactification
$\overline{({^{\diamond}\!X}^{bb}_{\mathbf{Q_s}})}^{tor}_{\Sigma_{
(\mathbf{Q_s})}}$, viz.,
\begin{equation}
\label{eq:where-end-Q-Gamma-appears-naturally-rajoute-0129}
({^{\diamond}\!X}^{bb}_{\mathbf{Q_s}})_{\mathbf{R_s},
\Sigma_{(\mathbf{Q_s})}}^{tor}=\left(\left[{^{\diamond}\Gamma}
(\mathbf{M}_{\mathbf{Q(s)},h})\right](\mathbf{M}_{\mathbf{R'_s},\ell}\,|\, \mathbf{R_s})\right)\backslash
{^{\diamond}\mathcal{B}}^c_{(\mathbf{Q_s},\mathbf{R_s}),
\Sigma_{(\mathbf{Q_s})}},
\end{equation}
where
$\mathbf{R'_s}$ is the maximal or improper parabolic subgroup of
$\mathbf{M}_{\mathbf{Q(s)},h}\simeq \mathbf{M}_{\mathbf{Q_s},h}$
to which $\mathbf{R_s}$ is subordinate.
(In the above formula, the arithmetic subgroup is given by
\eqref{eq:useful-arthmetic-subgr} with ${^{\diamond}\Gamma}
(\mathbf{M}_{\mathbf{Q(s)},h})$ instead of $\Gamma$ and
$\mathbf{R_s}$ instead if $\mathbf{R}$.)
Then
$\mathcal{W}^{tor}(\mathbf{Q},\mathbf{E})$ contains
${^{\diamond}\mathcal{B}}^c_{(\mathbf{Q_s},\mathbf{R_s}),\Sigma_{(\mathbf{Q_s})}}$
as an open dense subset, and we have a cartesian square
\begin{equation}
\label{eq:where-end-Q-Gamma-appears-naturally-rajoute-349}
\xymatrix@C=1.5pc@R=1.5pc{{^{\diamond}\mathcal{B}}^c_{(\mathbf{Q_s},\mathbf{R_s}),\Sigma_{(\mathbf{Q_s})}} \ar[r] \ar[d] &
\mathcal{W}^{tor}(\mathbf{Q},\mathbf{E})\ar[d] \\
({^{\diamond}\!X}^{bb}_{\mathbf{Q_s}})_{\mathbf{R_s},\Sigma_{(\mathbf{Q_s})}}^{tor} \ar[r] & {^{\diamond}T}^{tor}(\mathbf{Q_s},\mathbf{R_s})}
\end{equation}
where the vertical arrows are locally trivial Zariski covers.
These properties determines $\mathcal{W}^{tor}(\mathbf{Q},\mathbf{E})$
up to a canonical isomorphism.

The group $\Gamma$ acts on $\coprod_{(\mathbf{Q},\mathbf{E})\in
{\rm ob}(\mathscr{Q}_{\Gamma}(I_0,I_1))}
\mathcal{W}^{tor}(\mathbf{Q},\mathbf{E})$. The stabilizer
of the connected component $\mathcal{W}^{tor}(\mathbf{Q},\mathbf{E})$ acts through its quotient ${\rm
end}_{\mathscr{Q}_{\Gamma}(I_0,I_1)}(\mathbf{Q},\mathbf{E})$.
Hence, we have a diagram of schemes $\mathcal{W}^{tor}(I_0,I_1)$
indexed by $\mathscr{Q}_{\Gamma}(I_0,I_1)$ and a morphism
\begin{equation}
\label{eq:morphi-W-tor-V-tor}
\mathcal{W}^{tor}(I_0,I_1)
\hookrightarrow \mathcal{V}^{tor}(I_0,I_1)
\end{equation}
which is, for each object, the inclusion of a connected component. One
can check that \eqref{eq:morphi-W-tor-V-tor}
are natural in $(I_0,I_1)\in \mathcal{P}_2([\![1,r]\!])$, hence
they yield a morphism of diagrams in $\Dia(\Sch/\C)$
\begin{equation}
\label{eq:morphi-W-tor-V-tor-Dia-Dia-Sch}
(\mathcal{W}^{tor},\mathcal{P}_2([\![1,r]\!])
\hookrightarrow (\mathcal{V}^{tor},\mathcal{P}_2([\![1,r]\!])).
\end{equation}

\begin{proposition}
\label{prop:equival-entre-W-tor-V-tor}
For all $(I_0,I_1)\in \mathcal{P}_2([\![1,r]\!])$, the inclusion
\eqref{eq:morphi-W-tor-V-tor} yields an equivalence of diagrams
between $\mathcal{W}^{tor}(I_0,I_1)$ and
$\mathcal{V}^{\flat,\, tor}(I_0,I_1)$.
\end{proposition}

\begin{proof}
As $\mathcal{W}^{tor}(I_0,I_1)$ is objectwise connected,
\eqref{eq:morphi-W-tor-V-tor} induces a morphism
\begin{equation}
\label{eq-prop:equival-entre-W-tor-V-tor}
\mathcal{W}^{tor}(I_0,I_1) \!\xymatrix@C=1,7pc{\ar[r] & }\! \mathcal{V}^{\flat,\, tor}(I_0,I_1).
\end{equation}
This morphism is objectwise an isomorphism. Hence, it remains to show that
\eqref{eq-prop:equival-entre-W-tor-V-tor}
induces an equivalence on the indexing categories.

Lemma \ref{lemme-clef-pour-simplifier-V-tor}
implies that the functor underlying \eqref{eq:morphi-W-tor-V-tor}
is fully faithful (i.e., induces a bijection from the set of morphisms between two objects and to the set of morphisms between their images).
It remains to check the essential surjectivity.
By Lemma \ref{lemme:criterion-for-non-emptyness-cal-V},
every object of the indexing category of
$\mathcal{V}^{\flat, \, tor}(I_0,I_1)$ is isomorphic to one of the form
$(\mathbf{d}(\mathbf{Q},\mathbf{E}),C)$
where $(\mathbf{Q},\mathbf{E})\in \mathscr{Q}_{\Gamma}(I_0,I_1)$ and
$C$ is a connected component of
$\mathcal{V}^{tor}(\mathbf{d}(\mathbf{Q},\mathbf{E}))$.
On the other hand,
Lemma \ref{lemme-clef-pour-simplifier-V-tor} states that all the connected components $C$
are conjugate to $\mathcal{W}^{tor}(\mathbf{Q},\mathbf{E})$. This finishes the proof.
\end{proof}

Let
$(\varpi,\varsigma_r):(\mathcal{W}^{tor},\mathcal{P}_2([\![1,r]\!]))\to
(T^{tor},\mathcal{P}^*([\![0,r]\!])^{\rm op})$ and
$\Theta:(\mathcal{W}^{tor},\mathcal{P}_2([\![1,r]\!]))\to
\overline{X}^{bb}$ denote the usual morphisms. We deduce from
Propositions \ref{prop:comparison-cal-Y-and-cal-V} and
\ref{prop:equival-entre-W-tor-V-tor} the following result:

\begin{theorem}
\label{thm:final-form-comput-EE-X-bb}
There are canonical isomorphisms of commutative unitary algebras
$$\EE_{\overline{X}^{bb}}\simeq
\Theta_*(\varpi,\varsigma_r)^* \beta_{\overline{X}^{bb}} \qquad \text{and}
\qquad \An^*(\EE_{\overline{X}^{bb}})\simeq
\Theta^{an}_*(\varpi^{an},\varsigma_r)^* \beta^{an}_{\overline{X}^{bb}}.$$
Moreover, the following diagram
$$\xymatrix@C=1.5pc@R=1.5pc{g^*(\EE_{\overline{X}^{bb}}) \ar[d]_-{\sim} \ar[rr] & & \EE_{\overline{X'}^{bb}} \ar[d]^-{\sim} \\
g^*\Theta_* (\varpi,\varsigma_r)^* \beta_{\overline{X}^{bb}} \ar[r]  & \Theta'_*(\varpi',\varsigma_r)^* g^*\beta_{\overline{X}^{bb}} \ar[r] & \Theta'_* (\varpi',\varsigma_r)^* \beta_{\overline{X'}^{bb}}\!}$$
is commutative, as is the analogous diagram in the analytic context.

\end{theorem}

\begin{remark}
Using Corollary \ref{cor:not-used-for-main-thm} instead of Theorem
\ref{thm:final-form-for-applic-main-thm} in the above discussion,
one arrives at the conclusion that $\EE_{\overline{X}^{bb}}\simeq
\Theta_*\un_{\mathcal{W}^{tor}}$. However, this will not be needed
in the proof of Theorem \ref{thm:main-thm}.
\end{remark}

\subsubsection{End of the proof}
We now come to the final stage of the proof of Theorem
\ref{thm:main-thm}. We will work only with topological spaces and
complexes of sheaves on them. Thus, to simplify notation, we
identify a scheme with its variety of $\C$-points and use the same
symbol for both. The same applies to diagrams of schemes and
morphisms of diagrams of schemes. With this understood, let
$\vartheta^{tor}_{\Gamma \backslash
D}=(\varpi,\varsigma_r)^*\beta^{an}_{\overline{X}^{bb}}$.

It is clear that Theorem \ref{thm:main-thm} follows from Theorem
\ref{thm:final-form-comput-EE-X-bb} and the next proposition, the
proof of which is the subject of the rest of the article.

\begin{proposition}
\label{prop:last-prop-for-proving-main-thm}
Let $p:\overline{\Gamma \backslash D}^{rbs} \to \overline{\Gamma\backslash D}^{bb}$ be the quotient mapping. There is a canonical isomorphism of commutative unitary algebras
$$p_*\Q_{\overline{\Gamma \backslash D}^{rbs}} \simeq \Theta_*
\vartheta^{tor}_{\Gamma \backslash D}.$$
Moreover, the following diagram
$$\xymatrix@C=1.5pc@R=1.5pc{(g^{bb})^* p_*\Q_{\overline{\Gamma \backslash D}^{rbs}} \ar[r] \ar[d]_-{\sim} & p'_* (g^{rbs})^* \Q_{\overline{\Gamma \backslash D}^{rbs}} \ar[r]^-{\sim} & p'_* \Q_{\overline{\Gamma'\backslash D}^{rbs}} \ar[d]^-{\sim} \\
(g^{bb})^* \Theta_*\vartheta^{tor}_{\Gamma\backslash D} \ar[r] & \Theta'_* g^* \vartheta^{tor}_{\Gamma \backslash D} \ar[r] & \Theta'_* \vartheta^{tor}_{\Gamma'\backslash D} }$$
commutes.
\end{proposition}

The first step in the proof consists of bridging the gap between
the toroidal compactification and the Borel-Serre
compactifications. For this, we use the space
$\widehat{\overline{\Gamma \backslash D}}_{\Sigma}$ described in
\S \ref{subsection:toroidal-vs-rbs}. We need to introduce two
diagrams of topological spaces $\mathcal{W}^{bs}$ and
$\widehat{\mathcal{W}}$. These diagrams are, roughly, analogues for
$\overline{\Gamma\backslash D}^{bs}$ and
$\widehat{\overline{\Gamma \backslash D}}_{\Sigma}$ of what
$\mathcal{W}^{tor}$ was for the toroidal compactification
$\overline{X}^{tor}_{\Sigma}$. We present the details of their
construction.

For the construction of $\mathcal{W}^{bs}$, fix $(I_0,I_1)\in
\mathcal{P}_2([\![1,r]\!])$ and let $J=[\![0,r]\!]- I_0$ and
$\{0\}\bigsqcup I_1=\{i_0<\dots <i_s\}$. Let
$(\mathbf{Q},\mathbf{E})\in \mathscr{Q}_{\Gamma}(I_0,I_1)$ and, as
before, denote by $(\mathbf{Q_j},\mathbf{R_j})_{0\leq j \leq s}\in
\prod_{j=0}^s \mathcal{P}_{\Gamma}(J\cap [\![i_j, i_{j+1}]\!])$
(with $i_{s+1}=r$ as usual) its image by $\mathbf{d}(I_0,I_1)$.
Consider the $\mathbf{R_s}$-stratum $e(\mathbf{R_s})$ in the
partial Borel-Serre compactification
$\overline{e_h(\mathbf{Q_s})}^{bs}$. It admits an action of $\Gamma(\mathbf{E})$, and
we set:
$$\mathcal{W}^{bs}(\mathbf{Q},\mathbf{E})=
\Gamma(\mathbf{H_{Q,E}}) \backslash \overline{e(\mathbf{R_s})}.$$
Then $\Gamma$ acts naturally on
$\coprod_{(\mathbf{Q},\mathbf{E})\in \mathscr{Q}_{\Gamma}(I_0,I_1)}
\mathcal{W}^{bs}(\mathbf{Q},\mathbf{E})$.
An element $\gamma\in \Gamma$ takes $\mathcal{W}^{bs}(\mathbf{Q},\mathbf{E})$ isomorphically
onto
$\mathcal{W}^{bs}(\gamma\mathbf{Q}\gamma^{-1},\gamma\mathbf{E}\gamma^{-1})$.
Then $\Gamma(\mathbf{E})$ is the stabilizer in $\Gamma$ of
the connected component
$\mathcal{W}^{bs}(\mathbf{Q},\mathbf{E})$, and its action on
$\mathcal{W}^{bs}(\mathbf{Q},\mathbf{E})$ factors through ${\rm
end}_{\mathscr{Q}_{\Gamma}(I_0,I_1)}(\mathbf{Q},\mathbf{E})=\Gamma(\mathbf{E}/\mathbf{H_{Q,E}})$.
Thus, we get a diagram of topological spaces
$\mathcal{W}^{bs}(I_0,I_1)$ indexed by
$\mathscr{Q}_{\Gamma}(I_0,I_1)$. It is easy to see that the
assignment $(I_0,I_1)\rightsquigarrow \mathcal{W}^{bs}(I_0,I_1)$
defines a functor from $\mathcal{P}_2([\![1,r]\!])$ to
$\Dia({\rm Top})$.

The construction of $\widehat{\mathcal{W}}$ is parallel. Let
${^{\diamond}\widehat{B}}^{\circ}_{(\mathbf{Q_s},\mathbf{R_s}),\Sigma_{(\mathbf{Q_s})}}$
be the subspace of
$$[\Gamma(\mathbf{H_{Q,E}})\backslash e(\mathbf{R_s})]\times {^{\diamond}\mathcal{B}}^{\circ}_{(\mathbf{Q_s},\mathbf{R_s}),\Sigma_{(\mathbf{Q_s})}}$$
whose quotient by ${\rm
end}_{\mathscr{Q}_{\Gamma}(I_0,I_1)}(\mathbf{Q},\mathbf{E})$ is
the (corner-like) $\mathbf{R_s}$-stratum of
$$\widehat{\overline{{^{\diamond}\Gamma}(\mathbf{M}_{\mathbf{Q},h})\backslash e_h(\mathbf{Q_s})}}_{\Sigma_{(\mathbf{Q_s})}}.$$
We then define $\widehat{\mathcal{W}}(\mathbf{Q},\mathbf{E})$ to
be the closure in
$\mathcal{W}^{bs}(\mathbf{Q},\mathbf{E})\times\mathcal{W}^{tor}(\mathbf{Q},\mathbf{E})$
of
${^{\diamond}\widehat{B}}^{\circ}_{(\mathbf{Q_s},\mathbf{R_s}),\Sigma_{(\mathbf{Q_s})}}$.
The group $\Gamma$ acts on
$\coprod_{(\mathbf{Q},\mathbf{E})\in \mathscr{Q}_{\Gamma}(I_0,I_1)}
\widehat{\mathcal{W}}(\mathbf{Q},\mathbf{E})$,
and the stabilizer of the connected component
$\widehat{\mathcal{W}}(\mathbf{Q},\mathbf{E})$ is also
$\Gamma(\mathbf{E})$.  The action of the latter on
$\widehat{\mathcal{W}}(\mathbf{Q},\mathbf{E})$ factors through
${\rm end}_{\mathscr{Q}_{\Gamma}(I_0,I_1)}(\mathbf{Q},\mathbf{E})$.
Thus, we have a diagram of topological spaces
$\widehat{\mathcal{W}}(I_0,I_1)$ indexed by the groupoid
$\mathscr{Q}_{\Gamma}(I_0,I_1)$. Moreover, the assignment
$(I_0,I_1)\rightsquigarrow \widehat{\mathcal{W}}(I_0,I_1)$ gives a
functor from $\mathcal{P}_2([\![1,r]\!])$ to $\Dia({\rm Top})$.
By construction there are canonical morphisms in
$\Dia(\Dia({\rm Top}))$:
\begin{equation}
\label{eq:maps-between-3-diff-W-rajoute}
\xymatrix@C=1.7pc{\mathcal{W}^{bs}  & \widehat{\mathcal{W}} \ar[l]_-{p_1} \ar[r]^-{p_2} & \mathcal{W}^{tor},}
\end{equation}
which are the identity on the indexing categories (cf.
\cite[\S1.1]{rbs-2}). The argument in the proof of Lemma
\ref{lemma:two-proper-maps} shows that these morphisms are objectwise proper mappings. Indeed, over the object $(\mathbf{Q},\mathbf{E})\in \mathscr{Q}_{\Gamma}(I_0,I_1)$,
the arithmetic group in
\eqref{eq:where-end-Q-Gamma-appears-naturally-rajoute-0129}
acts properly discontinuously on the three topological spaces
in \eqref{eq:maps-between-3-diff-W-rajoute} and the induced maps on the quotients are proper.

Next, we construct complexes of sheaves of $\Q$-vector spaces
$\vartheta^{bs}_{\Gamma\backslash D}$ and
$\widehat{\vartheta}_{\Gamma\backslash D}$ on $\mathcal{W}^{bs}$
and $\widehat{\mathcal{W}}$ that are analogues of
$\vartheta^{tor}_{\Gamma \backslash D}$ on $\mathcal{W}^{tor}$.
Since we are now working in the setting of topological spaces and
complexes of sheaves, we can give a direct construction, as
follows. Fix a flasque resolution
$\digamma$ on topological spaces, that
is pseudo-monoidal and natural with respect to morphisms of
topological spaces as in
\S\ref{subsubsection:conclusion}.

Fix $(I_0,I_1)\in \mathcal{P}_2([\![1,r]\!])$ and let
$J=[\![0,r]\!]- I_0$ and $\{0\}\bigsqcup
I_1=\{i_0<\dots< i_s\}$.
Also, let
$K=\varsigma_r(I_0,I_1)=J\cap [\![i_s,r]\!]$ and write
$K=\{l_0<\dots<l_u\}$.
For $0\leq v \leq u$, we let $K_v=\{l_{v+1}<\dots< l_u\}$. (Note that
$K_{u}=\emptyset$ and $l_0\not \in K_v$.)
There is a chain of morphisms of diagrams
$$\mathcal{W}^{bs}(I_0\bigsqcup K_{u},I_1) \to \mathcal{W}^{bs}(I_0\bigsqcup K_{u-1},I_1) \to \dots \to
\mathcal{W}^{bs}(I_0\bigsqcup K_0,I_1),$$
and likewise for $\widehat{\mathcal{W}}$.
Now let $\mathcal{W}^{bs}(I_0\bigsqcup K_v,I_1)^{\circ}$
and $\widehat{\mathcal{W}}(I_0\bigsqcup K_v,I_1)^{\circ}$
denote the inverse images
of $X^{bb}_{l_v}$
in $\mathcal{W}^{bs}(I_0\bigsqcup K_v,I_1)$ and
$\widehat{\mathcal{W}}(I_0\bigsqcup K_v,I_1)$ respectively.
The inclusion
$$\mathcal{W}^{bs}(I_0\bigsqcup K_v,I_1)^{\circ}\hookrightarrow
\mathcal{W}^{bs}(I_0\bigsqcup K_v,I_1)$$
is an objectwise dense open immersion, and the same holds for $\widehat{\mathcal{W}}$.
With this notation we set
$\left(\vartheta^{bs}_{\Gamma \backslash D}\right)_{|\mathcal{W}^{bs}(I_0,I_1)}$ to be following complex of sheaves on $\mathcal{W}^{bs}(I_0,I_1)$:
$$\left([\mathcal{W}^{bs}(I_0\sqcup K_u,I_1)^{\circ} \!\to\! \mathcal{W}^{bs}(I_0\sqcup K_u,I_1)]_*
\digamma [\mathcal{W}^{bs}(I_0\sqcup K_u,I_1)^{\circ} \!\to\! \mathcal{W}^{bs}(I_0\sqcup K_{u-1},I_1)]^*\right)
$$
$$\cdots \left([\mathcal{W}^{bs}(I_0\sqcup K_1,I_1)^{\circ} \!\to\! \mathcal{W}^{bs}(I_0\sqcup K_1,I_1)]_*
\digamma [\mathcal{W}^{bs}(I_0\sqcup K_1,I_1)^{\circ} \!\to\! \mathcal{W}^{bs}(I_0\sqcup K_0,I_1)]^*\right)$$
$$
{\textstyle [\mathcal{W}^{bs}(I_0\sqcup K_0,I_1)^{\circ} \to \mathcal{W}^{bs}(I_0\sqcup K_0,I_1)]_* \digamma \Q_{\mathcal{W}^{bs}(I_0\sqcup K_0,I_1)^{\circ}}}.$$
We define $(\widehat{\vartheta}_{\Gamma \backslash
D})_{|\widehat{\mathcal{W}}(I_0,I_1)}$ analogously by replacing
everywhere the superscript ``$\emph{bs}$'' by a ``hat''. We leave
it to the reader to check that $(\vartheta^{bs}_{\Gamma \backslash
D})_{|\mathcal{W}^{bs}(I_0,I_1)}$ and
$(\widehat{\vartheta}_{\Gamma \backslash
D})_{|\widehat{\mathcal{W}}(I_0,I_1)}$, when $(I_0,I_1)$ varies,
define complexes of sheaves $\vartheta^{bs}_{\Gamma \backslash D}$
and $\widehat{\vartheta}_{\Gamma \backslash D}$ on
$\mathcal{W}^{bs}$ and $\widehat{\mathcal{W}}$ respectively.

\begin{remark}
\label{rem:direct-defn-of-vartheta-for-tor}
The analogue of the
above construction makes sense for $\mathcal{W}^{tor}$. That it
yields $\vartheta^{tor}_{\Gamma \backslash D}$ (up to a
canonical quasi-isomorphism) follows easily from
Lemma \ref{lemma:rajou-concrete-desc-beta-an-gen}
using Corollary \ref{cor:new-for-compar-beta-prime-et-theta-dp}.

\end{remark}

We let $\Theta^{bs}:\mathcal{W}^{bs} \to \overline{\Gamma \backslash D}^{bb}$
and let
$\widehat{\Theta}:\widehat{\mathcal{W}}\to \overline{\Gamma\backslash D}^{bb}$
denote the canonical morphisms.

\begin{lemma}
\label{lemma:bridge-toroidal-rbs-vartheta}
There is a canonical isomorphism of commutative unitary algebras:
\begin{equation}
\label{eq-lemma:bridge-toroidal-rbs-vartheta-1}
\Theta^{tor}_*\vartheta^{tor}_{\Gamma\backslash D} \simeq
\Theta^{bs}_*\vartheta^{bs}_{\Gamma\backslash D}\,.
\end{equation}
Moreover, the following diagram
$$\xymatrix@C=1.5pc@R=1.5pc{g^* \Theta^{tor}_*\vartheta^{tor}_{\Gamma\backslash D} \ar[r]\ar[d]_-{\sim} & \Theta'^{tor}_* g^* \vartheta^{tor}_{\Gamma\backslash D} \ar[r] & \Theta'^{tor}_* \vartheta^{tor}_{\Gamma'\backslash D} \ar[d]^-{\sim}  \\
g^* \Theta^{bs}_*\vartheta^{bs}_{\Gamma\backslash D} \ar[r] & \Theta'^{bs}_* g^* \vartheta^{bs}_{\Gamma\backslash D} \ar[r] & \Theta'^{bs}_* \vartheta^{bs}_{\Gamma'\backslash D}}$$
commutes.
\end{lemma}

\begin{proof}
As before, we construct only the isomorphism
\eqref{eq-lemma:bridge-toroidal-rbs-vartheta-1} after which
the commutation of the diagram follows. As we have the
commutative diagram
$$\xymatrix@C=1.5pc@R=1.5pc{\mathcal{W}^{bs} \ar@/_/[dr]_-{\Theta^{bs}} &\ar[l]_-{p_1} \widehat{\mathcal{W}} \ar[r]^-{p_2} \ar[d]^-{\widehat{\Theta}} & \mathcal{W}^{tor} \ar@/^/[dl]^-{\Theta^{tor}} \\
& \overline{\Gamma\backslash D}^{bb} & }$$
it suffices to construct isomorphisms of commutative unitary algebras
\begin{equation}
\label{eq-lemma:bridge-toroidal-rbs-vartheta-17}
\vartheta^{bs}_{\Gamma \backslash D} \simeq p_{1*}\widehat{\vartheta}_{\Gamma \backslash D} \qquad \text{and}\qquad
\vartheta^{tor}_{\Gamma\backslash D} \simeq p_{2*}\widehat{\vartheta}_{\Gamma \backslash D}.
\end{equation}
The construction is the same for both isomorphisms, and it relies on the fact that $p_1$ and $p_2$ are objectwise proper maps.
Thus, we will construct only $\vartheta^{bs}_{\Gamma\backslash D}\simeq p_{1*}\widehat{\vartheta}_{\Gamma \backslash D}$; using Remark
\ref{rem:direct-defn-of-vartheta-for-tor}, one repeats the construction to get $\vartheta^{tor}_{
\Gamma \backslash D}\simeq p_{2*}\widehat{\vartheta}_{\Gamma \backslash D}$.

Using the base change morphisms associated to
the commutative squares
\begin{equation}
\label{eq-lemma:bridge-toroidal-rbs-vartheta-29-tor}
\xymatrix@C=1.5pc@R=1.5pc{\widehat{\mathcal{W}}(I_0\sqcup K_v,I_1)^{\circ} \ar[r] \ar[d] & \widehat{\mathcal{W}}(I_0\sqcup K_{v-1},I_1) \ar[d] \\
\mathcal{W}^{bs}(I_0\sqcup K_v,I_1)^{\circ} \ar[r] & \mathcal{W}^{bs}(I_0\sqcup K_{v-1},I_1)}
\end{equation}
for $1\leq v\leq u$, we obtain a morphism
\begin{equation}
\label{eq-lemma:bridge-toroidal-rbs-vartheta-301-tor}
\xymatrix@C=1.7pc{(\vartheta^{bs}_{\Gamma \backslash D})_{|\mathcal{W}^{bs}(I_0,I_1)}\ar[r] &  p_2(I_0,I_1)_*
(\widehat{\vartheta}_{\Gamma \backslash D})_{|\widehat{\mathcal{W}}(I_0,I_1)}.}
\end{equation}
One easily checks that when $(I_0,I_1)$ varies, the morphisms
\eqref{eq-lemma:bridge-toroidal-rbs-vartheta-301-tor} form a
morphism $\vartheta^{bs}_{\Gamma \backslash D} \to p_{2*}
\widehat{\vartheta}_{\Gamma\backslash D}$ of complexes of sheaves
on $\mathcal{W}^{bs}$. We claim that
\eqref{eq-lemma:bridge-toroidal-rbs-vartheta-301-tor} is a
quasi-isomorphism. The vertical arrows in
\eqref{eq-lemma:bridge-toroidal-rbs-vartheta-29-tor} are objectwise
proper maps of topological spaces by Lemma
\ref{lemma:two-proper-maps}. Hence, by the topological base change
theorem for proper morphisms, the base change morphism associated
to \eqref{eq-lemma:bridge-toroidal-rbs-vartheta-29-tor} is
invertible. Our claim follows now as
$\mathcal{W}^{bs}(I_0\sqcup
K_0,I_1)^{\circ}=\widehat{\mathcal{W}}(I_0\sqcup
K_0,I_1)^{\circ}={^{\diamond}\!X}^{bb}_{\mathbf{Q_s}}$.
\end{proof}

Fix $(I_0,I_1)\in \mathcal{P}_2([\![1,r]\!])$ and
$(\mathbf{Q},\mathbf{E})\in \mathscr{Q}_{\Gamma}(I_0,I_1)$. Let
$J=[\![0,r]\!]- I_0$ and $\{0\}\bigsqcup I_1=\{i_0<\dots < i_s\}$.
Let $(\mathbf{Q_j},\mathbf{R_j})_{0\leq j \leq s}$ be the image of
$(\mathbf{Q},\mathbf{E})$ by $\mathbf{d}(I_0,I_1)$. Also write
$K=\{l_0<\dots < l_u\}$ and $K_v=\{l_{v+1},\dots, l_u\}$ for
$0\leq v \leq u$. Let $\mathbf{E}_{(v)}$ be the parabolic
$\Q$-subgroup of type $I_0\sqcup K_v$ containing $E$.
Then $\mathcal{W}^{bs}(\mathbf{Q},\mathbf{E}_{(0)})$ is the
Borel-Serre compactification of
${^{\diamond}\!X}^{bb}_{\mathbf{Q_s}}$, hence is a manifold with
corners. Moreover, for each $1\leq v \leq u$, the morphism
$\mathcal{W}^{bs}(\mathbf{Q},\mathbf{E}_{(v)})\to
\mathcal{W}^{bs}(\mathbf{Q},\mathbf{E}_{(0)})$ is locally
isomorphic to the inclusion of a stratum in the boundary. This
implies that
$$\left([\mathcal{W}^{bs}(\mathbf{Q},\mathbf{E}_{(u)})^{\circ} \to
\mathcal{W}^{bs}(\mathbf{Q},\mathbf{E}_{(u)})]_*\digamma
[\mathcal{W}^{bs}(\mathbf{Q},\mathbf{E}_{(u)})^{\circ} \to
\mathcal{W}^{bs}(\mathbf{Q},\mathbf{E}_{(u-1)})]^*\right)$$
$$\cdots \left([\mathcal{W}^{bs}(\mathbf{Q},\mathbf{E}_{(1)})^{\circ} \to
\mathcal{W}^{bs}(\mathbf{Q},\mathbf{E}_{(1)})]_*\digamma
[\mathcal{W}^{bs}(\mathbf{Q},\mathbf{E}_{(1)})^{\circ} \to
\mathcal{W}^{bs}(\mathbf{Q},\mathbf{E}_{(0)})]^*\right)$$
$$[\mathcal{W}^{bs}(\mathbf{Q},\mathbf{E}_{(0)})^{\circ} \to
\mathcal{W}^{bs}(\mathbf{Q},\mathbf{E}_{(0)})]_*\digamma
\Q_{\mathcal{W}^{bs}(\mathbf{Q},\mathbf{E}_{(0)})^{\circ}}$$ is
canonically quasi-isomorphic to
$\Q_{\mathcal{W}^{bs}(\mathbf{Q},\mathbf{E})}$. In other words,
there is a canonical quasi-isomorphism
$\vartheta^{bs}_{\Gamma\backslash D}\simeq \Q_{\mathcal{W}^{bs}}$.
Thus, it remains to show the following:

\begin{proposition}
\label{prop:remains}
There is a canonical isomorphism of commutative unitary
algebras
$$p_*\Q_{\overline{\Gamma\backslash D}^{rbs}}\simeq
\Theta^{bs}_*\Q_{\mathcal{W}^{bs}}.$$
Moreover, the following diagram
$$\xymatrix@C=1.5pc@R=1.5pc{g^*p_*\Q_{\overline{\Gamma\backslash D}^{rbs}}
\ar[r] \ar[d]_-{\sim} & p'_*g^*\Q_{\overline{\Gamma\backslash D}^{rbs}}
\ar[r]^-{\sim} & p'_*\Q_{\overline{\Gamma'\backslash D}^{rbs}} \ar[d]^-{\sim}\\
g^*\Theta^{bs}_*\Q_{\mathcal{W}^{bs}} \ar[r] & \Theta'^{bs}_*g^*
\Q_{\mathcal{W}^{bs}} \ar[r]^-{\sim} & \Theta'^{bs}_* \Q_{\mathcal{W}'^{bs}} }
$$
commutes.
\end{proposition}

To prove this proposition, we need to introduce a new diagram of
topological spaces $\mathcal{Z}^{bs}$. Let $(I_0,I_1)\in
\mathcal{P}_2([\![1,r]\!])$ and $(\mathbf{Q},\mathbf{E})\in
\mathscr{Q}_{\Gamma}(I_0,I_1)$. We denote by $\mathbf{F}$ the
image of $\mathbf{E}$ by the projection of $\mathbf{Q}$ to (the
quotient by a finite group of) $\mathbf{M_Q}$. Consider
$e(\mathbf{F})$, the $\mathbf{F}$-stratum in the Borel-Serre
partial compactification $\overline{\widehat{e}(\mathbf{Q})}^{bs}$
of $\widehat{e}(\mathbf{Q})$ (a stratum in the reductive
Borel-Serre partial compactification of $D$). We set
$$\mathcal{Z}^{bs}(\mathbf{Q},\mathbf{E})=\Gamma(\mathbf{H_{Q,E}})\backslash
\overline{e(\mathbf{F})}.$$
By construction, the action of
$\Gamma(\mathbf{E})$ on $\mathcal{Z}^{bs}(\mathbf{Q}, \mathbf{E})$ factors
through ${\rm
end}_{\mathscr{Q}_{\Gamma}(I_0,I_1)}(\mathbf{Q},\mathbf{E})$.
Thus, we have a diagram of topological spaces
$\mathcal{Z}^{bs}(I_0,I_1)$ indexed by
$\mathscr{Q}_{\Gamma}(I_0,I_1)$. Moreover, the assignment
$(I_0,I_1)\rightsquigarrow \mathcal{Z}^{bs}(I_0,I_1)$ defines a
functor $\mathcal{Z}^{bs}$ from $\mathcal{P}_2([\![1,r]\!])$ to
$\Dia({\rm Top})$.

The decomposition $\mathbf{M_Q}=\mathbf{M}_{\mathbf{Q},\ell}\times
\mathbf{M}_{\mathbf{Q},h}$ induces a decomposition
$\mathbf{F}=\mathbf{F_0}\times \mathbf{R_s}$. This gives a
decomposition $\overline{e(\mathbf{F})}\simeq
\overline{e(\mathbf{F_0})}\times \overline{e(\mathbf{R_s})}$.
Moreover, the action of $\Gamma(\mathbf{H_{Q,E}})$ respects this
decomposition and acts trivially on the first factor. Hence
$\mathcal{Z}^{bs}(\mathbf{Q},\mathbf{E})=\overline{e(\mathbf{F_0})}\times
\mathcal{W}^{bs}(\mathbf{Q},\mathbf{E})$.
The projection to the second factor yields a morphism
$\mathcal{Z}^{bs}(\mathbf{Q},\mathbf{E})
\to \mathcal{W}^{bs}(\mathbf{Q},\mathbf{E})$.
One immediately checks that these morphisms yield a morphism in $\Dia(\Dia({\rm Top}))$:
\begin{equation}
\label{eq:Z-bs-to-W-bs-prime}
z:\xymatrix@C=1.7pc{\mathcal{Z}^{bs} \ar[r] & \mathcal{W}^{bs}.}
\end{equation}
Now, note that $\overline{e(\mathbf{F_0})}$ is the closure of a
stratum in the Borel-Serre  partial compactification of
the symmetric space associated to
$\mathbf{M}_{\mathbf{Q},\ell}$. In particular,
$\overline{e(\mathbf{F_0})}$ is contractible and
\eqref{eq:Z-bs-to-W-bs-prime} is objectwise
a homotopy equivalence. We have proved
the following result.

\begin{lemma}
The canonical morphism $\Q_{\mathcal{W}^{bs}} \to z_*\Q_{\mathcal{Z}^{bs}}$ is invertible.
\end{lemma}

For every $(I_0,I_1)\in \mathcal{P}_2([\![1,r]\!])$, let
$\mathcal{U}^{bs}(I_0,I_1)$ be the quotient of
$\mathcal{Z}^{bs}(I_0,I_1)$ by the groupoid $\mathscr{Q}_{\Gamma}(I_0,I_1)$:
$$\mathcal{U}^{bs}(I_0,I_1) =
\mathscr{Q}_{\Gamma}(I_0,I_1)\backslash
\mathcal{Z}^{bs}(I_0,I_1).$$
Note that from the definitions we have:
\begin{equation}
\label{eqn:ubs-added-289}
\mathcal{U}^{bs}(\mathbf{Q},\mathbf{E})=
\Gamma(\mathbf{E})\backslash \overline{e(\mathbf{F})}
\end{equation}
where $\mathbf{F}$ is the
image of $\mathbf{E}$ by the projection of $\mathbf{Q}$ to (the
quotient by a finite group of) $\mathbf{M_Q}$.

We then have a diagram of topological
spaces $\mathcal{U}^{bs}$ indexed by $\mathcal{P}_2([\![1,r]\!])$
and a natural projection
\begin{equation}
\label{eq:Z-bs-to-U-bs-prime}
z':\mathcal{Z}^{bs} \!\xymatrix@C=1.7pc{\ar[r] &}\!
\mathcal{U}^{bs}.
\end{equation}
Note that for every $(\mathbf{Q},\mathbf{E})\in
\mathscr{Q}_{\Gamma}(I_0,I_1)$, the group ${\rm
end}_{\mathscr{Q}_{\Gamma}(I_0,I_1)}(\mathbf{Q},\mathbf{E})$ acts
properly discontinuously on the manifold with corners
$\mathcal{Z}^{bs}(\mathbf{Q},\mathbf{E})$. We obtain from this the
following:

\begin{lemma}
The canonical morphism $\Q_{\mathcal{U}^{bs}} \to z'_*\Q_{\mathcal{Z}^{bs}}$ is invertible.
\end{lemma}

There is a morphism of diagrams of topological spaces
\begin{equation}
\label{eq:U-bs-to-D-rbs-prime} u:\mathcal{U}^{bs}\to
\overline{\Gamma \backslash D}^{rbs},
\end{equation}
which sends $\mathcal{U}^{bs}(\mathbf{Q},\mathbf{E})=
\Gamma(\mathbf{E})\backslash \overline{e(\mathbf{F})}$ to
$\overline{\widehat{e}^{\, \prime}(\mathbf{E})}$.

\begin{lemma}
The canonical morphism
$\Q_{\overline{\Gamma\backslash D}^{rbs}} \to u_*\Q_{\mathcal{U}^{bs}}$ is invertible.
\end{lemma}

\begin{proof}
As $u$ is objectwise a proper mapping, this can be checked locally
over each stratum of $\overline{\Gamma\backslash D}^{rbs}$. Let
$\mathbf{P}$ be a parabolic subgroup of $\mathbf{G}$, and form the
cartesian square
$$\xymatrix@C=1.5pc@R=1.5pc{\mathcal{U}^{bs}_{\mathbf{P}} \ar[r] \ar[d]_-{u_{\mathbf{P}}} & \mathcal{U}^{bs} \ar[d]^-u \\
\widehat{e}^{\, \prime}(\mathbf{P}) \ar[r] & \overline{\Gamma
\backslash D}^{rbs}.}$$ We need to show that
$\Q_{\widehat{e}^{\,\prime}(\mathbf{P})} \to
(u_{\mathbf{P}})_*\Q_{\mathcal{U}^{bs}_{\mathbf{P}}}$ is
invertible.

Note that $\mathcal{U}^{bs}_{\mathbf{P}}(\mathbf{Q},\mathbf{E})$
is non-empty if and only if $\mathbf{E}$ is $\Gamma$-conjugate to
a parabolic $\Q$-subgroup containing $\mathbf{P}$. Let
$\mathcal{L}(\mathbf{P})$ be the set of pairs of parabolic
subgroups $(\mathbf{Q},\mathbf{E})$ such that $\mathbf{P}\subset
\mathbf{E} \subset \mathbf{Q}$. We endow $\mathcal{L}(\mathbf{P})$
with the order given by
$$(\mathbf{Q},\mathbf{E})\leq (\mathbf{Q'},\mathbf{E'})
\qquad \Leftrightarrow \qquad \mathbf{E}\subset \mathbf{E'}\subset \mathbf{Q'}\subset \mathbf{Q}.$$
We then have a fully faithful inclusion
$\mathcal{L}(\mathbf{P}) \hookrightarrow
\int_{\mathcal{P}_2([\![1,n]\!])} \mathscr{Q}_{\Gamma}$
sending $(\mathbf{Q},\mathbf{E})$ to
$((I_0,I_1),(\mathbf{Q},\mathbf{E}))$
where $I_0$ is the type of $\mathbf{E}$ and
$I_1$ is the cotype of $\mathbf{Q}$. Denote by
$\mathcal{U}_{\mathbf{P}}^{\flat}$ the restriction of
$\mathcal{U}^{bs}_{\mathbf{P}}$ to $\mathcal{L}(\mathbf{P})$ along this inclusion.
Also, let
$$u^{\flat}_{\mathbf{P}}:\xymatrix@C=1.7pc{(\mathcal{U}_{\mathbf{P}}^{\flat},
\mathcal{L}(\mathbf{P})) \ar[r] &
\widehat{e}^{\,\prime}(\mathbf{P})}$$
be the natural
projection. From the previous discussion, we deduce a canonical isomorphism
$(u_{\mathbf{P}})_*\Q_{\mathcal{U}^{bs}_{\mathbf{P}}} \simeq
(u^{\flat}_{\mathbf{P}})_*\Q_{\mathcal{U}^{\flat}_{\mathbf{P}}}$.
Thus, we are reduced to show that
$\Q_{\widehat{e}^{\,\prime}(\mathbf{P})} \to
(u^{\flat}_{\mathbf{P}})_*\Q_{\mathcal{U}^{\flat}_{\mathbf{P}}}$
is invertible.

Now, consider two elements
$(\mathbf{Q},\mathbf{E})$ and $(\mathbf{Q},\mathbf{E'})$
in $\mathcal{L}(\mathbf{P})$
with $\mathbf{E}\subset \mathbf{E'}$. Denote
by $\mathbf{F}$ and $\mathbf{F'}$ the images of
$\mathbf{E}$ and $\mathbf{E'}$
by the projection of $\mathbf{Q}$ to (the quotient by a finite group of) $\mathbf{M_Q}$.
Then,
$\Gamma(\mathbf{E})\backslash \overline{e(\mathbf{F})}$
and
$\Gamma(\mathbf{E'})\backslash \overline{e(\mathbf{F'})}$ are
the closures of the $\mathbf{F}$-stratum
and the $\mathbf{F'}$-stratum in the Borel-Serre compactification
of $\Gamma(\mathbf{Q})\backslash
\widehat{e}(\mathbf{Q})$. In particular, one has an isomorphism
$$\Gamma(\mathbf{E'})\backslash \overline{e(\mathbf{F'})} \times_{
\overline{\Gamma\backslash D}^{rbs}} \widehat{e}^{\,\prime}(\mathbf{P})
\simeq
\Gamma(\mathbf{E})\backslash \overline{e(\mathbf{F})} \times_{
\overline{\Gamma\backslash D}^{rbs}}
\widehat{e}^{\,\prime}(\mathbf{P}).$$
In fact, both sides can be identified with the stratum in
the Borel-Serre compactification
of $\Gamma(\mathbf{Q})\backslash
\widehat{e}(\mathbf{Q})$ corresponding to the image of
$\mathbf{P}$
by the projection of $\mathbf{Q}$ to (the quotient by a finite group of) $\mathbf{M_Q}$. In particular, we have shown that
the maps
$$\mathcal{U}^{\flat}_{\mathbf{P}}(\mathbf{Q},\mathbf{E})
\to \mathcal{U}^{\flat}_{\mathbf{P}}(\mathbf{Q},\mathbf{E'})$$
are isomorphisms (cf.~\eqref{eqn:ubs-added-289}). Thus, letting $i_{\mathbf{P}}:\mathcal{L}'(\mathbf{P})\hookrightarrow \mathcal{L}(\mathbf{P})$
be the inclusion of the ordered subset consisting of pairs of the form $(\mathbf{Q},\mathbf{P})$, one gets an isomorphism
$$(i_{\mathbf{P}})_*\Q_{\mathcal{U}_{\mathbf{P}}^{\flat} \circ i_{\mathbf{P}}} \simeq \Q_{\mathcal{U}_{\mathbf{P}}^{\flat}}.$$
(Use axiom \textbf{DerAlg 4'g} in \cite[Rem.~2.4.16]{ayoub-these-I}.)
Now, let
$$u'^{\,\flat}_{\mathbf{P}}:\xymatrix@C=1.7pc{(\mathcal{U}^{\flat}_{\mathbf{P}}\circ i_{\mathbf{P}},\mathcal{L}'(\mathbf{P})) \ar[r] &
\widehat{e}^{\,\prime}(\mathbf{P})}$$
be the natural projection. We are reduced to show that
$\Q_{\widehat{e}^{\,\prime}(\mathbf{P})} \to
(u'^{\,\flat}_{\mathbf{P}})_*\Q_{\mathcal{U}^{\flat}_{\mathbf{P}}\circ i_{\mathbf{P}} }$
is invertible.
But, $\mathcal{L}'(\mathbf{P})$ has a terminal
object, namely $(\mathbf{P},\mathbf{P})$. It follows that
$$(u'^{\,\flat}_{\mathbf{P}})_*\Q_{\mathcal{U}^{\flat}_{\mathbf{P}}\circ i_{\mathbf{P}}}
\simeq \{\mathcal{U}_{\mathbf{P}}(\mathbf{P},\mathbf{P})\to
\widehat {e}^{\,
\prime}(\mathbf{P})\}_*\Q_{\mathcal{U}_{\mathbf{P}}(\mathbf{P},\mathbf{P})}.$$
The lemma now follows, as
$\mathcal{U}_{\mathbf{P}}(\mathbf{P},\mathbf{P})= \widehat{e}^{\,
\prime}(\mathbf{P})$.
\end{proof}

Using the three lemmas above and the following commutative
diagram,
$$\xymatrix@C=1.5pc@R=1.5pc{& \mathcal{Z}^{bs} \ar[d]^-{z'} \ar@/_.4pc/[dl]_-{z} \\
\mathcal{W}^{bs} \ar[d]_-{\Theta^{bs}} & \mathcal{U}^{bs} \ar[d]^-{u} \\
\overline{\Gamma \backslash D}^{bb} & \overline{\Gamma \backslash D}^{rbs} \ar[l]_-{p}}$$
we can see that the proof of Proposition
\ref{prop:remains} is finished.
This completes the
proof of Proposition \ref{prop:last-prop-for-proving-main-thm} and
hence the proof of Theorem \ref{thm:main-thm}.

\end{document}